\documentclass[12pt,twoside,reqno]{amsart}

\usepackage{amsmath, amssymb, amsthm, enumitem, esint}

\usepackage[hyperindex]{hyperref}

\usepackage{hyperref}
\hypersetup{bookmarksdepth=3}

\newcounter{dummy}
\makeatletter
\newcommand\myitem[1][]{\item[#1]\refstepcounter{dummy}\def\@currentlabel{#1}}
\makeatother

\newcommand{\lto}{\longrightarrow}
\newcommand{\CONE}{\mathsf{Cone}}
\newcommand{\GenCONE}{\mathsf{GenCone}}
\newcommand{\CONEg}{\mathsf{Cone}_{R>0}}

\newcommand{\CONEgeq}{\mathsf{Cone}_{R\geq 0}}

\newcommand{\MMgrad}{\MM_{\grad}}
\newcommand{\MMg}{\MM_{R > 0}}
\newcommand{\MMgeq}{\MM_{R \geq 0}}
\newcommand{\MMggrad}{\MM_{\grad, R > 0}}
\newcommand{\MMgeqgrad}{\MM_{\grad,R \geq 0}}
\newcommand{\IR}{\mathbb{R}}
\newcommand{\lb}{\linebreak[1]}
\newcommand{\IC}{\mathbb{C}}

\newcommand{\IZ}{\mathbb{Z}}

\newcommand{\LL}{\mathcal{L}}

\newcommand{\MM}{\mathcal{M}}

\newcommand{\EE}{\mathcal{E}}

\newcommand{\WW}{\mathcal{W}}
\newcommand{\eps}{\varepsilon}
\newcommand{\la}{\lambda}

\newcommand{\ov}[1]{\overline{#1}}
\newcommand{\td}[1]{\widetilde{#1}}

\DeclareMathOperator{\Null}{null}

\DeclareMathOperator{\Image}{Image}
\DeclareMathOperator{\idx}{index}
\DeclareMathOperator{\grad}{grad}
\DeclareMathOperator{\Index}{index}

\DeclareMathOperator{\loc}{loc}

\DeclareMathOperator{\Int}{Int}
\DeclareMathOperator{\Ric}{Ric}
\DeclareMathOperator{\Rm}{Rm}

\DeclareMathOperator{\tr}{tr}
\DeclareMathOperator{\id}{id}

\DeclareMathOperator{\diam}{diam}
\DeclareMathOperator{\inj}{inj}
\DeclareMathOperator{\eucl}{eucl}
\DeclareMathOperator{\vol}{vol}

\DeclareMathOperator{\supp}{supp}

\DeclareMathOperator{\End}{End}

\DeclareMathOperator{\spann}{span}

\DeclareMathOperator{\rank}{rank}
\DeclareMathOperator{\DIV}{div}
\DeclareMathOperator{\proj}{proj}

\newcommand{\EMPTY}[1]{}

\newtheorem{Theorem}[equation]{Theorem}
\newtheorem{Lemma}[equation]{Lemma}
\newtheorem{Corollary}[equation]{Corollary}
\newtheorem{Proposition}[equation]{Proposition}
\newtheorem{Claim}[equation]{Claim}
\newtheorem{Question}[equation]{Question}
\newtheorem{Conjecture}[equation]{Conjecture}

\theoremstyle{definition}
\newtheorem{Definition}[equation]{Definition}
\theoremstyle{remark}
\newtheorem{Remark}[equation]{Remark}

\numberwithin{equation}{section}

\title[Degree theory for expanding solitons]{Degree theory for 4-dimensional asymptotically conical gradient expanding solitons}
\author{Richard H  Bamler and Eric Chen}
\address{Department of Mathematics, UC Berkeley, CA 94720, USA}
\email{rbamler@berkeley.edu}
\email{ecc@berkeley.edu}
\thanks{R.B. was supported by NSF grants DMS-1906500, DMS-2204364, E.C. was supported by NSF award DMS-3103392}
\date{\today}

\begin{document}

\begin{abstract}
We develop a new degree theory for 4-dimensional, asymptotically conical gradient expanding solitons.
Our theory implies the existence of gradient expanding solitons that are asymptotic to any given cone over $S^3$ with non-negative scalar curvature.
We also obtain a similar existence result for cones whose link is diffeomorphic to $S^3/\Gamma$ if we allow the expanding soliton to have orbifold singularities.

Our theory reveals the existence of a new topological invariant, called the \emph{expander degree,} applicable to a particular class of compact, smooth 4-orbifolds with boundary.
This invariant is roughly equal to a signed count of all possible gradient expanding solitons that can be defined on the interior of the orbifold and are asymptotic to any fixed cone metric with non-negative scalar curvature.
If the expander degree of an orbifold is non-zero, then gradient expanding solitons exist for any such cone metric.
We show that the expander degree of the 4-disk $D^4$ and any orbifold of the form $D^4/\Gamma$ equals $1$.
Additionally, we demonstrate that the expander degree of certain orbifolds, including exotic 4-disks, vanishes.

Our theory also sheds light on the relation between gradient and non-gradient expanding solitons with respect to their asymptotic model. 
More specifically, we show that among the set of asymptotically conical expanding solitons, the subset of those solitons that are \emph{gradient} forms a union of connected components.
\end{abstract}

\maketitle

\section{Introduction}
\subsection{Motivation and brief summary of the main results}
Ricci flows have had profound applications to the resolution of various topological conjectures in dimension 3.
In part, this was thanks to Perelman's surgery construction \cite{Perelman1, Perelman2}, which allowed the removal or resolution of singularities and the continuation of the flow past singular times.
It remains a central question in the field of Ricci flow whether a similar construction can be carried out in higher dimensions.
An affirmative answer to this question may lead to further interesting topological and geometric applications.

Recent work of the first author \cite{Bamler_HK_entropy_estimates,Bamler_RF_compactness,Bamler_HK_RF_partial_regularity} has revealed that in 4-dimensional Ricci flows, singularities may be described by cylindrical and conical singularity models -- at certain scales and in certain regions.
Notably, the \emph{conical} models represent a new type of model in dimension~4 (and higher); see also \cite{FIK_shrinker, Maximo_Ric_sing}.
So a new surgery procedure is required in order to continue a 4-dimensional Ricci flow beyond a conical singularity.
It has been proposed \cite{Gianniotis_Schulze_2018, Angenent_Knopf_2022} that asymptotically conical expanding solitons could be used to resolve such singularities.
This motivates the following question:

\begin{Question} \label{Q_main}
Given a 4-dimensional Riemannian cone with non-negative scalar curvature, is there a gradient expanding soliton with non-negative scalar curvature that is asymptotic to this cone?
\end{Question}

In this paper, we resolve Question~\ref{Q_main} affirmatively if the link of the given cone is diffeomorphic to a 3-sphere.
The topology of the expanding soliton will be diffeomorphic to $\IR^4$.
Our results will also resolve Question~\ref{Q_main} in the case where the link of the given cone is diffeomorphic to a spherical space form $S^3/\Gamma$.
In this case, the expanding soliton will be diffeomorphic to the orbifold $\IR^4/\Gamma$.
It is worth noting that orbifolds with isolated singularities are expected to naturally arise in the construction of 4-dimensional Ricci flow with surgery \cite{Simon_2020}, and thus the study of expanding solitons on orbifolds is highly relevant.

Our resolution of Question~\ref{Q_main} rests on a new degree theory for the space of asymptotically conical expanding solitons, which we believe is of independent interest.
The key concept of this theory is a new notion called the \emph{expander degree} $\deg_{\exp} (X) \in \IZ$, which is defined for a certain class of compact, 4-dimensional orbifolds with boundary.
This degree can be viewed as a signed count of all possible gradient expanding solitons on the interior of $X$ that are asymptotic (near $\partial X$) to a fixed cone metric with non-negative scalar curvature.
Therefore, if $\deg_{\exp} (X) \neq 0$, then gradient expanding solitons exist for any such prescribed asymptotic cone metric.
We compute the expander degree for $X \approx D^4$ or $D^4/\Gamma$ to be equal to $1$, which implies our aforementioned existence results.
We also show that if $\partial X \approx S^3$ and $\deg_{\exp}(X) \neq 0$, then $X$ must be diffeomorphic to a 4-disk.
In particular, this shows that the expander degree of any exotic 4-disk vanishes.

It is worth emphasizing that our degree theory differs from conventional degree theories in that it only holds in a certain localized sense.
Specifically, our theory is built upon the study of a much larger space which also includes \emph{non-gradient} expanding solitons that are asymptotic to a \emph{generalization} of a cone metric.
This is necessary due to the fact that the space of asymptotically conical \emph{gradient} expanding solitons, which is our main focus, has poor analytical properties, because it may not be a Banach manifold.
To overcome these analytical complications, we develop a degree theory for this larger space, which needs to be localized near the subset of \emph{gradient} expanding solitons.
As a byproduct of this theory, we show that the set of asymptotically conical \emph{gradient} expanding solitons that are asymptotic to a cone is a union of connected components of the space of asymptotically conical expanding solitons.

For a detailed discussion of our degree theory and the main results of this paper, we refer the reader to Subsection~\ref{subsec_statement_of_results}. 

\subsection{Historical context}
In recent years, there has been an increased interest in asymptotically conical expanding solitons, with some work focusing on their existence.
For instance, Schulze, Simon and Deruelle \cite{Schulze_Simon_2013, Deruelle} showed the existence of a gradient expanding soliton asymptotic to any cone metric with \emph{positive curvature operator,} in any dimension.
In this case, the soliton also has positive curvature operator and is unique up to diffeomorphisms within this class.
In addition, Conlon, Deruelle and Sun \cite{Conlon_Deruelle_2020, Conlon_Deruelle_Sun_2019}, in part building on earlier work of Siepmann \cite{Siepmann_thesis_2013}, established the existence and uniqueness of asymptotically conical \emph{K\"ahler} gradient expanding solitons on smooth canonical models of K\"ahler cones, again in any dimension.
This leads to a classification of all 4-dimensional, asymptotically conical, gradient expanding solitons in the K\"ahler case.

Deruelle's proof in \cite{Deruelle} used a continuity method that relied heavily on the positive curvature condition for two reasons.
First, it was necessary to establish the invertibility of the linearization of the soliton equation to guarantee the existence of nearby solitons corresponding to deformations of the cone metric.
Second, it was used to derive the gradient property for these nearby solitons, which was crucial for the analysis to continue.
By contrast, our degree theory does not rely on this curvature condition.
By focusing on the degree, we can remove the need for invertibility of the linearization, albeit at the expense of sacrificing uniqueness.
Furthermore, we develop a new continuity method that enables us to establish the preservation of the gradient condition within connected components of the space of asymptotically conical expanding solitons.
For further related work on expanding solitons see \cite{Chodosh_2014, Chen_Deruelle_2015,Deruelle_2015,Chodosh_Fong_2016, Deruelle_Lamm_2017,Deruelle_2017b,Deruelle_2017, Lott_Wilson_2017, Deruelle_Schulze_2021, Deruelle_Schulze_Simon_2022b,Lee_Topping_2022, Chan_2023, Cao_Liu_Xie_2023,Yudowitz}.

In mean curvature flow, expanders (the analogs of expanding solitons) have also received increased attention \cite{Bernstein_Wang_2018_degree_MCF, Bernstein_2020, Deruelle_Schulze_2020, Bernstein_Wang_2021, Bernstein_Wang_2021b, Bernstein_Wang_2022, Bernstein_Wang_2022b, Bernstein_Wang_2022c,Topping_Yin_2022}.
A useful feature of the expander equation is its equivalence to the minimal surface equation after conformally changing the Euclidean metric by the factor $e^{|x|^2/4}$.
This equivalence allows the mean curvature flow analog of Question~\ref{Q_main} to be approached using a minimization technique. 
However, expanding Ricci solitons do not enjoy a similar equivalence, rendering our degree technique seemingly the only viable approach for constructing expanding solitons with a prescribed asymptotic behavior.
We also refer to related work for expanders of the harmonic map heat flow in \cite{Deruelle_2019, Deruelle_Lamm_2021}.

Degree theories have found wide applications in the study of analysis, with Smale \cite{Smale_Sard_thm} laying the foundations for the $\IZ_2$-case and Elworthy and Tromba \cite{Elworthy_Tromba_1970} extending them to the integer case. 
Their significance in geometric analysis was first demonstrated for the minimal surface equation by Tomi, Tromba and White \cite{Tomi_Tromba_1978,Tromba_1985,White_1987, White_1989,White_1989b, White_1991}. 
Applications to other geometric equations are abundant. 
Of particular relevance to our work are the degree theories of Bernstein and Wang \cite{Bernstein_Wang_2021,Bernstein_Wang_2018_degree_MCF} for mean curvature flow expanders and by Anderson, Chang, Ge and Qing \cite{Anderson_2008,Chang_Ge_2020} for asymptotically hyperbolic Einstein manifolds.
In both cases, the authors show that the projection map in question is indeed a proper map between Banach manifolds and therefore it has a well defined degree.
Unfortunately, in our setting, the set of gradient expanding solitons only forms a possibly singular subvariety of a larger Banach manifold, on which the projection map may not be proper. 
To address this issue, we develop a new degree theory for the map on the larger Banach manifold, which only holds in a localized sense near the subvariety of gradient expanding solitons.
This localization poses particular challenges in the definition of the \emph{integer} degree, as the differential of the projection map may be degenerate everywhere along this subvariety.
For more details see Subsection~\ref{subsec_Overview}.

\subsection{Description of the degree theory and statement of the results}\label{subsec_statement_of_results}
Let us now provide a more detailed definition of the expander degree. 
For the sake of clarity and brevity, some technical aspects, such as regularity discussions, will be suppressed.
It is also worth noting that the expander degree can be defined as either  an integer or its reduction modulo 2.
The definition of the $\IZ_2$-degree is easier and sufficient for many purposes, for example for resolving Question~\ref{Q_main} in the case in which the expanding soliton is diffeomorphic to $\IR^4/\Gamma$.

Fix a compact, smooth 4-dimensional orbifold $X$ with isolated singularities and non-empty boundary $\partial X$ that satisfies the following two topological assumptions:
\begin{enumerate}[label=(\arabic*)]
\item \label{Property_1} The boundary $\partial X$ consists of regular points and admits a metric of positive scalar curvature.
So its components are diffeomorphic to connected sums of spherical space forms and copies of $S^2 \times S^1$, after possibly passing to orientable double covers.
\item \label{Property_2} $X$ admits a (possibly infinite) orbifold cover $\hat X \to X$, such that $\hat X$ is a manifold whose integer-valued  first and second homology groups satisfy the following properties:
\[ H_2(\hat X;\IZ) = 0, \qquad \text{$H_1(\hat X;\IZ)$ is torsion free} \]
\end{enumerate}
An interesting class of examples of such orbifolds is given by $X = D^4 / \Gamma$, where $D^4 \subset \IR^4$ is the closed unit 4-disk and $\Gamma \subset SO(4)$ is a finite subgroup that acts freely on the unit 3-sphere $S^3 = \partial D^4$.

Fix an integer or infinity $30 \leq k^* \leq \infty$.
We now define\footnote{See Definitions~\ref{Def_MM} and \ref{Def_MMX} for more details. We also remark that in the case $k^*=\infty$, we can take $g$ and $f$ to be smooth.} $\MMgrad^{k^*}(X)$ to be the space of isometry classes of complete, gradient expanding solitons on the interior $\Int X$ of $X$ that are asymptotic to a $C^{k^*}$-regular conical metric on a tubular neighborhood of $\partial X$.
More specifically, these classes are represented by triples of the form $(g, \nabla f,\gamma)$.
Here $g$ and $\nabla f$ denote a Riemannian metric and the gradient of a potential function $f$ on $\Int X$ that satisfy the gradient expanding soliton equation
\[ \Ric_g + \nabla^2 f + \tfrac12 g = 0. \]
Moreover, $\gamma$ is a $C^{k^*}$-regular Riemannian cone metric of the form $\gamma = dr^2 + r^2 h$ defined on $\IR_+ \times \partial X$, for some Riemannian metric $h$ on $\partial X$, which describes an asymptotically conical behavior of $g$ and $\nabla f$ at infinity. 
To make this last requirement more precise, we identify a tubular neighborhood $U \subset X$ of $\partial X$ with $(1,\infty] \times \partial X$ and require\footnote{Note that it is necessary for the first asymptotic condition in \eqref{eq_gdrhO} to hold for higher derivatives as well. Furthermore, we can ensure that the second condition in \eqref{eq_gdrhO} is satisfied by pullback through a suitable diffeomorphism, assuming the first condition is already met. For more details see Definition~\ref{Def_MM} and Lemma~\ref{Lem_iota}.} that for some $r_0 > 0$
\begin{equation} \label{eq_gdrhO}
 g = \gamma + O(r^{-2}), \qquad\qquad \nabla^g f= - \tfrac12 r \partial_r \qquad \text{on} \quad (r_0, \infty) \times \partial X,
\end{equation}
where $r$, $\partial_r$ denotes the first coordinate and coordinate vector field on on $U \approx (1,\infty] \times \partial X$.
Two triples $(g_1, \nabla f_1,\gamma_1)$ and $(g_2, \nabla f_2,\gamma_2)$ are said to represent the same class in $\MMgrad^{k^*}(X)$ if $g_2 = \phi^* g_1$ and $\nabla^{g_2} f_2 = \phi^*( \nabla^{g_1} f_1 )$ for a diffeomorphism $\phi : X \to X$ that fixes the boundary\footnote{Thanks to the second condition in \eqref{eq_gdrhO} such a diffeomorphism even fixes a neighborhood of the boundary of $\partial X$.}; this condition also implies that $\gamma_1=\gamma_2$.

Denote\footnote{Compare with Definition~\ref{Def_GenCONE}.} by $\CONE^{k^*}(\partial X)$ the set of all $C^{k^*}$-regular cone metrics $\gamma = dr^2 + r^2 h$ on $\IR_+ \times \partial X$ and let
\[ \Pi : \MMgrad^{k^*} (X) \longrightarrow \CONE^{k^*}(\partial X), \]
be the map that assigns to any gradient expanding soliton $[(g,\nabla f,\gamma)]  \in \MMgrad^{k^*}(X)$ its asymptotic cone metric $\gamma$.
We will denote by $$\MMgeqgrad^{k^*}(X) \subset \MMgrad^{k^*}(X), \qquad \CONEgeq^{k^*}(\partial X) \subset \CONE^\infty(\partial X)$$ the subsets of solitons and cone metrics, respectively, with non-negative scalar curvature.
Note that by restricting $\Pi$ to $\MMgeqgrad^{k^*}(X)$, we obtain a map of the form
\begin{equation} \label{eq_intro_Pi_restr}
\Pi \big|_{\MMgeqgrad^{k^*}(X)} \; : \; \MMgeqgrad^{k^*}(X)  \lto \CONEgeq^{k^*}(\partial X).
\end{equation}

We can now state our main result in a vague form:

\begin{Theorem} \label{Thm_main_vague}
The map \eqref{eq_intro_Pi_restr} is proper in a suitable topology.
In a certain generalized sense it has a well defined, integer-valued degree, called the \emph{expander degree,}
\[  \deg_{\exp} (X)  \in \IZ. \]
This degree is an invariant of the smooth structure of $X$.
For any fixed cone metric $\gamma \in\CONEgeq^{k^*}(\partial X)$ it can be solely determined from the local behavior of a certain extension of $\Pi$ near $\gamma$.
In particular, if $\deg_{\exp} (X)  \neq 0$, then the map \eqref{eq_intro_Pi_restr} is surjective.
\end{Theorem}

We will also compute the expander degree in the most elementary case.

\begin{Theorem} \label{Thm_deg_1}
If $\Gamma \subset SO(4)$ is a finite subgroup acting freely on $S^3$, then
\[  \deg_{\exp} (D^4/\Gamma) = 1. \]
In particular, $\deg_{\exp} (D^4) = 1$.
\end{Theorem}

Combining Theorems~\ref{Thm_main_vague}, \ref{Thm_deg_1} yields the following corollary:

\begin{Corollary}
Suppose that $\Gamma \subset SO(4)$ is a finite subgroup that acts freely on $S^3$.
Then for any $C^{k^*}$-regular cone metric $\gamma = dr^2 + r^2 h$ on $\IR_+ \times S^3/\Gamma$ with non-negative scalar curvature there is a gradient expanding soliton metric $g$ on $\IR^4/\Gamma$ with non-negative scalar curvature that is asymptotic to $\gamma$.
\end{Corollary}

One of the main complications in the definition of the expander degree is that the space  $\MMgrad^{k^*} (X)$ 
may not be a Banach manifold.
As a result, the usual method of assigning a degree to the map $\Pi$ or its restriction \eqref{eq_intro_Pi_restr} is not applicable.
To resolve this issue, we consider extensions of the spaces $\CONE^{k^*}(\partial X)$ and $\MMgrad^{k^*}(X)$.
More specifically, we define\footnote{Compare again with Definition~\ref{Def_GenCONE}.} $\GenCONE^{k^*} (\partial X) \supset \CONE^{k^*}(\partial X)$ as the space of \emph{generalized cone metrics} on $\IR_+ \times \partial X$, which are of the form
\begin{equation} \label{eq_gen_cone_intro}
 \gamma = dr^2 + r \, (dr \otimes \beta + \beta \otimes dr) + r^2 h, 
\end{equation}
for some 1-form $\beta$ and Riemannian metric $h$ on $\partial X$ of regularity $C^{k^*}$.
Furthermore, we define\footnote{See Definitions~\ref{Def_MM} and \ref{Def_MMX} for further details.} $\MM^{k^*}(X) \supset \MMgrad^{k^*}(X)$ to be the space of (not necessarily gradient) expanding solitons on $\Int X$ that are asymptotic to a generalized cone metric.
Similar to the definition of $\MMgrad^{k^*}(X)$, each element of $\MM^{k^*} (X)$ is represented by a triple $(g,V,\gamma)$, consisting of a Riemannian metric $g$, a vector field $V$, and a generalized cone metric $\gamma \in \GenCONE^{k^*}(\partial X)$.
Instead of the gradient soliton equation, we require the general soliton equation
\[ \Ric + \tfrac12 \LL_V g + \tfrac12 g = 0 \]
and we require a similar asymptotic assumption as in \eqref{eq_gdrhO}.
Note that we still have a projection of the form
\[ \Pi : \MM^{k^*} (X) \lto \GenCONE^{k^*} (\partial X), \]
which we will denote by the same letter.

Suppose that $k^* < \infty$ for the remainder of this discussion.
We can now state a more precise version of the second part of Theorem~\ref{Thm_main_vague}.

\begin{Theorem} \label{Thm_Banach_mf_intro}
There is an open neighborhood
\[ \MMgrad^{k^*} (X) \subset \MM' \subset \MM^{k^*} (X) \]
that can be equipped with a $C^{1,\alpha}$-Banach manifold structure such that the map $\Pi |_{\MM'}$ is smooth (even real-analytic) in appropriate coordinate charts.
The map $\Pi |_{\MM'}$ may not be proper, but it is still possible to compute a local degree of $\Pi |_{\MM'}$ by considering its local behavior near the preimage of any fixed $\gamma \in \CONEgeq^{k^*}(\partial X)$.
This degree, called the \emph{expander degree} $\deg_{\exp}(X) \in \IZ$, is independent of the choice of $\gamma$, so it is a topological invariant of $X$.
\end{Theorem}

The next theorem provides a formula for the expander degree over regular values of $\Pi$.
To make sense of this formula, we need to consider the Einstein operator
\[ L_{p} = \triangle_f + 2 \Rm_g = \triangle - \nabla_{\nabla f} + 2 \Rm_g, \]
associated with an element $p = [(g,\nabla f,\gamma)] \in \MMgeqgrad^{k^*}(X)$, which acts on symmetric $(0,2)$-tensors on $\Int X$.
In Section~\ref{sec_elliptic} we will see that this operator has well defined and finite nullity $\Null (-L_p)$ and index $\idx(-L_p)$, which are independent of the choice of the representative of $p$.
See Subsections~\ref{subsec_Lunardi_weighted}, \ref{subsec_spectral_theory} for more details.

\begin{Theorem} \label{Thm_degexp_identity}
Suppose now that $\gamma$ is a regular value of $\Pi$ over $\MMgeqgrad^{k^*}(X)$, by which we mean that for any $p  \in \MMgeqgrad^{k^*}(X) \cap \Pi^{-1}(\gamma)$ we have $\Null(-L_p)=0$.
Then
\begin{equation} \label{eq_degexp_identity_intro}
 \deg_{\exp}(X) = \sum_{p \in \MMgeqgrad^{k^*}(X) \cap \Pi^{-1}(\gamma)} (-1)^{\idx(-L_p)}. 
\end{equation}
For the $\IZ_2$-degree this formula simplifies to
\[ \deg_{\exp}(X) = \# \big( \MMgeqgrad^{k^*}(X) \cap \Pi^{-1}(\gamma) \big) \qquad \textnormal{mod} \; 2. \]
\end{Theorem}

It is important to note that, in contrast to standard degree theories, the existence of a regular value of $\Pi$ over $\MMgeqgrad^{k^*}(X)$ is not guaranteed, as the differential of the map $\Pi$ may be degenerate everywhere along $\MMgeqgrad^{k^*}(X) \subset \MM^{k^*}(X)$.
Consequently, the identity \eqref{eq_degexp_identity_intro} cannot serve as the primary definition of the expander degree; rather it is a consequence of it.
Nevertheless, Theorem~\ref{Thm_degexp_identity} shows that if $\deg_{\exp}(X) \neq 0$, then the map \eqref{eq_intro_Pi_restr} is surjective.
This follows from the fact that if $\MMgeqgrad^{k^*}(X) \cap \Pi^{-1}(\gamma) = \emptyset$, then $\gamma$ must be a regular value of $\Pi$ over $\MMgeqgrad^{k^*}(X)$ and thus $\deg_{\exp}(X) = 0$ by \eqref{eq_degexp_identity_intro}.

Our degree theory ensures the existence of gradient expanding solitons corresponding to many deformations of the asymptotic cone within the space of generalized cone metrics.
The following theorem expresses this existence statement by characterizing the image of $\Pi$. 

\begin{Theorem} \label{Thm_many_deformations}
Let $\gamma \in \CONE^{k^*}(\partial X)$.
Then the following is true:
\begin{enumerate}[label=(\alph*)]
\item \label{Thm_many_deformations_a}
If $\gamma \in \Image \Pi(\MMgrad^{k^*}(X))$, then the following is true. 
There is a real-analytic Banach-submanifold $\Gamma \lb \subset \lb \GenCONE^{k^*}(\partial X)$ of finite codimension such that
\[ \gamma \in \Gamma  \subset \Image \Pi. \]
\item \label{Thm_many_deformations_b}
If $\deg_{\exp}(X) \neq 0$ and $\gamma \in \CONEgeq^{k^*}(\partial X)$, then the following is true.
Let $Q \subset \GenCONE^{k^*}(\partial X)$ be a finite dimensional, real-analytic submanifold (for example, a subset of an affine linear subspace) such that $\gamma \in Q$.
Then $Q \cap \Image \Pi$ contains a neighborhood of $\gamma$ in $Q$.
\end{enumerate}
\end{Theorem}

Let us now turn our attention to the gradient property and the relationship between the spaces $\MM^{k^*}(X)$ and $\MMgrad^{k^*}(X)$.
A simple argument (see Remark~\ref{Rmk_gradient_implies_cone}) implies that the gradient condition necessitates the asymptotic model to be a cone, rather than a generalized cone.
In other words, if $(g,V,\gamma)$ is a representative of an element in $\MM^{k^*}(X)$ for which $V$ is a gradient vector field, then $\gamma \in \CONE^{k^*}(\partial X)$, i.e., $\beta \equiv 0$ in \eqref{eq_gen_cone_intro}.
So Theorem~\ref{Thm_many_deformations} implies the existence of many expanding solitons near $\gamma$ that are not gradient.
In the special case of expanding solitons asymptotic to cones near the Euclidean cone metric, this also follows from \cite{Koch_Lamm}.
Conversely, one may ask:

\begin{Question} \label{Q_cone_implies_gradient}
Does $\gamma \in \CONE^{k^*}(\partial X)$ imply that $V$ is a gradient vector field and thus $[(g,V,\gamma)] \in \MMgrad^{k^*}(X)$?
In other words, is $\Pi^{-1}(\CONE^{k^*}(\partial X))$ the same as $\MMgrad^{k^*}(X)$?
\end{Question}

In \cite{Deruelle}, Deruelle provided a positive answer to Question~\ref{Q_cone_implies_gradient} under the additional assumption that $g$ has positive curvature operator.
However, Question~\ref{Q_cone_implies_gradient} remains open in its general form.
In this paper, we make partial progress towards this question by showing that $\MMgrad^{k^*}(X)$ is a union of $C^1$-connected components of $\Pi^{-1}(\CONE^{k^*}(\partial X)) \subset \MM^{k^*}(X)$.
Specifically, we prove:

\begin{Theorem} \label{Thm_preservation_gradient_intro}
Let $\sigma: [0,1] \to \MM^{k^*} (X)$ be a curve that is of regularity $C^1$ in a neighborhood of $\MMgrad^{k^*}(X)$ (with respect to the Banach-manifold structure from Theorem~\ref{Thm_Banach_mf_intro}).
Suppose that $\Pi(\sigma(s)) \in \CONE^{k^*}(\partial X)$ for all $s \in [0,1]$ and that $\sigma(0) \in \MMgrad^{k^*}(X)$.
Then $\sigma(s) \in \MMgrad^{k^*}(X)$ for all $s \in [0,1]$.
\end{Theorem}

Our degree theory shifts the focus from understanding individual asymptotically conical expanding solitons to the question of determining the expander degree of a given orbifold.
Moreover, it provides a new invariant for the smooth structure, which in contrast to other such invariants, holds under vanishing of the second homology group (and even requires it).
The following theorem demonstrates that this invariant can be used to detect standard smooth structures.

\begin{Theorem}\label{Thm_standard_disk}
Let $X$ be as described in the beginning of this subsection and suppose that $\partial X \approx S^3$.
Then $\deg_{\exp}(X) = 1$ if and only if $X$ is diffeomorphic to the standard disk $D^4$. Otherwise  $\deg_{\exp}(X) = 0$.

More generally, if $X$ has a finite orbifold cover $\hat X$ with $\partial \hat X \approx S^3$, then  $\deg_{\exp}(X) = 1$ if and only if $X$ is diffeomorphic to the standard $D^4/\Gamma$ for some $\Gamma \subset SO(4)$. Otherwise $\deg_{\exp}(X) = 0$.
\end{Theorem}

\bigskip

\subsection{Open Questions}
It is an interesting question to compute the expander degree for more complicated orbifolds $X$.
The following question may be within closest reach:

\begin{Question}
Compute $\deg_{\exp} (S^1 \times D^3)$ and $\deg_{\exp} ((S^1 \times D^3)/\IZ_2)$.
\end{Question}

Here the $\IZ_2$-quotient should be taken with respect to the antipodal maps on each factor.
From a philosophical perspective, it would be convenient if $\deg_{\exp} (S^1 \times D^3) \neq 0$.
This may allow the resolution of a conical singularity in Ricci flow with link $\approx S^1 \times S^2$ without increasing the second Betti number.

In order to compute the expander degree of even more complicated orbifolds, it may be necessary to investigate the behavior of the expander degree under connected sums.
To be more precise, consider two orbifolds $X_1, X_2$ that satisfy the properties listed in the beginning of Subsection~\ref{subsec_statement_of_results}.
We define $X_1 \#_{\partial} X_2$ to be their connected sum at the boundary, i.e., the result of removing a copy of a half 4-ball $D^4_+$ from a tubular neighborhood of the boundary of each orbifold and gluing in a copy of $S^3_+ \times I$, where $S^3_+$ denotes the upper hemisphere.
Note that $\partial (X_1 \#_{\partial} X_2) = \partial X_1 \# \partial X_2$.

\begin{Question}\label{Q_relate_degX1X2}
Relate $\deg_{\exp} (X_1 \#_{\partial} X_2)$ with $\deg_{\exp}(X_1)$ and $\deg_{\exp}(X_2)$.
\end{Question}

An answer to Question~\ref{Q_relate_degX1X2} may allow us to find an orbifold $X$ such that $\partial X$ is diffeomorphic to any given closed, orientable 3-manifold.
Recall that any such 3-manifold must be diffeomorphic to a connected sum of the form
\[ (S^3/\Gamma_1) \#_{\partial} \ldots \#_{\partial} (S^3/\Gamma_k) \#_{\partial} (S^1 \times S^2)\#_{\partial} \ldots \#_{\partial} (S^1 \times S^2). \]
This may lead to the resolution of the following conjecture:

\begin{Conjecture}
For any orientable Riemannian cone metric $\gamma = dr^2 + r^2 h$ with non-negative scalar curvature there is a gradient expanding soliton with non-negative scalar curvature that is asymptotic to $\gamma$ and diffeomorphic to the interior of a connected sum of the form
\[ (D^4/\Gamma_1) \#_{\partial} \ldots \#_{\partial} (D^4/\Gamma_k) \#_{\partial} (S^1 \times D^3)\#_{\partial} \ldots \#_{\partial} (S^1 \times D^3). \]
\end{Conjecture}

A resolution of this conjecture would affirmatively answer Question~\ref{Q_main} and is expected to have direct applications in the construction of a 4-dimensional Ricci flow with surgery.
\medskip

Second, the role of the degree theory in showing the existence of expanding solitons with prescribed conical asymptotics raises the question of uniqueness.
More specifically, one may ask:

\begin{Question} \label{Q_non_unique}
Is there an example of two non-isometric gradient expanding solitons with non-negative scalar curvature, defined on the same orbifold $\Int X$, that are asymptotic to the same conical metric $\gamma \in \CONEgeq^\infty(\partial X)$?
\end{Question}

If we do not require that both expanding solitons are defined on the same orbifold, then Question~\ref{Q_non_unique} has a positive answer: the flat cone $\mathbb{C}^2/\mathbb{Z}_k$ for $k\geq 3$ is the asymptotic cone of both a smooth expanding soliton on a resolution of $\mathbb{C}^2/\mathbb{Z}_k$ by \cite{Conlon_Deruelle_Sun_2019}, and also the asymptotic cone of the flat orbifold expanding soliton on $\mathbb{C}^2/\mathbb{Z}_k$ itself.

Question~\ref{Q_non_unique} also has a positive answer in dimensions 5 and higher due to examples of \cite{Angenent_Knopf_2022}, in which the asymptotic cone of a gradient \emph{shrinking} soliton can be resolved both by expanding solitons on different smooth manifolds as well as non-isometric expanding solitons on the same smooth manifold. If there are also examples for Question~\ref{Q_non_unique} in dimension 4, this would show that the resolution of conical singularities in 4-dimensional Ricci flows may not be unique, and suggest that an analogous version of the uniqueness theorem from \cite{Bamler_Kleiner_uniqueness} may be false in dimension 4.

%An example for Question~\ref{Q_non_unique} for which $\gamma$ is also the asymptotic cone of a gradient \emph{shrinking} soliton would allow us to show, as in \cite{Angenent_Knopf_2022} for dimensions $5$ and higher, that the resolution of conical singularities in 4-dimensional Ricci flows may not be unique.
%This would imply that an analogous version of the uniqueness theorem from \cite{Bamler_Kleiner_uniqueness} may be false in dimension 4.

\medskip
Third, the role of the non-negative scalar curvature condition in our theory seems to be nebulous.
In fact, this condition is only used once in this paper: in the proof of the properness of the restricted projection map \eqref{eq_intro_Pi_restr} (see Lemma~\ref{Lem_nu_cone_exp} and Remark~\ref{Rmk_scal_bound_necessary}).
An alternative argument may allow us to remove the scalar curvature condition and affirmatively answer the following question.

\begin{Question}
Does $\deg_{\exp}(X) \neq 0$ imply surjectivity of the map
\[ 
\Pi \big|_{\MMgrad^\infty(X)} \; : \; \MMgrad^\infty(X)  \lto \CONE^\infty(\partial X). \]
\end{Question}

Lemma~\ref{Lem_some_lower_nu_bound} in Section~\ref{sec_properness} motivates the expectation that the surjectivity can at least be extended to all cones $\gamma = dr^2 + r^2 h$ whose link metric satisfies the bound $\lambda[h] \geq 2$, where the left-hand side denotes Perelman's $\lambda$-functional.
\medskip

Lastly, we remark that Question~\ref{Q_cone_implies_gradient} from the previous subsection is still open.

\subsection{Overview of the proof} \label{subsec_Overview}
Our proof draws inspiration from the deformation theory of expanding soliton by Deruelle \cite{Deruelle} and the degree theories for minimal surfaces by White \cite{White_1987} and conformally compact Einstein metrics by Anderson \cite{Anderson_2008}.
To this end, we adapt some of the analytical tools for elliptic operators with unbounded coefficients by Deruelle \cite{Deruelle_2015, Deruelle}, which builds on earlier work of Lunardi \cite{Lunardi_1998}.
In addition, the structure of several arguments is, to some extent, informed by the approach in \cite{White_1987}.
Furthermore, we employ recent asymptotic kernel estimates due to Deruelle and Schulze \cite{Deruelle_Schulze_2021}.

One of the main contributions of our paper is the development of a novel localized degree theory for improper maps.
Such a theory is necessary since the standard degree theories do not seem to be applicable to our setting.
This is due in part to the fact that the space of \emph{gradient} expanding solitons may not form a Banach manifold over the space of cones, but instead a possibly singular subvariety within the larger space of (possibly non-gradient) expanding solitons over generalized cones.
Consequently, the projection map defined on neither of these spaces satisfies the necessary conditions for a degree theory, as its domain may either not be a Banach manifold or the map itself may be improper.
This presents particular challenges in defining the \emph{integer} degree, as we cannot rely on the existence of regular values within the set of cones.

In addition, we address several specific issues that arise from the expanding soliton equation:
\begin{itemize}
\item We develop a framework to address the diffeomorphism invariance of the soliton equation and the effects of switching between different gauges while ensuring adequate control over the asymptotics at infinity.
\item We describe the relationship between the gradient condition and the asymptotic cone condition.
To achieve this, we present a new continuity method, which allows us to show the preservation of the gradient condition along $C^1$-families of asymptotically conical expanding solitons.
\item We establish a compactness theory for asymptotically conical gradient expanding solitons, which is based on the geometry of the asymptotic cone.
\end{itemize}
\medskip

Let us now provide a more detailed description of our proof strategy. 
We begin by addressing the diffeomorphism invariance of the soliton equation, which we eliminate by working in two different types of gauges.

The first gauge condition, which we will refer to as the \emph{gauge at infinity} for the purposes of this discussion, is used to define\footnote{For technical reasons, we actually use a slightly different topological setup in the definition of $\MM^{k^*}(X)$.
Specifically, we replace $X$ with an ensemble $(\Int X, \partial X, \iota)$, where $\iota$ describes the identification of a neighborhood of $\partial X$ with $(1,\infty] \times \partial X$. See Definition~\ref{Def_ensemble} for more details.} the space $\MM^{k^*}(X)$.
Given an orbifold $X$ satisfying the conditions outlined in the beginning of Subsection~\ref{subsec_statement_of_results} and an identification of a tubular neighborhood of $\partial X$ with $(1,\infty] \times \partial X$, we say that a metric $g$ and vector field $V$ on $\Int X$ satisfying the expanding soliton equation
\[ \Ric + \tfrac12 \LL_V g + \tfrac12 g = 0 \]
are gauged at infinity if there exists a generalized cone metric $\gamma \in \GenCONE^{k^*}(\partial X)$ such that for some $r_0 > 1$ (compare also with \eqref{eq_gdrhO})
\[ g = \gamma + O(r^{-2}), \qquad\qquad V= - \tfrac12 r \partial_r \qquad \text{on} \quad (r_0, \infty) \times \partial X. \]
This condition, which has also been used in \cite{Deruelle_Schulze_2021}, uniquely characterizes the metric on $(r_0 , \infty) \times \partial X$ and allows us to deduce a priori asymptotic estimates on $g$, which depend only on $\gamma$ and imply that $g$ possesses a polynomial asymptotic expansion.
However, the gauge at infinity does not determine the metric uniquely everywhere.
We are still left with a gauge group consisting of diffeomorphism that equal the identity on the complement of a compact subset; note that we have taken the quotient by this group in the definition of $\MM^{k^*}(X)$.
Nevertheless, the gauge at infinity condition allows us to define a metric structure and a topology on the space $\MM^{k^*}(X)$, which will later form the basis of the Banach manifold structure.

The second type of gauge, called the \emph{DeTurck gauge,} will be used to characterize local deformations of a given expanding soliton metric.
Consider two expanding soliton metrics $g,g'$ with the corresponding soliton vector fields $V, V'$.
We say that $(g',V')$ is in the DeTurck gauge with respect got $(g,V)$ if
\begin{equation} \label{eq_DeTurck_overview}
 V' = V + \DIV_g(g') - \tfrac12 \nabla^g \tr_g (g'), 
\end{equation}
where $\DIV_g$ denotes the divergence operator.
The DeTurck gauge condition allows us to recast the soliton equation for $(g',V')$ in terms of a strongly elliptic non-linear equation of the form
\begin{equation} \label{eq_Q_0_overview}
 Q_g(g') = - 2 \Ric(g') - g' - \LL_{V}g' + \LL_{\DIV_g (g') - \frac12 \tr_g (g') }g' = 0. 
\end{equation}
If $V = \nabla^g f$ is a gradient vector field, then the linearization of the equation \eqref{eq_Q_0_overview} is given by the operator
\begin{equation} \label{eq_linearization_overview}
 L_g h = \triangle_f h + 2 \Rm_g (h), 
\end{equation}
where $\triangle_f = \triangle - \nabla_{\nabla f}$.
Assuming that $(g,\nabla f, \gamma)$ represents an element of $\MMgrad^{k^*}(X)$, the ellipticity of this linearization enables us to apply an Implicit Function Theorem to equation \eqref{eq_Q_0_overview}.
This yields a description of the set of solutions of \eqref{eq_Q_0_overview} near $g$ that are asymptotic to some possibly different generalized cone metric $\gamma'$.
Specifically, we obtain that this set can be parameterized by an open subset of a Banach space.
Repeating this characterization at any point in $\MMgrad^{k^*}(X)$ leads to a Banach manifold structure on $\MM^{k^*}(X)$ near $\MMgrad^{k^*}(X)$.
We will revisit this step later in the overview to discuss further subtleties that affect our general strategy.

It is important to note that solutions $(g',V')$ of \eqref{eq_Q_0_overview} are generally not gauged at infinity.
Likewise, given a representative $(g',V',\gamma')$ of an element of $\MM^{k^*}(X)$ near the element represented by $(g, \nabla f, \gamma)$, it is generally not true that $(g',V')$ is in the DeTurck gauge with respect to $(g,V)$.
To complicate matters further, $(g',V')$ may not even be close to $(g,V)$ as tensor fields.
Therefore, in order to establish a local Banach manifold structure of $\MM^{k^*}(X)$ near $\MMgrad^{k^*}(X)$, it is necessary to convert between the gauge at infinity and the DeTurck gauge, and vice versa.
These conversions constitute a significant portion of the technical work in this paper, as it is crucial to ensure that the desired asymptotic properties at infinity remain preserved.

The metric on the space $\MM^{k^*}(X)$, which was defined via the gauge at infinity, allows us to express the properness of the map
\begin{equation} \label{eq_Pi_MMgeqgrad_overview}
 \Pi \big|_{\MMgeqgrad^{k^*}(X)} \; : \; \MMgeqgrad^{k^*}(X)  \lto \CONEgeq^{k^*} (\partial X) 
\end{equation}
To establish this properness, we first relate the entropy of a gradient expanding soliton to the entropy of its asymptotic cone.
Notably, this is the only step where the non-negative scalar curvature condition seems essential.
Second, we derive a bound on the scalar curvature of an asymptotically conical gradient expanding soliton, which only depends on its asymptotic geometry.
Third, we derive a bound on the entire curvature tensor by excluding the possibility of an instanton arising in the limit.
This step relies on the topological assumption on $X$ involving the first and second homology groups.
Finally, we deduce Cheeger-Gromov compactness of preimages of compact subsets with respect to the map \eqref{eq_Pi_MMgeqgrad_overview} and convert this to compactness with respect to the topology on $\MM^{k^*}(X)$.

Let us now explain why it is important to consider the larger space $\MM^{k^*}(X) \supset \MMgrad^{k^*}(X)$ of (possibly non-gradient) expanding solitons.
To do so, we need to revisit the application of the Implicit Function Theorem, which led to the local Banach manifold structure of $\MM^{k^*}(X)$.
We faced two constraints.
First, our analysis can only be carried out if the base metric $g$ is the metric of a \emph{gradient} expanding soliton.
Consequently, we only obtain a Banach manifold structure on a neighborhood $\MMgrad^{k^*}(X) \subset \MM' \subset \MM^{k^*}(X)$, as stated in Theorem~\ref{Thm_Banach_mf_intro}.
Second, the linearization \eqref{eq_linearization_overview} may have a non-trivial kernel $K_g$.
In this case, we use an asymptotic estimate for $h \in K_g$ due to Deruelle and Schulze \cite{Deruelle_Schulze_2021}, which is of the form
\begin{equation} \label{eq_h_K_asympt_overview}
 h = r^{-n} e^{-r^2/4} \dot\gamma + o(r^{-n} e^{-r^2/4}), 
\end{equation}
and holds in an $L^2$-sense over $\partial X$.
Here $\dot\gamma$ denotes an infinitesimal variation of a generalized cone metric, that is, $\dot\gamma$ is an element of the vector space that contains $\GenCONE^{k^*}(\partial X)$ as an open subset.
In order to eliminate the cokernel of the operator $L_g$, which is necessary to establish the Banach manifold structure, we need to allow variations of the asymptotic cone metric $\gamma$ in the directions $\dot\gamma$ appearing in \eqref{eq_h_K_asympt_overview}.
As these variations may be perpendicular to $\CONE^{k^*}(\partial X)$, we must consider variations of $\gamma$ in the direction of \emph{generalized} cones.
Consequently, we are only able to establish a local Banach manifold structure on the larger space $\MM^{k^*}(X)$, while the subset $\Pi^{-1}(\CONE^{k^*}(\partial X))$ may only be a real-analytic subvariety, potentially with singularities.

Let us now turn our attention to the gradient condition. 
The Implicit Function Theorem argument from the previous paragraph does not provide information on whether the vector field $V'$ of a nearby soliton metric $g'$, as given in \eqref{eq_DeTurck_overview}, is gradient.
In addition, as discussed in Subsection~\ref{subsec_statement_of_results}, the gradient condition can only hold in the non-generic case
when $\gamma'$ belongs to $\CONE^{k^*}(\partial X)$.
Therefore, we need to deduce the gradient condition \emph{a posteriori} under the assumption that $\gamma'$ is a cone.
We accomplish this via a continuity method; compare also with Theorem~\ref{Thm_preservation_gradient_intro}.
Consider a $C^1$-family of solutions $(g'_s)_{s \in [0,1]}$ of \eqref{eq_Q_0_overview} with $g'_0 = g$ and let $V'_s$ be the corresponding vector fields defined via \eqref{eq_DeTurck_overview}.
Recall that we have assumed that $V'_0 = V = \nabla f$ is gradient and assume furthermore that for all $s \in [0,1]$ the metric $g'_s$ is asymptotic to a cone $\gamma'_s \in \CONE^{k^*}(\partial X)$.
We prove that then $V'_s$ must be gradient for all $s \in [0,1]$.
To motivate the idea of our proof, let us consider the infinitesimal variation $h = \frac{d}{ds} |_{s=0} g'_s$.
Then \eqref{eq_DeTurck_overview} implies that
\begin{equation} \label{eq_ddsVpbgp}
 \frac{d}{ds} \bigg|_{s= 0} (V'_s)^{\flat, g'_s} = -h(\nabla f) + \big(\DIV_g ( h) \big)^{\flat,g} - \tfrac12 d \big(\tr_{g}(h) \big) =:  \big({\DIV_{g,f} ( h) }\big)^{\flat,g} - \tfrac12 d \big(\tr_{g}(h) \big). 
\end{equation}
On the other hand, as $h$ is an infinitesimal variation of the Einstein equation, we have $L_g h = 0$, which implies via a Bochner-type argument that
\[  \triangle_f {\DIV_{g,f} ( h)}  - \tfrac12 {\DIV_{g,f} ( h)}  = 0. \]
So we have ${\DIV_{g,f} ( h)}  \equiv 0$ by the maximum principle, which implies that the right-hand side of \eqref{eq_ddsVpbgp} is exact.
In other words, the gradient condition for $V'_s$ is preserved at first order near $s = 0$.
To extend this infinitesimal picture to the macroscopic level and to all $s$, we generalize the previous computations to the non-gradient case and we introduce a new analytical technique that allows us to bound the evolution of the deviation from the gradient condition with respect to the parameter $s$.

We emphasize that the argument described in the previous paragraph only establishes the \emph{preservation} of the gradient condition along $C^1$-families of asymptotically conical expanding solitons.
It does not rule out the existence of points in $\Pi^{-1}(\CONE^{k^*}(\partial X))$ that do not belong to $\MMgrad^{k^*}(X)$.
To address this issue, we need to consider a neighborhood of $\MMgrad^{k^*}(X)$ which is disjoint from $C^1$-connected components of $\Pi^{-1}(\CONE^{k^*}(\partial X))$ that consist of non-gradient expanding solitons.
This step requires that we work over a finite dimensional family of generalized cone metrics.
Specifically, fix a finite dimensional real-analytic submanifold $Q \subset \GenCONE^{k^*}(\partial X)$ containing a point $\gamma \in \CONEgeq^{k^*}(\partial X)$.
Suppose that $Q$ is transverse to $\Pi|_{\MM'}$, after possibly shrinking $\MM'$.
Then $P := \Pi^{-1}(Q) \cap \MM'$ must be a submanifold of $\MM^{k^*}(X)$.
It can be seen that the subset $\Pi^{-1}(\CONE^{k^*}(\partial X)) \cap P \subset P$ is a locally real-analytic subvariety of $P$.
Consequently, its connected components are $C^1$-path connected and thus each connected component is either fully contained in $\MMgrad^{k^*}(X)$ or disjoint from it.
So, after possibly shrinking the neighborhood $\MM'$ of $\MMgrad^{k^*}(X)$ and therefore $P$, we may assume that
\[ \Pi^{-1}(\CONE^{k^*}(\partial X)) \cap P = \MMgrad^{k^*}(X) \cap P. \]
This implies that we also have
\[ \Pi^{-1}(\CONEgeq^{k^*}(\partial X)) \cap P = \MMgeqgrad^{k^*}(X) \cap P. \]

We are now ready to sketch the definition of the expander degree.
Consider the restricted map
\[ \Pi |_P : P \lto Q \]
between finite dimensional manifolds of the same dimension.
While the map $\Pi|_P$ itself is generally not proper, the properness of the map \eqref{eq_Pi_MMgeqgrad_overview} implies properness of the restriction
\[ \Pi |_P :\MMgeqgrad^{k^*}(X) \cap P   \lto \CONEgeq^{k^*}(\partial X) \cap Q. \]
It turns out that this is sufficient to define a local degree of $\Pi |_P$ near $\MMgeqgrad^{k^*}(X) \cap P$.
Furthermore, we find that this degree is independent of the choices of $Q$, $\MM'$, $\gamma$ and $k^*$.
We define this to be the expander degree $\deg_{\exp}(X)$.

Our new degree construction faces particular challenges in the case of the \emph{integer} degree, as we also need to define an orientation on the manifolds $P$ and $Q$ in this case.
Defining this orientation involves the index of the operator $L_g$, as suggested by the formula \eqref{eq_degexp_identity_intro}.
This creates the following complication. 
On the one hand, this index is only defined in the case in which $g$ is a \emph{gradient} soliton.
On the other hand, the differential $\Pi|_P$ may be degenerate everywhere along $\MMgeqgrad^{k^*}(X) \cap P$.
This means that we cannot establish a relation between the orientations of both tangent spaces via the differential of $\Pi|_P$ as usual.
In addition, the corresponding operator $L_g$ would have a non-trivial kernel in this case.
To overcome this issue, we introduce a new method for defining a consistent orientation on the tangent spaces along $\MMgeqgrad^{k^*}(X) \cap P$.
With this method, we obtain a well defined orientation of $P$, and therefore, an integer degree for the map $\Pi|_P$.
In the case in which $\Pi|_P$ has a regular point in $\CONEgeq^{k^*}(\partial X) \cap Q$ (which we referred to as ``regular value over $\MMgeqgrad^{k^*}(X)$'' in Theorem~\ref{Thm_degexp_identity}), our new definition reduces the conventional definition of the orientation and the degree formula \eqref{eq_degexp_identity_intro} holds.
\medskip

Let us now describe the structure of our paper.
In Section~\ref{sec_preliminaries}, we recall and extend several known results involving expanding solitons.
These allow us to define the space $\MM^{k^*}(X)$ and its metric structure in Section~\ref{sec_MM}.
We also establish several basic properties relating to this definition, some of which explain the relation between different regularity assumptions.
In Section~\ref{sec_properness}, we establish the properness of the map \eqref{eq_Pi_MMgeqgrad_overview}.
In Section~\ref{sec_elliptic}, we recall and extend several analytical results due to Deruelle and Lunardi \cite{Deruelle, Deruelle_2015, Lunardi_1998} involving elliptic operators with unbounded coefficients, which will be needed frequently throughout this paper.
In Section~\ref{sec_gradientness}, we prove the preservation of the gradient condition for $C^1$-families of asymptotically conical expanding solitons.
In Section~\ref{sec_DT_gauge}, we prove that every expanding soliton near a fixed asymptotically conical gradient expanding soliton can be brought into the DeTurck gauge.
In Section~\ref{sec_smooth_dependence}, we define the local Banach manifold structure of $\MM^{k^*}(X)$ near $\MMgrad^{k^*}(X)$; this involves a process of gauging solitons at infinity.
In Section~\ref{sec_change_of_index}, we derive the necessary results to define the orientation on the manifold $P$; this section is only needed in the definition of the \emph{integer} degree.
In Section~\ref{sec_main_argument}, we combine all our results, define the expander degree and prove all our main theorems.
Appendix~\ref{appx_technical} contains further technical statements that are needed throughout the paper and Appendix~\ref{appx_real_analytic} discusses the local connectedness properties of real-analytic varieties.

Subsection~\ref{subsec_proofs_of_main_thms} contains the formal proofs of the main theorems, as stated in Subsection~\ref{subsec_statement_of_results}.
Note that many of these theorems are restated later in this paper, often in a more general form and using slightly different language.
So the proofs in Subsection~\ref{subsec_proofs_of_main_thms} will often only involve references to these theorems along with a short explanation.
\bigskip

\subsection{Acknowledgements}
We would like to thank Ronan Conlon and John Lott for helpful conversations.
\bigskip

\section{Preliminaries} \label{sec_preliminaries}
\subsection{Conventions}
In this paper, we will frequently work with smooth, 4-dimensional orbifolds $M$ that may have isolated conical singularities.
That is, the neighborhood of any point $p \in M$ can be modeled on an open subset in the quotient space $\IR^4/\Gamma$, where  $\Gamma \subset SO(4)$ is a finite group that acts freely on the unit sphere in $\IR^4$.
A point $p \in M$ is called \textbf{regular} if we can choose $\Gamma = 1$, in which case a neighborhood of $p$ corresponds to an open subset of $\IR^4$.
Otherwise, $p$ is called \textbf{singular.} 

A map $F : M \to M'$ between two such orbifolds is called \textbf{smooth} (resp. $C^k$-regular) if for any $p \in M$ there are open neighborhoods $p \in U \subset M$, $F(p) \in U' \subset M'$ with $F(U) \subset U'$ such that $U$ and $U'$ can be identified with open subsets $\hat U / \Gamma \subset \IR^4/\Gamma$ and $\hat U' / \Gamma' \subset \IR^4/\Gamma'$ and such that $F|_U : U  \to U'$ is covered by a smooth (resp. $C^k$-regular)  map of the form $\hat F : \hat U \to \hat U'$.
We say that $F$ has \textbf{invertible differential} if the maps $\hat F$ have invertible differential.
Similarly, we say that a path $\gamma : I \to M$ is smooth (resp. $C^k$-regular)  if it is locally the projection of a smooth (resp. $C^k$-regular)  path in $\IR^4$ under the projection map $\IR^4 \to \IR^4/\Gamma$.
A smooth map $F : M \to M'$ is called an \textbf{orbifold cover} if its differential is invertible and if it has the path-lifting property for $C^1$-paths.
We remark that all orbifolds considered in this paper are \textbf{good,} meaning they possess an orbifold cover, which is a smooth manifold.
In other words, any orbifold in this paper can be viewed as a quotient $\hat M / \Gamma$, where $\hat M$ is a smooth manifold and $\Gamma$ is a discrete group that acts properly on $\hat M$.

A continuous tensor field $u$ on an orbifold $M$ is given by a continuous tensor field $u_{\textnormal{reg}}$ on the set of regular points of $M$ with the property that around every point $p \in M$ there is a neighborhood $p \in U \subset M$ and an orbifold cover $\pi : \hat U \to U$ such that $\hat U$ is a manifold and such that $\pi^* u_{\textnormal{reg}}$ extends to a continuous tensor field on $\hat U$.
We say that $u$ has regularity $C^{k \leq \infty}$ if the pullbacks $\pi^* u_{\textnormal{reg}}$ have regularity $C^k$.
Note that if $M = \hat M / \Gamma$, then $u$ is given by a tensor field $\hat u$ on $\hat M$ that is invariant under the action of $\Gamma$ and $u$ is defined to be of regularity $C^k$ if $\hat u$ is of regularity $C^k$.
Differential operators acting on tensor fields on $M$, such as the Laplacian, are defined using pullbacks via local   orbifold covers $\pi : \hat U \to U$ as described above.
Similarly, a $(0,2)$-tensor field $g$ on $M$ is called \textbf{Riemannian metric,} if the local pullbacks $\pi^* g$ are Riemannian metrics.
If $(M,g)$ is a Riemannian orbifold and $p \in M$ is regular, then we define the \textbf{injectivity radius} $\inj(p)$ to be equal to the injectivity radius of $p$ as a point in the Riemannian manifold consisting of all regular points.
If $p$ is singular, then we set $\inj(p) := 0$. 

If $(M,g)$ is a Riemannian manifold and $E$ is a tensor bundle over $M$, then we define $C^{k < \infty}(M;E)$ to be the space of all sections $u$ of $E$ for which $|u|, \ldots, |\nabla^k u|$ are defined and uniformly bounded.
See Subsection~\ref{subsec_list_Holder} for more details.
We write $C^k_{\loc}(M;E)$ for the space of all sections of $E$ of \textbf{regularity} $C^k$.
Note that $C^k_{\loc}(M;E) \supsetneq C^k(M;E)$ if $M$ is non-compact.
If $M$ is an orbifold, then we adopt the same terminologies.
Note that in this case the covariant derivatives of $u$ are only defined after passing to appropriate local orbifold covers.

We use $\DIV=\DIV_g$ to denote the \textbf{divergence operator} with respect to a Riemannian metric $g$, dropping the subscript if it is clear which metric is used. 
That is, for any vector field $X$ and $(0,2)$-tensor $h$ we set 
\[ \DIV X := \sum_{i=1}^n e_i \cdot \nabla_{e_i} X, \qquad \DIV h := \sum_{i=1}^n (\nabla_{e_i} h)(e_i, \cdot), \]
where $\{ e_i \}_{i=1}^n$ denotes a local orthonormal frame.
Depending on the context, we will sometimes view $\DIV h$ as a 1-form or a vector field.
Given a $(0,2)$-tensor $h$, a vector field $V$ and a scalar function $f$, we define 
$$\DIV_V h :=\DIV h- h(V), \qquad \DIV_f h:=\DIV h- h(\nabla f).$$ 
 Similarly, we define
\[ \triangle_V:=\triangle-\nabla_V, \qquad \triangle_f:=\triangle-\nabla_{\nabla f}. \]

By a \textbf{$C^{1,\alpha}$-Banach manifold,} $\alpha \in (0,1)$, we mean a topological Hausdorff space $\MM$ together with a maximal atlas 
$$\big\{ \Phi_i : \MM \supset U_i \lto \Phi_i(U_i) \subset X_i \big\}_{i \in I},$$ 
where  $U_i \subset \MM$ are open subsets, $X_i$ are Banach spaces and each map $ \Phi_i : U_i \to \Phi_i(U_i)$ is a homeomorphism onto an open subset of $X_i$.
We require that the transition maps
\[ \Phi_j \circ \Phi_i^{-1} \big|_{\Phi_i (U_i \cap U_j) } : \Phi_i (U_i \cap U_j) \lto \Phi_j (U_i \cap U_j) \]
are of regularity $C^{1,\alpha}$.
Differentiability and H\"{o}lder continuity of maps between Banach spaces can be defined analogously to the situation for maps between finite-dimensional spaces.
Throughout this paper, all $C^{1,\alpha}$-Banach manifolds  will be metrizable and therefore paracompact.
So any such $C^{1,\alpha}$-Banach manifold is a $C^{1,\alpha}$-manifold if all Banach spaces $X_i$ are finite dimensional and their dimensions agree.
We remark that the Implicit Function Theorem still holds on $C^{1,\alpha}$-Banach manifolds.

If $f : M \to M'$ is a $C^1$-regular map between two manifolds, then we denote its differential at a point $p \in M$ by $df_p : T_p M \to T_p M'$.
If $F : \MM \to \MM'$ is a $C^1$-regular map between two $C^{1,\alpha}$-Banach manifolds, then we denote its differential point $p \in \MM$ by $DF_p : T_p \MM \to T_p \MM'$.
We use a capital `$D$' in this case in order to avoid confusion with spatial derivatives.
For instance, if $F : C^1(M) \to C^1(M)$ is a smooth map between Banach spaces and $f \in C^1(M)$, then $DF_{f}$ denotes the differential of the map $F$, while $d(F(f))$ denotes the differential of the function $F(f) : M \to \IR$, which is a 1-form on $M$.
Using a distinct notations allows us to drop the argument $f$ and distinguish $DF : TC^1(M) \to TC^1(M)$ from $dF : C^1(M) \to C^0(M; T^*M)$.
Note also that if $\MM$ and $\MM'$ happen to be finite dimensional, then we may use both notions interchangeably, so $DF = df$.
\bigskip

\subsection{Basic identities of expanding solitons} \label{subsec_basic_identities}
In this paper an {\bf expanding soliton} is given by a triple $(M,g,V)$, where $M$ is an $n$-dimensional smooth orbifold with isolated conical singularities, $g$ is a complete Riemannian metric of regularity $C^2$ and $V$ a complete vector field (in the sense that its trajectories are complete) of regularity $C^1$, that satisfy the following equation:
\begin{equation*} %
 \Ric_g + \tfrac12 \LL_V g + \tfrac12 g = 0 
\end{equation*}
Note that by Lemma~\ref{Lem_soliton_smooth} there is a (possibly different) smooth structure on $M$ with respect to which $g$ and $V$ are smooth.
If $V = \nabla f$ for some {\bf potential function} $f \in C^1(M)$, then the expanding soliton $(M,g,V)$ is called \textbf{gradient} and we have
\begin{equation} \label{eq_gradient_sol_eq}
 \Ric_g + \nabla^2 f + \tfrac12 g = 0. 
\end{equation}
We will sometimes write $(M,g,f)$ instead of $(M,g,V)$.
Note that if $V$ and $g$ are of regularity $C^k$, then $f$ is automatically of regularity $C^{k+1}$.
The following identities hold on a gradient expanding soliton:
\[ R + \triangle f + \tfrac{n}2 = 0, \qquad R + |\nabla f|^2 + f \equiv const. \]
Moreover, on a general expanding soliton $(M,g,V)$, where  $g$ is of regularity $C^4$, we have
\[ \triangle_V R + 2|{\Ric}|^2 = -R. \]

We remark that in \eqref{eq_gradient_sol_eq} we have chosen a different convention for the sign in front of the $\nabla^2 f$-term compared to some of the existing literature (for example \cite{Deruelle, Deruelle_Schulze_2021}).
Our convention is more consistent with the shrinking soliton equation.
However, one peculiarity of our convention is that the potential function $f$ often diverges to $-\infty$ at infinity.
For example, the {\bf standard Euclidean soliton} $(\IR^n, g_{\eucl}, f_{\eucl})$ satisfies $f_{\eucl} = - \frac14 r^2$ and $\nabla f = - \tfrac12 r \partial_r$, where $r$ and $\partial_r$ are the radial distance function and radial vector field.
We will observe similar asymptotic behavior in the asymptotically conical case.

An expanding soliton $(M,g,V)$ gives rise to its {\bf associated Ricci flow} $(g_t := t \phi_t^{*} g)_{t > 0}$, where $(\phi_t : M \to M)_{t > 0}$ is the flow of the time-dependent vector field $t^{-1} V$:
\begin{equation} \label{eq_ass_RF_phi}
 \phi_1 = \id_M, \qquad \partial_t \phi_t = t^{-1} V \circ \phi_t .
\end{equation}
\medskip

\subsection{Lower scalar curvature on expanding solitons}
The following theorem allows us to bound the scalar curvature on an expanding soliton from below.
The second part of the theorem seems to be new; we will use it in Lemma~\ref{Lem_Rgeq0_CONE_implies_MM} to show that an asymptotically conical expanding soliton must have non-negative scalar curvature if its asymptotic cone does.

\begin{Theorem} \label{Thm_soliton_PSC}
Let $(M,g,V)$ be an expanding soliton of dimension $n \geq 3$ with bounded curvature and suppose that $g$ is not Einstein.
Then $R > -\frac{n}2$.
If in addition
\begin{equation} \label{eq_R_minus_int}
\inf_M R > -\tfrac{n}2 \quad \text{and} \quad \int_M \max\{ -R,0 \} dg < \infty, 
\end{equation}
then we even have either $R > 0$ on $M$ or $(M,g)$ is isometric to a quotient of Euclidean space by an isometric action fixing the origin.
\end{Theorem}

See also \cite[Theorem~5]{Chan_2023} for a version of this theorem (and a different proof) for the case in which the soliton is gradient.

\begin{proof}
Consider the associated Ricci flow $(g_t := t\phi_t^* g)_{t > 0}$, where $(\phi_t)_{t > 0}$ is as in \eqref{eq_ass_RF_phi}.
On this Ricci flow we have the usual identity for the evolution of the scalar curvature
\begin{equation} \label{eq_scal_evolution_eq}
 \partial_t R = \triangle R + 2 |{\Ric}|^2 = \triangle R + 2 |{\Ric^\circ}|^2 + \tfrac2n R^2 \geq \triangle R  + \tfrac2n R^2, 
\end{equation}
where $\Ric^\circ$ denotes the traceless part of $\Ric$.
A well known application of the maximum principle shows that $R \geq -\frac{n}{2t}$ and strict inequality holds everywhere unless the flow (and thus $g$) consists of Einstein metrics.
The first statement now follows by setting $t = 1$.

Next, suppose that \eqref{eq_R_minus_int} holds.
Choose $0 < n' < n$ such that $R_g \geq - \frac{n'}2$; so $R_{g_t} \geq - \frac{n'}{2t}$ for the associated Ricci flow.
Fix an arbitrary point $p \in M$ and consider a kernel  $u \in C^\infty (M \times (0,1))$ to the following backward heat equation centered at $(p,1)$
\[ -\partial_t u = \triangle_{g_t} u - \big( 1 - \tfrac{2}{n} \big) R u, \qquad u \xrightarrow[t \nearrow 1]{\qquad} \delta_p. \]
See \cite[Theorems~24.40, 26.25]{Chow_RF_book_3} for the existence of such a solution and for the fact that this solution decays super-exponentially at spatial infinity, which implies that $u(\cdot, \frac12) \leq C$ on $M$ for some constant $C < \infty$.
By the maximum principle, we obtain that $u > 0$ and that for $t \in (0,\frac12)$
\[ -\frac{d}{dt} \max_M u(\cdot,t) \leq \big( 1 - \tfrac{2}{n} \big) \frac{n'}{2t} \max_M u(\cdot,t), \]
so
\begin{equation} \label{eq_u_improved_pol_bound}
  u(\cdot, t) \leq C t^{-(1-\frac{2}{n}) \frac{n'}{2} }.
\end{equation}

On the other hand, we compute, using \eqref{eq_scal_evolution_eq},
\begin{equation} \label{eq_dt_int_Ru}
 \frac{d}{dt} \int_M R u \, dg_t
\geq \int_M \big( (\triangle R)u + \tfrac2n R^2 u - R (\triangle u) + (1- \tfrac2n) R^2u - R^2 u \big) dg_t = 0. 
\end{equation}
Here we have used Green's identity, which is justified in the non-compact setting due to the uniform curvature bound and the decay bounds on $u$ mentioned earlier.
It follows, using \eqref{eq_dt_int_Ru}, \eqref{eq_u_improved_pol_bound}, and using the notion $a_- := \max \{ -a, 0 \}$ that
\begin{multline} \label{eq_scal_bound_computation}
 R_g (p) = R_{g_1} (p) \geq \int_M R u \, dg_t \geq -\int_M (R_{g_t})_-   u \, dg_t
\geq - C t^{-(1-\frac{2}{n}) \frac{n'}{2} } \, \int_M (R_{g_t})_-  \, dg_t  \\
= - C t^{-(1-\frac{2}{n})\frac{n'}{2} -1 + \frac{n}2}  \, \int_M (R_{g})_-  \, dg .
\end{multline}
Since
\[ -\big(1-\tfrac{2}{n} \big) \tfrac{n'}{2}  -1+\tfrac{n}2
> -\big(1-\tfrac{2}{n} \big) \tfrac{n}{2} -1+\tfrac{n}2 = 0,\]
the right-hand side in \eqref{eq_scal_bound_computation} goes to zero as $t \searrow 0$.
So $R_g \geq 0$.
If equality is attained at some point, then by the strong maximum principle applied to \eqref{eq_scal_evolution_eq}, we get that $\Ric_g \equiv 0$.
This implies that $\phi_{2}$ is a contraction, so it must have a fixed point $q \in M$.
It follows that $(M,g)$ is isometric to the blow-up limit of $(M,g)$ at $q$, which finishes the proof.
\end{proof}
\bigskip

\subsection{Asymptotics at infinity}
The following lemma, which will be used frequently in the course of the paper, provides an asymptotic description of certain gradient expanding solitons if coordinates at infinity are chosen appropriately.
Throughout this paper, we will refer to these coordinates as \emph{gauge at infinity.}

To be more specific we will consider (not necessarily gradient) expanding solitons that are asymptotic at infinity to what we will refer to as a \emph{generalized cone} (see Definition~\ref{Def_GenCONE} and subsequent discussion).
By this, we will mean a metric on $\IR_+ \times N$ of the form
\[ \gamma = dr^2 + r (dr \otimes \beta + \beta \otimes dr) + r^2 h, \]
where $\beta$ is a 1-form and $h$ is a metric on the link $N$.
In the case in which the expanding soliton is \emph{gradient,} and either $(N,h)$ has no Killing fields or, more generally, if the gradient vector field is asymptotic to the vector field $-\frac12 r \partial_r$, we have $\beta = 0$ and the metric $\gamma$ is a standard cone metric.
In this case, the following lemma is essentially well known (see, for example, \cite[Theorem~3.8]{Conlon_Deruelle_Sun_2019} and \cite[Subsec~4.3]{Siepmann_thesis_2013}); in the general case, however, a slightly different proof is necessary.

Note that the following lemma allows us to gain derivatives in the asymptotic expansion depending only on the regularity of $\gamma$.
More specifically, the lemma assumes only a low regularity on the soliton metric and its vector field ($C^k/C^1$ regularity), and a weak asymptotic convergence to the generalized cone at infinity.
Assuming that the regularity of the generalized cone metric is $C^{k^*}$, for some $k^* \gg k$, the lemma establishes strong asymptotic estimates up to $\approx k^*$ many derivatives, which in turn implies a stronger convergence to the generalized cone at infinity.

\begin{Lemma} \label{Lem_iota}
Let $(M,g,V)$ be an $n$-dimensional expanding soliton, where $g$ has regularity $C^k$ and $V$ has regularity $C^{1}$ for $k \geq 2$.
Suppose that $|V| : M \to \IR$ is proper and that for some $C < \infty$ and $p \in M$ we have the bounds
\begin{equation} \label{eq_Rm_V_bound}
 |{\Rm}| \leq C d^{-2}(p,\cdot), \qquad |V| \leq C d(p,\cdot) + C. 
\end{equation}
Moreover, suppose  that there is a sequence $\lambda_i \to 0$ such that we have Gromov-Hausdorff convergence of $(M,\lambda_i^2 g, p)$ to the metric completion of a generalized cone metric of the form
\begin{equation} \label{eq_ovg_lemma}
 \gamma = dr^2 + r (dr \otimes \beta + \beta \otimes dr) + r^2 h 
\end{equation}
on $\IR_+ \times N$, where $N$ is a smooth, compact, $(n-1)$-dimensional manifold, $r$ denotes the coordinate of the first factor and $\beta, h$ are a $1$-form and a Riemannian metric on $N$ of regularity $C^{k^*}$ for $k^* \geq k+2$.
We also suppose that under this convergence, trajectories of $V$ converge to radial lines of the form $\IR_+ \times \{ q' \}$.

Then there is a $C^{k+1}$-regular embedding $\iota : \IR_+ \times N \to M$ such that:
\begin{enumerate}[label=(\alph*)]
\item \label{Lem_iota_a} $M \setminus \iota ((r,\infty) \times N)$ is compact for all $r \geq 0$.
\item \label{Lem_iota_b} $\iota^* V = -\frac12 r \partial_r$, where $\partial_r$ denotes the standard vector field on $\IR_+$.
\item \label{Lem_iota_c} The pullback $\iota^* g$ has regularity $C^{k^*-2}$ and we have the following quantitative asymptotics to $\gamma$.
Suppose that for some constants $a > 0$, $A_0, \ldots, A_{k^*-4} < \infty$ we have the bounds
\[ \inj_{\gamma} \geq a r, \qquad |\nabla^{m,\gamma} {\Rm}_{\gamma}| \leq A_m r^{-2-m} \qquad \text{for} \quad m = 0, \ldots, k^*-4. \]
Then the following bounds hold on $(R_0(a,A_0,n), \infty) \times N$:
\begin{alignat}{2}
 |\nabla^{\gamma, m} (\iota^* g - \gamma  ) |_{\gamma} &\leq C_m (a,A_0, \ldots, A_{m},n) r^{-2-m}, \qquad & m &= 0,\ldots, k^*-4, \label{eq_ovg_asymp_1} \\
\qquad |\nabla^{\gamma, m} (\iota^* g - (\gamma - 2 \Ric_{\gamma} )) |_{\gamma} &\leq C_m (a,A_0, \ldots, A_{m+2},n) r^{-4-m}, \qquad & m &= 0,\ldots, k^*-6. \label{eq_ovg_asymp_2}
\end{alignat}
Moreover, if $(M,g,V = \nabla^g f)$ is gradient, then
\begin{equation} \label{eq_fp14r2}
 \big| f + \tfrac14 r^2 \big| \leq C < \infty. 
\end{equation}
If $f$ is normalized to satisfy $R + |\nabla f|^2 + f \equiv 0$, then we can choose $C = C(a,A_0)$.
\end{enumerate}
The map $\iota$ is unique in the following sense.
Suppose that $\iota' : \IR_+ \times N \to M$ is another map that satisfies Assertion~\ref{Lem_iota_a}, has the property that $s \mapsto \iota'(e^{s/2},z)$ is a trajectory of $-V$ for any $z \in N$ (this is equivalent to Assertion~\ref{Lem_iota_b} if $\iota'$ is $C^1$) and has the property that for any $(r_1, z_1), (r_2,z_2) \in \IR_+ \times N$ we have
\[ \lim_{s \to \infty} \frac1r d_g\big(\iota'(sr_1,z_1),\iota'(sr_2,z_2) \big) = d_{\gamma}\big( (r_1,z_1), (r_2,z_2) \big) \]
(this is weaker than Assertion~\ref{Lem_iota_c}).
Then there is an isometry $\psi : (N,h) \to (N,h)$ such that $\psi^* \beta = \beta$ and $\iota' = \iota \circ (\id_{\IR_+}, \psi)$.
\end{Lemma}

\begin{Remark}
The bound on $k^*$ and the (bounded) drop in regularity in Lemma~\ref{Lem_iota} is not optimal, but sufficient for our purposes.
\end{Remark}

\begin{Remark}
We will apply Lemma~\ref{Lem_iota} to expanding solitons whose metric $g$ and vector field $V$ admit an expansion at infinity of the form
\[ g = \gamma + O(r^{-2}), \qquad V = - \tfrac12 r \partial_r + O(1), \]
where $\gamma$ is as in \eqref{eq_ovg_lemma} and the first and second derivatives of $g - \gamma$ also decay like $O(r^{-2})$.
This assumption clearly implies the asymptotic assumption from Lemma~\ref{Lem_iota} involving Gromov-Hausdorff convergence.
Note, however, that in our applications, higher derivatives of $g - \gamma$ are, a priori, only known to decay like $O(r^{-2})$ and not like $O(r^{-2-m})$.
So it is crucial that we do not assume bounds on curvature derivatives of the form $|{\nabla^m \Rm}| \leq C d^{-2-m}(p,\cdot)$ in \eqref{eq_Rm_decay} or any form of Cheeger-Gromov convergence involving the convergence of higher derivatives in place of the Gromov-Hausdorff convergence.
This complicates the proof of Lemma~\ref{Lem_iota} since it leads to a lack of curvature derivative bounds in the Ricci flow $(g_t)_{t \in (0,1]}$ associated with $g$ near $t = 0$.
\end{Remark}

\medskip

\begin{proof}
Let us assume first that $g$ and $V$ are smooth.
Consider the associated Ricci flow $(g_t = t \phi^*_t g)_{t > 0}$, where $\phi_t$ is the flow of the time-dependent vector field $t^{-1} V$ with $\phi_1 = \id_M$.
Note that $g_1 = g$ and $\phi_t^* V = V$ for all $t > 0$ and
\begin{equation} \label{eq_phit_composition}
\phi_{t_1} \circ \phi_{t_2} = \phi_{t_1t_2} \qquad \text{for all} \quad t_1, t_2 > 0.
\end{equation}

For any $q \in M$ we have
\[ t \, \frac{d}{dt} |V_{\phi_t(q)}|^2 =  \nabla_{V_{\phi_t(q)}} |V_{\phi_t(q)}|^2
= \big( (\mathcal{L}_{V} g)(V,V) \big)_{\phi_t(q)}
= -2 \Ric(V_{\phi_t(q)},V_{\phi_t(q)}) - |V_{\phi_t(q)}|^2. \]
Since for some generic $C < \infty$, we have
\[ |{\Ric}(V,V)| \leq C|{\Ric}| \, |V|^2 \leq \frac{C}{d^2(p,\cdot) + 1} \, (d^2(p, \cdot) + C) \leq C, \]
we obtain
\[ \bigg|  \frac{d}{dt} \big( t |V_{\phi_t(q)}|^2 \big)  \bigg|
=
\bigg| t\, \frac{d}{dt} |V_{\phi_t(q)}|^2  + |V_{\phi_t(q)}|^2 \bigg| \leq C. \]
Integrating this ODI backwards implies that there is a constant $C_0 < \infty$ such that for any $q \in U_0 := \{ |V| > C_0 \}$ we have
\begin{equation} \label{eq_V_upper_lower}
 c t^{-1/2} |V_q| \leq |V_{\phi_t(q)}| \leq C t^{-1/2} |V_q|
\qquad \text{for all} 
\quad t \in (0,1],
\end{equation}
where $c > 0$ is a generic constant.
We may also assume without loss of generality that $C_0$ is chosen such that $U_0 \neq M$.
Combining the lower bound with \eqref{eq_Rm_V_bound} implies that, after possibly adjusting $C_0$ and $c$, we have for $q \in U_0$ and $t \in (0,1]$
\[ c t^{-1/2} |V_q| \leq d(p, \phi_t(q)). \]
Moreover, since for $t \in (0,1]$, we have by \eqref{eq_V_upper_lower},
\[ \bigg| \frac{d}{dt} d(p,\phi_t (q)) \bigg|
\leq t^{-1} |V_{\phi_t(q)}| 
\leq C t^{-3/2} |V_q|, \]
we also obtain that
\[ d(p, \phi_t(q)) \leq d(p,q) + C t^{-1/2} |V_q| \qquad \text{for all} \quad q \in U_0, 
\quad t \in (0,1]. \]
So, in summary,
\begin{equation} \label{eq_d_upper_lower}
 c t^{-1/2} |V_q| \leq d(p, \phi_t(q)) \leq d(p,q) + C t^{-1/2} |V_q|  \qquad \text{for all} \quad q \in U_0, \quad t \in (0,1]. 
\end{equation}

Combining \eqref{eq_Rm_V_bound}, \eqref{eq_d_upper_lower} implies that for any $q \in U_0$ we have
\begin{equation} \label{eq_Rm_decay}
 \big|{\Rm}_{\phi_t(q)}\big| \leq \frac{C}{d^2(p,\phi_t(q))} \leq C t \qquad \text{for all} \quad t \in (0,1]. 
\end{equation}
Similarly, if we set
\[ U := \big\{ q\in M \;\; : \;\; \text{$\phi_t(q) \in U_0$ for some $t > 0$} \big\} = \bigcup_{t > 0} \phi_t (U_0), \]
then by \eqref{eq_phit_composition} a bound of the form \eqref{eq_Rm_decay} holds for all $q \in U$, where $C$ may depend on $q$ in a continuous way.
Therefore the Ricci flow $(g_t)_{t \in (0,1]}$ restricted to $U \times (0,1]$ has locally uniformly bounded curvature.
It follows that $g_t |_{U}$ converges locally uniformly to some $C^0$-metric $g_0$ on $U$.

On the other hand, by assumption we know that $M \setminus U = \bigcap_{t > 0} \phi_t (M \setminus U_0)$ is compact and non-empty (as an intersection of compact subsets) and $(M, \lambda^2_i g)$ is isometric to $(M, g_{\lambda^2_i})$ via an isometry that maps  $M \setminus U, \phi_{\la_i^2}(q)$ to $M \setminus U, q$ for any $q \in U$ and that preserves the trajectories of $V$.
So by the assumption on the Gromov-Hausdorff convergence of the rescalings $(M, \la_i^2 g)$, it therefore follows that the length metrics of $(\IR_+ \times N, \gamma)$ and $(U, g_0)$ are locally isometric via a map $\iota : \IR_+ \times N \to U$ that maps radial lines of the form $\IR_+ \times \{ q' \}$ to the trajectories of $V$.
We are now in a position to apply Lemma~\ref{Lem_RF_Ck_convergence}, which implies that $g_0$ is of regularity $C^{k^*-2}$ and that 
\begin{equation} \label{eq_gtg0_Cksm2}
g_t \xrightarrow[t \searrow 0]{\quad C^{k^*-2}_{\loc} \quad} g_0. 
\end{equation}
So $\iota$ is even a \emph{Riemannian} isometry of regularity $C^{k^*-1}$ and $\iota^*V$ must be a scalar multiple of $\partial_r$ at every point.
Since the metrics $g_t$, for $t >0$, satisfy the soliton equation
\[ t \Ric_{g_t} + \tfrac12 \LL_V g_t + \tfrac1{2} g_t = 0, \]
we obtain that
\[ \LL_V g_0 + g_0 = 0, \]
which implies that
\[ \LL_{\iota^* V} \gamma + \gamma = 0. \]
Writing $\iota^*V = - \frac12 u r \partial_r$ for some $u \in C^1(\IR_+ \times N)$ gives
\[ u \LL_{-\frac12 r \partial_r} \gamma + du \otimes (-\tfrac12 r \partial_r)^\flat +  (-\tfrac12 r \partial_r)^\flat\otimes du = - \gamma, \]
so, since $\LL_{-\frac12 r \partial_r} \gamma = -\gamma$,
\[ r \, du \otimes (dr + r \beta) + r \, (dr + r \beta) \otimes du = 2(u-1) \gamma. \]
Since the left-hand side is nowhere definite, this implies that $u \equiv 1$ and thus Assertion~\ref{Lem_iota_b}.

Let us now consider again the Ricci flow $(g_t)_{t \in [0,1]}$ on $U$ and its pullback $(\ov g_t := \iota^* g_t)_{t \in [0,1]}$ on $\IR_+ \times N$.
Due to \eqref{eq_gtg0_Cksm2} we have for any $q = \iota(r,q') \in U$ with $r > 2$ and for any $m = 0, \ldots, k^*-4$
\[ \limsup_{t \searrow 0} \sup_{B(q,1)} |\nabla^{g_t,m}{\Rm_{g_t}}|_{g_t} \leq \sup_{B((r,z),1.1)} |\nabla^{\gamma,m}{\Rm_{\gamma}}|_{\gamma} \leq CA_m r^{-2-m}, \]
\[ \liminf_{t \searrow 0} \inj_{g_t}(q) \geq \tfrac12 \inj_{\gamma} (r,z) \geq \tfrac12 r. \]
So by applying Perelman's Pseudolocality Theorem \cite[10.3]{Perelman1} to the flow $(g_t)_{t \in [t_0,0]}$ for sufficiently small $t_0 > 0$, we obtain constants $R_0 (a, A_0,n), C'_0(a,A_0,n) < \infty$ such that
\[ |{\Rm}_{\ov g_t}|  \leq C'_0 r^{-2} \qquad \text{on} \quad (R_0, \infty) \times N, \qquad \text{for} \quad t \in [0,1]. \]
Similarly, after adjusting $R_0$, we can use Shi's estimates \cite[Theorem~14.16]{Chow_etal_book_II} to conclude that 
\[ |\nabla^{\ov g_t, m} {\Rm}_{\ov g_t}| \leq C(a, A_0, \ldots, A_{m},n) r^{-2-m} \qquad \text{for all} \quad t \in [0,1] \quad \text{and} \quad m = 0, \ldots, k^*-4. \]
Integrating this bound, via the evolution equation
\[ \partial_t \nabla^{\gamma,m} \ov g_t = - 2 \nabla^{\gamma,m} {\Ric}_{\ov g_t}, \]
inducting on $m$ and using the bound
\[ |\nabla^{\gamma,m} {\Rm}_{\ov g_t}| \leq C |\nabla^{\ov g_t,m} {\Rm}_{\ov g_t}|_{\ov g_t} + C \sum_{\substack{i + j_1 \ldots + j_l =m, \\ j_1, \ldots, j_l \geq 1}}
 |\nabla^{\gamma,i}  {\Rm}_{\ov g_t}|_{\gamma} |\nabla^{\gamma, j_1} \ov g_t|_{\gamma} \cdots |\nabla^{\gamma, j_l} \ov g_t|_{\gamma}
  \]
implies that
\begin{equation} \label{eq_ovg_asymp_11}
 |\nabla^{\gamma, m} (\ov g_t - \gamma  ) |_{\gamma} \leq C_m (a,A_0, \ldots, A_{m},n) r^{-2-m} \qquad \text{for all} \quad t \in [0,1] \quad \text{and} \quad  m = 0,\ldots, k^*-4,  
\end{equation}
which implies \eqref{eq_ovg_asymp_1} for $t = 1$.
The bound \eqref{eq_ovg_asymp_2} follows by integration from the fact that due to \eqref{eq_ovg_asymp_11} we have (compare also with \cite[Theorem~3.8]{Conlon_Deruelle_Sun_2019})
\[ \bigg| \frac{d}{dt} \nabla^{\gamma, m}{\Rm}_{\ov g_t} - \frac{d}{dt} \Big|_{t=0} \nabla^{\gamma, m}{\Rm}_{\ov g_t} \bigg| \leq C_m (a, A_0, \ldots, A_{m+2}) r^{-4-m}. \]
The bound \eqref{eq_fp14r2} is a direct consequence of \eqref{eq_ovg_asymp_1} and the fact that $R + |\nabla f|^2 + f \equiv const$.
This proves Assertion~\ref{Lem_iota_c}.

Consider now the general case in which $g,V$ only have regularity $C^k, C^{1}$, respectively.
By Lemma~\ref{Lem_soliton_smooth}, we may choose another smooth structure on $M$ with respect to which $g, V$ are smooth and we can carry out the previous construction.
Since $\iota$ has regularity $C^{k^*-1}$ with respect to the new smooth structure on $M$, we obtain that $\iota^* g$ has regularity $C^{k^*-2}$.
So since $k^*-2 \geq k$, we find that $\iota$ has regularity $C^{k+1}$ with respect to the original smooth structure on $M$.

To see the uniqueness statement, note that the first two conditions imply that the images of $\iota$ and $\iota'$ consist of the points whose $-V$-trajectories escape from any compact subset.
So these images agree and we must have $\iota' = \iota \circ \psi_0$ for some $\psi_0 : N \times \IR_+ \to N \times \IR_+$ of the form $\psi_0 (r,z) = (u(z) r, \psi(z))$, where $u : N \to \IR_+$, $\psi : N \to N$.
The last condition implies that $\psi_0$ is a Riemannian isometry for the metric $\gamma$.
So $u \equiv 1$ and $\psi$ is a Riemannian isometry for the metric $h$.
\end{proof}

\subsection{The elliptic problem and the gauge identity} \label{subsec_the_elliptic_problem}
Let $(M,g,f)$ be a fixed gradient expanding soliton.
A frequent goal in this paper will be to find and analyze nearby expanding soliton metrics $g'$ with soliton vector field $V'$ (which may not be gradient).
In order to overcome the lack of ellipticity of the expanding soliton equation, which is caused by its diffeomorphism invariance, we will often work in the {\bf DeTurck gauge,} which is defined by
\begin{equation} \label{eq_DT_gauge}
 P_g(g',V'):= V' - \nabla^g f + \DIV_g (g') - \tfrac12 \nabla^g \tr_g (g') = 0. 
\end{equation}
So if we set
\begin{equation} \label{eq_def_Qg}
 Q_g (g') = - 2 \Ric(g') - g' - \LL_{\nabla^g f}g' + \LL_{\DIV_g (g') - \frac12 \tr_g (g') }g', 
\end{equation}
then under the gauge condition \eqref{eq_DT_gauge} we have
\begin{equation} \label{eq_Q_soliton_equivalence}
 Q_g (g') = 0 \qquad \Longleftrightarrow \qquad \Ric_{g'} + \tfrac12 \LL_{V'} g + \tfrac12 g' = 0. 
\end{equation}
So, in particular, if $g'$ is a metric such that $Q_g (g') = 0$ and if we set
\[ V' := \nabla^g f - \DIV_g (g') + \tfrac12 \nabla^g \tr_g (g'), \]
then $(M,g',V')$ is an expanding soliton in the DeTurck gauge \eqref{eq_DT_gauge}.
Vice versa, if $(M, g', V')$ is an expanding soliton metric, then we often would like to find a diffeomorphism $\phi : M \to M$ such that $(M, g'' := \phi^* g', V'' := \phi^* V')$ satisfies the DeTurck gauge $P_g(g'',V'')=0$ and since the soliton equation is diffeomorphism invariant, the equivalence \eqref{eq_Q_soliton_equivalence} implies $Q_g(g'') = 0$.
As we will see in Section~\ref{sec_DT_gauge}, 
the equation $P_g(\phi^* g', \phi^* V') = 0$ is strongly elliptic in $\phi$, so it can often be solved.

Let us now focus on the identity
\[ Q_g (g') = 0. \]
A basic computation implies that
\begin{align*}
 \Ric (g')  &=  \Ric(g) + (g')^{-1} * \nabla^2 g' + (g')^{-1} * (g')^{-1} * \nabla g' * \nabla g', \\
 \LL_{\nabla^g f}(g')   &= \nabla_{\nabla f} g' + 2\nabla^2 f +  (\nabla^2 f + \tfrac12 g) *(g'-g)  - (g'-g), \\
 \LL_{\DIV_g ( g') - \frac12 \nabla^g \tr_g ( g') } (g')
&= \nabla g' * \nabla g' + g' * \nabla^2 g',
\end{align*}
where all covariant derivatives are taken with respect to $g$, $(g')^{-1}$ denotes the inverse of $g'$ viewed as a $(2,0)$-tensor and ``$*$'' denotes a multilinear contraction of the listed tensors, possibly after raising or lowering indices with respect to $g$.
Therefore, we obtain
\begin{multline} \label{eq_Q_expansion_intro}
 Q_g(g') = - \nabla_{\nabla f} (g'-g) + \Ric * (g'-g) \\ + (g')^{-1} * \nabla^2 g'  + g' * \nabla^2 g' + (g')^{-1} * (g')^{-1} * \nabla g' * \nabla g'  + \nabla g' * \nabla g',
\end{multline}
Linearizing this expression at $g$ gives
\begin{equation} \label{eq_DQg}
  (DQ_g)_g(h) = \triangle_f h + 2\Rm_g (h) =: L_g h, 
\end{equation}
where
\[ \big( {\Rm} (h) \big)_{ij} = g^{st} g^{kl} \Rm_{iskj} h_{tl}. \]
\bigskip

\section{The space of expanding solitons} \label{sec_MM}
In this section, we define the space of 4-dimensional expanding solitons that are asymptotic to a generalized cone metric.
We introduce a metric distance on this space, which induces a natural topology.
We also discuss regularity questions and various other elementary aspects, which will be needed frequently throughout this paper.

\subsection{Cones and generalized cones}
We first define the space of generalized cone metrics.
In the following let $N$ be a smooth, closed $(n-1)$-dimensional manifold (in this paper, we will take $n=4$).

\begin{Definition} \label{Def_GenCONE}
Given an integer or infinity $0 \leq k^* \leq \infty$, we define the set of \textbf{generalized cone metrics,} $\GenCONE^{k^*}(N)$, to be the set of Riemannian metrics $\gamma$ on $\IR_+ \times N$ of regularity $C^{k^*}$ and with the property that
\[ \LL_{\frac12 r \partial_r} \gamma = \gamma \qquad\text{and}\qquad \gamma(\partial_r,\partial_r ) = 1, \]
where $r$ and $\partial_r$ denote the standard coordinates and vector field on the $\IR_+$-factor.
We define the set of \textbf{cone metrics,} $\CONE^{k^*}(N) \subset \GenCONE^{k^*}(N)$, to be the subset of metrics $\gamma$ such that in addition
\[ \gamma(\partial_r, \cdot ) = dr. \]
In other words, any $\gamma \in \GenCONE^{k^*}(N)$ has the representation
\begin{equation} \label{eq_rep_gamma}
 \gamma = dr^2 +  r \big( dr \otimes (\pi^*\beta) + (\pi^*\beta) \otimes dr \big) + r^2 \pi^* h 
\end{equation}
for some $1$-form $\beta$ and some Riemannian metric $h$ on $N$ of regularity $C^{k^*}$, where $\pi : \IR_+ \times N \to N$ is the standard projection.
Note that $\gamma \in \CONE^{k^*}(N)$ is equivalent to $\beta = 0$; in this case $\gamma$ is uniquely determined by $h$, which we call the {\bf link metric.}

We  define $\CONEg^{k^*}(N),\CONEgeq^{k^*}(N) \subset \CONE^{k^*}(N)$ to be the subset of cone metrics with $R_\gamma >0$ and $R_\gamma \geq 0$, respectively.
Note that this is equivalent to $R_h > (n-2)(n-1)$ and $R_h \geq (n-2)(n-1)$, respectively, for the link metric $h$.

We also denote by $T\GenCONE^{k^*}(N)$ the associated vector space of symmetric $(0,2)$-tensors of regularity $C^{k^*}$ with the property that
\[ \LL_{\frac12 r \partial_r} \gamma = \gamma \qquad\text{and}\qquad \gamma(\partial_r,\partial_r ) = 0 \]
and we define $T\CONE^{k^*}(N) \subset T\GenCONE^{k^*}(N)$ to be the subspace of all $\gamma \in T\GenCONE^{k^*}(N)$ with
\[ \gamma(\partial_r, \cdot ) = 0. \]
If $k^* < \infty$, then we endow $T\GenCONE^{k^*}(N)$ and $T\CONE^{k^*}(N)$ with a standard Banach norm via the sum of  the $C^{k^*}$-norms on $\beta$ and $h$ on $N$ from \eqref{eq_rep_gamma} (since $N$ is compact, we can simply define these $C^{k^*}$-norms via a collection of coordinate charts).
\end{Definition}
\medskip

\begin{Remark} \label{Rmk_cone_vs_gencone}
If $\gamma \in \CONE^{k^*}(N)$, then the metric completion of $(\IR_+ \times N, \gamma)$ is a metric cone.
However, the converse is generally false: If the metric completion of $(\IR_+ \times N, \gamma)$ for some arbitrary $\gamma \in \GenCONE^{k^*}(N)$ is a metric cone, then we may have $\beta = u \, h(X,\cdot)$ for a Killing field $X$ and a suitable scalar function $u$ on $(N,h)$.
Only if $(N,h)$ has no Killing fields, then it is indeed possible to conclude that $\gamma \in \CONE^{k^*}(N)$.
\end{Remark}

\begin{Remark}
For any $\gamma_1, \gamma_2 \in \GenCONE^{k^*}(N)$ we have $\gamma_1 - \gamma_2 \in T\GenCONE^{k^*}(N)$ and for any $\gamma \in \GenCONE^{k^*}(N)$ we have $\gamma + \gamma' \in \GenCONE^{k^*}(N)$ for sufficiently small $\gamma' \in T\GenCONE^{k^*}(N)$.
The analogous statement holds for $\CONE^{k^*}(N)$.
\end{Remark}
\bigskip

\subsection{Ensembles and compactifications} \label{subsec_ensembles}
The main theorems of this paper, as stated in Subsection~\ref{subsec_statement_of_results}, concern expanding solitons defined on the interior of a compact orbifold $X$ with boundary.
Its conical asymptotics were described by identifying a tubular neighborhood of its boundary with $(1,\infty] \times \partial X$.
In the following we will make this characterization more precise.
For this purpose, we will consider an equivalent topological setup involving an ``ensemble'', where we only consider the interior of $X$ and fix a coordinate system at infinity.
This turns out to be more convenient for our analysis.

\begin{Definition} \label{Def_ensemble}
An \textbf{ensemble} is a triple $(M, N, \iota)$ consisting of a closed, smooth 3-manifold $N$, a smooth 4-orbifold $M$ with isolated, conical singularities and a smooth embedding $\iota : [1, \infty) \times N \to M$ such that the following is true:
\begin{enumerate}[label=(\arabic*)]
\item \label{Def_ensemble_1} $M \setminus \iota ((r,\infty) \times N)$ is compact for all $r \geq 1$.
\item \label{Def_ensemble_2} $N$ admits a metric of positive scalar curvature, i.e., the components of $N$ are diffeomorphic to connected sums of spherical space forms and copies of $S^2 \times S^1$, after possibly passing to their orientable double covers.
\item \label{Def_ensemble_3} $M$ has an orbifold cover $\hat M$ (of possibly infinite index), which is a smooth manifold with $H_2(\hat M; \IZ) = 0$ and torsion-free $H_1(\hat M; \IZ)$.
\end{enumerate}
\end{Definition}

The following lemma expresses the equivalence between both topological setups, modulo diffeomorphisms relative boundary and restriction of $\iota$.

\begin{Lemma} \label{Lem_ensemble_vs_X}
Let $X$ be a compact, smooth 4-orbifold with isolated singularities and non-empty boundary.
Suppose that $X$ satisfies Properties~\ref{Property_1}, \ref{Property_2} from the beginning of Subsection~\ref{subsec_statement_of_results}.
Then there is a smooth embedding $\iota : [1,\infty] \times \partial X \to X$  such that $(\Int X, \partial X, \iota)$ is an ensemble and such that $\iota(\cdot, \infty) = \id_{\partial X}$.
Moreover, $\iota$ is unique modulo diffeomorphisms of $X$ relative boundary in the following sense:
For any two such maps $\iota_1,\iota_2$ there is a diffeomorphism $\phi : X \to X$ such that $\phi|_{\partial X} = \id_{\partial X}$ and $\iota_2 = \phi \circ \iota_1$.

On the other hand, suppose that $(M,N,\iota)$ is an ensemble and let $X$ be the orbifold obtained by identifying $M$ with $[1,\infty] \times N$ via the map $\iota$.
Then $(M = \Int X,N = \partial X,\iota)$ is a possible ensemble as constructed in the first part of the lemma.
\end{Lemma}

\begin{proof}
The existence of $\iota$ follows from the existence of a collar neighborhood near $\partial X$.
For the uniqueness statement, choose a smooth family of maps $(\iota_s : [1, \infty] \times \partial X \to X)_{s \in [1,2]}$ interpolating between $\iota_1$ and $\iota_2$.
Then there is an $r_0 < \infty$ such that $\iota_s |_{[r_0, \infty] \times \partial X}$ consists of embeddings that move by smooth isotopies.
By the isotopy extension lemma 
(see for instance \cite[Chapter 8, Theorem 1.8]{Hirsch})
we can find an isotopy $(\phi'_s : X \to X)_{s \in [1,2]}$ such that $\iota_s = \phi'_s \circ \iota_1$ on $[r_1, \infty] \times \partial X$ for some $r_1 \in (r_0, \infty)$.
Let $\phi' := \phi'_2$.
Next, choose a diffeomorphism $\eta : [1, \infty] \to [r_1, \infty]$ and choose $\phi'', \phi''' : X \to X$ such that $\phi'' (\iota_1 (r,z)) =  \iota_1 (\eta(r),z)$ and $\phi''' (\iota_2 (r,z)) =  \iota_2 (\eta^{-1}(r),z)$.
Then $\phi := \phi''' \circ \phi' \circ \phi''$ has the desired property.

The last statement of the lemma is trivial.
\end{proof}

Throughout this paper, we will only consider metrics $g$ on $M$ that are asymptotic to a cone or a generalized cone, in the sense that $|\iota^* g - \gamma| \to 0$ as $r \to \infty$ for some $\gamma \in \GenCONE^{k^*}(N)$.
All such metrics $g$ are bilipschitz to one another.
Therefore, the following definition makes sense.

\begin{Definition}
If $(M, N,\iota)$ is an ensemble, then a map $\phi : M \to M$ is said to have {\bf finite displacement} if for one (and therefore all) Riemannian metrics $g$ on $M$ with the property that $|\iota^* g - \gamma| \to 0$ as $r \to \infty$ for some $\gamma \in \GenCONE^{k^*}(N)$ we have $\sup_{p\in M} d_g (p,\phi(p)) < \infty$.
\end{Definition}
\bigskip

\subsection{The space $\MM$}
We can now define the space $\MM^{k^*}$ of expanding solitons that are asymptotic to a generalized cone, as well as associated subspaces.
We will provide two different characterizations of these spaces, depending on different topological setups.
In Definition~\ref{Def_MM}, we first define $\MM^{k^*}(M,N,\iota)$ and its subspaces, which depend on the choice of an ensemble $(M,N,\iota)$. Based on this, we then define the space $\MM^{k^*}(X)$ and its subspaces in Definition~\ref{Def_MMX}, where $X$ is the orbifold that satisfies the assumptions of Lemma~\ref{Lem_ensemble_vs_X}.
Both spaces $\MM^{k^*}(M,N,\iota)$ and $\MM^{k^*}(X)$ will be canonically isomorphic if $(M,N,\iota)$ and $X$ are related as explained in Lemma~\ref{Lem_ensemble_vs_X}, so both terminologies can be used interchangeably in most situations. 
Throughout this paper, we mostly use the notation $\MM^{k^*}(M,N,\iota)$ because it is the less technical one. 
However, in order to state our results more succinctly, we have relied on $\MM^{k^*}(X)$ in the introduction.

\begin{Definition} \label{Def_MM}
Consider an ensemble $(M,N,\iota)$ and an integer or infinity $2 \leq k^* \leq \infty$.
Let $\MM^{k^*}_{\#}(M,N,\iota)$ be the set of triples $(g,V,\gamma)$ consisting of a Riemannian metric $g$ of regularity $C^{2}$, a vector field $V$ on $M$ of regularity $C^{1}$ and an element $\gamma \in \GenCONE^{k^*}(N)$ with the following properties:
\begin{enumerate}[label=(\arabic*)]
\item \label{Def_MM_1} $(M,g,V)$ is an expanding soliton, i.e., we have
\[ \Ric_g + \tfrac12 \LL_V g + \tfrac12 g = 0. \]
\item \label{Def_MM_2} We have $\iota^* V = -\frac12 r \partial_r$ on $(r_0, \infty) \times N$ for some $r_0 \geq 1$.
Here $r$ and $\partial_r$ denote the coordinate function and the standard vector field on the factor $(1, \infty)$.
\item \label{Def_MM_3} We have for $m = 0,1,2$
\begin{equation*} %
\sup_{(r_0, \infty) \times N} r^{m} |\nabla^m(\iota^* g - \gamma) | \xrightarrow[r_0 \to \infty]{\quad \quad} 0. 
\end{equation*}
\end{enumerate}
We now define the space of \textbf{expanding solitons that are asymptotic to a generalized cone} by
\[ \MM^{k^*}(M,N,\iota) := \MM_{\#}^{k^*} (M,N,\iota) \big/ \sim \]
where we write $(g_1, V_1, \gamma_1) \sim (g_2, V_2, \gamma_2)$ if $\gamma_1 = \gamma_2$ and there exists a $C^1$-diffeomorphism $\psi : M \to M$ that equals the identity on the complement of a compact subset of $M$ and such that
\[ \psi^* g_1 = g_2, \qquad \psi^* V_1 = V_2. \]
We also define the natural  projection map
\[ \Pi : \MM^{k^*} (M,N,\iota) \longrightarrow \GenCONE^{k^*} (N), \qquad [(g,V,\gamma)] \longmapsto \gamma. \]

Define 
\[  \MMg^{k^*} (M,N,\iota), \; \MMgeq^{k^*} (M,N,\iota), \; \MMgrad^{k^*}(M,N,\iota) \; \subset \; \MM^{k^*} (M,N,\iota) \]
to be the subsets of classes $[(g,V,\gamma)]$ for which $g$ has positive scalar curvature, for which $g$ has non-negative scalar curvature and for which $(M,g,V)$ is a \textbf{gradient expanding soliton that is asymptotic to a cone,} i.e., for which we have $\gamma \in \CONE^{k^*}(N)$ and $V = \nabla^g f$ for some $f \in C^1(M)$.
Set 
\begin{align*}
 \MMggrad^{k^*}(M,N,\iota) &:= \MMgrad^{k^*}(M,N,\iota)  \cap \MMg^{k^*}(M,N,\iota), \\
\MMgeqgrad^{k^*}(M,N,\iota) &:= \MMgrad^{k^*}(M,N,\iota)  \cap \MMgeq^{k^*}(M,N,\iota), 
\end{align*}
If the ensemble $(M,N,\iota)$ and $k^*$ are fixed, then we will often omit the superscript $k^*$ and the term $(M,N,\iota)$ and write $\MM, \MMgrad \MMggrad, \MMgeqgrad$ instead of $\MM^{k^*}(M,N,\iota)$, $\MMgrad^{k^*}(M,N,\iota)$,  $\MMggrad^{k^*}(M,N,\iota)$, $\MMgeqgrad^{k^*}(M,N,\iota)$.

Lastly, we say that a representative $(g,V,\gamma) \in \MM_{\#}^{k^*} (M,N,\iota)$ of an element $p = [(g,V,\gamma)] \in \MM^{k^*} (M,N,\iota)$ is \textbf{$C^{k}$-regular} for $1 \leq k \leq \infty$ if $g,V$ have regularity $C^k$.
\end{Definition}
\medskip

\begin{Remark} \label{Rmk_gradient_implies_cone}
In the definition of $\MMgrad^{k^*}(M,N,\iota)$ the condition that $\gamma \in \CONE^{k^*}(N)$ is dispensable.
To see this suppose that $p = [(g,V,\gamma)] \in \MM^{k^*}(M,N,\iota)$ such that $(M,g,V = \nabla^g f)$ is a gradient expanding soliton.
Then
\[ (\iota^* g)(-\tfrac12 r \partial_r, \cdot ) =
 (\iota^* g)(\iota^* V, \cdot) = \iota^* ( g(V, \cdot)) = \iota^* df = d (f \circ \iota). \]
is closed.
This implies via Definition~\ref{Def_MM}\ref{Def_MM_3} that $\gamma( r \partial_r, \cdot)$ is closed.
Using the representation \eqref{eq_rep_gamma}, we obtain that
\[ 0 = d \big( \gamma( r \partial_r, \cdot) \big)= d \big( r dr + r^2 \beta \big) = 2 r \, dr \wedge \beta. \]
It follows that $\beta \equiv 0$ and thus $\gamma \in \CONE^{k^*}(N)$.

Note that if we relax the gradient condition and only require that the metric $g$ belongs to gradient expanding soliton $(M,g,f)$ for a possibly different gradient vector field $\nabla f \neq V$, then we may have $\gamma \not\in \CONE^{k^*} (N)$.
Take for example a rotationally symmetric gradient expanding soliton and consider the discussion from Remark~\ref{Rmk_cone_vs_gencone}.

On the other hand, we are not aware of examples for elements $p = [(g,V,\gamma)]$ for which $\gamma \in \CONE^{k^*}(N)$, but for which $V$ is not a gradient vector field.
In the case in which $g$ has non-negative curvature operator the existence of such elements was ruled out by Deruelle \cite{Deruelle}.
In the general case, our discussion in Section~\ref{sec_gradientness} implies that such a $p$ must lie in a connected component of $\Pi^{-1}(\CONE^{k^*} (N))$ that is disjoint from $\MMgrad^{k^*}(M,N,\iota)$ with respect to the topology defined later.
\end{Remark}
\medskip

\begin{Remark}
Lemma~\ref{Lem_MM_regular_rep} below shows that every element in $\MM^{k^*} (M,N,\iota)$ has a $C^{k^*-2}$-regular representative.
Hence the space $\MM^{k^*} (M,N,\iota)$ remains the same if we required $g$ and $V$ to be of regularity $C^k$ and $C^{l}$, respectively, as long as $2 \leq k \leq k^* - 2$ and $1 \leq l \leq k^*-2$.
Likewise, the requirement that the isometry $\psi$ is of regularity $C^1$ is arbitrary.
If $g_1, g_2$ are of regularity $C^k$, then $\psi$ is automatically of regularity $C^{k+1}$.
\end{Remark}

\begin{Remark} \label{Rmk_differen_regularity_MM}
Our definition shows that for any $2 \leq k^*_1 \leq k^*_2 \leq \infty$ we have an inclusion map of the form
\[ \MM^{k^*_2} (M,N,\iota) \hookrightarrow \MM^{k^*_1} (M,N,\iota), \]
whose image is the preimage of $\GenCONE^{k^*_2}(N)$ under the map
\[ \Pi : \MM^{k_1^*} (M,N,\iota) \lb\longrightarrow\lb \GenCONE^{k_1^*} (N). \]
\end{Remark}
\medskip

Lemma~\ref{Lem_ensemble_vs_X} implies that the space $\MM^{k^*}(M,N,\iota)$ only depends on the associated orbifold $X$.
More specifically, we have the following result.

\begin{Lemma} \label{Lem_equivalence_MM}
Let $X$ be an orbifold satisfying the assumptions of Lemma~\ref{Lem_ensemble_vs_X} and $2 \leq k^* \leq \infty$. 
Consider two ensembles $(M = \Int X, N = \partial X, \iota_i)$, $i =1,2$, associated with $X$.
Then there is a unique map
\[ \Phi_{\iota_1,\iota_2} : \MM^{k^*}(M,N,\iota_2) \lto \MM^{k^*}(M,N,\iota_1) \]
such that for any $[(g,V,\gamma)] \in \MM^{k^*}(M,N,\iota_2)$ we have
\begin{equation} \label{eq_Phi_identity}
 \Phi_{\iota_1,\iota_2}\big( [(g,V,\gamma)]  \big) = [(\phi^* g, \phi^* V, \gamma)] 
\end{equation}
for some smooth diffeomorphism $\phi : X \to X$ with $\iota_2 = \phi \circ \iota_1$.
\end{Lemma}

\begin{proof}
Let $\phi$ be the map from Lemma~\ref{Lem_ensemble_vs_X}.
Note that $\phi$ is uniquely determined in a neighborhood of $\partial X$.
Then \eqref{eq_Phi_identity} determines $\Phi_{\iota_2,\iota_1}$ uniquely and does not depend on the choice of $\phi$.
\end{proof}

As a consequence of Lemma~\ref{Lem_equivalence_MM}, we can give an equivalent definition of the space $\MM^{k^*}$, and associated spaces, depending on the orbifold $X$ instead of the ensemble $(M,N,\iota)$.

\begin{Definition} \label{Def_MMX}
Let $X$ be an orbifold satisfying the assumptions of Lemma~\ref{Lem_ensemble_vs_X} and $2 \leq k^* \leq \infty$. 
Consider all possible associated ensembles $\{ (M = \Int X, N = \partial X, \iota_i) \}_{i \in I}$.
Then we define 
\[ \MM^{k^*}(X) := \bigcup_{i \in I} \MM^{k^*}(M,N,\iota_i) \Big\slash \sim, \]
where $p_1 \in \MM^{k^*}(M,N,\iota_{i_1})$ and $p_2 \in \MM^{k^*}(M,N,\iota_{i_2})$ are said to be equivalent if $\Phi_{\iota_{i_2},\iota_{i_1}}(p_1) = p_2$ for the map from Lemma~\ref{Lem_equivalence_MM}.
We define the subspaces $\MMgrad^{k^*}(X), \MMgeq^{k^*}(X), \ldots \subset \MM^{k^*}(X)$ similarly.
\end{Definition}
\medskip

Let us now switch back to the notation $\MM^{k^*}(M,N,\iota)$.
The following lemma discusses some basic properties of the spaces $\MM^{k^*} (M,N,\iota)$.
For example, it clarifies some of the regularity issues and introduces the embedding $\td\iota : \IR_+ \times N \to M$, which we will frequently use.
Note that, as discussed before, the regularity degrees (e.g. the $C^{k^*-2}$-regularity in Assertion~\ref{Lem_MM_regular_rep_a}) may not be optimal and the precise regularity drops are not important for this paper.
The reader may replace these by significantly lower regularity assertions, e.g., by $C^{k^*-10}$-regularity, for easier recollection.

\begin{Lemma} \label{Lem_MM_regular_rep}
Consider an ensemble $(M,N,\iota)$, an integer or infinity $4 \leq k^* \leq \infty$ and an element $p \in \MM^{k^*}(M,N,\iota)$.
Then the following is true:
\begin{enumerate}[label=(\alph*)]
\item \label{Lem_MM_regular_rep_a} $p$ has a $C^{k^*-2}$-regular representative $(g,V,\gamma)$ such that $\iota^* V = - \frac12 r \partial_r$ on $(2,\infty) \times N$.
\item \label{Lem_MM_regular_rep_c} For any representative $(g,V,\gamma)$ of $p$ there is a unique map $\td\iota : \IR_+ \times N \to M$ of regularity $C^{1}$ such that $\td\iota^* V = - \frac12 r \partial_r$ and such that $\td\iota = \iota$ on $(r_0,\infty) \times N$ for some $r_0 > 1$.
Moreover, if $(g,V,\gamma)$ is a $C^k$-regular representative, $2 \leq k \leq k^* - 2$, then $\td\iota$ has regularity $C^{k+1}$.
If $(g,V,\gamma)$ is the representative from Assertion~\ref{Lem_MM_regular_rep_a}, then $\td\iota$ has regularity $C^{k^*-1}$ and we can set $r_0 = 2$.
\item \label{Lem_MM_regular_rep_b} For any representative $(g,V,\gamma)$ of $p$, Assertion~\ref{Lem_iota_c} of Lemma~\ref{Lem_iota} holds for $\iota$ replaced with $\td\iota$, where the constants $a,\lb A_0, \lb\ldots,\lb A_{k^*-4}$ only depend on $\gamma$ in the indicated way.
In particular, if $p  \in \MMgrad^{k^*}(M,N,\iota)$ and $V = \nabla^g f$, then we have
\begin{equation} \label{eq_f_asymptotics_iota}
  \big| f \circ \iota + \tfrac14 r^2 \big| \leq C < \infty. 
\end{equation}
\item \label{Lem_MM_regular_rep_d} Consider the setting of Lemma~\ref{Lem_iota} and suppose that  the map $\iota : \IR_+ \times N \to M$ from this lemma agrees with the map $\iota: (1,\infty) \times N \to M$ on a subset of the form $(r_0,\infty) \times N$, $r_0 \geq 1$ and that $k^* \geq 4$.
Then $(g,V,\gamma)$ represents an element in $\MM^{k^*}(M,N,\iota)$.
\end{enumerate}
\end{Lemma}
\medskip

\begin{Remark}
Throughout this paper we will often work with a $C^{k^*-2}$-regular representative $(g,V,\gamma)$ of a given element $p \in \MM^{k^*}(M,N,\iota)$; the existence of such a representative is guaranteed by Assertion~\ref{Lem_MM_regular_rep_a}. 
However, we may sometimes generate nearby elements $p' \in \MM^{k^*}(M,N,\iota)$ by specifying representatives $(g',V',\gamma')$ of lower regularity.
\end{Remark}

\medskip

We record the following special case of Lemma~\ref{Lem_MM_regular_rep}\ref{Lem_MM_regular_rep_a}:

\begin{Corollary} \label{Cor_smooth_rep}
Consider an ensemble $(M,N,\iota)$ an integer $k^* \geq 4$ and an element $p \in \MM^{k^*}(M,\lb N, \lb \iota)$.
If $\Pi(p) \in \GenCONE^\infty (N)$, then $p$ has a smooth representative $(g,V,\gamma)$.
\end{Corollary}
\medskip

\begin{proof}[Proof of Lemma~\ref{Lem_MM_regular_rep}.]
Let $\td\iota : \IR_+ \times N \to M$ be the unique map with the property that for any $z \in N$ the curve $s \mapsto \td\iota(e^{s/2}, z)$ is the trajectory of $-V$ that agrees with $s \mapsto \iota(e^{s/2}, z)$ for large $s$.
In other words, $\td\iota$ is the unique map such that $\td\iota^* V = -\frac12 r \partial_r$ and $\td\iota = \iota$ on $(r_0, \infty) \times N$ for some large enough $r_0 \geq 1$.
Then $\td\iota$ has regularity $C^{1}$ (the same regularity as $V$), and it satisfies the assertions of Lemma~\ref{Lem_iota} (for $k=k^*-2$), due to the uniqueness statement in the end.
This shows Assertions~\ref{Lem_MM_regular_rep_c}, \ref{Lem_MM_regular_rep_b}.
In particular, it follows that $\td\iota^* g$ has regularity $C^{k^*-2}$ and $\td\iota^* V = -\frac12 r \partial_r$ is smooth.

Next, we prove Assertion~\ref{Lem_MM_regular_rep_a}.
By Lemma~\ref{Lem_soliton_smooth} there is a smooth structure on $M$ with respect to which $g,V$ are smooth, so there is a smooth orbifold $M'$ and a $C^1$-diffeomorphism $\phi : M' \to M$ such that $g' := \phi^* g$, $V' := \phi^* V$ are smooth.
Since $\iota^*g = \td\iota^* g$ (which holds on $(r_0, \infty) \times N$) and $\phi^* g$ both have regularity $C^{k^*-2}$, we obtain that $\iota^{-1} \circ \phi$, and therefore $\phi$, restricted to $\phi^{-1} (\iota((r_0, \infty) \times N))$ has regularity $C^{k^*-1}$.
So by smoothing $\phi$, we can construct a diffeomorphism $\phi' : M' \to M$ of regularity $C^{k^*-1}$ such that $\phi' = \phi$ on $\phi^{-1} (\iota((r_0+1, \infty) \times N))$.
Then $g'' := (\phi')_* \phi^* g = (\phi')_* g'$ and $V'' := (\phi')_* \phi^* V = (\phi')_* V'$ have regularity $C^{k^*-2}$.
Since $\phi \circ (\phi')^{-1}$ is equal to the identity on the complement of a compact subset, we obtain that $[(g,V,\gamma)] = [(g'',V'',\gamma)]$.

Suppose now without loss of generality that we have started with a representative $(g,V,\gamma)$ such that $g,V$ have regularity $C^{k^*-2}$.
Our goal will be to modify $g, V$ by a diffeomorphism such that $\iota^* V = -\frac12 r \partial_r$ on $(2,\infty) \times N$.
We will achieve this by applying an isotopy along $\td\iota$ and then along $\iota$.
Note by Assertion~\ref{Lem_MM_regular_rep_b} the map $\td\iota$ has regularity $C^{k^*-1}$.
Let $\phi : M \to M$ be a smooth diffeomorphism such that equals the identity on $M \setminus \Image \iota$ and is of the form $\phi (\iota(r,z)) = \iota(\eta(r),z)$ for some $\eta : (1,\infty) \to (1,\infty)$ with $\eta(r) = r$ for $r \in (1,1.1)$ and $\eta(r) = br$ on $(2,\infty)$ for some $b > 1$.
Then
\[ \phi \big( \iota((2,\infty) \times N) \big) \subset \iota( (r_0,\infty) \times N) \qquad \text{and} \qquad \phi^* V =  V \qquad \text{on} \quad \iota((2,\infty) \times N). \]
Similarly, let $\td\phi : M \to M$ be a $C^{k^*-1}$-diffeomorphism that equals the identity on $M \setminus \Image \td\iota$ and is of the form $\td\phi (\td\iota(r,z)) = \td\iota(\td\eta(r),z)$ for some smooth $\td\eta : \IR_+ \to \IR_+$ with $\td\eta(r) = r$ on $(0, \frac12 r_0)$ and $\td\eta(r) = b^{-1} r$ on $(r_0,\infty)$.
Then
\[  \td\phi^* V =  V \qquad \text{on} \quad \iota((r_0,\infty) \times N) \]
and $\td\phi \circ \phi$ equals the identity on $\iota (( \max\{ r_0, 2 \}, \infty) \times N)$.
It follows that
\[ g' := (\td\phi \circ \phi)^* g = \phi^* \td\phi^* g, \qquad V' := (\td\phi \circ \phi)^* V = \phi^* \td\phi^* V   \]
are of regularity $C^{k^*-2}$ and $(g',V',\gamma)$ satisfies Assertion~\ref{Lem_MM_regular_rep_a}.

Assertion~\ref{Lem_MM_regular_rep_d} is an immediate consequence of Lemma~\ref{Lem_iota}.
\end{proof}
\medskip

The following lemma shows that any (gradient) expanding that is asymptotic to a cone with positive (or non-negative) scalar curvature has non-negative scalar curvature.
Although the statement is for four-dimensional expanding solitons, the proof works in all dimensions except for dimension $n=2$ due to the application of Theorem \ref{Thm_soliton_PSC}.

\begin{Lemma} \label{Lem_Rgeq0_CONE_implies_MM}
Consider an ensemble $(M,N,\iota)$ and an integer $k^* \geq 2$.
Then
\begin{align*}
\Pi \big( \MMgeq^{k^*}(M,N,\iota) \big) &\subset \CONEgeq(N) \\
 \Pi^{-1} \big( \CONEg^{k^*}(N) \big) &\subset \MMg^{k^*}(M,N,\iota), \\
 \Pi^{-1} \big( \CONEgeq^{k^*}(N) \big) \cap \MMgrad^{k^*}(M,N,\iota) &\subset \MMgeqgrad^{k^*}(M,N,\iota) 
\end{align*}
\end{Lemma}

\begin{proof}
Let $(g,V,\gamma)$ be a representative of an element of $\MM^{k^*}(M,N,\iota)$.
By Definition~\ref{Def_MM}\ref{Def_MM_3}, if $g$ has non-negative scalar curvature, then so does $\gamma$.
On the other hand, if $\gamma \in  \CONEg^{k^*}(N)$, then by Definition~\ref{Def_MM}\ref{Def_MM_3} we have $R_g > 0$ on the complement of a compact subset.
So we can apply Theorem~\ref{Thm_soliton_PSC} and obtain that $R > 0$ on $M$.

Next assume that $\gamma \in  \CONEgeq^{k^*}(N)$ and that $V = \nabla f$ is gradient.
Then by Definition~\ref{Def_MM}\ref{Def_MM_3} we obtain that for any $a > 0$ we have $R \circ \iota \geq - a r^{-2}$ for large enough $r$.
Without loss of generality, we may assume that $f$ is normalized such that
\begin{equation} \label{eq_f_normalize_2}
 R + |\nabla f|^2 = - f + 10 . 
\end{equation}
Combining Definition~\ref{Def_MM}\ref{Def_MM_2}, \ref{Def_MM_3} yields that for some generic constant $C < \infty$ we have
\[ - f \circ \iota \leq  C + |\nabla f|^2 \circ \iota \leq C + C r^2 \]
and similarly $-f \circ \iota \geq -C-c r^2$ for some $c > 0$.
So we can choose a compact subset $K \subset M$ with smooth boundary such that
\begin{equation} \label{eq_fgeq1_RleqC}
 -f \geq 10^3, \qquad |R| \leq 1 \qquad \text{on} \quad M \setminus K. 
\end{equation}
Moreover, since $\gamma \in \CONEgeq^{k^*}(N)$, for any $a > 0$ we have $R \geq - a(-f)^{-1}$ on the complement of some (possibly larger) compact subset.

Let $C_0 < \infty$ be a constant such that
\begin{equation} \label{eq_RC0onboundary}
  R > -C_0(-f)^{-3} \qquad \text{on} \quad \partial K. 
\end{equation}
We claim that
\begin{equation} \label{eq_RgeqC010}
 R \geq -C_0(-f)^{-3} \qquad \text{on} \quad M \setminus K. 
\end{equation}
Note that \eqref{eq_RgeqC010} allows us to apply Theorem~\ref{Thm_soliton_PSC} as before to conclude that $R \geq 0$ on $M$, hence finishing the lemma.
So assume that \eqref{eq_RgeqC010} is false.
In the following we will work with the smooth structure supplied by Lemma~\ref{Lem_soliton_smooth}.
By the discussion from the previous paragraph there is an $a > 0$, which may be very large, such that
\begin{equation} \label{eq_RgeqC010witha}
 R \geq -a(-f)^{-1} -C_0(-f)^{-3} \qquad \text{on} \quad M \setminus K. 
\end{equation}
Choose $a \geq 0$ minimal with the property that \eqref{eq_RgeqC010witha} holds.
Then $a > 0$, since we have assumed that \eqref{eq_RgeqC010} is false.
Moreover, again by the discussion from the previous paragraph and by \eqref{eq_RC0onboundary}, we must have equality in \eqref{eq_RgeqC010witha} at some point $p \in M \setminus K$.
At this point we have
\begin{equation} \label{eq_true_at_p}
 R = -a(-f)^{-1} -C_0(-f)^{-3} , \qquad 
\triangle_f R \geq \triangle_f \big(-a(-f)^{-1} -C_0(-f)^{-3} \big)  
\end{equation}
By the soliton identities and \eqref{eq_f_normalize_2} we have
\begin{equation} \label{eq_Revol_1}
 \triangle_f R = - 2|{\Ric}|^2 - R 
\end{equation}
and
\[ \triangle_f f = \triangle f - |\nabla f|^2
= -R-2 - |\nabla f|^2 = f-12. \]
Therefore, for any $b \in \{ 1, 3 \}$, we have by \eqref{eq_fgeq1_RleqC}
\begin{align*}
 \triangle_f (-f)^{-b}
&=  b(\triangle_f f) (-f)^{-b-1} + b(b+1) |\nabla f|^2 (-f)^{-b-2} \\
&=  -b (-f)^{-b} -12b (-f)^{-b-1} + b(b+1)(-f-R+10) (-f)^{b-2} \\
&\leq  -b (-f)^{-b} +b(b-11) (-f)^{-b-1} + 11b(b+1)(-f)^{b-2} < - b(-f)^{-b},
\end{align*}
which implies
\[ \triangle_f \big(-a(-f)^{-1} -C_0(-f)^{-3} \big)
> a(-f)^{-1} + C_0(-f)^{-3}. \]
Combining this with \eqref{eq_true_at_p}, \eqref{eq_Revol_1}, each time evaluated at $p$, implies that
\begin{equation*}
 -R > \triangle_f R \geq \triangle_f \big(-a(-f)^{-1} -C_0(-f)^{-3} \big) 
> a(-f)^{-1} + C_0(-f)^{-3} = - R,
\end{equation*}
which is a contradiction.
\end{proof}
\medskip

The following lemma shows that the geometry of a representative $(g,V,\gamma)$ of an element of $\MM^{k^*}(M,N,\iota)$ is already determined by its geometry restricted to the image of the map $\td\iota : \IR_+ \times N \to M$ from Assertion~\ref{Lem_MM_regular_rep_c} of Lemma~\ref{Lem_MM_regular_rep}.
This will be the basis of the definition of a metric on $\MM^{k^*}(M,N,\iota)$ in the next subsection.

\begin{Lemma} \label{Lem_Im_iota_dense}
Consider an ensemble $(M,N,\iota)$, an integer $k^* \geq 4$, an element $p = [(g,V,\gamma)]  \in \MM^{k^*}(M,N,\iota)$ and the map $\td\iota : \IR_+ \times N \to M$ from  Lemma~\ref{Lem_MM_regular_rep}\ref{Lem_MM_regular_rep_c}.
Then $\td\iota(  \IR_+ \times N) \subset M$ is dense.
Moreover, if $\inf_M R \geq -2 + \eps$, then we have for $0 < r_1 \leq r_2$
\[ \big|M \setminus \td\iota ( (r_1,\infty) \times N ) \big|_g \leq \Big( \frac{r_1}{r_2} \Big)^{2\eps} \big|M \setminus \td\iota ( (r_2,\infty) \times N ) \big|_g. \]
If we even have $R \geq 0$, then for any $r > 0$
\[ \big|M \setminus \td\iota ( (r,\infty) \times N ) \big|_g \leq r^4 \big| (0,1) \times N \big|_{\gamma}. \]
\end{Lemma}

\begin{proof}
By Theorem~\ref{Thm_soliton_PSC}, and the asymptotics in Definition~\ref{Def_MM}\ref{Def_MM_3} we may always assume that $R > - 2 + \eps$ for some $\eps > 0$.
Consider the compact sets $K_r := M \setminus \td\iota ( (r,\infty) \times N )$.
If $(\phi_t : M \to M)_{t \in \IR_+}$ denotes the flow of $t^{-1} V$, then we have $\phi_t (K_r) = K_{ t^{-1/2} r }$, which implies via $-\DIV V = -\frac12 \tr (\LL_V g) =  2 +  R \geq \eps$ that
\[ \frac{d}{dt} \big|K_{ t^{-1/2} r } \big|_g 
= t^{-1} \int_{K_{ t^{-1/2} r }} (\DIV V) dg 
\leq - \eps t^{-1} |K_{ t^{-1/2} r }|_g . \]
Integrating this inequality implies the first statement.
For the second statement, we set $\eps = 2$ and observe that
\[ \big|K_{  r } \big|_g \leq \lim_{t \to 0} t^{2} \big|K_{t^{-1/2} r } \big| _g = r^4 | (0,1) \times N |_{\gamma}. \qedhere \]
\end{proof}
\bigskip

\subsection{A metric on $\MM$} \label{subsec_metric_on_MM}
Fix an ensemble $(M,N,\iota)$ and an integer $k^* \geq 2$ and set in the following $\MM := \MM^{k^*}(M,N,\iota)$.
We will now define a metric on $\MM$.

\begin{Definition} \label{Def_MM_metric}
We define a map $d : \MM \times \MM \to [0,\infty)$ as follows.
Let $[(g_i,V_i,\gamma_i)] \in \MM$, $i=1,2$.
For each $i = 1,2$ let $\td\iota_i : \IR_+ \times N \to M$ be the unique map from Lemma~\ref{Lem_MM_regular_rep}\ref{Lem_MM_regular_rep_c} with the property that $\td\iota_i^* V_i = -\frac12 r \partial_r$ and $
\td\iota_i = \iota$ on $(r_0, \infty) \times N$ for some  $r_0 > 1$.
Using the function $F(x) := \frac{x}{1+x}$, we then define
\begin{multline*}
 d\big([(g_1,V_1,\gamma_1)],[(g_2,V_2,\gamma_2)] \big)
:= \sum_{j=1}^\infty 2^{-j} F \big( d'_j \big([(g_1,V_1,\gamma_1)],[(g_2,V_2,\gamma_2)] \big)\big) \\ + \big| \max |{\Rm}_{g_1}|_{g_1} -\max |{\Rm}_{g_2}|_{g_2} \big| + \Vert \gamma_1 - \gamma_2 \Vert, 
\end{multline*}
where the last norm is the Banach norm on $T\GenCONE^{k^*}(N)$ and
\[ d'_j \big([(g_1,V_1,\gamma_1)],[(g_2,V_2,\gamma_2)] \big) := \sup_{x,y \in [j^{-1},j] \times N} 
|d_{g_1}(\td\iota_1(x),\td\iota_1(y))-d_{g_2}(\td\iota_2(x),\td\iota_2(y))|. \]
\end{Definition}
\medskip

The following lemma states that $d$ defines a metric on $\MM$.

\begin{Lemma}
$d$ is well defined, $(\MM,d)$ is a metric space and $\Pi$ is continuous (even $1$-Lipschitz) with respect to $d$.
\end{Lemma}

\begin{proof}
To see that $d$ is well defined, suppose that $[(g_i, V_i, \gamma_i)] = [(g'_i, V'_i, \gamma_i)]$, $i=1,2$, and fix $C^1$-dif\-feo\-mor\-phisms $\psi_i : M \to M$ that equal the identity outside of some compact subset such that $\psi_i^* g_i = g'_i$ and $\psi_i^* V_i = V'_i$.
Let $\td\iota_i, \td\iota'_i$ be the corresponding maps from Definition~\ref{Def_MM_metric}.
Then
\[ \td\iota_i^* V_i  = -\tfrac12 r \partial_r, \qquad (\psi \circ \td\iota'_i)^* V_i = (\td\iota'_i)^* \psi^*_i V_i =(\td\iota'_i)^* V'_i = -\tfrac12 r \partial_r \]
and on $(r_0, \infty) \times N$ for some $r_0 > 1$ we have $\td\iota_i = \iota = \psi_i \circ \iota = \psi_i \circ\td\iota'_i$.
This shows that $\td\iota_i = \psi \circ\td\iota'_i$ on $\IR_+ \times N$ and therefore we have for any $x,y \in \IR_+ \times N$
\[ d_{g_i} (\td\iota_i(x),\td\iota_i(y))
= d_{g_i} \big( \psi_i(\td\iota'_i(x)), \psi_i(\td\iota'_i(y)) \big)
= d_{\psi^*_i g_i} \big( \td\iota'_i(x), \td\iota'_i(y) \big) = d_{g'_i} \big( \td\iota'_i(x), \td\iota'_i(y) \big). \]
So $d'_j$ and thus $d$ are well defined.

To see that $(\MM,d)$ is a metric space, observe first that the $d'_j$ satisfy the triangle inequality.
So since $F(0) = 0$ and $F'' <0$, the same is true for $d$.
To prove the positive definiteness, suppose that $d([(g_1,V_1,\gamma_1)],[(g_2,V_2,\gamma_2)] ) = 0$.
It follows that
\[ d_{g_1}(\td\iota_1(x),\td\iota_1(y)) = d_{g_2}(\td\iota_2(x),\td\iota_2(y)) \qquad \text{for all} \quad x,y \in \IR_+ \times N. \]
So if we set $\psi' := \td\iota_2 \circ \td\iota^{-1}_1 : \td\iota_1 ( \IR_+ \times N) \to \td\iota_2 ( \IR_+ \times N)$, then $(\psi')^* g_2 = g_1$ and $(\psi')^* V_2 = (\td\iota_1)_* \td\iota_2^* V_2 =  (\td\iota_1)_*  (-\frac12 r \partial_r) = V_1$.
Moreover, $\psi'$ and its inverse map Cauchy sequences to Cauchy sequences.
So by Lemma~\ref{Lem_Im_iota_dense} we can extend $\psi'$ to a map $\psi : M \to M$ with $\psi^* g_2 = g_1$ and $\psi^* V_2 = V_1$.
Since $\td\iota_1 = \td\iota_2 = \iota$ on $(r_0,\infty) \times N$ for some $r_0 \geq 1$, we know that $\psi$ agrees with the identity outside some compact subset.
Hence $[(g_1,V_1,\gamma_1)]=[(g_2,V_2,\gamma_2)]$.

The continuity of $\Pi$ follows directly from the definition of $d$.
\end{proof}
\bigskip

\subsection{The map $T$} \label{subsec_map_T}
An ensemble $(M,N,\iota)$ allows us to define a map $T$, which relates the set of cone deformations in $T \GenCONE^{k^*}(N)$ with metric deformations on $M$.

\begin{Definition} \label{Def_T}
Fix some smooth cutoff function $\eta : (1,\infty) \to [0,1]$ such that $\eta \equiv 0$ on $(1,2)$ and $\eta \equiv 1$ on $[3,\infty)$.
If $(M,N,\iota)$ is an ensemble and $0 \leq k^* \leq \infty$, then we define
\[ T : T\GenCONE^{k^*} (N) \longrightarrow C^{k^*} (M; S^2 T^*M), \qquad \gamma \longmapsto \iota_* ( \eta(r) \gamma ), \]
where $r$ denotes the coordinate on the $(1,\infty)$ factor and the tensor $\gamma$ is restricted from $\IR_+ \times N$ to $(1,\infty) \times N$.
\end{Definition}

We remark that the precise choice of the function $\eta$ will be unimportant in the following and the function $\eta$ will not be used anymore.
We will only use the fact that $T$ is linear, that $T\gamma = \iota_* \gamma$ on the complement of some compact subset and that $T$ composed with the restriction to any compact subset is a bounded operator with respect to the $C^{k^*}$-norm.

We will sometimes need the following lemma.

\begin{Lemma} \label{Lem_nabf_Tgamma}
Let $(M, N, \iota)$ be an ensemble and $k^* \geq 4$ and let $(g,V,\gamma)$ be a $C^{k^*-2}$-regular representative of some $p \in \MM^{k^*}(M,N,\iota)$.
Then there is a constant $C < \infty$ such that for any $\gamma' \in T\GenCONE^{k^*}(N)$ and on $(1,\infty) \times N$
\begin{alignat*}{2}
 \big| \nabla^m (T(\gamma')) \big| \circ \iota &\leq C r^{-m} \Vert \gamma' \Vert & \qquad &\text{for} \quad m = 0, \ldots, k^*-4, \\
 \big| \nabla^m \big( \nabla_{V}( T(\gamma')) \big) \big| \circ \iota &\leq C r^{-2-m} \Vert \gamma' \Vert  & \qquad &\text{for} \quad m = 0, \ldots, k^*-5.
\end{alignat*}
\end{Lemma}
\bigskip

\begin{proof}
Note that $\nabla^\gamma_{\partial_r} \gamma' = 0$.
So for sufficiently large $r_0 \geq 1$ we have on $(r_0, \infty) \times N$
\[ \iota^* \nabla^m (T(\gamma')) = \nabla^{\iota^* g, m} (\iota^* T(\gamma'))
= \nabla^{\iota^* g, m} \gamma',
\]
\begin{multline} \label{eq_iotanabm}
 \iota^* (\nabla^m (\nabla_{V} T( \gamma' )) )
= \nabla^{\iota^* g, m} (\nabla^{\iota^* g}_{\iota^* V} (\iota^* T(\gamma')))
= - \tfrac12\nabla^{\iota^* g, m} (\nabla^{\iota^* g}_{ r \partial_r} \gamma') \\
=  -\tfrac12 \nabla^{\iota^* g, m} (\nabla^{\iota^* g}_{  r \partial_r} \gamma' -\nabla^{\gamma}_{  r \partial_r} \gamma' ). %
\end{multline}
Moreover,
\[ |\nabla^{\gamma, m} \gamma' | \leq C r^{-m} \]
and the difference of the two covariant derivatives in \eqref{eq_iotanabm} can be computed in terms of $\nabla^\gamma (\iota^* g)$ as follows:
\begin{equation} \label{eq_diff_connections}
 \nabla^{\iota^* g}_{  r \partial_r} \gamma' -\nabla^{\gamma}_{  r \partial_r} \gamma' =  (\iota^* g)^{-1}  * \nabla^\gamma (\iota^* g) * (r\partial_r) * \gamma', 
\end{equation}
where $(\iota^* g)^{-1}$ denotes the dual $(2,0)$-tensor.
By Lemma~\ref{Lem_MM_regular_rep}\ref{Lem_MM_regular_rep_b} we have for all $m = 0, \ldots, k^* - 4$
\begin{equation} \label{eq_iotag_minus_gamma}
 \big| \nabla^{\gamma,m} \big( \iota^* g - \gamma \big) \big| \leq C r^{-2-m}, 
\end{equation}
where $C$ is some generic constant.
Hence by differentiating \eqref{eq_diff_connections}, we obtain the following bound for $m = 0, \ldots, k^*-5$
\[ \big| \nabla^{\gamma, m} \big( \nabla^{\iota^* g}_{  r \partial_r} \gamma' -\nabla^{\gamma}_{  r \partial_r} \gamma' \big) \big| \leq C r^{-2-m} \Vert \gamma' \Vert. \]
Due to \eqref{eq_iotag_minus_gamma}, this implies that we have a similar bound with respect to the $\nabla^{\iota^* g}$-derivatives:
\[ \big|\nabla^{\iota^* g, m} \gamma' \big| \leq C r^{-m} , \qquad
 \big| \nabla^{\iota^* g, m} \big( \nabla^{\iota^* g}_{  r \partial_r} \gamma' -\nabla^{\gamma}_{  r \partial_r} \gamma' \big) \big| \leq C r^{-2-m} \Vert \gamma' \Vert. \]
Combining this with \eqref{eq_iotanabm} finishes the proof.
\end{proof}
\bigskip

\subsection{Convergence and closeness in $\MM$}
The following proposition shows that convergence within the space $\MM^{k^*}(M,N,\iota)$ is essentially equivalent to  $C^{k^*-5}_{\loc}$-convergence for appropriately chosen representatives.

\begin{Proposition} \label{Prop_converging_representatives}
Consider an ensemble $(M,N,\iota)$ and an integer $k^* \geq 7$.
Suppose that $p_i \to p_\infty \in \MM := \MM^{k^*}(M,N,\iota)$ and let $(g_\infty, V_\infty, \gamma_\infty)$ be a $C^{k^*-5}$-regular representative of $p_\infty$.
Then we can find $C^{k^*-5}$-regular representatives $(g_i, V_i, \gamma_i)$ of $p_i$ such that
\begin{equation} \label{eq_conv_giVi}
 g_i \xrightarrow[i\to\infty]{\quad C^{k^*-5}_{\loc} \quad} g_\infty, \qquad  V_i \xrightarrow[i\to\infty]{\quad C^{k^*-5}_{\loc} \quad} V_\infty 
\end{equation}
and such that $\iota^* V_i =- \frac12 r \partial_r$ on $(r_0,\infty) \times N$ for some uniform $r_0 > 2$.
\end{Proposition}

\begin{Remark}
Again, the regularity properties in this proposition are not optimal. 
\end{Remark}

Combined with Lemma~\ref{Lem_iota}, we even obtain uniform convergence with respect to a decaying weight at infinity:
\begin{Corollary} \label{Cor_conv_weighted}
Let $\mathbf{r} : M \to \IR_+$ be a smooth function such that $\mathbf{r} \circ \iota = r$ on $(r_0,\infty) \times N$ for some $r_0 \geq 1$.
Then in the setting of Proposition~\ref{Prop_converging_representatives} we even have for any $\eps > 0$ and $m = 0, \ldots, k^*-5$
\begin{equation} \label{eq_conv_decaying_weight}
 \sup_M \mathbf{r}^{2-\eps} |\nabla^{g_\infty,m} (g_i - g_\infty - T(\gamma_i - \gamma_\infty))|_{g_\infty} \xrightarrow[i\to\infty]{} 0. 
\end{equation}
\end{Corollary}

\begin{Remark}
Expressed in terms of the norms introduced in Subsection~\ref{subsec_list_Holder} and \eqref{eq_f_asymptotics_iota}, the convergence \eqref{eq_conv_decaying_weight} implies that in the case in which $p_\infty \in \MMgrad^{k^*} (M,N,\iota)$ we have
\begin{equation} \label{eq_Rmk_conv_giginfty}
 \big\Vert g_i - g_\infty - T(\gamma_i - \gamma_\infty) \big\Vert_{C^{k^*-5}_{-2+\eps}} \xrightarrow[i\to\infty]{} 0. 
\end{equation}
\end{Remark}

\begin{Remark}
It follows from \eqref{eq_ovg_asymp_2} in Lemma~\ref{Lem_iota} that we can even take $\eps = 0$ or even $\eps > -2$ in \eqref{eq_conv_decaying_weight} as long as $m \leq k^*-6$.
Moreover, by computing more precise asymptotic expansions of $\iota^* g$, we may choose $\eps$ arbitrarily small, as long as $k^*- m$ is large enough.
However, Corollary~\ref{Cor_conv_weighted} and \eqref{eq_Rmk_conv_giginfty} will be enough for our purposes.
\end{Remark}

The proof of Proposition~\ref{Prop_converging_representatives} is based on the following technical lemma, which we will also invoke in Section~\ref{sec_properness} to establish properness of the map $\Pi|_{\MMgeqgrad}$.

\begin{Lemma} \label{Lem_good_reps_technical}
Consider an ensemble $(M,N,\iota)$, an integer $k^* \geq 7$ and a sequence $p_i = [(g_i, V_i, \gamma_i)] \in \MM := \MM^{k^*}(M,N,\iota)$ such that $(g_i, V_i,\gamma_i)$ are $C^{k^*-2}$-regular representatives with the property that $\iota^* V_i = - \frac12 r \partial_r$ on $(2,\infty) \times N$; these exist due to Lemma~\ref{Lem_MM_regular_rep}\ref{Lem_MM_regular_rep_a}.
Suppose that:
\begin{enumerate}[label=(\roman*)]
\item \label{Lem_good_reps_technical_i} $\gamma_i \to \gamma_\infty \in \GenCONE^{k^*}(N)$.
\item \label{Lem_good_reps_technical_ii} For some $y_0 \in M$ we have pointed Cheeger-Gromov convergence
\begin{equation} \label{eq_CG_MgVy}
 (M, g_i, V_i, y_0) \xrightarrow[i\to\infty]{\quad \textrm{\rm CG} \quad} (M_\infty, g_\infty, V_\infty, y_\infty), 
\end{equation}
where $M_\infty$ is an orbifold with isolated singularities.
By this we mean that there are open subsets $U_1 \subset U_2 \subset \ldots \subset M_\infty$ with $\bigcup_{i=1}^\infty U_i = M_\infty$ and $C^1$-diffeomorphisms $\phi_i : U_i \to \phi_i(U_i) \subset M$ with $\phi_i(y_\infty) = y_0$ such that
\begin{equation} \label{eq_phi_gi_to_g_infty}
 \phi_i^* g_i \xrightarrow[i\to\infty]{\quad C^{\infty}_{\loc} \quad} g_\infty, \qquad \phi_i^* V_i \xrightarrow[i\to\infty]{\quad C^{\infty}_{\loc} \quad} V_\infty. 
\end{equation}
\item \label{Lem_good_reps_technical_iii} For any $r > 0$ the diameters $\diam(M \setminus \iota((r,\infty) \times N), g_i)$ are uniformly bounded.
\end{enumerate}
Then, after passing to a subsequence, we can find representatives $p_i = [(g'_i, V'_i, \gamma_i)]$ for $i \leq \infty$ such that the assertions of Proposition~\ref{Prop_converging_representatives} hold for $(g_i, V_i)$ replaced with $(g'_i, V'_i)$.
Moreover, we have $p_i \to p_\infty$.
\end{Lemma}

\begin{proof}
Recall from Lemma~\ref{Lem_MM_regular_rep} that we have local uniform $C^{k^*-4}$-bounds on $g_i$ restricted to $\iota((R_0, \infty) \times N)$ for some uniform $R_0 > 1$.

We first argue that the maps $\phi_i$ restricted the complement of a compact subset $K \subset M_\infty$ subsequentially converge to a map $\phi_\infty : M_\infty \setminus K \to M$.
To see this, note that by Assumption~\ref{Lem_good_reps_technical_iii} the set $M \setminus \iota ( (2,\infty) \times N)$ is contained in $g_i$-balls around $\phi_i(y_\infty) = y_0$ of uniform radius.
Therefore, after possibly increasing $R_0$ there is a compact subset $K \subset M_\infty$ such that for large $i$ we have for some $r_i \to \infty$
\[ \iota((2R_0,r_i) \times N) \subset \phi_i (M_\infty \setminus K) \subset \iota((R_0,\infty) \times N). \]
Moreover, for any $y' \in M_\infty$ the sequence $\phi_i(y')$ must be bounded.
Since we have local uniform $C^{k^*-4}$-bounds on both $g_i$ restricted to $\iota((R_0,\infty))$ and $\phi_i^* g_i$, we obtain local uniform $C^{k^*-3}$-bounds on the maps $\phi_i$.
So by Arzela-Ascoli there is a subsequence such that for some sequence $r_i \to \infty$
\begin{equation} \label{eq_beta_conv}
 \phi_i |_{M_\infty \setminus K} \xrightarrow[i\to\infty]{\quad C^{k^*-4}_{\loc} \quad} \phi_\infty, 
\end{equation}
where $\phi_\infty : M \setminus K \to \iota((R_0,\infty) \times N)$ is an embedding of regularity $C^{k^*-4}$ whose image contains.

Due to \eqref{eq_beta_conv} we may, for each large $i$, construct a $C^{k^*-4}$-regular interpolation $\phi'_i : M_\infty \to M$ between $\phi_i$ and $\phi_\infty$ such that for some slightly larger compact subset $K \subset K' \subset M_\infty$ we have $\phi'_i = \phi_\infty$ on $M_\infty \setminus K'$ and $\phi'_i = \phi_i$ on a neighborhood of $K$ and such that \eqref{eq_beta_conv} still holds with $\phi_i$ replaced by $\phi'_i$.
Then for large $i$, $\phi'_i$ is a $C^{k^*-4}$-regular diffeomorphism and we have for some uniform $R'_0 > 2 R_0$
\[ \iota((R'_0,\infty) \times N) \subset \phi'_i (M_\infty \setminus K'). \]
So after passing to a subsequence, the map $\psi_i := \phi'_i \circ (\phi'_1)^{-1} : M \to M$ is a diffeomorphism that equals the identity on $\iota((R'_0,\infty) \times N)$ for any $i$.
Set, for $i \leq \infty$,
\[ g'_i := \psi_i^* g_i, \qquad V'_i := \psi_i^* V_i. \]
Then $[(g'_i,V'_i,\gamma_i)] = [(g_i,V_i,\gamma_i)]$. and in a neighborhood of $\phi'_1(K)$ we have
\begin{align*}
 g'_i = \psi_i^* g_i 
= (\phi'_1)_* (\phi'_i)^* g_i
= (\phi'_1)_* (\phi_i)^* g_i 
&\xrightarrow[i\to\infty]{\quad C^{k^*-5}_{\loc} \quad }
(\phi'_1)_* g_\infty 
=g'_\infty  \\
V'_i = \psi_i^* V_i 
= (\phi'_1)_* (\phi'_i)^* V_i
= (\phi'_1)_* (\phi_i)^* V_i 
&\xrightarrow[i\to\infty]{\quad C^{k^*-5}_{\loc} \quad }
(\phi'_1)_* V_\infty 
=V'_\infty.
\end{align*}
On the other hand on $M \setminus \phi'_1(K)$
\begin{align*}
 g'_i = \psi_i^* g_i = (\phi'_1)_* (\phi'_i)^* g_i = (\phi'_1)_* (\phi'_i)^* (\phi_i)_* \phi_i^* g_i&\xrightarrow[i\to\infty]{\quad C^{k^*-5}_{\loc} \quad }  (\phi'_1)_* \phi_\infty^* (\phi_\infty)_* g_\infty  = (\phi'_1)_* g_\infty = g'_\infty , \\
  V'_i = \psi_i^* V_i = (\phi'_1)_* (\phi'_i)^* V_i = (\phi'_1)_* (\phi'_i)^* (\phi_i)_* \phi_i^* V_i &\xrightarrow[i\to\infty]{\quad C^{k^*-5}_{\loc} \quad }   (\phi'_1)_* \phi_\infty^* (\phi_\infty)_* V_\infty = (\phi'_1)_* V_\infty = V'_\infty.
\end{align*}

It remains to show that $p_i \to p_\infty$.
To see this, consider the maps $\td\iota'_i : \IR_+ \times N \to M$ from Lemma~\ref{Lem_MM_regular_rep} with respect to the representatives $(g'_i, V'_i, \gamma_i)$ for $i \leq \infty$ and note that $\td\iota'_i = \iota$ on $(r_1,\infty) \times N$.
Since $s \mapsto \td\iota'_i(e^{s/2},z)$ are trajectories of $-V'_i$, we obtain that $\td\iota'_i \to \td\iota'_\infty$ in $C^0_{\loc}$.
This proves that $p_i \to p_\infty$ according to Definition~\ref{Def_MM_metric}.
\end{proof}
\bigskip

\begin{proof}[Proof of Proposition~\ref{Prop_converging_representatives}.]
Use Lemma~\ref{Lem_MM_regular_rep}\ref{Lem_MM_regular_rep_a} to choose $C^{k^*-2}$-regular representatives $(g_i, V_i,\gamma_i)$ with the property that $\iota^* V_i = - \frac12 r \partial_r$ on $(2,\infty) \times N$ and consider the maps $\td\iota_i : \IR_+ \times N \to M$, $i \leq \infty$ with $\td\iota = \iota$ on $(2, \infty) \times N$, from Lemma~\ref{Lem_MM_regular_rep}\ref{Lem_MM_regular_rep_c}.

Next, observe that it is enough to show that the convergence \eqref{eq_conv_giVi} holds for \emph{some} subsequence of any \emph{given} subsequence of the original sequence (and for possibly different representatives $(g_i,V_i,\gamma_i)$ of $p_i$).
So let us now pass to an arbitrary subsequence of the original sequence.
We will verify the assumptions of Lemma~\ref{Lem_good_reps_technical}.
This will then imply that \eqref{eq_conv_giVi} holds for a subsequence, which will finish the proof.
Assumption~\ref{Lem_good_reps_technical_i} holds trivially.

By the convergence $p_i \to p_\infty$, the metrics $g_i$ have uniformly bounded curvature.
So due to Shi's estimates applied to the flows associated with each soliton metric, we obtain uniform bounds on the curvature derivatives of these metrics (possibly after choosing a different smooth structure, see Lemma~\ref{Lem_soliton_smooth}).
By Lemma~\ref{Lem_iota} we moreover have uniform lower injectivity radius bounds at some point $y_0 \in \iota ((R_0,\infty) \times N)$ for some uniform $R_0 > 1$.
So, after passing to a subsequence, we have pointed Cheeger-Gromov convergence
\begin{equation*} %
 (M,g_i, y_0) \xrightarrow[i \to \infty]{\quad\text{CG}\quad} (M_\infty, g^*_\infty, y_\infty), 
\end{equation*}
where $M_\infty$ is a smooth orbifold (see \cite{Fukaya_1986, Lu_2001} for a discussion of Cheeger-Gromov compactness for orbifolds; note that by Definition~\ref{Def_ensemble}\ref{Def_ensemble_3} all orbifolds have manifold covers).
Due to the convergence $p_i \to p_\infty$ and Lemma~\ref{Lem_Im_iota_dense}, we moreover have pointed Gromov-Hausdorff convergence of
\begin{equation} \label{eq_GH_conv_Wi}
 (W_i := \td\iota_i ((r_i,\infty) \times N)), g_i, y_0) \xrightarrow[i \to \infty]{\quad\text{GH}\quad} (M, g_\infty, y_0) 
\end{equation}
for some $r_i \to 0$.
This implies that there is an isometric embedding of the pointed metric space $(M,d_{g_\infty},y_0)$ into $(M_\infty, d_{g^*_\infty}, y_\infty)$.
Since both spaces are complete and connected, it follows that they are isometric.
Thus we may even take $W_i = M$ in \eqref{eq_GH_conv_Wi}, which implies Assumption~\ref{Lem_good_reps_technical_iii} of Lemma~\ref{Lem_good_reps_technical}.

To verify Assumption~\ref{Lem_good_reps_technical_ii} of Lemma~\ref{Lem_good_reps_technical}, choose open subsets $U_1 \subset U_2 \subset \ldots \subset M_\infty$ with $\bigcup_{i=1}^\infty U_i = M_\infty$ and $C^1$-diffeomorphisms $\phi_i : U_i \to \phi_i(U_i) \subset M$ such that $\phi_i (y_\infty) = y_0$ and $\phi_i^* g_i \to g^*_\infty$ in $C^\infty_{\loc}$.
It remains to show that after passing to a subsequence we have local smooth convergence of the vector fields $\phi_i^* V_i$.
To do this, observe that by Lemma~\ref{Lem_iota} we have local uniform bounds on $|V_i|$ on $\iota ((R_0,\infty) \times N)$ for some uniform $R_0 > 0$.
We claim that we even have uniform bounds on $a_i := \max_{M \setminus \iota ((R_0,\infty) \times N)} |V_i|$.
Set $V^*_i := a_i^{-1} V_i$ and observe that by the soliton equation
\[ \triangle V_i + \Ric(V_i) 
= \DIV \LL_{V_i} g_i - \frac12 \nabla \tr \LL_{V_i} g_i
= 0, \]
so
\begin{equation*} %
\triangle V_i^* + \Ric(V^*_i) = 0.
\end{equation*}
Since $|V^*_i| \leq 1$  on $M \setminus \iota ((R_0,\infty) \times N)$, this implies local uniform higher derivative bounds on $V^*_i$ by elliptic regularity, whenever $a_i$ is bounded away from $0$ (which is the case the we are interested in).
Now if for a subsequence we had $a_i \to \infty$, then a subsequence of the vector fields $\phi_i^* V^*_i$ would converge to a vector field $V^*_\infty$ on $M_\infty$.
Since the $V_i$ were locally uniformly bounded on $\iota ((R_0,\infty) \times N)$, we would get $V^*_\infty = 0$ outside a compact set.
However, passing the equation $a_i^{-1} \Ric + \frac12 \LL_{V^*_i} g_i + \frac12 a_i^{-1} g_i = 0$ to the limit implies that $V^*_\infty$ is a Killing field of $g^*_\infty$.
Since Killing fields restricted to any geodesic are Jacobi fields, we therefore find that $V^*_\infty \equiv 0$, a contradiction.
So $a_i$ is uniformly bounded and by the same arguments we have local smooth convergence $\phi^*_i V_i \to V_\infty$ for some subsequence.
This concludes the proof of Assumption~\ref{Lem_good_reps_technical_ii} of Lemma~\ref{Lem_good_reps_technical}.

Consider now the representatives $(g'_i, V'_i,\gamma_i)$, $i \leq \infty$ from Lemma~\ref{Lem_good_reps_technical}.
If $g'_\infty = g_\infty$ and $V'_\infty = V_\infty$, then we are done.
Otherwise, choose a $C^1$-diffeomorphism $\psi : M \to M$ that equals the identity on the complement of a compact subset such that $\psi^* g'_\infty = g_\infty$ and $\psi^* V'_\infty = V_\infty$.
Since $g_\infty, g'_\infty$ are both of regularity $C^{k^*-5}$, the map $\psi$ is of regularity $C^{k^*-4}$ and we have
\[ \psi^* g_i \xrightarrow[i\to\infty]{\quad C^{k^*-5}_{\loc} \quad} g_\infty, \qquad  \psi^* V_i \xrightarrow[i\to\infty]{\quad C^{k^*-5}_{\loc} \quad} V_\infty. \]
So \eqref{eq_conv_giVi} holds after replacing $g_i, V_i$ with $\psi^* g_i, \psi^* V_i$.
\end{proof}
\bigskip

\section{Properness of $\Pi$ restricted to $\MMgeqgrad$} \label{sec_properness}
\subsection{Statement of the main result}
The main result of this section is the following.

\begin{Proposition} \label{Prop_properness}
Let $(M,N,\iota)$ be an ensemble and $k^* \geq 7$.
Consider the space $\MM := \MM^{k^*}(M, \lb N, \lb \iota)$ and the subset $\MMgeqgrad := \MMgeqgrad^{k^*}(M,N,\iota) \subset \MM$.
Then the following map is proper
\[ \Pi\big|_{\MMgeqgrad} : \MMgeqgrad \lto \CONEgeq^{k^*}(N). \]
\end{Proposition}

We will prove this proposition using Lemma~\ref{Lem_good_reps_technical}, which will reduce the problem to proving Cheeger-Gromov convergence and a uniform diameter bound.

As a byproduct we also obtain:

\begin{Proposition} \label{Prop_eucl_case}
Let $(M,N,\iota)$ be an ensemble and $k^* \geq 7$.
Suppose that $N \approx S^3/\Gamma$ and that $M$ has an orbifold cover such that $\iota$ lifts to a map from $(1,\infty) \times S^3$.
Consider the projection
\[ \Pi : \MM^{k^*}(M,N,\iota) \lto \GenCONE^{k^*}(N). \]
Let $\gamma_{\eucl} \in \CONE^{k^*}(N)$ be a locally Euclidean cone metric (so its link metric is isometric to a quotient of the standard round sphere).
Then 
\begin{equation} \label{eq_numberPiinvleq1}
 \# \big( \Pi^{-1}(\gamma_{\eucl}) \big) \leq 1 
\end{equation}
 and the following is true:
\begin{enumerate}[label=(\alph*)]
\item 
If $(M,N,\iota)=(\IR^4/\Gamma,S^3/\Gamma, \iota_\Gamma)$, where $\iota_\Gamma : [1,\infty) \times S^3/\Gamma \to \IR^4/\Gamma$ is the standard radial embedding and if $\gamma_{\eucl}$ is the standard Euclidean cone metric on $\IR^4/\Gamma$, then $\Pi^{-1} (\gamma_{\eucl}) \subset \MM^{k^*}(M,N,\iota)$ consists of a single element, which is represented by $(g_{\eucl}, V_{\eucl},\gamma_{\eucl})$, where $(\IR^4/\Gamma, g_{\eucl}, V_{\eucl} = -\tfrac12 r \partial_r)$ (see Subsection~\ref{subsec_basic_identities}) is the quotient of the standard Euclidean soliton.
\item If we have equality in \eqref{eq_numberPiinvleq1}, then $(M,N,\iota)$ is isomorphic to the standard ensemble $(\IR^4/\Gamma, \lb S^3/\Gamma, \lb \iota_\Gamma)$ in the sense that
 there are diffeomorphisms $\phi : M \to \IR^4 / \Gamma$ and $\psi : N \to S^3/\Gamma$ such that
\[ \iota_\Gamma \circ (\id_{(1,\infty)} ,\psi) = \phi \circ \iota, \]
\end{enumerate}
\end{Proposition}

\subsection{A non-collapsing property for expanding solitons} \label{subsec_soliton_noncollapsing}
In this subsection we establish non-collapsedness bounds for asymptotically conical expanding solitons.
These bounds only depend on certain geometric data of the asymptotic cones.
The following proposition is our main result.
It is kept more general than needed in this paper as it may be of independent interest.

\begin{Proposition}\label{Prop_NLC_bound}
Let $(M,g,V)$ be an $(n\geq 3)$-dimensional, expanding soliton on an orbifold with isolated singularities such that $g$ is of regularity $C^2$ and $V$ is of regularity $C^1$.
Let $(N,h)$ be a closed $(n-1)$-dimensional Riemannian manifold such that $h$ is of regularity $C^2$.
Suppose that $(M,g)$ is 
asymptotic to the conical metric $\gamma := dr^2 + r^2 h$ on $\IR_+ \times N$ in the following sense:
For any $\eps > 0$ there is a $D_\eps > 0$ and a diffeomorphism onto its image $\psi_\eps : (D_\eps,\infty) \times N \to M$ such that $M \setminus \Image \psi_\eps$ is compact and
\begin{equation} \label{eq_psi_eps_char}
 e^{-2\eps} \gamma \leq \psi_\eps^* g  \leq e^{2\eps} \gamma, \qquad  R_g \circ \psi_\eps \geq  R_{\gamma} - \frac{\eps}{r^2},
\end{equation}
where $r$ denotes the first coordinate of $(D_\eps, \infty) \times N$.
Suppose also $\gamma$ and $g$ have non-negative scalar curvature.

Then we have the following bounds involving Perelman's $\mu$ and $\nu$-functionals (see the discussion after this proposition in which these functionals are recalled and extended to the non-compact setting):
\begin{equation} \label{eq_nu_lower_prop}
 \nu[g] \geq \nu[\gamma] \geq \inf_{0 < \tau \leq 1}  \mu[h,\tau] - C(n).  
\end{equation}
If we have $|{\Rm}| \leq A$ on $(N,h)$ and $\inj(N,h) \geq A^{-1}$ for some $A < \infty$, then the right-hand side of \eqref{eq_nu_lower_prop} is bounded from below by $- C(n,A)$.

Moreover, there is a constant $c > 0$, which only depends on $n$ and a lower bound on $\nu[g]$ (for example the right-hand side in \eqref{eq_nu_lower_prop}) such that the following is true for any $r > 0$ and $y \in M$:
\[ R \leq r^{-2} \quad \text{on} \quad B(y,r) \qquad \Longrightarrow \qquad \vol_g (B(y,r)) \geq c r^n. \]
\end{Proposition}

Let us briefly recall Perelman's $\nu$ and $\mu$-functionals \cite{Perelman1} and discuss the modifications that have to be made in the non-compact setting.
Let, for a moment, $(M,g)$ be an arbitrary $n$-dimensional and possibly incomplete Riemannian orbifold with isolated singularities, where $g$ is assumed to have regularity $C^2$.
In the remainder of this chapter we will denote by $f : M \to (-\infty, \infty]$ an arbitrary function ($f$ is not the soliton potential!), in order to stay close to Perelman's notation.
More specifically, we allow $f$ to attain the value $\infty$, but we require that $f$ is continuous and that $e^{-f} \in C^\infty_c (M)$ is smooth and has compact support via the standard convention $e^{-\infty} = 0$.
Note that then $|\nabla f|^2 e^{-f} , f e^{-f} \in C^0_c (M)$, so the integral
\[ \WW[g,f,\tau] = \int_M \big( \tau (R + |\nabla f|^2) + f - n \big)(4\pi \tau)^{-n/2} e^{-f} dg \in (-\infty, \infty] \]
is well defined.
The functional $\mu[g,\tau]$ is defined as usual by taking the infimum of $\WW[g,f,\tau]$ over all $e^{-f} \in C^\infty_c (M)$ with $\int_M (4\pi\tau)^{-n/2} e^{-f} dg = 1$.
Note that, equivalently, we may also define $\mu[g,\tau]$ by taking the infimum over all $e^{-f}$ that have sufficient decay on complete ends of $M$; for example if an end has bounded curvature and Euclidean volume growth, then it suffices to assume exponential decay of $e^{-f}$ instead of compact support within this end.
The definition $\nu[g] := \inf_{\tau > 0} \mu[g,\tau]$ carries over from the compact case.
We will also use Perelman's $\mathcal{F}$-functional
\[ \mathcal{F}[g,f] = \int_M (R+ |\nabla f|^2) e^{-f} dg \]
on a compact manifold $M$ and define $\lambda[g]$ to be the infimum of $\mathcal{F}[g,f]$ over all $f \in C^\infty(M)$ with $\int_M  e^{-f} dg = 1$.
\medskip

The following lemma is a slight improvement of \cite[Theorem 1.3]{Ozuch} using similar methods, which also shows that the $\nu$-functional of a cone metric $\gamma = dr^2 + r^2 h$ satisfies $\nu[\gamma] = -\infty$ if $\lambda[h] < n-2$. The difference is that we additionally bound $\nu[\gamma]$ from below based on certain geometric data of the link metric $h$ when $\lambda[h] \geq n-2$.
We remark that Lemma~\ref{Lem_some_lower_nu_bound} will only be needed in the proof of Proposition~\ref{Prop_NLC_bound}, where we assumed that $\gamma$ has non-negative scalar curvature and consequently $\lambda[h] \geq \inf_N R \geq (n-1)(n-2) \geq n-2$.
In fact, some computations in the proof of Lemma~\ref{Lem_some_lower_nu_bound} can be slightly simplified if we impose the additional assumption that $\gamma$ has non-negative scalar curvature.
We have, however, decided to state Lemma~\ref{Lem_some_lower_nu_bound} in its most general form, as it may be of independent interest.

\begin{Lemma} \label{Lem_some_lower_nu_bound}
Let $(N,h)$ be a closed $(n-1)$-dimensional Riemannian manifold, where $n \geq 2$ and $h$ is assumed to have regularity $C^2$.
Consider the cone metric $\gamma = dr^2 + r^2 h$ on $\IR_+ \times N$.
If  \[ \lambda[h] \geq n-2, \] then
\begin{equation} \label{eq_nu_bound_inf}
 \nu[\gamma] \geq \inf_{0 < \tau \leq 1}  \mu[h,\tau] - C(n) . 
\end{equation}
In particular, if $|{\Rm}| \leq A$ on $(N,h)$ and $\inj(N,h) \geq A^{-1}$, then we have
\begin{equation} \label{eq_nu_bound_simplified}
 \nu[\gamma] \geq  - C(n, A) > - \infty.
\end{equation}
Vice versa, if $\lambda[h] < n-2$, then $\nu[\gamma] = - \infty$.\end{Lemma}

We remark that due to scaling invariance we have $\mu[\gamma, \tau] = \nu[\gamma]$ for all $\tau > 0$.

\begin{proof}
Let $f : \IR_+ \times N \to (-\infty,\infty]$ such that $e^{-f}$ is smooth with compact support and
\begin{equation} \label{eq_ovM_int_1}
 \int_{\IR_+ \times N} (4\pi)^{-n/2} e^{-f} d\gamma = 1. 
\end{equation}
Define $F : \IR_+ \to (-\infty, \infty]$ via
\begin{equation} \label{eq_def_F}
 e^{-F(r)} :=  \int_N (4\pi)^{-(n-1)/2} e^{-f(r,\cdot)} dh. 
\end{equation}
So $e^{-F}$ has compact support and
\begin{equation} \label{eq_F_int_1}
\int_N (4\pi)^{-(n-1)/2} e^{-(f(r,\cdot)-F(r))} dh = 1, \qquad
 \int_0^\infty (4\pi)^{-1/2} e^{-F(r)} r^{n-1} dr = 1. 
\end{equation}
Next, we use separation of variables to compute
\begin{align}
 \WW[\gamma, f, 1]
&= \int_0^\infty \int_N \big( r^{-2} |\nabla^h f(r,\cdot)|^2 + (\partial_r f)^2 + r^{-2} R_h - (n-1)(n-2) r^{-2} \notag \\
&\qquad\qquad\qquad\qquad\qquad\qquad\qquad\qquad + f - n \big) (4\pi)^{-n/2} e^{-f} r^{n-1} dh \, dr \notag \\
&= (4\pi)^{-1/2} \int_0^\infty \bigg( \WW[h,f(r,\cdot)-F(r),r^{-2}] - \frac{(n-1)(n-2)}{r^{2}} + F(r)\bigg)  e^{-F(r)} r^{n-1}   dr \notag \\
&\qquad\qquad + \int_0^\infty  \int_N (\partial_r f)^2  (4\pi)^{-n/2} e^{-f(r,\cdot)} r^{n-1} dh \, dr - 1. \label{eq_WW_cone_1}
\end{align}
If $0 < r < 1$, then we can bound the $\WW$-functional in \eqref{eq_WW_cone_1} as follows, using \eqref{eq_F_int_1},
\begin{multline} \label{eq_WW_rl1}
 \WW[h,f(r,\cdot)-F(r),r^{-2}] = (r^{-2}-1) \mathcal{F}[h, f(r,\cdot)-F(r)](4\pi r^{-2})^{-(n-1)/2}  + \WW[h,f(r,\cdot)-F(r),1] r^{n-1}      \\
\geq \Big( (r^{-2}-1) \lambda[h] + \mu[h,1]  \Big) r^{n-1} e^{-F(r)} 
\geq \Big( \frac{n-2}{r^2} - (n-2)  + \mu[h,1]  \Big) r^{n-1} e^{-F(r)}.
\end{multline}
On the other, hand, if $r \geq 1$, then using the fact that
\[ \int_N (4\pi r^{-2})^{-(n-1)/2} e^{-(f(r,\cdot) - F(r) + (n-1) \log r)} dr = 1, \]
we obtain
\begin{multline} \label{eq_WW_rg1}
 \WW[h,f(r,\cdot)-F(r),r^{-2}]    
=  \Big( \WW\big[h,f(r,\cdot) - F(r) + (n-1) \log r, r^{-2} \big]  - (n-1) \log r \Big) r^{n-1}    \\
\geq  \Big( \mu[h,r^{-2}]  - (n-1) \log r \Big) r^{n-1} .
\end{multline}
Plugging \eqref{eq_WW_rl1}, \eqref{eq_WW_rg1} back into \eqref{eq_WW_cone_1}, using \eqref{eq_F_int_1} and applying Cauchy-Schwarz to the second integral in \eqref{eq_WW_cone_1} gives
\begin{multline} \label{eq_WW_in_F}
\WW[\gamma, f, 1]
\geq -C(n) + \inf_{0<\tau\leq 1} \mu[h,\tau] \\
+ \int_0^\infty \bigg({ -\frac{(n-2)^2}{r^2} + F(r) - (n-1)(\log r)_+ + (F')^2} \bigg)  (4\pi)^{-1/2}  e^{-F(r)}r^{n-1} dr.
\end{multline}
In order to establish a lower bound on the last integral, fix a smooth function $q : \IR_+ \to \IR_+$ such that $q(r) = r$ for $0 < r < 1$ and $q(r) = 1$ for $r \geq 2$.
In the following, we will use the rotationally symmetric, asymptotically cylindrical metric $g^* := dr^2 + q^2(r) g_{S^1}$ on $M^* := \IR^2$.
Set 
\[ H(r) := F(r) - (n-1) \log r + \log q(r) - \tfrac12 \log(4\pi) + \log(2\pi) . \]
Then \eqref{eq_F_int_1} implies
\begin{equation} \label{eq_eH_mass_1}
 \int_{M^*} (4\pi)^{-1}e^{-H} dg^* = \int_0^\infty (4\pi)^{-1} e^{-H(r)} \, 2\pi q(r) dr = 1 
\end{equation}
and, after integration by parts, the integral in \eqref{eq_WW_in_F} becomes with $C_0 =  \tfrac12 \log(4\pi) - \log(2\pi)$
\begin{align*}
\int_0^\infty &\bigg({ - \frac{(n-2)^2}{r^2} + H(r) + (n-1) \log r - \log q(r) - C_0 - (n-1) (\log r)_+} \\
&\qquad + (H'(r))^2 + 2 H'(r) \Big( \frac{n-1}r - \frac{q'(r)}{q(r)} \Big) + \Big( \frac{n-1}{r} - \frac{q'(r)}{q(r)} \Big)^2 \bigg) (4\pi)^{-1}e^{-H(r)} \, 2\pi q(r) dr \\
&=\int_0^\infty \bigg({ - \frac{(n-2)^2}{r^2} + \Big( \frac{n-1}{r} - \frac{q'(r)}{q(r)} \Big)^2 
+ \frac{2}{q(r)}  \Big( \frac{n-1}r q(r) - q'(r) \Big)'}\bigg) (4\pi)^{-1}e^{-H(r)} \, 2\pi q(r) dr \\
&\qquad + \int_0^\infty \Big( (n-1) \log r - (n-1) \log q(r) - (n-1) (\log r)_+ \Big) (4\pi)^{-1}e^{-H(r)} \, 2\pi q(r) dr\\
&\qquad + \int_{M^*} \big( |\nabla H|^2 + H + (n-2) \log q \big) (4\pi)^{-1} e^{-H} dg^* - C(n).
\end{align*}
Note that the term in the parenthesis in the first integrand on the right-hand side vanishes for $0 < r < 1$ and equals $\frac{-(n-2)^2 + (n-1)^2 - 2(n-1)}{r^2} \geq -1$ for $r \geq 2$.
So this term is bounded from below by a uniform constant for all $r > 0$.
Similarly, the term in the parenthesis of the second integral on the right-hand side vanishes on $(0,1] \cup [2, \infty)$, so it is uniformly bounded.
Combining this with \eqref{eq_eH_mass_1} implies that the integral in \eqref{eq_WW_in_F} can be bounded from below by
\begin{equation} \label{eq_int_H_tbb}
 \int_{M^*} \big( |\nabla H|^2 + H + (n-2) \log q \big) (4\pi)^{-1} e^{-H} dg^* - C(n). 
\end{equation}

It remains to bound the integral in \eqref{eq_int_H_tbb} from below under the assumption \eqref{eq_eH_mass_1}.
To do this, we define $u \in C^1_c (M^*)$ by
\begin{equation} \label{eq_u_eH}
 u^2 := (4\pi)^{-1} e^{-H}. 
\end{equation}
Then \eqref{eq_eH_mass_1} becomes $\int_{M^*} u^2 dg^* = 1$ and the integral in \eqref{eq_int_H_tbb} becomes
\[ \int_{M^*} \big( 4 |\nabla u|^2 - 2 u^2 \log u - u^2 \log(4\pi) + (n-2) u^2 \log q \big) dg^*.  \]
Using the Sobolev inequality \cite{Aubin_1976}, the H\"older inequality  and the fact that $\log u \leq u$, this integral can be bounded from below by
\begin{align*}
 c \bigg(& \int_{M^*} u^6 dg^* \bigg)^{1/3} 
- C \int_{M^*} u^2 dg^*
- 2 \int_{M^*} u^3 dg^*
- (n-2) \bigg( \int_{M^*} u^4 dg^* \bigg)^{1/2} \bigg( \int_{M^*} (\log q)^2 dg^* \bigg)^{1/2} \\
&\geq c \bigg( \int_{M^*} u^6 dg^* \bigg)^{1/3}  - C - 2 \bigg( \int_{M^*} u^6 dg^* \bigg)^{1/4} \bigg( \int_{M^*} u^2 dg^* \bigg)^{3/4} - C(n-2) \bigg( \int_{M^*} u^4 dg^* \bigg)^{1/2} \\
&\geq c \bigg( \int_{M^*} u^6 dg^* \bigg)^{1/3}  - C - 2 \bigg( \int_{M^*} u^6 dg^* \bigg)^{1/4} - C(n-2) \bigg( \int_{M^*} u^6 dg^* \bigg)^{1/4}\bigg( \int_{M^*} u^2 dg^* \bigg)^{1/4} \\
&\geq c \bigg( \int_{M^*} u^6 dg^* \bigg)^{1/3}  - C - 2 \bigg( \int_{M^*} u^6 dg^* \bigg)^{1/4} - C(n-2) \bigg( \int_{M^*} u^6 dg^* \bigg)^{1/4}
\geq - C(n).
\end{align*}
Here $c, C$ denote generic constants.
This finishes the proof of \eqref{eq_nu_bound_inf}.
The statement involving \eqref{eq_nu_bound_simplified} also follows via Lemma~\ref{Lem_mu_bounded_geometry}.

Now suppose that $\lambda[h] < n-2$ and choose $f' \in C^\infty(N)$ such that $\int_N (4\pi)^{-(n-1)/2} e^{-f'} dh = 1$ and $\mathcal{F}[h,f'] < (n-2 - \delta)(4\pi)^{-(n-1)/2}$ for some $\delta > 0$.
Set $S := \int_N f' \, (4\pi)^{-(n-1)/2} e^{-f'} dh$.
Let $F : \IR_+ \to (-\infty,\infty]$ be a function, which we will choose later, and set $f := F + f'$.
Then \eqref{eq_def_F} holds by definition and \eqref{eq_ovM_int_1} is equivalent to the second identity in \eqref{eq_F_int_1}; assume from now on that $F$ is chosen such that these identities hold.
The first equality in \eqref{eq_WW_cone_1} implies that 
\begin{multline*} \WW[\gamma, f, 1] \leq \int_0^\infty \Big(  \frac{n-2-\delta}{r^2} - \frac{(n-1)(n-2)}{r^{2}} + (F'(r))^2 + F(r) + S  - n \Big) (4\pi)^{-1/2} r^{n-1} e^{-F(r)}  dr. 
\end{multline*}
A similar computation as before shows that for $H(r) := F(r) - (n-2) \log r  -\frac12 \log(4\pi) + \log(2\pi)$ we have (set $q(r) = r$)
\[ \WW[\gamma, f, 1] \leq \int_0^\infty \Big({ - \frac{\delta}{r^2} + (H'(r))^2 + H(r) + (n-2) \log r }\Big) (4\pi)^{-1} e^{-H(r)} 2\pi r \, dr + C(n,S), \]
while \eqref{eq_ovM_int_1} is equivalent to
\[ \int_0^\infty (4\pi)^{-1} e^{-H(r)} 2\pi r \, dr = 1. \]

As before, using the radially symmetric function $u \in C^1_c(\IR^2 \setminus \{ 0 \})$ from \eqref{eq_u_eH}, this reduces our problem to showing that the integral
\begin{equation} \label{eq_R2_integral}
 \int_{\IR^2} \Big( -\frac{\delta}{r^2} u^2 + 4 |\nabla u|^2 - 2 u^2 \log u - u^2 \log(4\pi) + (n-2) u^2 \log r \Big) dg_{\eucl} 
\end{equation}
can be made arbitrarily small while keeping the integral
\[ \int_{\IR^2} u^2 \, dg_{\eucl} = 1. \]
To see this, fix a sequence of compactly supported, radial functions $u_i \in C^\infty_c(\IR^2 \setminus \{ 0 \})$ such that for some uniform $A < \infty$ we have
\[ \int_{\IR^2}  u_i^2 \,  dg_{\eucl} = 1, \qquad \sup u_i + \int_0^\infty |\nabla u_i|^2  \, dg_{\eucl} \leq A, \qquad \supp u_i \subset B(\vec 0, A) \]
and such that we have local uniform convergence $u_i \to a > 0$ for on $(0 , r_0)$ for some $r_0 > 0$.
Then
\[ \int_{\IR^2} \frac{\delta}{r^2} u^2 \, dg_{\eucl} \xrightarrow[i \to \infty]{}  \infty. \]
Since $u^2 \log u \geq - C$, this shows that \eqref{eq_R2_integral}, for $u = u_i$ and $i \to \infty$, diverges to $-\infty$.
This finishes the proof of the lemma.
\end{proof}
\bigskip

The next lemma shows that the $\nu$-entropy of an expanding soliton is bounded from below by the $\nu$-entropy of its asymptotic cone. 
The idea underlying this result is that the asymptotic cone can be thought of as the initial condition of the associated Ricci flow, so the bound essentially follows from the monotonicity of the $\nu$-functional.

\begin{Lemma} \label{Lem_nu_cone_exp}
In the setting of Proposition~\ref{Prop_NLC_bound} we have
\[ 0 \geq \nu[g]  \geq \nu[\gamma]. \]
\end{Lemma}

\begin{Remark} 
In fact \cite[Corollary 1.7]{Ozuch} also implies that $\nu[g]\leq\nu[\gamma]$, so the second inequality above is actually an equality.
\end{Remark}

\begin{Remark} \label{Rmk_scal_bound_necessary}
Lemma~\ref{Lem_nu_cone_exp} is also the only place where the non-negative scalar curvature assumption on $\gamma$ and $g$ from Proposition~\ref{Prop_NLC_bound} is essential. (It is also used in the proof of Proposition~\ref{Prop_curv_diam_bound}, but could be removed by a modification of the argument.)
It is an interesting question whether Lemma~\ref{Lem_nu_cone_exp} also holds under the more general assumption that $\lambda[h] \geq n-2$.
This would potentially allow us to extend our degree theory from $\CONEgeq(N)$ to the larger subspace of $\CONE(N)$ consisting of cones whose link metrics $h$ satisfy $\lambda[h] \geq 2$.
\end{Remark}

\begin{proof}[Proof of Lemma~\ref{Lem_nu_cone_exp}.]
Consider the expanding soliton $(M,g,V)$.
By Lemma~\ref{Lem_soliton_smooth} we may assume that $g$ and $V$ are smooth.
The upper bound always holds on a Riemannian orbifold.
To see the lower bound, consider the associated Ricci flow $(M, (g_t = t \phi_t^* g)_{t > 0})$, where $(\phi_t : M \to M)_{t > 0}$ is the flow of the time-dependent vector field $t^{-1} V$, as explained in Subsection~\ref{subsec_basic_identities}, and fix some $\tau_1 > 0$ and $e^{-f_1} \in C^\infty_c (M)$ with
\[ \int_M (4\pi \tau_1)^{-n/2} e^{-f_1} dg_1 = 1. \]
We need to show that
\[ \WW[g, f_1, \tau_1] \geq \nu[\gamma]. \]
To this end let $\tau_t := \tau_1 + 1 - t$ and consider the solution to the conjugate heat equation
\begin{equation} \label{eq_conj_HE_on_expander}
 \square^* \big( (4\pi \tau_t)^{-n/2} e^{-f_t} \big)
=
(-\partial_t - \triangle + R) \big( (4\pi \tau_t)^{-n/2} e^{-f_t} \big) = 0 
\end{equation}
with initial condition $(4\pi \tau_1)^{-n/2} e^{-f_1}$.
Set $f'_t := f_t \circ \phi_t^{-1}$ and $\tau'_t := \frac{\tau_t}t$; note that $\lim_{t \to \infty} \tau'_t = \infty$.
Then for any $t \in (0,1]$ we have
\[ \int_M (4\pi\tau'_t)^{-n/2} e^{-f'_t} dg
= \int_M (4\pi \tau_t)^{-n/2} e^{-f_t} dg_t = 1, \qquad
 \WW[g, f_1, \tau_1] \geq \WW[g_t, f_t, \tau_t] = \WW [g, f'_t , \tau'_t ]. \]
So it suffices to prove that
\begin{equation} \label{eq_need_show_lim_WW_Cone}
 \limsup_{t \searrow 0} \WW [g, f'_t , \tau'_t ] \geq \nu[\gamma]. 
\end{equation}
Applying the maximum principle to the evolution equation \eqref{eq_conj_HE_on_expander} and using the fact that $g_t$ has non-negative scalar curvature, implies
\[ \sup_M (4\pi\tau_t)^{-n/2} e^{-f_t} \leq \sup_M (4\pi\tau_1)^{-n/2} e^{-f_1}, \]
so
\begin{equation} \label{eq_inff_C0}
 \inf_M f_t \geq - C_0 
\end{equation}
for some constant $C_0 < \infty$, which is independent of $t$.

Fix some small constant $\eps > 0$, whose value we will later send to zero, and consider the map $\psi_\eps : (D_\eps,\infty) \times N \to M$ from \eqref{eq_psi_eps_char}.
Suppose without loss of generality that $D_\eps \geq 1$.
Choose a cutoff function $\eta \in C^\infty(M)$, $0 \leq \eta \leq 1$, with support in $\Image \psi_\eps$ that is $\equiv 1$ on the complement of a compact subset $S \subset M$.
For any $t \in (0,1]$ we define $e^{-f''_t} = \eta^2 e^{-f'_t}$.
Then
\begin{equation} \label{eq_fpp_int_1}
 \bigg| \int_{M} (4\pi \tau'_t)^{-n/2} e^{-f''_t} dg - 1 \bigg|
\leq  \int_{S} (4\pi \tau'_t)^{-n/2} e^{-f} dg
\leq  (4\pi\tau'_t)^{-n/2} |S|_g e^{C_0} 
\end{equation}
and, using the shorthand $\tau = \tau'_t$, $f = f'_t$, we have for some constant $C_\eps < \infty$, which may depend on $\eps$ and $\eta$, but is independent of $t$,
\begin{align}
\WW[g,f'_t,\tau'_t]- \WW[g,f''_t,\tau'_t]
&= \int_{S} \big( \tau (R + |\nabla f|^2 ) + f - n \big) (4\pi \tau)^{-n/2} (1-\eta^2) e^{-f} dg \notag \\
&\qquad - \int_{S} \big( \tau ( - 4 \eta^{-1} \nabla \eta \cdot \nabla f + 4\eta^{-2} |\nabla \eta|^2)  - 2\log \eta \big) (4\pi \tau)^{-n/2} \eta^2 e^{-f} dg \notag \\
&= \int_{S} \big( \tau (R + |\nabla f|^2 ) + f - n \big) (4\pi \tau)^{-n/2} (1-\eta^2) e^{-f} dg \notag \\
&\qquad - \int_{S} \big(  - 4\tau \eta \triangle \eta   - 2\eta^2 \log \eta \big) (4\pi \tau)^{-n/2}  e^{-f} dg \notag \\
&\geq - (  C_0 + n \big) (4\pi \tau'_t)^{-n/2}  e^{C_0} |S|_g
- C_\eps \tau'_t   (4\pi \tau'_t)^{-n/2}  e^{C_0} |S|_g. \label{eq_fpp_WW}
\end{align}
Note that the right-hand sides of \eqref{eq_fpp_int_1}, \eqref{eq_fpp_WW} go to zero as $t \searrow 0$.

Define $\ov f_t :  \IR_+ \times N \to (-\infty,\infty]$ by $\ov f_{t} := f''_t \circ \psi_\eps$ on $(D_\eps,\infty) \times N$ and $\ov f_{t} := \infty$ on $(0,D_\eps] \times N$.
Then by \eqref{eq_fpp_int_1} and \eqref{eq_psi_eps_char} we have for sufficiently small $t$, depending on $\eps$,
\begin{equation} \label{eq_a_t_close_1}
 a_{t} := \int_{ \IR_+ \times N} (4\pi \tau'_t)^{-n/2} e^{-\ov f_{t}} d\gamma  \in [e^{-2n\eps}, e^{2n\eps}]. 
\end{equation}
Moreover, using \eqref{eq_psi_eps_char}, \eqref{eq_inff_C0}, \eqref{eq_a_t_close_1}, and letting $r$ be the coordinate of the first factor of $\IR_+ \times N$, we get for sufficiently small $t > 0$
\begin{align*}
 e^{-(n+2)\eps} \WW[\gamma, \ov f_t, \tau'_t]& - \WW[g, f''_t, \tau'_t] \\
 &\leq \int_{ \IR_+ \times N} \big( \tau'_t (e^{-(n+2)\eps} R_\gamma - (R_g \circ \psi_\eps) |\det d\psi_\eps| ) \\
 &\qquad\qquad +  e^{-(n+2)\eps} (\ov f_t - n) - (\ov f_t - n) |\det d\psi_\eps| \big) (4 \pi \tau'_t)^{-n/2} e^{-\ov f_t} d\gamma \\
 &\leq \int_{ \IR_+ \times N} \big(\eps \tau'_t r^{-2} + (C_0 + n) (e^{n\eps} - e^{-(n+2)\eps}) \big)  (4\pi\tau'_t)^{-n/2} e^{-\ov f_t} d\gamma \\
 &\leq \int_{(0,(\tau'_t)^{1/2}) \times N} \eps \tau'_t r^{-2} \,   (4\pi\tau'_t)^{-n/2} e^{-\ov f_t} d\gamma \\
 &\qquad\qquad + e^{2n\eps} \big( \eps + (C_0+n)(e^{n\eps} - e^{-(n+2)\eps}) \big)  \\
 &\leq \vol (N,h) \int_0^{(\tau'_t)^{1/2}} \eps \tau'_t r^{n-1-2}  \,   (4\pi\tau'_t)^{-n/2} e^{C_0} dr\\
 &\qquad\qquad   + e^{2n\eps} \big( \eps + (C_0+n)(e^{n\eps} - e^{-(n+2)\eps}) \big) \\
 &\leq \tfrac1{n-2}  \vol(N,h)(4\pi)^{-n/2} e^{C_0} \eps
  + e^{2n\eps} \big( \eps + (C_0+n)(e^{n\eps} - e^{-(n+2)\eps}) \big)  =: \Psi (\eps),
\end{align*}
Note that in the first inequality above, the $e^{-(n+2)\epsilon}$ factor allows us to control the difference between the  $|\nabla^\gamma \ov f_t|_\gamma^2$ and $|\nabla^g f''_t|^2_g$ terms by accounting for the distortion from $\gamma$ to $g$ in their respective norms and volume forms.
So for sufficiently small $t$, depending on $\eps$, we have by \eqref{eq_fpp_WW}
\begin{equation} \label{eq_WW_s_p}
 e^{-(n+2)\eps} \WW[\gamma, \ov f_t, \tau'_t] - \WW[g, f'_t, \tau'_t] \leq 2\Psi(\eps). 
\end{equation}
Note that $\lim_{\eps \to 0} \Psi(\eps) =0$.
Set now $\ov f^{*}_{t} := \ov f_{t} + \log a_{t}$.
Then
\[ \int_{ \IR_+ \times N} (4\pi\tau'_t)^{-n/2} e^{-f^{**}_{t}} d\gamma = 1 \]
and
\begin{equation} \label{eq_nu_cone_WW_s}
 \nu [\gamma] \leq \WW[\gamma, \ov f^*_{t}, \tau'_t] = \frac1{a_{t}} \WW[\gamma, \ov f_{t} , \tau'_t] + \log a_{t}. 
\end{equation}
We now obtain \eqref{eq_need_show_lim_WW_Cone} by combining \eqref{eq_a_t_close_1}, \eqref{eq_WW_s_p} and \eqref{eq_nu_cone_WW_s} and letting $\eps \to 0$.
\end{proof}
\bigskip

\begin{proof}[Proof of Proposition~\ref{Prop_NLC_bound}.]
The first statement follows directly from Lemmas~\ref{Lem_some_lower_nu_bound}, \ref{Lem_nu_cone_exp}.
The last statement is a restatement of an improvement of Perelman's No Local Collapsing Theorem \cite{Perelman1}; see \cite[Remark~13.13]{Kleiner_Lott_notes}.
\end{proof}
\bigskip

\subsection{A uniform curvature and diameter bound}
Next, we use the non-collapsing statement from Proposition~\ref{Prop_NLC_bound} to derive geometric bounds for gradient expanding solitons with non-negative scalar curvature, which only depend on geometric bounds on their asymptotic cones.

\begin{Proposition} \label{Prop_curv_diam_bound}
Let $(M,N,\iota)$ be an ensemble and consider a representative $(g,V =\nabla^g f,\gamma)$ of an element in $\MMgeqgrad^6(M,N,\iota)$, where we write $\gamma = dr^2 + r^2 h$ for some Riemannian metric $h$ on $N$ of regularity $C^{6}$.
Suppose that for some $A < \infty$ we have
\begin{equation} \label{eq_bounds_on_link}
 |{\Rm_h}| \leq A, \qquad \inj(N,h) \geq A^{-1}, \qquad \mu[h,1] \geq - A. 
\end{equation}
If $\td\iota : \IR_+ \times N \to M$ denotes the embedding from Lemma~\ref{Lem_MM_regular_rep}, then the following is true:
\begin{enumerate}[label=(\alph*)]
\item \label{Prop_curv_diam_bound_a} $|{\Rm_g}| \leq C(A)$ on $M$.
\item \label{Prop_curv_diam_bound_b} $|\nabla f|\leq C(A) (r + 1)$ on $M \setminus \td\iota((r,\infty) \times N)$ for all $r > 0$.
\item \label{Prop_curv_diam_bound_c} For any $y \in M$ and $0 < r < 1$ we have the volume bound $|B(y,r)| \geq c(A) r^4$.
\item \label{Prop_curv_diam_bound_d} If in addition $\vol(N,h) \leq A$, then for all $D \geq 1$ we have
\[ \diam (M \setminus \td\iota((D,\infty) \times N) \leq C(A,D). \]
\end{enumerate}
\end{Proposition}

\begin{proof}
By Lemma~\ref{Lem_soliton_smooth} we may work with a smooth structure on $M$ with respect to which $M,\nabla f$, and therefore $f$, are smooth.

We first prove a uniform scalar curvature bound.
By Lemmas~\ref{Lem_MM_regular_rep}\ref{Lem_MM_regular_rep_b} and \ref{Lem_iota} we obtain a uniform $1 < R_0 (A) < \infty$ such that on $[R_0,\infty) \times N$ we have
\begin{equation} \label{eq_control_on_R0}
 0.9 \gamma \leq \td\iota^* \td g \leq 1.1 \gamma, \qquad
 0.49 r \leq |\nabla f| \circ \td\iota \leq 0.51 r, \qquad |{\Rm}| \circ \td\iota \leq 2A r^{-2}. 
\end{equation}
So, using the soliton equation $R + |\nabla f|^2 = - f$, we obtain that
\[   -f \leq - C(A) \qquad \text{on} \quad \td\iota (\{ R_0\}\times N). \]
On the other hand, since by tracing the soliton equation we have $\triangle(- f )=  R + 2 > 0$, we obtain via the maximum principle applied to $M \setminus \td\iota([R_0,\infty) \times N)$
\[ R + |\nabla f|^2 = -f \leq C(A) \qquad \text{on} \quad M \setminus \td\iota ((R_0,\infty)\times N). \]
This proves Assertion~\ref{Prop_curv_diam_bound_b} and establishes a uniform scalar curvature bound on $M$.
Combining this scalar curvature bound with Proposition~\ref{Prop_NLC_bound} implies Assertion~\ref{Prop_curv_diam_bound_c}.

Next, we prove Assertion~\ref{Prop_curv_diam_bound_a} by contradiction.
So consider a sequence of gradient expanding solitons $(g_i, V_i = \nabla^{g_i} f_i, \gamma_i)$ corresponding to  ensembles $(M_i, N_i, \iota_i)$ whose cone metrics $\gamma_i = dr^2 + r^2 h_i$ satisfy \eqref{eq_bounds_on_link} for some uniform $A$, but suppose that there are points $y_i \in M_i$ such that $Q_i := |{\Rm}_{g_i}| (y_i) \to \infty$.
Without loss of generality, we may assume that $y_i$ has been chosen at the maximum of each $|{\Rm}_{g_i}|$.
Consider the orbifold covers $\pi_i : \hat M_i \to M_i$ such that $\hat M_i$ is a smooth manifold, $H_2(\hat M_i;\IZ) = 0$ and $H_1(\hat M_i; \IZ)$ is torsion-free (see Property~\ref{Def_ensemble_3} in Definition~\ref{Def_ensemble}), set $\hat g_i := \pi_i^* g_i$ and choose lifts $\hat y_i$, $\pi_i(\hat y_i) = y_i$.

Due to Shi's derivative estimates for the associated Ricci flows and Assertion~\ref{Prop_curv_diam_bound_c}, we find that after passing to a subsequence we have smooth Cheeger-Gromov convergence
\[ (\hat M_i,Q_i \hat g_i, \hat y_i) \xrightarrow[i\to\infty]{\text{CG}} (M_\infty, g_\infty, y_\infty) \]
to a non-flat pointed Riemannian manifold with positive asymptotic volume ratio.
Due to the uniform scalar curvature bound on the original metrics $\hat g_i$, we obtain that $g_\infty$ is Ricci flat, so by \cite{Cheeger_Naber_codim4} it is a non-trivial ALE space and there is a compact domain $\Omega_\infty \subset M_\infty$ such that $\Int \Omega_\infty \approx M_\infty$ and $\partial \Omega_\infty \approx S^3/\Gamma$ for some non-trivial $\Gamma \subset SO(4)$.
This implies that for large $i$ there are compact smooth domains $\Omega_i \subset \hat M_i$ with $\Omega_i \approx \Omega_\infty$.
Using Lemma~\ref{Lem_no_instanton} below, we obtain a contradiction, which finishes the proof of Assertion~\ref{Prop_curv_diam_bound_a}.

Lastly, in order to prove Assertion~\ref{Prop_curv_diam_bound_d}, note that it suffices to consider the case $D=R_0$.
Set $K := M \setminus \td\iota ((R_0, \infty) \times N)$.
Then by Lemma~\ref{Lem_Im_iota_dense} we have the volume bound 
\[ |K|_g \leq R_0^4 \big| (0,1) \times N \big|_\gamma \leq C(A). \]
We can now derive a diameter bound on $K$ using Assertion~\ref{Prop_curv_diam_bound_c} via a ball packing argument.
Let $m$ be maximal such that we can find points $y_1, \ldots, y_m \in K$ such that the balls $B(y_j,1)$ are pairwise disjoint.
Using Assertion~\ref{Prop_curv_diam_bound_c} and the fact that we can bound the volume of a 1-tubular neighborhood of $K$ due to \eqref{eq_control_on_R0}, we obtain a bound of the form $m \leq c^{-1}(A)  (|K|_g + C(A)) \leq C(A)$.
By the maximal choice of $m$, the balls $B(y_j,2)$ cover $K$.
Consider the graph consisting of the vertices $\{ y_1, \ldots, y_m \}$ with the property that $y_{i}, y_j$ are connected by an edge if and only if $B(y_i,2)$ and $B(y_j,2)$ intersect.
Since $K$ is connected, this graph must be connected as well.
So we have $d(y_i,y_j) \leq 4(m-1)$ for any $i,j=1,\ldots,m$, which implies $\diam(K) \leq 2+4(m-1)+2 = 4m \leq C(A)$.
\end{proof}

The following lemma, which is a slight modification of \cite[Theorem~6.1]{Zhang_ALE_exclusion_2019}, goes back to an idea of Anderson \cite[Proposition~3.10]{Anderson_2008}, \cite{Anderson_1989, Anderson_2010}.

\begin{Lemma} \label{Lem_no_instanton}
Let $X$ be a (possibly non-compact) smooth 4-manifold such that $H_2(X;\IZ) = 0$ and $H_1(M;\IZ)$ is torsion-free.
Consider a compact smooth domain $\Omega \subset X$.
Then $\Int \Omega$ is not diffeomorphic to a non-trivial Ricci flat ALE space.
\end{Lemma}

\begin{proof}
This follows by a modification of the proof of
\cite[Theorem~6.1]{Zhang_ALE_exclusion_2019}. 
In this proof the compactness was used in the first step of \cite[Lemma~6.2]{Zhang_ALE_exclusion_2019} to show that $H^2(X;\IZ) = 0$ (due to Poincar\'e duality), which implied that $H_1(X;\IZ)$ is torsion-free (due to the Universal Coefficient Theorem).
Since we have assumed that $H_1(X;\IZ)$ is torsion-free, we obtain that $H^2(X;\IZ) = 0$, so we can skip the first step.
\end{proof}
\bigskip

\subsection{Proof of Proposition~\ref{Prop_properness}}

\begin{proof}[Proof of Proposition~\ref{Prop_properness}.]
Consider a sequence 
\[ p_i  \in \ \MMgeqgrad^{k^*}(M, \lb N, \lb \iota).  \]
For each $i$ let $(g_i, V_i = \nabla f_i, \gamma_i)$ be a $C^{k^*-2}$-regular representative of $p_i$ and $\td\iota_i : \IR_+ \times N \to M$ the map supplied by  Lemma~\ref{Lem_MM_regular_rep}\ref{Lem_MM_regular_rep_a}; recall that $\td\iota_i^* V_i = - \tfrac12 r \partial_r$.
Set $y_i := \td\iota_i(1,z_0)$ for some fixed $z_0 \in N$.
Note that by Lemma~\ref{Lem_soliton_smooth} there is a smooth structure on $M$, for each $i$, with respect to which $g_i, f_i$ are smooth.

Suppose that $\gamma_i \to \gamma_\infty \in \CONEgeq^{k^*}(N)$.
Our goal is to show that for a subsequence $p_i \to p_\infty \in \MMgeqgrad^{k^*}(M, \lb N, \lb \iota)$.
To do this, it suffices we verify the assumptions of Lemma~\ref{Lem_good_reps_technical}.
Assumption~\ref{Lem_good_reps_technical_i} holds trivially and Assumption~\ref{Lem_good_reps_technical_iii} follows from Proposition~\ref{Prop_curv_diam_bound}.
Assumption~\ref{Lem_good_reps_technical_ii} follows from the curvature, potential and non-collapsing bounds of Proposition~\ref{Prop_curv_diam_bound}, Shi's estimates and Cheeger-Gromov compactness for orbifolds with isolated singularities (see again \cite{Fukaya_1986, Lu_2001} for a discussion of Cheeger-Gromov compactness for orbifolds, as in the proof of Proposition~\ref{Prop_converging_representatives}).
So by Lemma~\ref{Lem_good_reps_technical} we have $p_i \to p_\infty \in \MM^{k^*}(M,N,\iota)$ and after possibly choosing different representatives via Proposition~\ref{Prop_converging_representatives}, we even have local $C^2_{\loc}$ convergence of $g_i,V_i$, which implies that $p_\infty \in \MMgeqgrad^{k^*}(M,N,\iota)$.
\end{proof}
\bigskip

\subsection{Proof of Proposition~\ref{Prop_eucl_case}}

\begin{proof}[Proof of Proposition~\ref{Prop_eucl_case}.]
Without loss of generality, we may pass to a finite orbifold cover and assume that $N = S^3$ and that $\gamma_{\eucl}$ is the standard Euclidean cone, which corresponds to the standard round metric on $S^3$.
Let $(g,V,\gamma_{\eucl})$ be a representative of an element of $\Pi^{-1}(\gamma_{\eucl})$.
We will show that then $(M,N,\iota)$ isomorphic to $(\IR^4, S^3, \iota_{\{ 0\}})$ and that $(g,V,\gamma_{\eucl})$ is equivalent to $g_{\eucl}, V_{\eucl}, \gamma_{\eucl})$.

By Lemma~\ref{Lem_nu_cone_exp}, we have $\nu[g] = 0$.
So if we consider the associated Ricci flow $(g_t)_{t > 0}$, then $\mu[g_t,\tau] = 0$ for all $t, \tau > 0$. Considering the $\WW$-functional with respect to a conjugate heat kernel concentrating at some space-time point of $(M, g_t)$ then shows that $(M,g=g_1)$ must also be a shrinking Ricci soliton, so it must be isometric to Euclidean space by \cite{Kotschwar_Wang}.
So $\LL_V g = - g$.
Let $\td\iota : \IR_+ \times S^3 \to M$ be the embedding from Lemma~\ref{Lem_MM_regular_rep}.
Then 
\[ \LL_{\frac12 r \partial_r} \td\iota^* g 
= \LL_{\td\iota^* V} \td\iota^* g 
= \iota^* \big( \LL_V g \big)
= \td\iota^* g. \]
Since $\td\iota^*g$ is asymptotic to $\gamma_{\eucl}$, this implies that $\td\iota^* g = \gamma_{\eucl}$ .
So if we identify $\IR_+ \times S^3$ with $\IR^4 \setminus \{ 0 \}$, then $\td\iota$ extends to the inverse map of an isometry $\phi :  (M,g) \to (\IR^4, g_{\eucl})$.
Since $\td\iota = \iota$ on $(r_0,\infty) \times N$ for some $r_0 > 1$, this shows that $(M,N,\iota)$ is isomorphic to $(\IR^4,S^3,\iota_{\{ 1 \}})$ and that that $(g,V,\gamma_{\eucl})$ is equivalent to $g_{\eucl}, V_{\eucl}, \gamma_{\eucl})$.
\end{proof}
\bigskip

\section{Elliptic estimates} \label{sec_elliptic}
Most of the analysis in this paper relies on estimates for solutions to elliptic equations of the form $\triangle_V u = \triangle u -  \nabla_V u = h$, where $V$ denotes a vector field on a Riemannian orbifold, which may be unbounded, but satisfies $|\nabla V| \leq C < \infty$, and $u, h$ are tensor fields.
Usually, we will have $V = \nabla f$, in which case we write $\triangle_f u = \triangle_{\nabla f}$.
We will need both Schauder-type estimates, as well as $L^2$-type spectral estimates (the latter will only be important in the degree theory with \emph{integer} coefficients).
The Schauder-type estimates follow from an approach of Lunardi \cite{Lunardi_1998}; see also \cite{Deruelle} for an adaptation to the manifold and weighted case.
As we need slightly more general estimates than what is available in these references, we have included a full and somewhat streamlined discussion of these estimates.
They can be found in Subsection~\ref{subsec_Lunardi_unweighted} (the main estimates in the unweighted case) and Subsection~\ref{subsec_Lunardi_weighted} (the weighted case).
We have prefaced our discussion with Subsection~\ref{subsec_list_Holder}, which contains a list of all H\"older norms that will be used in the paper, for easy reference.
Lastly, in Subsection~\ref{subsec_spectral_theory} we discuss the $L^2$ and spectral theory associated with these equations.

\subsection{List of H\"older norms}\label{subsec_list_Holder}
Let $(M,g)$ be a complete Riemannian orbifold 
with bounded curvature and $E$ a tensor bundle over $E$, i.e., a bundle of the form $E = T^{d}_{d^*} M = (TM)^{\otimes d} \otimes (T^* M)^{\otimes d^*}$.
Recall that $E$ inherits a natural inner product structure, norm $| \cdot |_g$ and connection $\nabla^g$ from $TM$.
Let $k \geq 0$ be an integer and $\alpha \in (0,1)$.

\begin{Definition} \label{Def_Ck}
The {\bf $C^k$ and $C^{k,\alpha}$-norms} of a section $u \in C^k_{\loc} (M;E)$ are defined by
\[ \Vert u \Vert_{C^k_g} := \sum_{l=0}^m \sup_M |\nabla^{l,g} u|_g, \qquad
\Vert u \Vert_{C^{k,\alpha}_g} := \Vert u \Vert_{C^k_g} + \sup_{\substack{\gamma : [0,r] \to M \\ 0 < r < 1}} \frac{|P^{\gamma,g} ( \nabla^k u(\gamma(0))) - \nabla^k u(\gamma(r))|_g}{r^\alpha},
 \]
where $P^\gamma$ denotes the parallel transport along $\gamma$ from $\gamma(0)$ to $\gamma(r)$ and $\gamma$ is assumed to be a unit-speed geodesic.
The spaces $C^k_{g} (M;E), C^{k,\alpha}_{g} (M;E) \subset C^k_{\loc} (M;E)$ are subspaces on which the corresponding norms are finite.
We will often drop the ``$g$'' or ``$\loc$''-subscripts if there is no chance of confusion.
\end{Definition}

\begin{Remark}
Given a cover of $M$ by suitable coordinate charts, the norms from Definition~\ref{Def_Ck} are equivalent to standard $C^k$ and $C^{k,\alpha}$-norms in these coordinates.
See \eqref{eq_sup_equivalent} in Lemma~\ref{Lem_Holder_norms_properties} below for more details.
\end{Remark}

Next, let $w \in C^\infty_{\loc} (M)$, $w > 0$, be a weight function and $V \in C^\infty_{\loc} (M; TM)$ a vector field.

\begin{Definition} \label{Def_Ckw}
The {\bf weighted $C^k_{g,w}$ and $C^{k,\alpha}_{g,w}$-norms} are defined by
\[ \Vert u \Vert_{C^k_{g,w}} := \big\Vert w^{-1} u \big\Vert_{C^k_g}, \qquad
\Vert u \Vert_{C^{k,\alpha}_{g,w}} := \big\Vert w^{-1} u \big\Vert_{C^{k,\alpha}_g}. \]
Moreover, if $k \geq 2$, then we define the {\bf weighted $C^k_{w,V}$ and $C^{k,\alpha}_{w,V}$-norms} by
\[ \Vert u \Vert_{C^k_{g,w,V}} := \Vert u \Vert_{ C^k_{g,w}} + \Vert \nabla_V u \Vert_{C^{k-2}_{g,w}}, \qquad
\Vert u \Vert_{C^{k,\alpha}_{g,w,V}} := \Vert  u \Vert_{C^{k,\alpha}_{g,w}}+ \Vert \nabla_V u \Vert_{C^{k-2,\alpha}_{g,w}}. \]
Again, we denote the corresponding Banach spaces by $C^k_{g,w}, C^{k,\alpha}_{g,w}, C^k_{g,w,V}, C^{k,\alpha}_{g,w,V}$ and we will drop the ``$g$''-subscripts if there is no chance of confusion.
\end{Definition}

In this paper, the weight $w$ near infinity will often be of the form $r^{-a}$, for $a > 0$, and the vector field $V$ will often grow at most linearly such that $\nabla V$ is uniformly bounded.

\begin{Remark}
Bounds on the $C^k_{g,w}$ and $C^{k,\alpha}_{g,w}$-norms do not imply faster decay bounds for higher derivatives.
For example, if $(M,g)$ is Euclidean space and $w = (1+r^2)^{-a/2}$ for some $a > 0$ and the radial distance function $r : \IR^n \to \IR$, then the bound $\Vert u \Vert_{C^k_{g,w}} < \infty$ is equivalent to $|\nabla^m u| \leq C r^{-a}$ for all $m = 0, \ldots, k$ (see Lemma~\ref{Lem_Holder_norms_properties} below).
\end{Remark}

In this paper we will only apply the norms from Definitions~\ref{Def_Ck}, \ref{Def_Ckw} when $(M,g)$ is the Riemannian orbifold associated with an asymptotically conical,  gradient expanding soliton $(M,g,f)$.
In this case we will use the potential function $f$ to define a weight function with decay $O(r^{-a})$.
More specifically, recall that on asymptotically conical solitons we have $f = - \frac14 r^2 + O(1)$ (see \eqref{eq_f_asymptotics_iota}).
So if we set $S := \sup_M f$, then for any $a \geq 0$ the function $w := (S - f + 1)^{-a}$ satisfies $w \leq 1$ and $w \sim r^{-a})$ near infinity.
Note that the choice of the weight depending on the potential function is mainly for convenience, as the potential function possesses good derivative bounds.
Using the potential function, we will introduce the following conventions, which will  simplify our notation.

\begin{Definition}
If $(M,g,f)$ is a gradient expanding soliton such that $g,f$ are of regularity $C^2$, $S := \sup_M f < \infty$ and $a \geq 0$, then using the unique smooth structure on $M$ supplied by Lemma~\ref{Lem_soliton_smooth}, we set for any $u \in C^k_{\loc}(M;E)$
\[ \Vert u \Vert_{C^k_{g,-a}} := \Vert u \Vert_{C^k_{g,(S-f+1)^{-a/2}}}, \qquad \Vert u \Vert_{C^{k,\alpha}_{g,-a}} := \Vert u \Vert_{C^{k,\alpha}_{g,(S-f+1)^{-a/2}}} \]
and
\begin{equation} \label{eq_CkaV}
 \Vert u \Vert_{C^k_{g,-a,V}} := \Vert u \Vert_{C^k_{g,(S-f+1)^{-a/2},V}}, \qquad \Vert u \Vert_{C^{k,\alpha}_{g,-a,V}} := \Vert u \Vert_{C^{k,\alpha}_{g,(S-f+1)^{-a/2},V}}. 
\end{equation}
As before, we will often drop the ``$g$''-subscripts if there is no chance of confusion.
\end{Definition}

We remark that the vector field $V$ in \eqref{eq_CkaV} will often be taken to be $\nabla f$.

The following lemma discusses some basic properties of these H\"older norms.
For example, it shows that a bound of the form $\Vert u \Vert_{C^k_{-a}} < \infty$ is equivalent to $|\nabla^m u| \leq C (S-f+1)^{-a/2} \lesssim r^{-a}$ for $m = 0,\ldots, k$.

\begin{Lemma} \label{Lem_Holder_norms_properties}
Suppose that $(M,g,f)$ is an $n$-dimensional gradient expanding soliton on an orbifold with isolated singularities,
$k \geq 0$ is an integer and $\alpha \in [0,1)$ and assume that  $f$ attains a maximum $S$ and
\[ |{\Rm}| \leq A,   \qquad n, k, \rank E \leq A. \]
Then for any $u \in C^{k,\alpha}_{-a} (M;E)$, $0 \leq a^* \leq a$, $k' \geq 0$, $\alpha' \in [0,1)$, we have $u \in C^{k,\alpha} (M; E)$, $\nabla u \in C^{k-1,\alpha}_{-a} (M; T^* M \otimes E)$ and
\begin{align}
\Vert  u \Vert_{C^{k,\alpha}_{-a}} &= \Vert  u \Vert_{C^{k,\alpha}} \qquad\;\, \text{if} \quad a= 0, \label{eq_a_is_0} \\
\Vert u \Vert_{C^{k,\alpha}_{-a}} &\leq \Vert u \Vert_{C^{k',\alpha'}_{-a}} \qquad \text{if} \quad k+\alpha \leq k'+\alpha', \label{eq_ka_kpap} \\
  \Vert  u \Vert_{C^{k,\alpha}_{-a^*}} &\leq C(A,a,\alpha) \Vert u \Vert_{C^{k,\alpha}_{-a}},  \label{eq_Cka_Ckaa_bound} \\
\Vert \nabla u \Vert_{C^{k-1,\alpha}_{-a}} &\leq C(A,a,\alpha) \Vert u \Vert_{C^{k,\alpha}_{-a}},\label{eq_Ckm1a_Cka_bound} \\
c(A,a,\alpha)\Vert u \Vert_{C^{k,\alpha}_{-a}}
&\leq  \bigg( \sum_{m=0}^{k-1} \Vert \nabla^m u \Vert_{C^0_{-a}} + \Vert \nabla^{k} u \Vert_{C^{0,\alpha}_{-a}} \bigg) \leq C(A,a,\alpha) \Vert u \Vert_{C^{k,\alpha}_{-a}}, \label{eq_equivalent_norms}
\end{align}
where in the case $\alpha = 0$ the norms $C^{k,\alpha}_{-a}$ can be replaced with $C^k_{-a}$.
Moreover, suppose that $\{ \vec x_p : U_p \to B_{2r} \}_{p \in M}$ denotes a collection of orbifold-coordinate charts, where $B_r := B(0, r) \subset \IR^n$, $r > 0$ such that $\vec x_p (p) = \vec 0$.
If $p$ is a singular point, then $U_p$ denotes the corresponding finite orbifold cover of an open neighborhood of $p$ and $\vec x_p$ is understood to be equivariant with respect to the group describing the singularity.
Suppose that in each local coordinate chart we have
\[ |\partial^m (g_{ij} - \delta_{ij}) | \leq A \quad \text{on} \quad B_{2r}, \qquad m = 0, \ldots, k+1. \]
Then
\begin{equation} \label{eq_sup_equivalent}
 c(A,a, \alpha,r) \Vert u \Vert_{C^{k,\alpha}_{-a}}
\leq \sup_{p \in M} (S-f(p)+1)^{a/2} \big\Vert (\vec x_p)_* u \big\Vert_{C^{k,\alpha}_{\eucl}( B_r )} \leq C(A,a,\alpha,r) \Vert u \Vert_{C^{k,\alpha}_{-a}}, 
\end{equation}
where $\Vert (\vec x_p)_* u \Vert_{C^{k,\alpha}_{\eucl}( B_r )}$  denotes the standard Euclidean H\"older norm of the tensor coefficients of $u$ in the coordinates $\vec x_p$, restricted to $B_r \subset \IR^n$.

Lastly, for any section $u' \in C^{k,\alpha}_{-a'}(M;E')$, $a' \geq 0$, of some tensor bundle $E'$ and any global bilinear operation ``$*$'' we have
\begin{equation} \label{eq_u_star_up}
 \Vert u * u' \Vert_{C^{k,\alpha}_{-(a+a')}} \leq C \Vert u \Vert_{C^{k,\alpha}_{-a}} \Vert u' \Vert_{C^{k,\alpha}_{-a'}}. 
\end{equation}
\end{Lemma}

\begin{proof}
The identities \eqref{eq_a_is_0} and \eqref{eq_ka_kpap} are direct consequences of the definition.
To see the remaining bounds, we may for simplicity adjust $f$ by a bounded constant such that $S = \sup_M f = -1$; so $(S-f+1)^{-a/2}=(-f)^{-a/2}$.
Then, evaluating $|\nabla f|^2 + R + f \equiv const$ at the maximum of $f$ implies that
\begin{equation} \label{eq_const_is_CA}
 \big| |\nabla f|^2 + R + f  \big| \leq C(A). 
\end{equation}
Using Shi's estimates and the soliton equation implies that for $m \geq 0$
\begin{equation} \label{eq_nab_higher_f}
 |\nabla^{m+2} f| = |\nabla^{m} {\Ric}| \leq C(m,A). 
\end{equation}
So for any $b \geq 0$ and $m \geq 0$ we have by \eqref{eq_const_is_CA} and \eqref{eq_nab_higher_f}
\begin{equation} \label{eq_nabmfmb}
 \big\Vert\nabla \big( \log(-f)\big) \big\Vert_{C^m}, \;\big\Vert (-f)^{-b} \big\Vert_{C^m} \leq C(m,b,A). 
\end{equation}

To see \eqref{eq_Cka_Ckaa_bound} note that, using \eqref{eq_nabmfmb},
\begin{multline*}
  \Vert  u \Vert_{C^{k,\alpha}_{-a^*}}
= \big\Vert (-f)^{a^*/2} u \big\Vert_{C^{k,\alpha}}
= \big\Vert (-f)^{-(a-a^*)/2} (-f)^{a/2} u \big\Vert_{C^{k,\alpha}} \\
= \big\Vert (-f)^{-(a-a^*)/2} \big\Vert_{C^{k,\alpha}} \big\Vert (-f)^{a/2} u \big\Vert_{C^{k,\alpha}}
\leq C(A,a,\alpha) \big\Vert u \big\Vert_{C^{k,\alpha}_{-a}}. 
\end{multline*}
For \eqref{eq_Ckm1a_Cka_bound} note that, again using \eqref{eq_nabmfmb},
\begin{multline*}
 \Vert \nabla u \Vert_{C^{k-1,\alpha}_{-a}}
= \big\Vert (-f)^{a/2} \nabla u \big\Vert_{C^{k-1,\alpha}_{-a}}
\leq \big\Vert \nabla \big( (-f)^{a/2} u \big) \big\Vert_{C^{k-1,\alpha}} + \frac{a}2 \bigg\Vert \frac{\nabla f}{f} \otimes \big( (-f)^{a/2}  u \big)  \bigg\Vert_{C^{k-1,\alpha}} \\
\leq \big\Vert  (-f)^{a/2} u \big\Vert_{C^{k,\alpha}} + \frac{a}2 \big\Vert \nabla (\log(-f)) \big\Vert_{C^{k-1,\alpha}} \big\Vert (-f)^{a/2}    u  \big\Vert_{C^{k-1,\alpha}} \leq C(A,a,\alpha) \Vert u \Vert_{C^{k,\alpha}_{-a}}.
\end{multline*}
The second inequality in \eqref{eq_equivalent_norms} follows by repeatedly applying \eqref{eq_Ckm1a_Cka_bound} and \eqref{eq_ka_kpap}.
The first inequality follows by repeated application of
\begin{multline*}
 \Vert u \Vert_{C^{k,\alpha}_{-a}}
= \big\Vert (-f)^{a/2} u \big\Vert_{C^{k,\alpha}}
\leq \big\Vert (-f)^{a/2} u \big\Vert_{C^{0}} + \big\Vert \nabla ( (-f)^{a/2} u) \big\Vert_{C^{k-1,\alpha}} \\
\leq  \Vert  u \Vert_{C^{0}_{-a}} + \Vert  \nabla u \Vert_{C^{k-1,\alpha}_{-a}} + \frac{a}2 \big\Vert \nabla (\log(-f)) \big\Vert_{C^{k-1,\alpha}} \big\Vert (-f)^{a/2} u \big\Vert_{C^{k-1,\alpha}} \\
\leq C(A,a,\alpha)\big(\Vert  u \Vert_{C^{k-1,\alpha}_{-a}} + \Vert \nabla u \Vert_{C^{k-1,\alpha}_{-a}} \big).
\end{multline*}

To see \eqref{eq_sup_equivalent}, we first consider the case $a = 0$. 
In local coordinates $\vec x_p$ the differences $\nabla^m u - \partial^m u$, $m = 0, \ldots, k$, are bounded in terms of $u, \ldots, \partial^{m-1} u$ and metric derivatives up to order $k$.
Moreover, due to the derivative bound on $g_{ij}$, we obtain that the parallel transport map in the same coordinates satisfies $|P^{\gamma,g} - \id_{\IR^n} | \leq C(A) r'$, where $\gamma :[0,r'] \to U_p$ is a unit speed geodesic with $0 < r' < r$.
Combining these two bounds implies \eqref{eq_sup_equivalent} for $a = 0$.
Suppose now that $a > 0$ and write $Q_1 \sim Q_2$ if $c(A,a,\alpha,r) Q_1 \leq Q_2 \leq C(A,a,\alpha,r) Q_1$.
Then
\begin{multline*}
 \Vert u \Vert_{C^{k,\alpha}_{-a}} 
= \Vert (-f)^{a/2} u \Vert_{C^{k,\alpha}} 
\sim \sup_{p \in M}  \big\Vert (\vec x_p)_* ((-f)^{a/2}u) \big\Vert_{C^{k,\alpha}_{\eucl}(B_r)} \\
= \sup_{p \in M}(-f(p))^{a/2} \bigg\Vert (\vec x_p)_* \Big( \frac{-f}{-f(p))} \Big)^{a/2} \bigg\Vert_{C^{k,\alpha}_{\eucl}(B_r)}\big\Vert (\vec x_p)_* u \big\Vert_{C^{k,\alpha}_{\eucl}(B_r)}.
\end{multline*}
Due to \eqref{eq_const_is_CA} and \eqref{eq_nab_higher_f}, the second last H\"older norm is bounded from above and below by a constant of the form $C(A,\alpha,r)$.

To see \eqref{eq_u_star_up} note that
\begin{multline*}
 \Vert u * u' \Vert_{C^{k,\alpha}_{-(a+a')}}
= \big\Vert (-f)^{-(a+a')/2} u * u' \big\Vert_{C^{k,\alpha}} \\
\leq  C \big\Vert (-f)^{-a/2} u \big\Vert_{C^{k,\alpha}} \big\Vert (-f)^{-a'/2} u' \big\Vert_{C^{k,\alpha}} 
=  C \Vert u \Vert_{C^{k,\alpha}_{-a}} \Vert u' \Vert_{C^{k,\alpha}_{-a'}}. \qedhere
\end{multline*}
\end{proof}
\bigskip

\subsection{Unweighted estimates} \label{subsec_Lunardi_unweighted}
The following proposition is a generalization of the estimates in \cite{Lunardi_1998,Deruelle}.
The key point of the proposition is that it does not depend on an $L^\infty$-bound on $V$ and $V$ may even be unbounded.

\begin{Proposition} \label{Prop_Lunardi}
Let $(M,g)$ be a complete, $n$-dimensional Riemannian
orbifold, $k \geq 0$ and $\alpha \in (0,1)$.
Let $V \in C^\infty (M;TM)$ be a vector field on $M$, $E$ a tensor bundle over $M$ and $b_1, b_0$ sections of the bundles $TM \otimes \End(E)$, $\End(E)$, respectively.
Consider the following equation for sections $u, v$ of $E$
\begin{equation} \label{eq_Lu_is_f}
 L u = \triangle u + \nabla_V u + b_1 * \nabla u + b_0 * u = v, 
\end{equation}
where ``$*$'' denotes the appropriate contraction.
Suppose that there is a $\lambda > 0$ such that at any point $p \in M$ and any $e \in E_p$ we have
\begin{equation} \label{eq_b1b0_la_bound}
 \tfrac14 |(b_1 * \cdot) \cdot e|^2 +  (b_0 * e) \cdot e \leq - \lambda |e|^2, 
\end{equation}
and suppose that
\begin{multline} \label{eq_b1b0_A_bound}
 \Vert {\Rm} \Vert_{C^{k+5}}, 
 \Vert \nabla V \Vert_{C^{k+3}},
 \Vert b_1 \Vert_{C^{k+3}}, 
 \Vert b_0 \Vert_{C^{k+3}}, \lambda^{-1}, \\
 n,
 \rank E,
k,
\alpha^{-1}, (1-\alpha)^{-1},
(\inj (M,g))^{-1}
  \leq A < \infty 
\end{multline}

Then for every $v \in C^{k,\alpha}(M;E)$ the equation \eqref{eq_Lu_is_f} has a bounded solution $u \in C^{k+2,\alpha}_{0,V}(M;E)$, which is unique among all bounded sections in $C^2_{\loc}(M;E)$.
Moreover, we have
\begin{equation} \label{eq_u_f_bound_m_alph}
 \Vert u \Vert_{C^{k+2,\alpha}_{0,V}} \leq C(A)  \Vert v \Vert_{C^{k,\alpha}} 
\end{equation}
and
\begin{equation} \label{eq_u_f_bound_L_inf}
 \Vert u \Vert_{L^\infty} \leq \frac1{\lambda}  \Vert v \Vert_{L^\infty}. 
\end{equation}
\end{Proposition}

The most commonly used consequence of this proposition is the following:

\begin{Corollary} \label{Cor_Lundardi}
Let $(M,g,f)$ be an $n$-dimensional gradient expanding soliton with bounded curvature such that $f$ attains a maximum, let $E$ be a tensor bundle over $M$ and suppose that $k \geq 0$ and $\alpha \in (0,1)$.
Then for any $b > 0$ the operator
\[ L := \triangle_f - b : C^{k+2,\alpha}_{0,\nabla f} (M;E) \longrightarrow C^{k,\alpha} (M;E) \]
is well defined and invertible.
Here $\triangle_f = \triangle_{\nabla f} = \triangle - \nabla_{\nabla f}$.
Moreover, the operator norm of $L^{-1}$ can be bounded depending only on bounds on $|{\Rm}|$, $\inj(M,g)$, $n$, $\rank E$, $k$, $\alpha$ and $b$.
\end{Corollary}

\begin{proof}
This follows by applying Proposition~\ref{Prop_Lunardi} for $V = -\nabla f$ and $b_0 = -b \id$.
\end{proof}
\bigskip

\begin{proof}[Proof of Proposition~\ref{Prop_Lunardi}.]
Let us first establish \eqref{eq_u_f_bound_L_inf}.
The equation \eqref{eq_Lu_is_f}, the bound \eqref{eq_b1b0_la_bound} and Young's inequality imply that
\begin{multline} \label{eq_Lu_cdot_u_computation}
 2 v \cdot u =  2(Lu) \cdot u = \triangle |u|^2 - 2|\nabla u|^2 + \nabla_V |u|^2 + 2(b_1 * \nabla u) \cdot u + 2(b_0 * u) \cdot u \\
\leq \triangle |u|^2 - 2 |\nabla V|^2 + \nabla_V |u|^2 - 2\lambda |u|^2.
\end{multline}
So if $M$ is compact, then the maximum principle implies that $\lambda \max |u|^2 \leq \max (- v \cdot u)$, which implies \eqref{eq_u_f_bound_L_inf}.
In order to apply the maximum principle in the non-compact case, we need to assume that $|u|$ is bounded. 
Choose a positive function $w \in C^\infty (M)$ with uniformly bounded first and second derivatives that grows linearly at infinity \cite{Greene_Wu_1976}.
So $w \geq c |V|$ for some $c > 0$ and for sufficiently small $\eps > 0$ we have
\[ \triangle w^\eps + \nabla_V w^\eps - 2 \lambda w^\eps < 0. \]
We may then apply the maximum principle to $|u|^2 - \delta w^\eps$ and let $\delta \searrow 0$.
Thus \eqref{eq_u_f_bound_L_inf}, and consequently the uniqueness statement, also hold in the non-compact case.

Next, we reduce the proposition to the case in which $M$ is compact.
Fix some $v \in C^{k,\alpha}(M;E)$, a point $p \in M$ and choose a sequence of bump functions $\eta_i \in C^\infty(M)$ such that $\supp \eta_i \subset B(p,2i)$, $\eta_i \equiv 1$ on $B(p,i)$, $0 \leq \eta_i \leq 1$ and $\Vert \nabla \eta_i\Vert_{C^{k+3,\alpha}} \leq C(n,A) i^{-1}$; for example, one may choose local smoothings of radial functions $\td\eta_i$ with $|\nabla\td\eta_i| \leq C i^{-1}$.
Using a doubling construction (for example, using \cite{Cheeger_Gromov_chopping}), we can represent $(M,g,p)$ as a smooth pointed Cheeger-Gromov limit of a sequence of compact, pointed Riemannian orbifolds $(M_i, g_i, p_i)$, which satisfy a similar curvature bound as in \eqref{eq_b1b0_A_bound}.
So, after possibly  passing to a subsequence, there are diffeomorphisms $\psi_i : U_i := B(p,2i) \to \psi_i(U_i) \subset M_i$ such that $\psi_i^* g_i = g$ on $U_i$.
Define the following tensor fields on $M$
\[  v'_i := v \eta_i, \qquad V'_i := V \eta_i, \qquad b'_{1,i} := b_{1,i} \eta_i, \qquad b'_{0,i} := (b_0 + \lambda \id) \eta_i - \lambda \id. \]
Define $v_i := (\psi_i)_* v'_i$ on $\psi_i (U_i)$ and $v_i := 0$ on $M_i \setminus \psi_i(U_i)$.
Define $V_i, b_{1,i}, b_{0,i}$ similarly, with the exception that $b_{0,i} := - \lambda \id$ outside $\psi_i (U_i)$.
Then the equation (for some unknown solution $u_i$)
\[ \triangle u_i + \nabla_{V_i} u_i + b_{1,i} * \nabla u_i + b_{0,i} * u_i = v_i \]
satisfies \eqref{eq_b1b0_la_bound}, \eqref{eq_b1b0_A_bound} for large $i$, after adjusting $A$ in a controlled way.
So, assuming that the proposition is true in the compact case, we obtain solutions $u_i \in C^{k+2,\alpha}(M;E)$ that satisfy a bound of the form \eqref{eq_u_f_bound_m_alph}, where $\Vert v_i \Vert_{C^{k,\alpha}} \leq C(A) \Vert v \Vert_{C^{k,\alpha}}$.
So, after passing to a subsequence, we have $\psi_i^* u_i \to u \in C^{k+2,\alpha}(M;E)$ in $C^{k+2}_{\loc}$ and the limit $u$ satisfies \eqref{eq_u_f_bound_m_alph}.

So let us assume from now on that $M$ is compact.
In the following we show the formally weaker bound
\begin{equation} \label{eq_u_f_bound_m_alph_weaker}
 \Vert u \Vert_{C^{k+2,\alpha}} \leq C(A)  \Vert v \Vert_{C^{k,\alpha}} .
\end{equation}
This bound implies \eqref{eq_u_f_bound_m_alph}, because
\begin{multline*}
 \Vert u \Vert_{C^{k+2,\alpha}_V} 
= \Vert u \Vert_{C^{k+2,\alpha}} + \Vert \nabla_V u \Vert_{C^{k,\alpha}}
\leq \Vert u \Vert_{C^{k+2,\alpha}} + \Vert v - \triangle u - b_1 * \nabla u + b_0 *u \Vert_{C^{k,\alpha}} \\
\leq C(A) \Vert u \Vert_{C^{k+2,\alpha}} +  \Vert v \Vert_{C^{k,\alpha}}
\leq C(A)  \Vert v \Vert_{C^{k,\alpha}} .
\end{multline*}

By applying a covariant derivative on both sides of \eqref{eq_Lu_is_f}, we find that $\nabla u$ satisfies an equation of the form
\begin{multline} \label{eq_commutation_nab_u}
 \triangle \nabla u + \nabla_V \nabla u + b_1 * \nabla^2 u + b_0 * \nabla u + R(\cdot, V) u \\
+ \Rm * \nabla u + \nabla \Rm * u + \nabla V * \nabla u + \nabla b_1 * \nabla u + \nabla b_0 * u = \nabla v. 
\end{multline}
So if we set $E' := E \oplus (T^*M \otimes E)$ and consider the associated sections $u' = (u, c \nabla u), v' = (v, c \nabla v) \in C^{m+1,\alpha}(M; E')$ for some $c > 0$, which we will determine later, then equations \eqref{eq_Lu_is_f} and \eqref{eq_commutation_nab_u} can be expressed as an equation of the same form 
\begin{equation} \label{eq_uprime_main}
 \triangle u' + \nabla_V u' + b'_1 * \nabla u' + b'_0 * u' = v', 
\end{equation}
where
\[ b'_0 = \left[\begin{array}{cc}b_0 & 0 \\c \nabla \Rm + c R(\cdot, V) + c\nabla b_0 & b_0\end{array}\right], \qquad
b'_1 = \left[\begin{array}{cc}b_1 & 0 \\c  \Rm + c \nabla V + c \nabla b_1  & b_1\end{array}\right]. \]
So for some sufficiently small $c (A) > 0$, the coefficient functions of equation \eqref{eq_uprime_main} satisfy bounds of the form \eqref{eq_b1b0_la_bound}, \eqref{eq_b1b0_A_bound} if we replace $k$ with $k-1$, $\lambda$ with $\frac12 \lambda$ and adjust $A$; note that the curvature operator $R(\cdot, V)$, viewed as a section of $\End E$, is bounded in $C^{k+2}$, because $\Rm(\cdot, V, \cdot, \cdot) = \Rm (\cdot, \cdot, V, \cdot)$ is a linear expression of $\nabla^2 V$.
After successively reducing $k$ via this construction, we may assume in the following that $k = 0$.

\begin{Claim} \label{Cl_heat_flow_Lunardi}
There is a family of operators $(P_t : C^0(M;E) \to C^0(M;E))_{t > 0}$ such that for any $u \in C^0(M;E)$ the family $(u_t := P_t u)_{t > 0}$ is a solution to the parabolic equation
\begin{equation} \label{eq_heat_eq_u_E}
 \partial_t u_t = (L+\lambda) u_t, \qquad u_t \xrightarrow[t \searrow 0]{} u \quad \text{in} \quad C^0(M;E) 
\end{equation}
Moreover, for any $0 \leq \beta_1 \leq \beta_2 \leq 1$ we have
\begin{equation} \label{eq_P_t_bound}
 \Vert P_t u \Vert_{C^{2,\beta_2}} \leq C(A) \Big( \frac{1}{t^{1+(\beta_2 - \beta_1)/2}} + 1 \Big) \Vert u \Vert_{C^{0,\beta_1}}. 
\end{equation}
\end{Claim}

\begin{proof}
The first part of the claim follows from standard parabolic theory.
To see \eqref{eq_P_t_bound}, note that a computation of the form \eqref{eq_Lu_cdot_u_computation} implies that for $u_t := P_t u$
\begin{equation} \label{eq_evolution_norm_u}
 \partial_t |u_t|^2 \leq \triangle |u_t|^2 - 2 |\nabla u|^2 + \nabla_V |u_t|^2 
\end{equation}
So by the maximum principle, we have $\Vert u_t \Vert_{C^0} \leq \Vert u \Vert_{C^0}$.
Thus, if the claim is known to be true for $t < 1$, then the case $t \geq 1$ follows by applying \eqref{eq_P_t_bound} to $u = u_{t - \frac12}$ for $t = \frac12$ and $\beta_1 = 0$.
So assume for the remainder of the proof that $t < 1$.

Next, we establish \eqref{eq_P_t_bound} in the case $(\beta_1, \beta_2) = (0,1)$.
To do this, we repeat the construction involving \eqref{eq_commutation_nab_u} three times and denote the corresponding parabolic solutions to \eqref{eq_heat_eq_u_E} by $u_t^{\prime,1}, u_t^{\prime,2}, u_t^{\prime,3}$.
Each of these time-dependent tensors fields satisfies a bound of the form \eqref{eq_evolution_norm_u}, so if we set
\[ z := |u_t|^2 + t |u^{\prime,1}_t|^2 + \tfrac12 t^2 |u^{\prime,2}_t|^2 + \tfrac13 t^3 |u^{\prime,3}_t|^2, \]
then
\[ \partial_t z_t \leq \triangle z_t + \nabla_V z_t. \]
The bound \eqref{eq_P_t_bound} now follows by applying the maximum principle to this equation and noticing that $\Vert u_t \Vert_{C^{2,1}} \leq \Vert u_t \Vert_{C^{3}} \leq C(A) \Vert u^{\prime,3}_t \Vert_{C^0}$.

We may establish \eqref{eq_P_t_bound} in the cases in which $\beta_1, \beta_2 \in \{ 0,1 \}$ in a similar way.
Using the interpolation inequality, we obtain the case in which $\beta_1 \in \{ 0,1 \}$ and $\beta_2 \in [0,1]$.
To establish \eqref{eq_P_t_bound} if $\beta_1 \in (0,1)$, choose a decomposition $u = u' + u''$ (for example, using \cite[Proposition~2.8]{Deruelle}) such that
\[ \Vert u' \Vert_{C^0} \leq C(A) t^{\beta_1/2} \Vert u \Vert_{C^{0,\beta_1}}, \qquad \Vert u'' \Vert_{C^1} \leq C(A) t^{(\beta_1 - 1)/2} \Vert u \Vert_{C^{0,\beta_1}}. \]
Then
\[ \Vert P_t u \Vert_{C^{2,\beta_2}}
\leq \Vert P_t u' \Vert_{C^{2,\beta_2}} + \Vert P_t u'' \Vert_{C^{2,\beta_2}}
\leq \frac{C(A)}{t^{1+\beta_2/2}} \Vert u' \Vert_{C^{0}} + \frac{C(A)}{t^{1+(\beta_2-1)/2}} \Vert u'' \Vert_{C^{1}} 
\leq \frac{C(A)}{t^{1+(\beta_2 - \beta_1)/2}} \Vert u \Vert_{C^{0,\beta_1}}, \]
which finishes the proof of the claim.
\end{proof}

We can now construct a solution $u$ to \eqref{eq_Lu_is_f} via the representation (for any $\xi > 0$)
\[ u = \int_0^\infty (P_t v) e^{-\lambda t} dt = \int_0^\xi (P_t v) e^{-\lambda t} dt + \int_\xi^\infty (P_t v) e^{-\lambda t} dt =: u'_\xi + u''_\xi. \]
Choose $\delta > 0$ such that $0 < \alpha' := \alpha - \delta < \alpha < \alpha'' := \alpha + \delta < 1$.
Then any $\xi \in (0,1]$ we have by the Claim
\[ \Vert u'_\xi \Vert_{C^{2,\alpha'}}
\leq \int_0^\xi \Vert P_t v \Vert_{C^{2,\alpha'}} e^{-\lambda t} dt
\leq \int_0^\xi \frac{C(A)}{t^{1 - \delta}} \Vert v \Vert_{C^{0,\alpha}} dt \leq C \xi^\delta \Vert v \Vert_{C^{0,\alpha}} \]
\[ \Vert  u''_\xi \Vert_{C^{2,\alpha''}}
\leq \int_\xi^\infty \Vert P_t v \Vert_{C^{2,\alpha''}} e^{-\lambda t} dt
\leq \int_\xi^\infty \frac{C(A)}{t^{1 + \delta}} \Vert v \Vert_{C^{0,\alpha}} dt \leq C \xi^{-\delta}\Vert v \Vert_{C^{0,\alpha}} \]
This implies \eqref{eq_u_f_bound_m_alph_weaker} via \cite[\S 2.7.2, Thm. 1]{Triebel_1995_book} or \cite[Proposition~2.8]{Deruelle} (the second reference suffices in our case).
\end{proof}

\subsection{Weighted estimate} \label{subsec_Lunardi_weighted}
The following estimates, which hold on gradient expanding solitons, are analogous to \cite[Corollary~2.4]{Deruelle}.
We refer to Subsection~\ref{subsec_list_Holder} for a definition of the H\"older norms used here.
We recall that $\triangle_f = \triangle - \nabla_{\nabla f}$.

\begin{Proposition} \label{Prop_Lundardi_weighted}
Let $(M,g,f)$ be an $n$-dimensional gradient expanding soliton with bounded curvature on an orbifold with isolated singularities.
Assume that $f$ attains a maximum, let $E$ be a tensor bundle over $M$ and suppose that $k \geq 0$, $\alpha \in (0,1)$ and $a, a' > 0$.
Then for any $\lambda \geq 0$ and $b \in C^{k,\alpha}_{-a'}(M;\End E)$ the operator
\[ L := \triangle_f - \lambda + b : C^{k+2,\alpha}_{-a,\nabla f} (M;E) \longrightarrow C^{k,\alpha}_{-a} (M;E) \]
is well defined and bounded and has Fredholm index $0$.
If $b = 0$, then $L$ is invertible and for every $v \in C^{k,\alpha}_{-a} (M; E)$ the equation $L u = v$ has a unique solution $u \in C^{k+2,\alpha}_{-a} (M; E)$ among all $C^2_{\loc}$ solutions for which $|u| \to 0$ at infinity.
\end{Proposition}

Proposition~\ref{Prop_Lundardi_weighted} is a consequence of the following lemma, which is essentially the same as \cite[Thm.~2.3]{Deruelle}.

\begin{Lemma} \label{Lem_delta_f_ell_bounds}
Let $(M,g,f)$ be an $n$-dimensional gradient expanding soliton on an orbifold with isolated singularities, let $E$ be a tensor bundle over $M$ and suppose that $k \geq 0$, $\alpha \in (0,1)$, $\lambda \geq 0$, $a > 0$ and $\max_M f = -1$.
Consider the following equation for sections $u$, $v$ of $E$:
\begin{equation} \label{eq_Lap_f_u_v}
 \triangle_f u + (b_0 - \lambda) u = v, 
\end{equation}
where, for $W := |\nabla f|^2 + R + f \equiv const$,
\begin{equation} \label{eq_b0_def}
 b_0 :=  \frac{a}2 \cdot \frac{\frac{n}2 + W}{-f} 
-   \Big( \frac{a^2}{4} + \frac{a}2 \Big) \frac{|\nabla f|^2}{(-f)^2}. 
\end{equation}
Suppose that for some $A > 0$
\[ \Vert {\Rm} \Vert_{C^0}, \big| |\nabla f|^2 + R + f \big|, \lambda, n, \rank E, k, |a| \leq A, \qquad  a \leq - A^{-1}, \qquad \alpha \in [A^{-1}, 1-A^{-1}]. \]
Then for every $v \in C^{k,\alpha}_{-a} (M; E)$ the equation \eqref{eq_Lap_f_u_v} has a solution $u \in C^{k+2,\alpha}_{-a} (M;E)$ and this solution is unique among $C^2_{\loc}$ solutions with the property that $(-f)^{a/2} |u|$ is uniformly bounded.
Moreover
\[ \Vert u \Vert_{C^{k+2,\alpha}_{-a,\nabla f}} \leq C(A) \Vert v \Vert_{C^{k,\alpha}_{-a}}. \]
\end{Lemma}

\begin{proof}
The proof is the same as that of \cite[Thm.~2.3]{Deruelle}. 
Setting $u' := (-f)^{a/2} u$, $v' := (-f)^{a/2}v$, and applying the soliton equation, we obtain that equation \eqref{eq_Lap_f_u_v} is equivalent to
\[ \triangle_{f+a\log (-f)} u' - \Big( \frac{a}2 + \lambda \Big)  u' = v'. \]
We can then apply Proposition~\ref{Prop_Lunardi} along with the bounds $|\nabla f|^2 \leq f + C(A)$, $|\nabla^{k+2} f| = |\nabla^k {\Ric}| \leq C(A)$ via Shi's estimates to finish the proof.
\end{proof}
\bigskip

\begin{proof}[Proof of Proposition~\ref{Prop_Lundardi_weighted}.]
To see that $L$ is well defined and bounded, note that by Lemma~\ref{Lem_Holder_norms_properties} we have for any $u \in C^{k+2,\alpha}_{-a,\nabla f}(M;E)$
\[ \Vert L u \Vert_{C^{k,\alpha}_{-a}} 
\leq \Vert \triangle u - \nabla_{\nabla f} u \Vert_{C^{k,\alpha}_{-a}} + 
\Vert b u \Vert_{C^{k,\alpha}_{-a}}   
\leq C\Vert u \Vert_{C^{k+2,\alpha}_{-a,\nabla f}} + \Vert b \Vert_{C^{k,\alpha}} \Vert u \Vert_{C^{k,\alpha}_{-a}} \leq C\Vert u \Vert_{C^{k+2,\alpha}_{-a,\nabla f}}. \]
If $b = b_0$, where $b_0$ is chosen as in \eqref{eq_b0_def}, then $L$ is invertible by Lemma~\ref{Lem_delta_f_ell_bounds}.

Let now $b \in C^{k,\alpha}_{-a'}(M;\End E)$ be arbitrary.
To show that $L$ has Fredholm index $0$, we write $L = (\triangle_f + b_0) + K$, where $K u = (b-b_0)u$.
We need to show that $K : C^{k+2,\alpha}_{-a,\nabla f} (M;E) \to C^{k,\alpha}_{-a} (M;E)$ is compact.
To see this, let $u_i \in C^{k+2,\alpha}_{-a,\nabla f} (M;E)$ with $\Vert u_i \Vert_{C^{k+2,\alpha}_{-a, \nabla f}} \leq 1$.
Using Arzela-Ascoli, we can pass to a subsequence such that $u_i \to u_\infty \in C^{k+1}_{-a}(M;E)$ in $C^{k+1}_{\loc}$.
Choose $0 < a'' < a$ such that $a'' + a' \geq a$.
We claim that $u_i \to u_\infty$ in $C^{k+1}_{-a''}$.
Indeed, for any $p \in M$, $r > 0$ and $m = 0, \ldots, k+1$ we have
\begin{align*}
 \sup_M &\big| (-f)^{a''/2} \nabla^m (u_i - u_\infty) \big| \\
&\leq \sup_{B(p,r)} \big|(-f)^{a''/2} \nabla^m (u_i - u_\infty) \big| \\
&\qquad\qquad + \sup_{M \setminus B(p,r)} \big|(-f)^{a''/2} \nabla^m u_i  \big|  +  \sup_{M \setminus B(p,r)} \big|(-f)^{a''/2} \nabla^m u_\infty \big| \\
&\leq \sup_{B(p,r)} \big|(-f)^{a''/2} \nabla^m (u_i - u_\infty) \big| + C \sup_{M \setminus B(p,r)}  (-f)^{a''
/2-a/2}.
\end{align*}
The first term on the right-hand side goes to zero as $i \to \infty$ and the second term goes to zero as $r \to \infty$.
We now obtain, using Lemma~\ref{Lem_Holder_norms_properties}, that
\[ \Vert Ku_i - Ku_\infty \Vert_{C^{k,\alpha}_{-a}}
\leq \Vert b \Vert_{C^{k,\alpha}_{-a'}} \Vert u_i - u_\infty \Vert_{C^{k,\alpha}_{-a''}}  \longrightarrow 0. \]
This proves that $K$ is compact and hence that $L$ has Fredholm index $0$.

To see the last statement, suppose that $b = 0$ and suppose that we had $Lu =0$ for some non-zero $u \in  C^{k+2,\alpha}_{-a,\nabla f} (M;E)$.
A computation similar to \eqref{eq_Lu_cdot_u_computation} shows that $0 \leq \triangle |u|^2 + \nabla_V |u|^2 - 2\lambda |u|^2$.
Since $|u| \to 0$ near infinity, we obtain a contradiction by the maximum principle.
So $L$ is injective and thus also surjective.
The fact that $u$ is unique among all $C^2_{\loc}$-solutions that decay at infinity also follows from the maximum principle.
\end{proof}

\subsection{Spectral theory} \label{subsec_spectral_theory}
In this subsection we establish a spectral theory for self-adjoint elliptic operators on a gradient expanding soliton equipped with an inverse Gaussian measure with density $\sim e^{r^2/4}$ and we define the index of such operators.
This definition will be crucial for the definition of the \emph{integer} expander degree.
We remark that the following discussion is not used in  the $\IZ_2$-degree theory; so a reader who is only interested in this theory may skip this subsection.

The following theory builds on earlier work of Deruelle \cite{Deruelle_2015}.
Let us first define the spaces for the $L^2$-theory.

\begin{Definition}
Let $(M,g)$ be a complete Riemannian orbifold with isolated singularities, $f \in C^0(M)$ a potential function and let $E$ be a Euclidean vector bundle equipped with a metric connection.
We define
\[ L^2_{g,f}(M;E) := \left\{ u \in L^2_{\loc}(M;E) \;\; : \;\; \Vert u \Vert_{L^2_{g,f}} := \int_M |u|^2 e^{-f} dg < \infty \right\} \]
and for $u_1, u_2 \in L^2_{g,f}(M;E)$ we define the inner product
\[ \langle u_1, u_2 \rangle_{g,f} := \int_M (u_1 \cdot u_2) e^{-f} dg. \]
Note that $L^2_{g,f}(M;E)$ together with this inner product is a Hilbert space.
We also define $H^1_{g,f} (M;E) \lb \subset \lb L^2_{g,f}(M;E)$ as the subspace of $H^1_{\loc}$-regular sections $u$ such that $\Vert u \Vert_{H^1_{g,f}}^2  = \langle u, u \rangle_{H^1_{g,f}}  < \infty$, where 
\[ \langle u_1, u_2 \rangle_{H^1_{g,f}} := \int_M \big( (u_1 \cdot u_2) + (\nabla u_1 \cdot \nabla u_2) \big) e^{-f} dg . \]
As before, we will sometimes drop the metric ``$g$'' in the subscript if there is no chance of confusion.
\end{Definition}

The following proposition is the main result of this subsection.
Note that we will be specifically interested in the case in which $L = \triangle_f + 2 \Rm$ is the Einstein operator on a gradient expanding soliton, which is defined on the vector bundle of symmetric $(0,2)$-tensors.
Note also that in the setting of the following proposition, the density function behaves like $e^{-f} \sim e^{r^2/4}$.

\begin{Proposition} \label{Prop_L2_theory}
Consider an $n$-dimensional gradient expanding soliton $(M,g,f)$ with bounded curvature on an orbifold with isolated singularities and a tensor bundle $E$ over $M$.
Suppose that $f$ is bounded from above and proper. 
Let $b \in C^\infty_{\loc}(M; \End E)$ be a self-adjoint endomorphism field with $b \to 0$ at infinity and consider the operator $L = \triangle_f + b = \triangle - \nabla_{\nabla f} + b$.
Then the following is true:
\begin{enumerate}[label=(\alph*)]
\item \label{Prop_L2_theory_a} For any $u_1, u_2 \in H^1_f (M; E)$ with the property that $u_1 \in H^2_{\loc} (M; E)$ and $Lu_1 \in L^2_f(M; E)$ we have 
\[  -\langle L u_1, u_2 \rangle_f  = \int_M \big( (\nabla u_1 \cdot \nabla u_2) - b(u_1, u_2) \big) e^{-f} dg =: I_f (u_1, u_2), \]
where we view $b$ as a symmetric $(0,2)$-tensor on $E$.
\item \label{Prop_L2_theory_b} For any $u \in L^2_{f}(M;E)$ or $u \in C^{2,\alpha}_{-a,\nabla f} (M; E)$, $\alpha \in (0,1)$, $a > 0$, with the property that $L u \in L^2_f (M; E)$ we have $u \in H^1_f (M; E)$.
\item \label{Prop_L2_theory_c} There are eigenvalues $\lambda_0 < \lambda_1 < \ldots$ of $-L$ such that $\lambda_i \to \infty$ and an orthogonal decomposition
\[ L^2_f(M;E) = \ov{ \bigoplus_{i=1}^\infty \EE_{\lambda_i}}, \]
where 
\[ \EE_{\lambda} = \big\{ u \in L^2_f(M;E) \cap C^{2}_{\loc}(M;E) \;\; : \;\; -Lu = \lambda u \big\}. \]
Moreover, the eigenspaces $\EE_{\lambda_i}$ are finite dimensional, non-trivial and contained in $H^1_f (M; \lb E) \lb \cap \lb C^\infty_{\loc}(M;E)$.
\item \label{Prop_L2_theory_d} Suppose that $\sup_M f = -1$ and that we have the  bounds $\inj(M,g) > A^{-1} > 0$ and $|{\Rm}|, n \leq A$ and $|R|, |b| \leq A(-f)^{-1}$ for some $A < \infty$.
Then for any $u \in \EE_\lambda$ with $|\lambda| \leq A$ and $\Vert u \Vert_{L^2_f} = 1$ and any $\eps > 0$ we have
\[ |u| \leq C(A,\eps) (-f)^{-n/2+\la+\eps} e^{f}. \]
\end{enumerate}
\end{Proposition}

Proposition~\ref{Prop_L2_theory} allows us to define the index and nullity of $L$:

\begin{Definition} \label{Def_idx_null}
In the setting of Proposition~\ref{Prop_L2_theory}, we define the {\bf index} and {\bf nullity} of $-L$ by
\[ \Index (-L) := \sum_{\lambda_i < 0} \dim \EE_{\lambda_i}, \qquad \Null(-L) := \dim \EE_0 = \dim \ker L. \]
\end{Definition}

We obtain the following characterization of the index.

\begin{Corollary} \label{Cor_index}
In the setting of Proposition~\ref{Prop_L2_theory} the following holds:
\begin{enumerate}[label=(\alph*)]
\item \label{Cor_index_a} There is an orthogonal decomposition
\[ L^2_{f}(M;E) = N_{f} \oplus K_{f} \oplus P_{f} \]
such that the following is true: $\dim N_{f} = \Index(-L) < \infty$ and $\dim K_{f} = \Null(-L) < \infty$ and $N_f, K_f$ are contained in $H^1_{f}(M; E) \cap C^\infty_{\loc}(M;E)$ and $P_{f} \cap H^1_{f} (M; E)$ is dense in $P_{f}$.
Moreover, the bilinear form $I_f$ is diagonal with respect to this decomposition, negative definite on $N_{f}$, vanishes on $K_{f}$ and positive definite on $P_{f} \cap H^1_{f} (M; E)$.
\item \label{Cor_index_b} Suppose there is a (not necessarily orthogonal) decomposition
\[ L^2_{f}(M;E) = N' \oplus P' \]
such that the following holds: $N'$ is finite dimensional and contained in $H^1_{f}(M; E)$ and $I_f$ is negative definite on $N'$ and positive semidefinite on $P' \cap H^1_{f} (M; E)$.
Then $\Index (-L) = \dim N'$.
\end{enumerate}
\end{Corollary}
\bigskip

The proof of Proposition~\ref{Prop_L2_theory} relies on the following Poincar\'e inequality (compare with \cite[Lemma 5.1]{Munteanu_Wang} and \cite[Lemma~3.2]{Deruelle_Schulze_2021}):

\begin{Lemma} \label{Lem_L2f_Poincare}
Consider the setting of Proposition~\ref{Prop_L2_theory}.
Then for any $u \in H^1_{f}(M;E)$ we have
\[ \int_M |\nabla f|^2 |u|^2 e^{-f} dg \leq 4 \int_M |\nabla u|^2 e^{-f}dg. \]
\end{Lemma}

\begin{proof}
Since $C^1_c(M;E) \subset H^1_f(M;E)$ is dense, it suffices to establish the inequality for $u \in C^1_c(M; E)$.
By Theorem~\ref{Thm_soliton_PSC} we have $\triangle f = - R - \frac{n}2 \leq 0$.
So integration by parts gives
\begin{multline*}
\int |\nabla f|^2 |u|^2 e^{-f} dg =  \int_M \big(\nabla f \cdot \nabla |u|^2 + (\triangle f )  |u|^2 \big) e^{-f} dg \\
\leq \frac12 \int_M |\nabla f|^2 |u|^2 e^{-f} dg + 2\int_M |\nabla u|^2 e^{-f} dg  
\end{multline*}
This proves the lemma.
\end{proof}
\bigskip

\begin{proof}[Proof of Proposition~\ref{Prop_L2_theory}.]
Note first that the set of compactly supported sections, $C^\infty_c (M;E)$, is dense in both $L^2_f (M; E)$ and $H^1_f (M; E)$.
Due to integration by parts, Assertion~\ref{Prop_L2_theory_a} holds if $u_2 \in C^\infty_c (M;E)$; the assertion follows for general $u_2$ by approximation.

Next, let $\beta : M \to [0,\infty)$ be a smooth and bounded function, which we will choose later and define $\hat L := L - \beta$ and
\[ \hat I_f (u_1, u_2) := \int_M \big( (\nabla u_1 \cdot \nabla u_2) - b(u_1, u_2) + \beta (u_1 \cdot u_2) \big) e^{-f} dg.  \]
Note that Assertion~\ref{Prop_L2_theory_a} continues to hold if we replace $L, I_f$ with $\hat L, \hat I_f$, respectively.
Next we will prove Assertion~\ref{Prop_L2_theory_c}.

\begin{Claim} \label{Cl_hatI_inner_prod}
There is a smooth, compactly supported function $\beta_0 : M \to [0,\infty)$ such that if $\beta \geq \beta_0$, then $\hat I_f$ is positive definite and the associated norm $\hat I_f^{1/2}$ is equivalent to $\Vert \cdot \Vert_{H^1_f}$.
Moreover, there is a bounded operator $\hat L^{-1} : L^2_f (M;E) \to H^1_f(M;E)$ such that for any $v \in L^2_f(M;E)$ we have
\[ - \langle v, w \rangle_f = \hat I_f (\hat L^{-1} v, w ) \qquad \text{for all} \quad w \in H^1_f(M;E). \]
So $\hat L^{-1} v$ is a weak solution to $\hat L u = v$.
\end{Claim}

\begin{proof}
The first statement is a direct consequence of Lemma~\ref{Lem_L2f_Poincare}; note that $|\nabla f|^2 = C  - R - f$ is proper.
The existence of $\hat L^{-1}$ follows from Riesz representation applied to the inner product $\hat I_f$ on $H^1_f$.
To see that $\hat L^{-1}$ is bounded, note that for $u := \hat L^{-1} v$
\[ c \Vert u \Vert^2_{H^1_f} \leq \hat I_f (u, u) = - \langle v, u \rangle_f 
\leq \Vert v \Vert_{L^2_f} \Vert u \Vert_{L^2_f}
\leq \Vert v \Vert_{L^2_f} \Vert u \Vert_{H^1_f}, \]
so $c\Vert u \Vert_{H^1_f} \leq \Vert v \Vert_{L^2_f}$.
\end{proof}

\begin{Claim} \label{Cl_Rellich}
The embedding $L^2_f(M;E) \to H^1_f(M;E)$ is compact.
\end{Claim}

\begin{proof}
Consider a bounded sequence $u_i \in H^1_f(M;E)$.
By Lemma~\ref{Lem_L2f_Poincare} there is a constant $C < \infty$ such that for any $A > 1$
\begin{equation} \label{eq_superlevelset_bound}
 \int_{\{|\nabla f| > A \}} |u_i|^2 e^{-f} dg \leq \frac{C}{A^2} 
\end{equation}
Since $\{|\nabla f| \leq A \}$ is compact, we can pass to a subsequence, such that for some $u_\infty \in L^2_f(M;E)$ and sequence $A_i \to \infty$ we have
\[ \int_{\{ |\nabla f| \leq A_i \}} |u_i - u_\infty|^2 e^{-f} dg \to 0. \]
So $u_\infty$ also satisfies \eqref{eq_superlevelset_bound} and thus
\[ \int_{\{|\nabla f| > A_i \}}|u_i - u_\infty|^2 e^{-f} dg \leq \frac{4C}{A_i^2} \to 0. \qedhere\]
\end{proof}
\medskip

Combining Claims~\ref{Cl_hatI_inner_prod}, \ref{Cl_Rellich} implies that $-\hat L^{-1} : L^2_f(M;E) \to L^2_f(M;E)$ is compact, self-adjoint and positive definite.
Therefore it has a discrete spectrum and finite dimensional eigenspaces, which implies Assertion~\ref{Prop_L2_theory_c} for $\hat L$.
We obtain the statement for $L$ if we choose $\beta$ to be constant and adjust the eigenvalues accordingly.
The regularity statement follows from standard local elliptic regularity theory.

Next, let us prove Assertion~\ref{Prop_L2_theory_b} in the (easier) case in which $u \in L^2_f(M;E)$.
We obtain that $\triangle_f u =  L u - b(u) \in L^2(M;E)$.
Let $\eta \in C^\infty_c(M)$, $0 \leq \eta \leq 1$ be a cutoff function with $0 \leq \eta \leq 1$ and $|\nabla \eta| \leq 1$.
Integration by parts yields
\begin{multline*}
  \int_M \eta^2 |\nabla u|^2 e^{-f} dg 
  = \int_M -2\eta \nabla \eta \cdot u \cdot \nabla u \,e^{-f} dg + \int_M \eta^2 \triangle_f u \cdot u \, e^{-f} dg  \\
  \leq \frac12 \int \eta^2 |\nabla u|^2 dg + 2 \int |\nabla \eta|^2 |u|^2 e^{-f} dg +  \Vert \triangle_f u \Vert_{L^2_f} \Vert u \Vert_{L^2_f}.   
\end{multline*}
So
\[ \int \eta^2 |\nabla u|^2 e^{-f} dg \leq 2\Vert u \Vert_{L^2_f}^2 + 2\Vert \triangle_f u \Vert_{L^2_f} \Vert u \Vert_{L^2_f}. \]
Letting the support of $\eta$ go to $M$ implies that $u \in H^1_f(M;E)$, as desired.

Let us now prove Assertion~\ref{Prop_L2_theory_b} in the (harder) case in which $u \in C^{2,\alpha}_{-a,\nabla f}(M;E)$.
By the previous argument, it suffices to check that $u \in L^2_f(M;E)$.
We will achieve this by representing $u$ as a time-integral of the heat flow starting from $\hat L u$, as in the proof of Proposition~\ref{Prop_Lunardi}.
We will then bound this heat flow in terms of $L^2_f$.
As a preparation for our arguments, suppose without loss of generality that $\sup f = -1$ and suppose that $\beta$ is compactly supported.
Then $\hat L u = L u - \beta u \in L^2_f(M;E)$.
As in Claim~\ref{Cl_heat_flow_Lunardi} in the proof of Proposition~\ref{Prop_Lunardi}, there is a family of bounded operators $(\hat P_t : C^0(M; E) \to C^0(M; E))_{t > 0}$ such that for any $u \in C^0(M; E)$ the family $(u_t := \hat P_t u \in C^2_{\loc}(M;E))_{t > 0}$ is a solution to the parabolic equation
\begin{equation*} %
 \partial_t u_t = \hat L u_t, \qquad u_t \xrightarrow[t \searrow 0]{} u \quad \text{in} \quad C^0(M;E) 
\end{equation*}
For example, $\hat P_t$ can be obtained by approximating $\hat L$ by operators with bounded coefficients.

The following claim is, at its heart, a restatement of Lemma~\ref{Lem_delta_f_ell_bounds}.

\begin{Claim} \label{Cl_C0_decay}
For any $a > 0$ there is a smooth, bounded and compactly supported function $\beta'_{a} : M \to [0,\infty)$ such that if $\beta \geq \beta'_a$, then the following holds.
If $u$ is continuous and $(-f)^{a/2}|u|$ is bounded, then for some $c_a > 0$ the quantity
\begin{equation} \label{eq_C0_monotonicity}
 e^{c_a t} \max_M (-f)^{a/2} |u_t| 
\end{equation}
is non-increasing in $t$.
Moreover, $v := \int_0^\infty u_t \, dt$ exists and $(-f)^{a/2} |v|$ is bounded.
If $u$ is of regularity $C^{0,\alpha}_{\loc}$, then $v$ is of regularity $C^{2,\alpha}_{\loc}$ and we have $-\hat Lv = u$.
\end{Claim}

\begin{proof}
Let $\td u_t := (-f)^{a/2} |u_t|$ and note that in the viscosity sense
\begin{multline} \label{eq_sqtdufau}
 (\partial_t - \triangle_f) \td u_t
\leq -(\triangle_f (-f)^{a/2}) |u_t| + a (-f)^{a/2} \frac{\nabla f}{-f} \cdot \nabla |u_t| +(-f)^{a/2}  \frac{u_t}{|u_t|} \cdot (\partial_t - \triangle_f )u_t \\
= \left(\frac{a}2 \, \frac{\triangle_f f}{-f} - \frac{a(a-2)}{4} \, \frac{ |\nabla f|^2}{(-f)^2}  \right) \td u_t + a \frac{\nabla f}{-f} \cdot \nabla \td u_t + \frac{a^2}{2} \frac{|\nabla f|^2}{(-f)^2} \td u_t + (|b| - \beta) \td u_t.
\end{multline}
Since $\triangle f = -R - \frac{n}2$ and $|\nabla f|^2 = C - f - R$ and since $f$ is proper, we find that $\frac{|\nabla f|^2}{-f} \to 1$ at infinity and that all other terms on the right-hand side of \eqref{eq_sqtdufau} decay.
So for an appropriate choice of $\beta$ we have
\[ \left(\partial_t - \triangle_f -  a \frac{\nabla f}{-f} \cdot \nabla \right) \td u_t \leq - \frac{a}{4} \td u_t. \]
The monotonicity of \eqref{eq_C0_monotonicity} follows using the maximum principle; the fact that $M$ may be non-compact can be dealt with as explained in the beginning of the proof of Proposition~\ref{Prop_Lunardi}.
To see the last statement, write $v = v' + v'' := \int_0^1 u_t \, dt + \int_1^\infty u_t \, dt$.
Standard local parabolic theory implies that $v'$ is of regularity $C^{2,\alpha}$ and that $- \hat Lv' = u -u_1$.
The exponential decay of $(-f)^{a/2} |u_t|$ combined with local parabolic derivative estimates implies local exponential decay of higher derivatives and thus $v''$ is also $C^{2,\alpha}$ and satisfies $-\hat L v'' = u_1$.
\end{proof}

In the next claim we establish decay of $u_t$ in $L^2_f$.

\begin{Claim}
There is a smooth, compactly supported function $\beta'' : M \to [0,\infty)$ such that if $\beta \geq \beta''$, then the following holds.
If $u \in L^2_f(M;E) \cap C^0(M;E)$, then $u_t \in L^2_f (M; E)$ for all $t \geq 0$ and there is a $c > 0$ such that we have for $t \geq 0$
\[ \Vert u_t \Vert_{L^2_f} \leq e^{-ct} \Vert u \Vert_{L^2_f}. \]
\end{Claim}

\begin{proof}
Let $\td u_t :=  |u_t|e^{-f/2}$ and note that in the viscosity sense
\begin{multline*}
 (\partial_t - \triangle ) \td u_t 
=    (\partial_t - \triangle) |u_t|e^{-f/2} + \nabla f \cdot \nabla |u_t| e^{-f/2} -  |u_t|\triangle e^{-f/2} \\
= \big( (\partial_t - \triangle_f)|u_t| \big) e^{-f/2} + \tfrac14 (2\triangle f - |\nabla f|^2 ) \td u 
\leq \left( |b| - \beta  + \tfrac14 \big( -2R - n  - |\nabla f|^2 ) \right) \td u.
\end{multline*}
By Theorem~\ref{Thm_soliton_PSC} and since, again, $|\nabla f|^2 = C - f  - R$ and $f$ is proper, we may choose $\beta$ such that the term in the parentheses is $\leq -c < 0$.
This implies that the time-derivative of $e^{2ct} \Vert u_t \Vert^2_{L^2_f} = e^{2ct}\Vert \td u_t \Vert^2_{L^2}$ is non-positive.
\end{proof}

Suppose now that $\beta \geq \max \{ \beta'_a, \beta'' \}$.
Applying Claim~\ref{Cl_C0_decay} and the uniqueness statement of Proposition~\ref{Prop_Lundardi_weighted}, yields that 
\[ u = \int_0^\infty \hat P_t(\hat Lu) dt \in L^2_f (M;E), \]
which finishes the proof of Assertion~\ref{Prop_L2_theory_b}.

Lastly, we prove Assertion~\ref{Prop_L2_theory_d}.
Let $u \in \EE_\lambda$ and $A$ as in the statement of this assertion.
We first establish a weaker decay bound.

\begin{Claim} \label{Cl_prelim_decay}
We have the bound
\[ |u| \leq C(A) (-f)^{n/4} e^{f/2}. \]
\end{Claim}

\begin{proof}
We argue similarly as in \cite{Deruelle_2015}.
Fix $p \in M$ and choose $r_p := (-f(p))^{-1/2} \leq 1$.
Due to the equation $|\nabla f |^2 + R = - f$ we have 
\begin{equation} \label{eq_ffp_diff}
|f - f(p) | \leq C(A) \qquad \text{on} \quad B(p,r_p).
\end{equation}
Consider the rescaled metric $g'_p := r_p^{-2} g$.
Then $B_{g'_p} (p,1) = B_g(p,r_p)$ and we can choose local coordinates around $p$ in which the metric coefficients of $g'_p$ and their derivatives are uniformly bounded.
Moreover, the eigenvalue equation becomes
\begin{equation} \label{eq_rescaled_elliptic_eq}
 -\triangle_{g'_p} u +  \nabla_{r_p^{2} \nabla^g f} u - r_p^2 b(u) = r_p^2 \lambda u. 
\end{equation}
Note that on $B_{g'_p}(p,1)$ we have $|r_p^{2} \nabla^g f|^2_{g'_p} = r^2_p |\nabla^g f|_g^2 \leq r^2_p (-f + R) \leq C(A)$, so the coefficients of \eqref{eq_rescaled_elliptic_eq} are uniformly bounded.
Moreover, we have due to \eqref{eq_ffp_diff}
\[ \Vert u \Vert_{L^2_{g'_p}(B(p,r))}^2 
\leq C(A) r_p^{-n} e^{f(p)} \Vert u \Vert_{L^2_{g,f}(B_g(p,r))}^2 
= C(A)(-f(p))^{n/2} e^{f(p)}. \]
The claim now follows from standard elliptic estimates applied to \eqref{eq_rescaled_elliptic_eq}.
\end{proof}

Let $a > 0$ and $H < \infty$ be constants whose values we will determine later and let $\delta \geq 0$ be minimal with the property that
\begin{equation} \label{eq_u_bounded_H_delta}
 |u| \leq H (-f)^{-n/2 + \la +\eps} e^f + \delta (-f)^{-a}. 
\end{equation}
Note that due to Claim~\ref{Cl_prelim_decay}, such a $\delta$ exists and if $\delta > 0$, then we have equality in \eqref{eq_u_bounded_H_delta} at some point.

\begin{Claim} \label{Cl_no_equality_far}
There is a constant $F(A,\eps) < \infty$ such that if $\lambda + \frac\eps{2}  \leq \frac{a}2$, then we have strict inequality in \eqref{eq_u_bounded_H_delta} on $\{ f \leq - F \}$.
\end{Claim}

\begin{proof}
Let $p \in M$ be a point where equality in \eqref{eq_u_bounded_H_delta} holds.
We will derive a contradiction assuming that $-f(p)$ is large enough.
At $p$ we have
\begin{equation} \label{eq_Lap_ineq}
 \triangle_f |u| \leq \triangle_f \big( H (-f)^{-n/2+\lambda +\eps} e^f + \delta (-f)^{-a}\big). 
\end{equation}
We compute that
\[ |u|  \triangle_f |u| + |\nabla |u| |^2 
= \tfrac12 \triangle_f |u|^2 
= u \cdot \triangle_f u +   |\nabla u |^2 . \]
So, using Kato's inequality,
\begin{equation} \label{eq_Lap_u_bound}
 \triangle_f |u|  
\geq \frac{u}{|u|} \cdot \triangle_f u 
= \frac{u}{|u|} \cdot  (Lu - b(u))
= - \lambda |u| - \frac{u \cdot b(u)}{|u|}
\geq - \lambda |u| - C(A) (-f)^{-1} |u|. 
\end{equation}
On the other hand,
\begin{align*}
 \triangle_f \big( (-f)^{-n/2 +\la+ \eps} e^f  \big)
&= \big( \triangle_f (-f)^{-n/2+\la+\eps} \big) e^f - 2( -\tfrac{n}{2} +\la+ \eps ) \frac{ |\nabla f|^2}{-f} (-f)^{-n/2+\la+\eps} e^f \\
&\qquad + (-f)^{-n/2+ \la + \eps} \triangle_f e^f \\
&= \left(- (-\tfrac{n}2 +\la + \eps) \frac{\triangle_f f}{-f} + (-\tfrac{n}2 + \la + \eps) (-\tfrac{n}2 + \la -1+\eps) \frac{|\nabla f|^2}{(-f)^2} \right. \\
&\qquad \left. - (-n + 2\la + 2\eps) \frac{|\nabla f|^2}{-f} + \triangle_f f + |\nabla f|^2 \right)(-f)^{-n/2 + \la + \eps} e^f.
\end{align*}
Since $|\nabla f|^2 = -f-R$ and
\[ \triangle_f f = \triangle f - |\nabla f|^2 = -R - \frac{n}2 - |\nabla f|^2 = f - \frac{n}2, \]
this implies that
\begin{multline} \label{eq_lap_main_term}
 \triangle_f \big( (-f)^{n/2 + \la + \eps} e^f  \big)
\leq \left( - \frac{n}2 + \la + \eps + n - 2\la - 2\eps - \frac{n}2  + C(A) (-f)^{-1} \right)(-f)^{-n/2 + \la + \eps} e^f \\
=\left( - \eps - \la + C(A) (-f)^{-1} \right)(-f)^{-n/2 + \la + \eps} e^f. 
\end{multline}
Similarly,
\begin{equation} \label{eq_lap_pert_term}
 \triangle_f (-f)^{-a} 
=  a(\triangle_f f) (-f)^{-a-1}  -a(a+1) |\nabla f|^2 (-f)^{-a-2}
\leq -a(-f)^{-a} + C(A) (-f)^{-a-1} . 
\end{equation}
Combining \eqref{eq_Lap_ineq}, \eqref{eq_Lap_u_bound}, \eqref{eq_lap_main_term}, \eqref{eq_lap_pert_term} implies that at $p$ we have
\begin{multline*}
 - \lambda |u| - C(A) (-f)^{-1} |u| \\
\leq \left(-(\la +\eps) H + C(A)H (-f)^{-1}  \right)(-f)^{-n/2 + \la + \eps} e^f - \delta a (-f)^{-a} + C(A)\delta (-f)^{-a-1}. 
\end{multline*}
If $-f(p)$ is sufficiently large, depending on $A,\eps$, then this implies that at $p$
\[ - (\lambda+\tfrac{\eps}2) |u| 
< - (\la+\tfrac{\eps}2) H(-f)^{-n/2 + \la + \eps} e^f - \delta \tfrac{a}2 (-f)^{-a}. \]
Plugging in the fact that we have equality in \eqref{eq_u_bounded_H_delta} at $p$ and using the bound $\la + \frac{\eps}2 \leq \frac{a}2$ yields a contradiction.
So $f(p)$ must be bounded from below.
\end{proof}

By Claim~\ref{Cl_prelim_decay}, we can choose $H(A,\eps)$ large enough such that \eqref{eq_u_bounded_H_delta} holds with $\delta = 0$ on $\{ f \geq - F \}$.
Let now $\delta \geq 0$ be minimal such that \eqref{eq_u_bounded_H_delta} holds on all of $M$.
By Claim~\ref{Cl_no_equality_far}, we must have $\delta =0$, because we cannot have strict inequality anywhere.
\end{proof}
\bigskip

\begin{proof}[Proof of Corollary~\ref{Cor_index}.]
Assertion~\ref{Cor_index_a} follows by setting $N_f := \oplus_{\lambda_i < 0} \EE_{\lambda_i}$, $K_f := \EE_0$ and letting $P_f$ be the closure of $\oplus_{\lambda_i > 0} \EE_{\lambda_i}$.

For Assertion~\ref{Cor_index_b} suppose first that by contradiction $\dim N' > \dim N_f$.
Then there is an element $u \in N' \cap (K_f \oplus P_f) \subset H^1_f(M;E)$.
However, this would imply that $I_f(u,u) < 0$ and $I_f (u,u) \geq 0$.
On the other hand, suppose that $\dim N' < \dim N_f$.
Then there is an element $u \in N_f \cap P' \subset H^1_f(M;E)$ and we obtain a contradiction as before.
\end{proof}
\bigskip

\section{Preservation of the gradient condition}\label{sec_gradientness}
\subsection{Main result}
In this section we show that the gradient property of an asymptotically conical expanding soliton remains preserved under $C^1$-deformations, provided that all solitons in this deformations remain asymptotic to a cone (not only a generalized cone!).
Our main result is the following:

\begin{Proposition}\label{Prop_gradientness}
Let $M$ be a smooth, $n$-dimensional orbifold with isolated singularities, $N$ a closed, smooth $(n-1)$-dimensional manifold and let $\iota : (1,\infty) \times N \to M$ an embedding such that $M \setminus \iota ((r,\infty) \times N)$ is compact for all $r \geq 1$.
Consider a family $(M,g_s, V_s)_{s \in (-\eps, \eps)}$, $\eps > 0$,  of expanding solitons on $M$, where $g_s$ and $V_s$ have regularity $C^{11}$ for all $s \in (-\eps, \eps)$.
Suppose that the following is true for some constant $C < \infty$:
\begin{enumerate}[label=(\roman*)]
\item \label{Prop_gradientness_i}  The dependence of the tensor fields $g_s, V_s$ on the parameter $s$ is of regularity $C^1$ in the $C^{11}_{\loc}$-sense (i.e., derivatives of the form $\partial_s \partial^m g_s$ and $\partial_s \partial^m V_s$ exist for $0 \leq m \leq 11$ and are continuous on $M \times (-\eps,\eps)$).
\item \label{Prop_gradientness_ii} For any $s \in (-\eps,\eps)$ the expanding soliton $(M, g_s)$ is asymptotic to a conical metric in the following sense.
There is a continuous family of conical metrics $(\gamma_s \in \CONE^{11}(N))_{s \in (-\eps,\eps)}$ 
on $\IR_+ \times N$ such that for all $s \in (-\eps, \eps)$ we have on $(1, \infty) \times N$:
\begin{align*}
 |\nabla^{m, \gamma_s} (\iota^* g_s - \gamma_s)|_{\gamma_s} &\leq  C r^{-1} \qquad \text{for all} \quad m = 0,1,2,\ldots,11,  \\
  |\iota^* V_s + \tfrac12 r \partial_r|_{\gamma_s} &\leq C.
\end{align*}
Here $r$ and $\partial_r$ denote the coordinate function and standard vector field of the $\IR_+$-factor.
\item \label{Prop_gradientness_iii} For all $s \in (-\eps,\eps)$ we have 
\begin{alignat*}{2}
 |\nabla^{m,g_s} V_s |_{g_s} &\leq C \qquad &\text{for all} \quad m &= 1,\ldots,8, \\
 |\nabla^{m,g_s} (\partial_s V_s) |_{g_s} &\leq Cr^{-1} \qquad &\text{for all} \quad m &= 0,\ldots,5,  \\
 | \partial_s g_s |_{g_s} \leq C  , 
\qquad |\nabla^{m,g_s} (\partial_s g_s) |_{g_s} &\leq C r^{-1} \qquad &\text{for all} \quad m &= 1,\ldots,6
\end{alignat*}
and
\begin{equation} \label{eq_psgV_bounded}
 | (\partial_s g_s) (V_s)|_{g_s} \leq C, \qquad 
\end{equation}
Here $r : M \to \IR_+$ denotes a smooth function that agrees with $r \circ \iota^{-1}$ on $\iota((2,\infty) \times N)$.
\item \label{Prop_gradientness_iv} $(M, g_0, V_0)$ is gradient, i.e., we have $V_0 = \nabla^{g_0} f_0$ for some $f_0 \in C^{1}(M)$.
\end{enumerate}
Then $(M,g_s, V_s)$ is gradient for all $s \in (-\eps,\eps)$ and $V_s = \nabla^{g_s} f_s$ for some family $(f_s \in C^{12}(M))_{s \in (-\eps,\eps)}$ with $C^1$-dependence on $s$ in the $C^{12}_{\loc}$-sense that extends $f_0$.
Moreover, if we assume that $g_s$ and $V_s$ are of regularity $C^k$ and depend on $s$ in the $C^k_{\loc}$-sense for some $k \geq 11$, then the potentials $f_s$ are of regularity $C^{k+1}$ and depend on $s$ in the $C^{k+1}_{\loc}$-sense.
\end{Proposition}

\begin{Remark} \label{Rmk_nogencone_asspt}
The assumption that all metrics $g_s$ are asymptotic to cones (not only generalized cones!) is key here.
Otherwise, we cannot guarantee the preservation of the gradient condition.
In fact, Theorem~\ref{Thm_many_deformations} implies that in dimension 4, \emph{any} asymptotically conical gradient expanding soliton admits infinitely many deformations through expanding solitons that are asymptotic to non-trivial generalized cones.
Such deformations \emph{cannot} be gradient as discussed in Remark~\ref{Rmk_gradient_implies_cone}.
The assumption that all metrics $g_s$ are asymptotic to cones is expressed in Assumption~\ref{Prop_gradientness_ii} and also in Assumption~\ref{Prop_gradientness_iii} via the condition that $|\partial_s g_s(V_s)| \leq C$, which can be understood as the corresponding infinitesimal statement.
As it turns out, our proof depends on the latter assumption in a crucial way, while Assumption~\ref{Prop_gradientness_ii} will only be used to guarantee basic asymptotic bounds.
\end{Remark}

The following corollary, which is stated in a way that will be most convenient for us, is an immediate consequence of Proposition~\ref{Prop_gradientness}.

\begin{Corollary} \label{Cor_pres_gradientness}
Let $(M,N,\iota)$ be an ensemble, $11 \leq k \leq k^* - 2$ be integers, $\alpha \in (0,1)$ and fix some $C^{k^*-2}$-regular representative $(g,V = \nabla^g f,\gamma)$ of an element of $\MM_{\grad}^{k^*} (M,N,\iota)$.
We will use the gradient soliton $(M,g,f)$ to define the weighted H\"older spaces as in Subsection~\ref{subsec_list_Holder}; note that by Lemma~\ref{Lem_soliton_smooth} there is a smooth structure on $M$ with respect to which $g,f$ are smooth.

Let $U \subset X$ be an open neighborhood of the origin of some Banach space $X$ that is star-shaped with respect to $0$.
Suppose there are $C^1$-regular families
\begin{alignat*}{2}
 U &\longrightarrow \CONE^{k^*}(N), \qquad & x &\longmapsto \gamma_x, \\
 U &\longrightarrow C^{k,\alpha}_{-1}(M; S^2 T^*M), \qquad &x &\longmapsto  h_x 
\end{alignat*}
such $\gamma_0 = \gamma$ and $h_0 = 0$ and such that for every $x \in U$ the metric ($T$ denotes the map from Subsection~\ref{subsec_map_T})
\[ g_x := g + h_x + T(\gamma_x) - T(\gamma) \]
and the vector field
\[ V_x := \nabla^g f - \DIV_g (g_x) + \tfrac12 \nabla^{g} \tr_g (g_x) \]
form an expanding soliton $(M, g_x, V_x)$.
Then for every $x \in U$ the soliton $(M, g_x, V_x)$ is a gradient expanding soliton.
\end{Corollary}

The proof of Proposition~\ref{Prop_gradientness} is carried out in two steps.
First, we need to modify the family $(g_s,V_s)_{s \in (-\eps,\eps)}$ by a family of diffeomorphisms $(\chi_s : M \to M)_{s \in (-\eps,\eps)}$ such that the families of pullbacks $(\td g_s := \chi^*_s g_s, \td V_s := \chi^*_s V_s)$ satisfies the DeTurck gauge condition
\begin{equation} \label{eq_inf_dt}
 \partial_s \td V_s = - \DIV_{\td g_s}(\partial_s \td g_s) + \frac12 \nabla^{\td g_s} \tr_{\td g_s} (\partial_s \td g_s). 
\end{equation}
Note that this is the infinitesimal version of the standard DeTurck gauge condition in the sense that if it \emph{were} true that for some fixed $s$ the pair $(\td g_{s'}, \td V_{s'})$ is in the DeTurck gauge with respect to $(\td g_{s}, \td V_{s})$ for $s'$ near $s$, then \eqref{eq_inf_dt} would hold by differentiating this condition in $s'$ for $s' = s$.
The advantage of working in the infinitesimal DeTurck gauge \eqref{eq_inf_dt} is that the infinitesimal variations $\td h_s := \partial_s \td g_s$ satisfy an elliptic equation of the following form (all covariant derivatives and curvature terms are taken with respect to the metric $\td g_s$ some fixed parameter $s$, which we will frequently drop in the index):
\[ \triangle_{\td V} \td h  + 2 \Rm(\td h) + \nabla \td V * \td h = 0, \]
where the last term vanishes if $\td V$ is a gradient vector field.
We refer to Subsection~\ref{subsec_inf_DT} for more details.
We will moreover see that
\begin{equation} \label{eq_ddstdV}
   \frac{d}{ds} \td V_s^{\flat} = - \DIV_{\td V_s} \td h_s + \tfrac12 d \tr \td h_s. 
\end{equation}
So the infinitesimal deviation from being a gradient vector field is caused by the term $\DIV_{\td V} \td h$.
If $\td V$ happens to be gradient, then this term satisfies the following identity for $\DIV_{\td V} \td h = \DIV \td h - \td h (\td V)$
\begin{equation} \label{eq_trtdV}
 \triangle_{\td V} \DIV_{\td V} \td h - \tfrac12 \DIV_{\td V} \td h  = 0, 
\end{equation}
which implies via the maximum principle that $\DIV_{\td V} \td h = 0$.
So in this case the gradient condition of $\td V$ is preserved infinitesimally.
In the general case in which $\td V$ is \emph{not} gradient, we obtain an additional term of the form
\[  d\td V^\flat * \DIV_{\td V} \td h + \nabla d\td V^\flat * \td h + d\td V^\flat * \nabla \td h  \]
on the right-hand side of \eqref{eq_trtdV}; note that $d\td V^\flat = 0$ if $\td V$ is gradient.
Using the theory from Section~\ref{sec_properness}, we will then bound the $C^{2,\alpha}$-norm of $\DIV_{\td V} \td h$ in terms of the $C^{1,\alpha}$-norm of $d \td V^\flat$.
Taking the exterior derivative on both sides of \eqref{eq_ddstdV} implies that $\frac{d}{ds} d \td V^\flat = - d \DIV_{\td V} \td h$.
So we can bound the $C^{1,\alpha}$-norm of $\frac{d}{ds} d \td V^\flat$ in terms of the $C^{1,\alpha}$-norm of $d \td V^\flat$.
An application of Gronwall's Lemma then shows that $d \td V^\flat_s = 0$ for all $s$, so $\td V_s$ is a gradient vector field.
\bigskip

\subsection{The infinitesimal DeTurck gauge and other important equations} \label{subsec_inf_DT}
For the sake of the following discussion, consider a family of expanding solitons $(M,g_s,V_s)_{s \in (-\eps,\eps)}$ of high enough regularity and let $g = g_0$ and $V = V_0$.
All geometric quantities will be taken with respect to $g$.
We will write $h := \partial_s |_{s=0} g_s$ and $\dot V := \partial_s |_{s=0} V_s$.
Differentiating the soliton equation at $s = 0$ implies that
\[  \triangle h + 2 \Rm (h) - h (\cdot, \Ric(\cdot)) - h (\Ric(\cdot), \cdot) - \LL_{\DIV h - \frac12 \nabla \tr h} g - \LL_{\dot V} g - \LL_V h - h = 0 \]
Using $(\LL_{V} h)_{ij} = \nabla_V h_{ij} + h_{is} \nabla_j V^s + h_{js} \nabla_i V^s$ (here coefficients are taken with respect to a local orthonormal frame and we use the Einstein summation convention) and the soliton equation, this can be simplified to
\begin{equation} \label{eq_first_variation}
  \triangle_V h + 2 \Rm (h)   - \frac12 \big( (dV^\flat)_{is} \,h_{sj} + (dV^\flat)_{js} \, h_{si} \big) =  \LL_{\dot V + \DIV h - \frac12 \nabla \tr h} g  ,
\end{equation}
where $d V^\flat_{ij} = g_{js} \nabla_i V^s - g_{is} \nabla_j V^s$.
Note that if $V$ is gradient, then $d V^\flat = 0$.
We call an infinitesimal variation $(h, \dot V)$ {\bf gauged} if the term in the Lie-derivative vanishes:
\begin{equation} \label{eq_inf_gauge_condition}
 \dot V + \DIV h - \frac12 \nabla \tr h = 0. 
\end{equation}
Note that this is equivalent to
\[ \partial_s|_{s=0} V^\flat_s = - \DIV_V h + \frac12 \nabla \tr h. \]
For a gauged variation $(h,\dot V)$ the equation \eqref{eq_first_variation} simplifies to
\begin{equation} \label{eq_sol_gauged_variation}
     \triangle_V h + 2 \Rm (h)    - \frac12 \big( (dV^\flat)_{is} \,h_{sj} + (dV^\flat)_{js} \, h_{si} \big) = 0 . 
\end{equation}
The following lemma shows that any infinitesimal variation can be gauged by solving an elliptic equation.

\begin{Lemma} \label{Lem_inf_gauge}
Let $(M^n,g,V)$ be an expanding soliton, $h$ a symmetric $(0,2)$-tensor field and $\dot V, Y$ vector fields on $M$.
Then
\begin{equation} \label{eq_gauge_condition_Y}
 \triangle_V Y - \frac12 Y + \frac12 \big( d V^\flat  (Y, \cdot) \big)^\# = U := \dot V + \DIV h - \frac12 \nabla \tr h, 
\end{equation}
is equivalent to the condition that $(h - \LL_Y g, \dot V - \LL_Y V)$ is gauged.
\end{Lemma}

\begin{proof}
The gauge condition is equivalent to
\begin{multline*}
 U 
= \LL_Y V + \DIV \LL_Y g - \frac12 \nabla \tr \LL_Y g
= [Y,V] + \triangle Y + \nabla \DIV Y + \Ric(Y) - \nabla \DIV Y \\
= \nabla_Y V + \triangle_V Y + \Ric(Y) . 
\end{multline*}
This is equivalent to \eqref{eq_gauge_condition_Y} via the soliton equation.
\end{proof}

The next lemma provides an identity for the divergence $\DIV_V h$ of a gauged variation.

\begin{Lemma} \label{Lem_gauged_div_equation}
Let $(M,g,V)$ be an expanding soliton and $h$ a symmetric $(0,2)$-tensor that satisfies \eqref{eq_sol_gauged_variation}.
Then
\begin{equation*}
 \big( \triangle_V\DIV_V h-\tfrac12 \DIV_V h \big)_i %
=   \tfrac12(dV^\flat)_{is} (\DIV_V h)_s -   (\nabla_u dV^\flat)_{ti}  h_{tu} - (dV^\flat)_{js}  \nabla_j h_{si}. %
\end{equation*}
Note that the right-hand side only depends on $V$ via its derivatives or via $\DIV_V h$.
Moreover, if $V = \nabla f$ is a gradient vector field, then (recall that $\DIV_f h = \DIV_{\nabla f} h$)
\[ \triangle_f\DIV_f h-\tfrac12 \DIV_f h =0. \]
\end{Lemma}

The proof of Lemma~\ref{Lem_gauged_div_equation} uses the following commutator identity.

\begin{Lemma} \label{Lem_Bochner}
For any tensor field $T$ and vector field $V$ on a Riemannian orbifold $(M^n,g)$ we have
\[ (\triangle \nabla T)(V) - \nabla_V \triangle T
= \sum_{i=1}^n \big( (\nabla_{e_i} R)(e_i, V) T + 2 R(e_i, V) \nabla_{e_i} T \big) + \nabla_{\Ric(V)} T , \]
where $\{ e_i \}_{i=1}^n$ denotes a local orthonormal frame.
\end{Lemma}

\begin{proof}
We may assume that $\nabla e_i = \nabla V = 0$ at the point in question.
Then
\begin{multline*}
 \nabla^3_{e_i, e_i, V} T - \nabla^3_{e_i, V, e_i} T
= \nabla_{e_i} ( \nabla^2_{e_i, V} T - \nabla^2_{V,e_i} T ) \\
= \nabla_{e_i} ( R (e_i, V)T )
= (\nabla_{e_i} R) (e_i, V)T + R(e_i, V) \nabla_{e_i} T 
\end{multline*}
and
\begin{equation*}
 \nabla^3_{e_i, V, e_i} T - \nabla^3_{V, e_i, e_i} T
= (R(e_i, V) \nabla T)(e_i) 
= R(e_i, V) \nabla_{e_i} T - \nabla_{R(e_i, V)e_i} T. 
\end{equation*}
Combining both equations and summing over $i$ implies the lemma.
\end{proof}

\begin{proof}[Proof of Lemma~\ref{Lem_gauged_div_equation}.]
Using Lemma~\ref{Lem_Bochner} and viewing $h$ as a $(1,1)$-tensor, we compute
\begin{align*}
\DIV \triangle h - \triangle \DIV h
&=  \sum_{s,t=1}^n \left( - \big((\nabla_{e_s} R)(e_s, e_t) h\big)(e_t) - 2 \big( R(e_s, e_t) \nabla_{e_s} h \big)(e_t) \right) - \sum_{t=1}^n (\nabla_{\Ric(e_t)} h)(e_t)  \\
&=  \sum_{s,t=1}^n \left( - (\nabla_{e_s} R)(e_s, e_t) (h(e_t)) + h \big( (\nabla_{e_s} R)(e_s, e_t)e_t \big) \right. \\
&\qquad \left. - 2 R(e_s, e_t) \big( (\nabla_{e_s} h) (e_t) \big) 
+ 2 (\nabla_{e_s} h) ( R(e_s, e_t) e_t )\right) - \sum_{t=1}^n(\nabla_{\Ric(e_t)} h)(e_t) 
 \\
&=  - \sum_{s,t=1}^n  (\nabla_{e_s} R)(e_s, e_t) (h(e_t)) + \tfrac12 h ( \nabla R )  \\
&\qquad  - 2 \sum_{s,t=1}^n R(e_s, e_t) \big( (\nabla_{e_s} h) (e_t) \big) 
+ 2 \sum_{s=1}^n (\nabla_{e_s} h) ( \Ric(e_s) ) - \sum_{t=1}^n (\nabla_{\Ric(e_t)} h)(e_t) ,
\end{align*}
so, expressed in index notation and viewing $h$ as a $(0,2)$-tensor,
\begin{equation} \label{eq_commute_1}
 (\DIV \triangle h - \triangle \DIV h)_i =  - \nabla_s R_{stui} h_{tu} + \tfrac12 h_{iu} \nabla_u R - 2 R_{stui} \nabla_s h_{tu} + R_{st} \nabla_s h_{ti}.  
\end{equation}
Next, we have
\begin{multline*} -\DIV \nabla_V h + \nabla_V \DIV h
= - (\nabla^2_{e_t, V} h)(e_t) - (\nabla_{\nabla_{e_t} V} h)(e_t) + (\nabla^2_{V, e_t} h)(e_t) \\
= \big( R(V, e_t) h\big) (e_t)  - (\nabla_{\nabla_{e_t} V} h)(e_t)
= R(V,e_t)( h(e_t)) - h(\Ric(V))- (\nabla_{\nabla_{e_t} V} h)(e_t),  
\end{multline*}
which gives
\begin{equation} \label{eq_commute_2}
 (-\DIV \nabla_V h + \nabla_V \DIV h)_i 
= R_{stui} h_{tu} V^s - R_{su} h_{ui} V^s  - \nabla_s h_{ti} \nabla_t V^s. 
\end{equation}
The following identities can be checked easily:
\begin{align}
 \big( -(\triangle h)(V) + \triangle (h(V)) \big)_i
 &= 2 \nabla_s h_{ti} \nabla_s V^t + \big( h(\triangle V) \big)_i  \label{eq_commute_3} \\
(\nabla_V h)(V)-\nabla_V(h(V)) &= -h(\nabla_V V)  \label{eq_commute_4} \\
  \big(\DIV(2\Rm(h) \big)_i &= 2\nabla_t R_{tsui} h_{su} + 2 R_{tsui} \nabla_t h_{su}   \label{eq_commute_5} \\
  \big( - 2(\Rm(h))(V) \big)_i &= -2 R_{tsui} h_{su} V^t  \label{eq_commute_6}
\end{align}
 Combining equations (\ref{eq_commute_1})--(\ref{eq_commute_6}) and applying the soliton equation gives
\begin{multline} \label{eq_commute_intermediate}
 \big( \DIV_V ( \triangle_V h + 2 \Rm(h) ) - \triangle_V \DIV_V h \big)_i \\
  = \big( -\tfrac12 \DIV_V h + h \big( \tfrac12 \nabla R - \Ric(V) + \triangle_V V- \tfrac12 V \big) \big)_i \\
  + \tfrac32 \nabla_s h_{ti} (\nabla_s V^t - \nabla_t V^s) 
  + \nabla_s R_{stui} h_{tu} - R_{tsui} h_{su} V^t.
\end{multline}
Taking the divergence of the soliton equation and plugging in $V$ gives
\begin{align}
\tfrac12 \nabla_j R+\triangle V_j+\tfrac{1}{2}\nabla_s(\nabla_j V^s-\nabla_s V^j) &= 0, \label{eq_sol_id_1} \\
-\Ric(V)_j-(\nabla_V V)_j+\tfrac12 (\nabla_s V^j-\nabla_j V^s) V^s -\tfrac12 V_j &= 0.  \label{eq_sol_id_2}
\end{align}
In addition, the following identity follows from the second Bianchi identity and the soliton equation
\begin{multline} \label{eq_sol_id_3}
 \nabla_s R_{stui} h_{tu} - R_{tsui} h_{su} V^t = 
 (\nabla_i R_{tu} - \nabla_u R_{ti}) h_{tu} - R_{tsui} h_{su} V^t \\
 = \tfrac12 \big(-\nabla_i \nabla_t V^u - \nabla_i \nabla_u V^t + \nabla_u \nabla_t V^i + \nabla_u \nabla_i V^t \big) h_{tu}  - R_{tsui} h_{su} V^t 
 = \tfrac12 \nabla_u ( \nabla_t V^i - \nabla_i V^t ) h_{tu}. 
 \end{multline}
Combining (\ref{eq_commute_intermediate}) with (\ref{eq_sol_id_1})--(\ref{eq_sol_id_3}) implies
\begin{multline} \label{eq_final_commute}
 2\big( \DIV_V ( \triangle_V h + 2 \Rm(h) ) - \triangle_V \DIV_V h + \tfrac12 \DIV_V h \big)_i 
= -  (\nabla_s V^j - \nabla_j V^s) V^s h_{ij} \\
-  \nabla_s (\nabla_j V^s - \nabla_s V^j) h_{ij} 
+ 3(\nabla_s V^t - \nabla_t V^s) \nabla_s h_{ti}
+  \nabla_u ( \nabla_t V^i - \nabla_i V^t ) h_{tu}.
\end{multline}

Equation \eqref{eq_sol_gauged_variation} implies
\begin{align}
 2\big( \DIV_V ( \triangle_V h + 2 \Rm(h) ) \big)_i &= \nabla_j \big( (\nabla_i V^s - \nabla_s V^i) h_{sj} + (\nabla_j V^s - \nabla_s V^j ) h_{si} \big) \notag \\
  &\qquad -  (\nabla_i V^s - \nabla_s V^i) V^j h_{sj} - (\nabla_j V^s - \nabla_s V^j ) V^j h_{si} \notag\\
&= \nabla_j (\nabla_i V^s - \nabla_s V^i) h_{sj} + \nabla_j (\nabla_j V^s - \nabla_s V^j ) h_{si} \notag \\
&\qquad +(\nabla_i V^s - \nabla_s V^i) \nabla_j h_{sj}+ (\nabla_j V^s - \nabla_s V^j ) \nabla_j h_{si} \notag \\
&\qquad -  (\nabla_i V^s - \nabla_s V^i) V^j h_{sj}  - (\nabla_j V^s - \nabla_s V^j ) V^j h_{si} . \label{eq_div_dVflat_terms}
\end{align}
Combining \eqref{eq_final_commute} and \eqref{eq_div_dVflat_terms} therefore yields
\begin{align}
2\big(- \triangle_V \DIV_V h + \tfrac12 \DIV_V h\big)_i&=2\nabla_u ( \nabla_t V^i - \nabla_i V^t ) h_{tu}\notag\\
&\qquad -(\nabla_i V^s - \nabla_s V^i) \nabla_j h_{sj}+2 (\nabla_j V^s - \nabla_s V^j ) \nabla_j h_{si}\notag\\
&\qquad+(\nabla_i V^s - \nabla_s V^i) V^j h_{sj}.\notag
\end{align}
Note, in particular, that the terms of the form $(\nabla_j V^s-\nabla^s V^j) V^j h_{si}$ canceled out.
\end{proof}
\bigskip

\subsection{Gauging the family of solitons}
The following lemma shows that we can modify the family of solitons from Proposition~\ref{Prop_gradientness} by a family of diffeomorphisms such that the infinitesimal gauge condition \eqref{eq_inf_gauge_condition} holds.

\begin{Lemma} \label{Lem_gauged_variation}
Consider the setting of Proposition~\ref{Prop_gradientness}.
After possibly shrinking $\eps$, there is a continuous family $(\chi_s: M \to M)_{s \in (-\eps,\eps)}$ of $C^6$-diffeomorphisms of uniform finite displacement such that $(\td g_s := \chi_s^* g_s)_{s \in (-\eps,\eps)}$ and $(\td V_s := \chi_s^* V_s)_{s \in (-\eps,\eps)}$ are of regularity $C^1$ in the $C^{5}_{\loc}$-sense and such that $\td h_s := \partial_s \td g_s$ satisfies the gauge condition:
\begin{equation} \label{eq_gauge_condition_V}
 \partial_s \td V_s + \DIV_{\td g_s} \td h_s - \frac12 \nabla^{\td g_s} \tr_{\td g_s}  \td h_s = 0.
\end{equation}
Moreover we have for some uniform constant $C' < \infty$ \begin{equation} \label{eq_tdhtdV_bounds}
 |\td h_s|, |\nabla^{2,\td g_s} \td h_s|, |\nabla^{3,\td g_s} \td h_s| \leq C', \qquad |\nabla^{\td g_s} \td h_s | \leq C' r^{-1}, \qquad |\td h_s (\td V_s)| \leq C' 
\end{equation}
and
\begin{equation} \label{eq_tdg_bounds}
|\nabla^{m,g_0} (\td g_s - g_0) |_{g_0} \leq C'|s| \qquad \text{for all} \quad m = 0,\ldots, 3.
\end{equation}

\end{Lemma}

\begin{proof}
For any $s \in (-\eps,\eps)$ set $h_s := \partial_s g_s$ and define, as in \eqref{eq_gauge_condition_Y},
\[ U_s := \partial_s V_s + \DIV_{g_s} h_s - \frac12 \nabla^{g_s} \tr_{g_s} h_s. \]
Recall that, as mentioned in the statement of Proposition~\ref{Prop_gradientness}, we view $r \in C^\infty(M)$ as a positive function, which agrees with the radial coordinate on $\iota((2,\infty) \times N)$.
The next claim shows that we can solve the equation \eqref{eq_gauge_condition_Y} with the necessary bounds on the solution.

\begin{Claim} \label{Cl_solve_Y_eq}
For any $s \in (-\eps,\eps)$ the equation \eqref{eq_gauge_condition_Y} has a unique bounded and $C^2$-regular solution $Y_s$.
Moreover, we have $\Vert r Y_s \Vert_{C^{6}_{g_s}} \leq C$ (see Subsection~\ref{subsec_list_Holder}) for some uniform constant $C <\infty$ and the dependence $s \mapsto  Y_s$ is continuous in the $C^{6}_{\loc}$-sense.
\end{Claim}

\begin{proof}
Adapting the computation of Lemma~\ref{Lem_delta_f_ell_bounds} to our setting, we rewrite \eqref{eq_gauge_condition_Y} using the asymptotically linear weight of the form $w := r + A$, where $A > 0$ is a constant whose value we will determine later.
Fix some $s \in (-\eps,\eps)$ for a moment and take all covariant derivatives with respect to $g_s$.
Since
\[ \triangle_{V_s} ( w Y ) -  \tfrac12 (w Y)
= w  \triangle_{V_s} Y 
+ \left( \frac{\triangle_{V_s} w}{w}  - \frac12 \right)  (w Y) + \nabla_{2\frac{\nabla w}w} (w Y) - 2 \frac{|\nabla w|^2}{w^2} (wY), \]
the equation \eqref{eq_gauge_condition_Y} is equivalent to
\begin{equation} \label{eq_gauge_condition_Y_times_w}
 \triangle_{V_s+2\frac{\nabla w}w} (wY) 
 - \left(1+   \left( \frac{\triangle_{V_s} w}{w}  - \frac12 \right)  - 2 \frac{|\nabla w|^2}{w^2} \right) (wY) + \frac12 \big( dV^\flat (wY, \cdot) \big)^{\sharp} = w U 
\end{equation}
Note that if we consider the push-forward of the field $\partial_r$ via $\iota$, then using Assumption~\ref{Prop_gradientness_ii} of Proposition~\ref{Prop_gradientness},
\[
 \frac{\triangle_{V_s} w}{w}  - \frac12 
= \frac{\triangle r}{r + A} - \frac{dr (V_s)}{r + A} - \frac12 
= \frac{\triangle r}{r + A} - \frac{dr(V_s + \frac12 r \partial_r)}{r + A} 
+ \frac{r }{2(r + A)} - \frac12 .
\]
So we can fix a sufficiently large $A$, independently of $s$, such that
\begin{equation} \label{eq_zeroth_order_term}
 1 + \frac{\triangle_{V_s} w}{w}  - \frac12 - 2 \frac{|\nabla w|^2}{w^2} > \frac1{10}. 
\end{equation}
Moreover, we have uniform bounds on the left-hand side of \eqref{eq_zeroth_order_term}, as well as on its derivatives up to order $7$.
We seek to apply Proposition~\ref{Prop_Lunardi} with $k=4$, so we also need the bounds
\[ \bigg\Vert \nabla \bigg( V + 2 \frac{\nabla w}{w} \bigg)\bigg\Vert_{C^{7}}, \;\; \Vert {\Rm} \Vert_{C^{9}}, \;\; \Vert d V^\flat \Vert_{C^{7}} \leq C.\]
These follow from the assumptions of Proposition~\ref{Prop_gradientness}. Note also that by these assumptions we have for some generic constant $C$, which is independent of $s$, that
\[ |\nabla^m U_s| \leq Cr^{-1} \quad \text{and} \quad |\nabla^m w | \leq C r \qquad \text{for all} \quad m = 0,\ldots, 5, \]
so by the product rule $\Vert w U_s \Vert_{C^5} \leq C$.
We can now apply Proposition~\ref{Prop_Lunardi} with $k=4$,
and obtain solutions $w Y_s$ to \eqref{eq_gauge_condition_Y_times_w} such that $\Vert w Y_s \Vert_{C^{6,\alpha}_{g_s}} \leq C$ for some constants $C$ and $\alpha \in (0,1)$ that are independent of $s$.

To see that $Y_s$ depends continuously on $s$ in the $C^{6}_{\loc}$-sense, consider a sequence $s_i \to s_\infty \in (-\eps,\eps)$.
By a similar argument as in the proof of Lemma~\ref{Lem_Holder_norms_properties}, we have local uniform $C^{6,\alpha}$-bounds on $Y_{s_i}$.
So using Arzel\'a-Ascoli, we can pass to a subsequence such that $w Y_{s_i}$ converges to some $w Y'$ in $C^{6}_{\loc}$, which is bounded and solves \eqref{eq_gauge_condition_Y_times_w} for $s = s_\infty$.
Thus $Y' = Y_{s_\infty}$ by the uniqueness statement of Proposition~\ref{Prop_Lunardi}.
\end{proof}

Let $(\chi_{s,s_0})_{s,s_0 \in (-\eps,\eps)}$ be the flow of the time-dependent vector field $-Y_s$, meaning that
\[ \partial_s \chi_{s, s_0} = -Y_s \circ \chi_{s,s_0}, \qquad \chi_{s_0,s_0} = \id_M. \]
After possibly shrinking $\eps$, we may assume that all maps $\chi_{s,s_0}$ are $C^6$-diffeomorphisms.
Set $\chi_s := \chi_{s,0}$ and note that since $\partial_s \chi_s = - Y_s \circ \chi_s$, Claim~\ref{Cl_solve_Y_eq} implies that the map $s \mapsto \chi_s$ is of regularity $C^1$ in the $C^6_{\loc}$-sense.
Now for any $s_0 \in (-\eps,\eps)$
\[  \td h_{s_0} = \frac{d}{ds} \bigg|_{s=s_0} (\chi_s^* g_s) 
= \frac{d}{ds} \bigg|_{s=s_0} (\chi_{s_0,0}^*\chi_{s,s_0}^* g_s)
= \chi_{s_0}^* \left(h_{s_0} - \LL_{Y_{s_0}} g_{s_0}  \right). \]
Similarly,
\[ \partial_s \td V_{s_0} =  \frac{d}{ds} \bigg|_{s=s_0} (\chi_s^* V_s) 
= \chi_{s_0}^* \left(\partial_s V_{s_0} - \LL_{Y_{s_0}} V_s   \right). \]
It follows from Lemma~\ref{Lem_inf_gauge} that $( h_{s_0} - \LL_{Y_{s_0}} g_{s_0},   \partial_s V_{s_0} - \LL_{Y_{s_0}} V_{s_0})$ is gauged with respect to $(M, \lb  g_{s_0}, \lb V_{s_0})$, and thus its pullback via $\chi_{s_0}$, satisfies the gauging condition \eqref{eq_gauge_condition_V} with respect to $(M,  \td g_{s_0}, \td V_{s_0})$.

It remains to verify the bounds on $\td h_s$, $\td V_s$ and $\td g_s$ from \eqref{eq_tdhtdV_bounds} and \eqref{eq_tdg_bounds}.
Since $(\chi_{s})_* \td h_s = h_s - \LL_{Y_s} g_s$ and $(\chi_{s})_* \partial_s \td V_s = V_s - \LL_{Y_s} V_s$, the bounds \eqref{eq_tdhtdV_bounds} are implied by uniform bounds of the form
\begin{multline*}
|\LL_{Y_s} g_s|_{g_s}, |\nabla^{2,g_s} (\LL_{Y_s} g_s)|_{g_s}, |\nabla^{3,g_s} (\LL_{Y_s} g_s)|_{g_s} \leq C, \qquad 
|\nabla^{g_s} (\LL_{Y_s} g_s)|_{g_s} \leq Cr^{-1}, \\ \qquad  |(\LL_{Y_s} g_s)(V_s)|_{g_s}, |h_s (\LL_{Y_s} V_s)|_{g_s} \leq C.
\end{multline*}
These bounds follow immediately from the assumptions of the lemma together with the result of Claim~\ref{Cl_solve_Y_eq}.
To see the bounds \eqref{eq_tdg_bounds}, we integrate the bounds on $\partial_s \td g_s = \td h_s$ from \eqref{eq_tdhtdV_bounds}.
We obtain first that the metrics $\td g_s$ are uniformly $(1+C'|s|)$-bilipschitz and then by induction that for $m = 1, \ldots, 3$ and for some generic constant $C < \infty$, which does not depend on $s$,
\[ |\partial_s \nabla^{m,g_0} \td g_s|_{g_0} \leq |\nabla^{m,g_0} \td h_s |_{g_0} \leq   C \bigg( \sum_{i=0}^{m} |\nabla^{i,g_0} \td g_s|_{g_0} \bigg) \bigg( \sum_{i=0}^{m}|\nabla^{i,\td g_s}  \td h_s|_{g_0} \bigg)   \leq C + C | \nabla^{m,g_0}\td g_s|_{g_0}. \]
This shows that $|\partial_s \nabla^{m,g_0} \td g_s|_{g_0} \leq C$, which implies the desired bound.
\end{proof}
\medskip

\subsection{Proof of Proposition~\ref{Prop_gradientness}}
\begin{proof}[Proof  of Proposition~\ref{Prop_gradientness}.]
Note first that it suffices to show that for all $s \in (-\eps, \eps)$ the vector field $V_s$ is a gradient vector field with respect to $g_s$ for some potential $f_s \in C^1(M)$.
The higher regularity and $C^1$-dependence follows by the corresponding assumptions on $g_s$ and $V_s = \nabla^{g_s} f_s$. 
In addition, since the condition of being a gradient vector field is closed under local $C^1$-convergence, it remains to show the openness of this condition.
So without loss of generality, we are free to shrink $\eps$ in the following discussion.

Consider the gauged family $(M, \td g_s, \td V_s)$ from Lemma~\ref{Lem_gauged_variation} and recall that $\td h_s := \partial_s \td g_s$.
By diffeomorphism invariance, it suffices to show that for $s \in (-\eps,\eps)$ the vector field $\td V_s$ is a gradient vector field with respect to $\td g_s$.
Define the $1$-forms
\[  \xi_s := \td V_s^{\flat, \td g_s}, \qquad \beta_s := \big({\DIV_{\td g_s,\td V_s} \td h_s }\big)^{\flat, \td g_s} . \]
We need to show that $\xi_s$ is exact.
Expressing \eqref{eq_gauge_condition_V} in these 1-forms yields
\begin{equation} \label{eq_dsxi_beta_dtr}
 \partial_s \xi_s = - \beta_s + \tfrac12 d \tr_{\td g_s} \td h_s . 
\end{equation}
So
\begin{equation} \label{eq_nongradient_evolution}
\partial_s d\xi_s = - d \beta_s.
\end{equation}
Due to Lemmas \ref{Lem_gauged_div_equation} and \ref{Lem_gauged_variation} we have for any $s \in (-\eps, \eps)$
\begin{equation} \label{eq_beta_equation}
   \triangle_{\td V_s} \beta_s  -\tfrac12 \beta_s - \tfrac12 d\xi_s (\beta_s^{\sharp}, \cdot) = \nabla d\xi_s\ast  \td h_s + d\xi_s \ast \nabla \td h_s, 
\end{equation}
where the Laplacian, gradient, and the musical operator are taken with respect to $\td g_s$.

Our goal will be to apply Proposition~\ref{Prop_Lunardi} with $k=0$ to \eqref{eq_beta_equation}.
To do this, note first that by \eqref{eq_tdhtdV_bounds} we have
\[ |\beta_s| \leq C |\nabla \td h_s| + |\td h_s(\td V_s)| \leq C, \]
where $C$ is a generic constant that does not depend on $s$ and all covariant derivatives and norms are taken with respect to $\td g_s$.
Note that here we have used the condition \eqref{eq_psgV_bounded}; in fact, it is the only place where we have used the critical assumption that $\td g_s$ is asymptotic to a cone (see also Remark~\ref{Rmk_nogencone_asspt}).
Next, we need to ensure that the zeroth order term in \eqref{eq_beta_equation} is bounded by a uniform negative constant.
To do this, note that by \eqref{eq_tdhtdV_bounds} and Assumptions~\ref{Prop_gradientness_ii}, \ref{Prop_gradientness_iii} the right-hand side of \eqref{eq_nongradient_evolution} is bounded by 
\[ |\nabla \beta_s| \leq C|\nabla^2 \td h_s | + C|\nabla \td h_s| \, |\td V_s| + C |\td h_s | \, |\nabla \td V_s | \leq C, \]
where we used that $|\td V_s|_{\td g_s} = |V_s|_{g_s} \circ \chi_s \leq Cr$ and $|\nabla^{\td g_s} \td V_s|_{\td g_s} = |\nabla^{g_s} V_s|_{g_s} \circ \chi_s \leq C$ due to diffeomorphism invariance.
So since by Assumption~\ref{Prop_gradientness_iv} we have $d\xi_0 = 0$, we obtain that for some uniform $C < \infty$
\[ |d\xi_s| \leq C |s|. \]
Therefore after shrinking $\eps$, we may assume that the zeroth order term in \eqref{eq_beta_equation} satisfies
\[ \Big( - \frac12 \beta_s - d\xi_s^\flat (\beta_s^{\sharp}, \cdot) \Big) \cdot \beta_s \leq - \frac14 |\beta_s|^2. \]
In addition, we have the bounds
\[ \Vert {\Rm_{\td g_s}} \Vert_{C^5_{\td g_s}}, \;\; 
\Vert \nabla^{\td g_s} \td V_s \Vert_{C^3_{\td g_s}}, \;\; \Vert d\xi_s\Vert_{C^3_{\td g_s}} \leq C, \]
which follow again by the diffeomorphism invariance from Assumptions~\ref{Prop_gradientness_ii}, \ref{Prop_gradientness_iii} for the original metric $g_s$ and vector field $V_s$.
Thus all conditions of Proposition~\ref{Prop_Lunardi} are satisfied, and we obtain that for some generic constant $C < \infty$, which does not depend on $s$
\[ 
\big\Vert \beta_s \big\Vert_{C^{2,\alpha}_{\td g_s}}
\leq C \big\Vert \nabla d\xi_s\ast  \td h_s + d\xi_s \ast \nabla \td h_s \big\Vert_{C^{0,\alpha}_{\td g_s}} \leq C \big\Vert d\xi_s \big\Vert_{C^{1,\alpha}_{\td g_s}} \big\Vert \td h_s \big\Vert_{C^{1,\alpha}_{\td g_s}} 
\leq C \big\Vert d\xi_s \big\Vert_{C^{1,\alpha}_{\td g_s}},\]
where the last bound follows from \eqref{eq_tdhtdV_bounds}.
Using Lemma~\ref{Lem_Holder_norms_properties} and \eqref{eq_tdg_bounds}, this bound can be expressed in terms of H\"older norms with respect to a \emph{fixed} background metric $g_0$:
\begin{equation} \label{eq_beta_bound_dxi}
 \big\Vert \beta_s \big\Vert_{C^{2,\alpha}_{g_0}}
\leq  C \big\Vert d\xi_s \big\Vert_{C^{1,\alpha}_{ g_0}}. 
\end{equation}

Due to \eqref{eq_nongradient_evolution} we have for any $s_1, s_2 \in (-\eps, \eps)$
\[  \Big| \big\Vert d\xi_{s_2} \big\Vert_{C^{1,\alpha}_{g_0}} -  \big\Vert d\xi_{s_1} \big\Vert_{C^{1,\alpha}_{g_0}} \Big|
\leq \big\Vert d\xi_{s_2} - d\xi_{s_1} \big\Vert_{C^{1,\alpha}_{g_0}}
= \bigg\Vert \int_{s_1}^{s_2} d\beta_s \, ds \bigg\Vert_{C^{1,\alpha}_{g_0}}
\leq \int_{s_1}^{s_2} \big\Vert d\beta_s \big\Vert_{C^{1,\alpha}_{g_0}} ds. \]
So $s \mapsto \big\Vert d\xi_{s} \big\Vert_{C^{1,\alpha}_{g_0}}$ is Lipschitz and for almost all $s$ we have by \eqref{eq_beta_bound_dxi}
\[ \frac{d}{ds} \big\Vert d\xi_s \big\Vert_{C^{1,\alpha}_{g_0}} \leq  \big\Vert d \beta_s \big\Vert_{C^{1,\alpha}_{g_0}} 
\leq C \big\Vert \beta_s \big\Vert_{C^{2,\alpha}_{g_0}} 
\leq C \big\Vert d\xi_s \big\Vert_{C^{1,\alpha}_{g_0}}. \]
%The bound on the derivative in the first inequality above can be justified by estimating with the appropriate difference quotients.
Combined with $d\xi_0 = 0$, this implies $d\xi_s = 0$ and $\beta_s = 0$ for all $s$.
If $H^1(M;\IR) = 0$, then this finishes the proof.
Otherwise, recall that by \eqref{eq_dsxi_beta_dtr} we have $\partial_s \xi_s = \tfrac12 d \tr_{\td g_s} \td h_s$. 
So if we integrate $\partial_s f'_s = \frac12 \tr_{\td g_s} \td h_s$ with initial condition $f'_0 = f_0$, then $\xi_s = df'_s$.
\end{proof}
\bigskip

\section{Finding the DeTurck gauge} \label{sec_DT_gauge}

\subsection{Main result}
Consider an expanding gradient soliton $(M,g,f)$ and another expanding soliton metric $g'$ with soliton vector field $V'$.
As in Subsection~\ref{subsec_the_elliptic_problem}, we say that $(g',V')$ is in the {\bf DeTurck gauge} if
\[ P_g(g',V'):= V' - \nabla^g f + \DIV_g (g') - \tfrac12 \nabla^g \tr_g (g') = 0. \]
Recall from this subsection that if $(g',V')$ is in the DeTurck gauge, then the expanding soliton equation for $(g',V')$ is equivalent to the equation $Q_g(g') = 0$.

Suppose now that $(g',V')$, close to $(g,V := \nabla^g f)$, is given such that $(M,g',V')$ is an expanding soliton.
Our goal will be to find a diffeomorphism $\phi : M \to M$ such that $(g'',V'') = (\phi^* g', \phi^* V')$ is in the DeTurck gauge, i.e., $P_g(g'',V'') =0$, and therefore $Q_g(g'') = 0$.
Our main result of this section is the following  proposition.

\begin{Proposition} \label{Prop_gauge}
Let $(M,N,\iota)$ be an ensemble, $0 \leq k \leq k^* -10$ integers and $\alpha \in (0,1)$.
Fix a $C^{k^*-2}$-regular representative $(g,V=\nabla^g f, \gamma)$ of an element of $\MM^{k^*}_{\grad}(M,N,\iota)$.
All H\"older norms will be taken with respect to $(g,f)$ (see Subsection~\ref{subsec_list_Holder}).

Then there is a constant $\eps > 0$ with the following properties:

\begin{enumerate}[label=(\alph*)]
\item  \label{Prop_gauge_a}
For any $C^{k+2,\alpha}$-metric $g' = g+h' + T(\gamma'-\gamma)$, $\gamma' \in \GenCONE^{k^*}(N)$, and any $C^{k+1,\alpha}$-regular vector field $V'$ on $M$ with 
\begin{equation} \label{eq_hpgpVp_small}
 \eps' := \Vert h' \Vert_{C^{k+3,\alpha}_{-1}} + \Vert \gamma'-\gamma  \Vert + \Vert V'-\nabla^g f \Vert_{C^{k+2,\alpha}_{-1}} \leq \eps 
\end{equation}
there is an orientation preserving $C^{k+2,\alpha}$-diffeomorphism $\phi : M \to M$ such that 
\begin{equation} \label{eq_P_is_0}
  P_g (\phi^* g', \phi^* V') = \phi^* V' - \nabla^g f + \DIV_g (\phi^* g') - \tfrac12 \nabla^g \tr_g (\phi^* g') = 0 .
\end{equation}
Moreover, we have
\begin{equation*} 
 \sup_{x \in M} d_g (\phi(x),x), \, \Vert \phi^*g - g \Vert_{C^{k+1,\alpha}}, \, \Vert \phi^* g' - g - T(\gamma'-\gamma ) \Vert_{C^{k+1,\alpha}_{-1}} \leq C\eps'. 
\end{equation*}
\item  \label{Prop_gauge_b}
If $k \geq 2$, then $\phi$ is uniquely determined by \eqref{eq_P_is_0}, among all $C^3$-diffeomorphisms satisfying 
\begin{equation} \label{eq_phi_close_id}
 \sup_{x \in M} d_g (\phi(x),x) \leq \eps, \qquad \Vert \phi^*g - g \Vert_{C^{2}} \leq \eps. 
\end{equation}
\item  \label{Prop_gauge_c}
The assignment $(h',\gamma', V'') \mapsto h''$,  where $V''=V'-\nabla^g f$ and $\phi^* g' = g + h'' + T(\gamma' -\gamma)$, induces a $C^{1,\alpha}$-map of the form
\[ \qquad\qquad C^{k+3,\alpha}_{-1}(M; S^2 T^* M) \times \GenCONE^{k^*} (N) \times C^{k+2,\alpha}_{-1}(M;TM) \supset U \longrightarrow C^{k+1,\alpha}_{-1}(M; S^2 T^* M), \]
where $U$ is a neighborhood of $(0,\gamma,0)$.
Its differential at the origin equals
\[ (h',\gamma',V'') \longmapsto h' + \LL_Y g, \]
where $Y \in C^2_{-1,\nabla f} (M ; TM)$ is the unique vector field satisfying
\begin{equation} \label{eq_gauging_differential}
 \triangle_f Y - \tfrac12 Y = - V' - \DIV_g (h'+T(\gamma')) + \tfrac12 \nabla \tr_g (h' + T(\gamma')). 
\end{equation}
\end{enumerate}
\end{Proposition}

\subsection{Proof}
Let $(M,g)$ be a complete Riemannian manifold with uniformly bounded curvature and curvature derivatives.
For any continuous vector field $Z \in C^0(M; TM)$ define the continuous map $\phi_Z : M \to M$ by 
\[ \phi_Z (p) := \exp_p (Z_p). \]
Note that if $\Vert Z \Vert_{C^1}$ is sufficiently small, depending on curvature and injectivity radius bounds, then $\phi_Z$ is an orientation-preserving diffeomorphism.
Suppose now that $(M,g)$ is a Riemannian orbifold with isolated singularities.
Then any continuous vector field $Z \in C^0(M; TM)$ must vanish at each singular point $p \in M$, because its neighborhood can be modeled on a quotient of the form $\IR^n/\Gamma$, where the origin is the only fixed point of $\Gamma$.
Consequently, we can set $\phi_Z(p) = p$ for all singular points $p \in M$.
Equivalently, we can define $\phi_Z$ near any singular point by passing to a smooth  local orbifold cover.

The following lemma characterizes the analytic properties of the map $Z \mapsto \phi_Z$.

\begin{Lemma} \label{Lem_composition}
For any $A < \infty$ there is an $\eps > 0$ such that the following holds.

Suppose that $k,a,b \geq 0$ are integers, $d_1, d_2 \geq 0$ are real numbers, $\alpha \in (0,1)$.
Let $(M,g,f)$ be a complete, $n$-dimensional gradient expanding soliton on an orbifold with isolated singularities, where we suppose that $g$ and $f$ are smooth, and let $V \in C^{k+10}_{\loc}(M;TM)$ be a vector field such that:
\begin{enumerate}[label=(\roman*)]
\item $k,a,b, n \leq A$ and $d_1, d_2 \leq A$ and $A^{-1} \leq \alpha \leq 1 - A^{-1}$.
\item $f$ attains a maximum and $\max_M f = -1$.
\item \label{Lem_composition_iii} We have for all $m = 0, \ldots, k+10$ and $p \in M$
\[ 
|\nabla^m {\Rm}| \leq A (-f)^{-(2+m)/2}, \qquad \inj(p) \geq A^{-1} (-f)^{1/2}. \]
\item \label{Lem_composition_iv} We have $|V| \leq A (-f)^{1/2}$ and 
\[  |\nabla^m V| \leq A \qquad \text{for} \quad m = 1, \ldots, k+10 \]
\end{enumerate}
Denote by $T^a_b M$ the bundle of $(a,b)$-tensors
and consider the weighted H\"older norms with respect to $(M,g,f)$ (see Subsection~\ref{subsec_list_Holder}).
For any Banach space $X$ we denote by $[X]_\eps := B(0,\eps) \subset X$ the $\eps$-ball around the origin.
Then the maps
\begin{align*}
 H_1 : \big[ C_{-d_1}^{k+1,\alpha}(M;TM) \times C^{k+2,\alpha}_{-d_2} (M; T^a_b M) \big]_\eps &\longrightarrow C^{k,\alpha}_{-d_1-d_2}(M; T^a_b M), \\ (Z,u) &\longmapsto \phi_Z^* u - u, \\
 H_2 : \big[ C_{-d_1}^{k+1,\alpha}(M;TM) \big]_\eps &\longrightarrow C^{k,\alpha}_{-d_1}(M; T M), \\ Z &\longmapsto \phi_Z^* V - V - (\id + \nabla Z)^{-1} (\nabla_{V} Z ),
\end{align*}
where $(\id + \nabla Z)^{-1}$ denotes the inverse of the endomorphism $v \mapsto v + \nabla_v Z$,
are defined and of regularity $C^{1,\alpha}$.
Moreover their $C^{1,\alpha}$-norms are bounded by a constant of the form $C(A)$.
\end{Lemma}

\begin{proof}
For any $r > 0$ we set $B_r := B(\vec 0, r) \subset \IR^n$.
We first establish a basic local and scalar case of the lemma.

\begin{Claim} \label{Cl_scalar_pullback}
Denote by $C^{k}(B_1; B_2) \subset C^{k}(B_1; \IR^n)$ the subset of vector-valued functions that take values in $B_2$.
Then the map
\begin{equation*} 
H : C^{k}(B_1; B_2) \times C^{k+1,\alpha}(B_2) \longrightarrow C^{k}(B_1), \qquad (\phi, u) \longmapsto u \circ \phi 
\end{equation*}
is of regularity $C^{1,\alpha}$.
\end{Claim}

\begin{proof}
We first use induction to reduce the claim to the case $k=0$.
So suppose that $k \geq 1$ and that the claim is true for $k$ replaced with $k-1$.
Then all maps
\begin{alignat*}{2}
 C^{k} (B_1; B_2) \times C^{k+1,\alpha}(B_2) &\longrightarrow C^{k-1}(B_1), &\qquad 
(\phi,u) &\longmapsto u \circ \phi  \\
 C^{k} (B_1; B_2) \times C^{k+1,\alpha}(B_2) &\longrightarrow C^{k-1}(B_1), &\qquad 
(\phi,u) &\longmapsto \frac{\partial}{\partial x^i} (u \circ \phi)  = \Big( \frac{\partial u}{\partial x^j} \circ \phi \Big) \, \frac{\partial \phi^j}{\partial x^i} \bigg|_{B_1} 
\end{alignat*}
have the desired regularity property, hence $H$ does as well.

So assume now that $k = 0$.
It suffices to show that the map
\[ C^{0} (B_1; B_2)  \times C^{1,\alpha}(B_2) \times C^{0}(B_1; \IR^n) \times C^{1,\alpha}(B_2) \rightarrow C^{0}(B_1), \qquad (\phi, u, \dot\phi, \dot u) \mapsto DH_{(\phi, u)}(\dot\phi, \dot u) \]
is $C^{0,\alpha}$.
Note that
\[ (DH)_{(\phi, u)} (\dot\phi, \dot u) =  \dot u \circ \phi + du_{\phi} (\dot \phi) . \]
It suffices to establish the desired continuity property on each summand.
This reduces the claim to showing that the map
\[ H' : C^{0} (B_1; B_2) \times C^{0,\alpha}(B_2) \longrightarrow C^{0}(B_1), \qquad (\phi, u) \longmapsto u \circ \phi \]
is of regularity $C^{0,\alpha}$.
To see this note that for any two $(\phi_i, u_i) \in C^{0} (B_1; B_2) \times C^{0,\alpha}(B_2)$, $i = 1,2$, we have
\begin{multline*}
 \Vert u_1 \circ \phi_1 - u_2 \circ \phi_2 \Vert_{C^{0}} 
\leq \Vert (u_1 - u_2) \circ \phi_1 \Vert_{C^{0}}  + \Vert u_2 \circ \phi_1 - u_2 \circ \phi_2 \Vert_{C^{0}} \\
\leq \Vert u_1 - u_2 \Vert_{C^0} + \Vert u_2 \Vert_{C^{0,\alpha}} \Vert \phi_1 - \phi_2 \Vert_{C^0}^\alpha,
\end{multline*}
which finishes the proof of the claim.
\end{proof}
\medskip

Consider now the maps $H_1$ and $H_2$ from the statement of the lemma.
Let $\eps > 0$ be a generic constant, which we assume to be small, depending on $A$, in the course of the proof.

We first argue that it suffices to consider a localized problem.
Set for simplicity $\rho := (-f)^{1/2}$ and $\rho_p := \rho(p)$ for $p \in M$.
By the decay assumptions on the curvature derivatives there is a constant $c(A) > 0$ such that 
around every point $p \in M$ we can find a smooth coordinate chart $\vec x_p : U_{p,2 c \rho_p} \to \vec x_p (U_{p,2c \rho_p}) = B_{2c \rho_p} \subset \IR^n$ (where $U_{p,2c\rho_p}$ may denote a local orbifold cover if $p$ is near a singular point) within which we have
\begin{equation} \label{eq_gij_bounds}
 |\partial^m (g_{ij} - \delta_{ij})| \leq c^{-1} \rho_p^{-m} \leq c^{-1} \qquad \text{for} \qquad m = 0, \ldots, k+10. 
\end{equation}
For any $r \in (0,2 c\rho_p ]$ set $U_{p,r} := \vec x_p^{-1} (B_r)$.
Fix $p \in M$ and identify $B_{2c \rho_p}$ with $U_{2c \rho_p}$.
Assuming $\eps \leq \ov\eps(A)$, we may consider the maps (note that here the norms are only taken over $B_c, B_{2c} \subset B_{2c\rho_p}$)
\begin{align*}
 H'_1 : \big[ C^{k+1,\alpha} (B_c; \IR^n) \times C^{k+2,\alpha} (B_{2c} ; T^a_b\IR^n) \big]_\eps &\longrightarrow C^{k,\alpha} (B_{c}; T^a_b \IR^n), \\ 
 (Z,u) &\longmapsto \rho_p^{d_1+d_2} \big( \phi^*_{\rho_p^{-d_1} Z} (\rho_p^{-d_2} u) - (\rho_p^{-d_2} u) \big), \\
 H'_2 : \big[ C^{k+1,\alpha}(B_c;\IR^n) \big]_\eps &\longrightarrow C^{k,\alpha}(B_c; \IR^n), \\
  Z &\longmapsto \rho_p^{d_1} \big( \phi_{\rho_p^{-d_1} Z}^* V - V \\ &\qquad\qquad\qquad - (\id + \nabla (\rho_p^{-d_1} Z))^{-1}( \nabla_V (\rho_p^{-d_1} Z)) \big), 
\end{align*}
where maps $\phi_Z : B_c \to B_{2c}$ are defined via the exponential map with respect to the metric $g_{ij}$ on $B_{3c}$, the covariant derivative is taken with respect to the metric $g_{ij}$ and the $C^{k,\alpha}$, $C^{k+1,\alpha}$-norms are defined with respect to the Euclidean background structure.
Due to \eqref{eq_gij_bounds} and \eqref{eq_sup_equivalent} in Lemma~\ref{Lem_Holder_norms_properties}, the lemma is implied by the following claim.

\begin{Claim}
If $\eps \leq \ov\eps(A)$, then the maps $H'_1, H'_2$ are of regularity $C^{1,\alpha}$ and their $C^{1,\alpha}$-norms are bounded by a constant of the form $C(A)$, which is independent of $p$.
\end{Claim}

\begin{proof}
We may assume that $d_2 = 0$ by linearity and $d_1 = 0$ since $H'_1(0,0) = 0$ and $\rho_p \geq 1$ and since the map $(Z,u) \mapsto \rho_p^{d_1} H'_1(\rho_p^{-d_1} Z,u)$ has smaller derivative bounds than $H'_1$ (the analogous statement holds for $H'_2$).
So after adding the identity map to $H'_1$, it suffices to consider the following maps instead of $H'_1, H'_2$:
\begin{align*}
 H''_1 : \big[ C^{k+1,\alpha} (B_c; \IR^n) \times C^{k+2,\alpha} (B_{2c} ; T^a_b \IR^n) \big]_\eps &\longrightarrow C^{k,\alpha} (B_{c}; T^a_b \IR^n), \\ (Z,u) &\longmapsto   \phi^*_{ Z}  u,  \\
 H''_2 : \big[ C^{k+1,\alpha}(B_c;\IR^n) \big]_\eps &\longrightarrow C^{k,\alpha}(B_c; \IR^n), \\ Z &\longmapsto   \phi_{ Z}^* V - V - (\id + \nabla Z)^{-1} (\nabla_V Z).
\end{align*}

For $\delta \leq \ov\delta (A)$ the exponential map restricted to vectors on $U_{p, c\rho_p}$ of length $< \delta \rho_p$ can be expressed in the coordinates $\vec x_p$ as a map of the form
\[ B_{c \rho_p} \times B_{\delta \rho_p} \ni (\vec q, \vec v)  \longmapsto \vec q + \vec v + \rho_p Y_p ( \rho_p^{-1} \vec q, \rho_p^{-1}  \vec v) \in B_{2c \rho_p} \]
for some smooth $Y_p : B_{c} \times B_{\delta} \to B_{c}$ whose derivatives up to order $k+5$ are uniformly bounded and which satisfies
\begin{equation} \label{eq_Y_vanish}
 Y(\vec q, \vec 0)= d_2 Y (\vec q , \vec 0) = \vec 0 \qquad \text{for all} \quad \vec q \in B_{c}. 
\end{equation}
Thus if $\eps \leq \ov\eps(\delta)$, then
\begin{equation} \label{eq_dphiij_formula}
 (d\phi_Z)_i^j = \delta_i^j + \partial_i Z^j + \partial_i Y_p^j (\rho_p^{-1} \vec q,  \rho_p^{-1} Z(\vec q)) + \partial_{l+n} Y_p^j ( \rho_p^{-1} \vec q, \rho_p^{-1}  Z(\vec q)) \partial_{i}Z^l. 
\end{equation}
It follows that if $\delta \leq \ov\delta (A)$, then the maps
\begin{alignat*}{2}
 C^{k+1,\alpha} (B_c; B_\delta) &\longrightarrow C^{k+1,\alpha} (B_c; B_{2c}), & \qquad Z &\mapsto \phi_Z |_{B_c}, \\
 C^{k+1,\alpha} (B_c; B_\delta) &\longrightarrow C^{k,\alpha} (B_c), & \qquad Z &\mapsto (d\phi_Z)_i^j \quad \text{and} \quad Z \mapsto  ((d\phi_Z)^{-1})_i^j  
\end{alignat*}
are of regularity $C^2$ with uniformly bounded first and second derivatives.

Since in the coordinates $\vec x_p$ we have
\[ (\phi_Z^* u)_{j_1 \ldots j_b}^{i_1 \ldots i_a} = (d\phi_{Z})_{i'_1}^{i_1} \cdots (d\phi_{Z})_{i'_a}^{i_a} ((d\phi_{Z})^{-1})_{j_1}^{j'_1} \cdots ((d\phi_{Z})^{-1})_{j_b}^{j'_b} \big(u_{j'_1 \ldots j'_b}^{i'_1 \ldots i'_a} \circ \phi_Z \big), \]
and since $C^{1,\alpha}$-regularity and bounds are preserved under sums and products, we obtain the $C^{1,\alpha}$-property of $H''_1$ using Claim~\ref{Cl_scalar_pullback} (where we need to substitute $k$ with $k+1$ and use the embeddings $C^{k+1,\alpha} \to C^{k+1}$, $C^{k+1} \to C^{k,\alpha}$).

Let us now consider the map $H''_2$.
By Assumption~\ref{Lem_composition_iv} and since by \eqref{eq_gij_bounds} the Christ\"offel symbols satisfy 
\begin{equation} \label{eq_Gamma_bounds}
 |\partial^m \Gamma| \leq C(A) \rho^{-1-m}, \qquad \text{for all} \quad  m = 0, \ldots, k+9,
\end{equation}
we obtain uniform $C^{k+3}$-bounds on the derivatives $\partial_j V^i$ of the coefficient functions $V^i$.
Thus 
\begin{equation} \label{eq_V_diff_rho_bounds}
 \Vert V^i - V_{\vec 0}^i\Vert_{C^{k+2,\alpha}}, \qquad \Vert \rho_p^{-1} V^i \Vert_{C^{k+2,\alpha}} \leq C(A). 
\end{equation}
Next, we rewrite the component functions $H^{\prime\prime,1}_2, \ldots, H^{\prime\prime,n}_2$ of $H''_2$ as
\[ H^{\prime\prime,i}_2 (Z) = ((d\phi_Z)^{-1})^i_j \big( H^j_3(Z) + H^j_4(Z) + H^j_5(Z) \big), \]
where
\begin{align*}
H^j_3, H^j_4, H^j_5 : \big[ C^{k+1,\alpha}(B_c;B_\delta) \big]_\eps  &\longrightarrow C^{k,\alpha}(B_c), \\
H^j_3 : Z &\longmapsto V^j \circ \phi_Z - V^j, \\
H^j_4 : Z &\longmapsto V^j - (d\phi_Z)^j_s V^s + \partial_l Z^j \, V^l, \\
H^j_5 : Z &\longmapsto -\partial_l Z^j \, V^l +(d\phi_Z)^j_s \big((\id + \nabla Z)^{-1} (\nabla_V Z) \big)^s. 
\end{align*}
It remains to verify the $C^{1,\alpha}$-regularity for each of the maps $H^j_3, H^j_4, H^j_5$ separately.

The $C^{1,\alpha}$-regularity of the maps $H^j_3$ is a direct consequence of Claim~\ref{Cl_scalar_pullback} and \eqref{eq_V_diff_rho_bounds} since
\[ H^j_3(Z) = (V^j - V^j_{\vec 0} )\circ \phi_Z - (V^j - V^j_{\vec 0} ). \]
To analyze the maps $H^j_4, H^j_5$, we combine \eqref{eq_Y_vanish}, \eqref{eq_dphiij_formula} to write
\[ (d\phi_Z)_s^j = \delta^j_s + \partial_s Z^j + \rho_p^{-1} F^j_{sl} (\cdot, \ldots, \cdot, Z^1, \ldots, Z^n) \, Z^l + \rho_p^{-1} G^j_{ml} (\cdot, \ldots, \cdot, Z^1, \ldots, Z^n) \, Z^m \, \partial_s Z^l, \]
where the coefficient functions $F^j_{sl}$, $G^j_{ml}$, which arise from the function $Y_p$, are smooth with uniformly bounded $C^{k+3}$-norm and the omitted arguments denote spatial dependence.
It follows that
\[ H^j_4(Z) = - F^j_{sl} (\cdot, \ldots, \cdot, Z^1, \ldots, Z^n) \, Z^l \, (\rho_p^{-1} V^s) - G^j_{ml} (\cdot, \ldots, \cdot, Z^1, \ldots, Z^n) \, Z^m \, \partial_s Z^l \, (\rho_p^{-1} V^s) \]
and
\begin{align*}
 H^j_5(Z) &= -\partial_l Z^j \, V^l
+ \big(\delta_s^j + (\nabla Z)_s^j\big) \big((\id + \nabla Z)^{-1} (\nabla_V Z) \big)^s \\
&\qquad + \rho_p^{-1} \big( - \rho_p \Gamma_{si}^j Z^i + F^j_{sl} (\cdot, \ldots, \cdot, Z^1, \ldots, Z^n) \, Z^l \\
&\qquad\qquad\qquad +  G^j_{ml} (\cdot, \ldots, \cdot, Z^1, \ldots, Z^n) \, Z^m \, \partial_s Z^l \big) \big((\id + \nabla Z)^{-1} (\nabla_V Z) \big)^s \\
&= \rho_p \Gamma_{il}^j \, (\rho_p^{-1} V^i) \, Z^l + \big( - \rho_p \Gamma_{si}^j Z^i + F^j_{sl} (\cdot, \ldots, \cdot, Z^1, \ldots, Z^n) \, Z^l \\
&\qquad\qquad\qquad +  G^j_{ml} (\cdot, \ldots, \cdot, Z^1, \ldots, Z^n) \, Z^m \, \partial_s Z^l \big) \big((\id + \nabla Z)^{-1} (\nabla_{\rho_p^{-1} V} Z) \big)^s
\end{align*}
which implies the desired $C^{1,\alpha}$-property of $H^j_4, H^j_5$ via \eqref{eq_Gamma_bounds}, \eqref{eq_V_diff_rho_bounds}.
This concludes the proof of the claim.
\end{proof}

\end{proof}
\medskip

We can now prove Proposition~\ref{Prop_gauge}.

\begin{proof}[Proof of Proposition~\ref{Prop_gauge}.]
We use the same notation as in the statement of Lemma~\ref{Lem_composition}.
Our first goal will be to show that for sufficiently small $\eps > 0$ the following map is well defined and $C^{1,\alpha}$:
\begin{multline*}
 \td P : \big[ C^{k+2,\alpha}_{-1,\nabla f} (M; TM) \times  C^{k+3,\alpha}_{-1} (M; S^2 T^*M) \times T\GenCONE^{k^*}(N)  \times  C^{k+2,\alpha}_{-1} (M;TM) \big]_\eps \\ \longrightarrow C^{k,\alpha}_{-1} (M; TM) 
\end{multline*}
\begin{equation} \label{eq_tdPZhpgpVp}
 (Z,  h', \gamma',V') \longmapsto P(\phi_Z^* ( g+ h'+T(\gamma')), \phi_Z^*  (\nabla^g f + V') ). 
\end{equation}
(Note in contrast to the statement of the proposition we reparameterize $\gamma'$ and $V'$ as $\gamma + \gamma'$ and $\nabla^g f + V'$, respectively.)

To see this, we may temporarily work with the smooth structure on $M$ supplied by Lemma~\ref{Lem_soliton_smooth}, with respect to which $g, V$ are smooth.
Since $g$ is $C^{k^*-2}$ with respect to both smooth structures, we obtain that both smooth structures induce the same $C^{k^*-1}$-atlas on $M$.
Let us now express $\td P(Z, h',\gamma', V')$ as a sum of the following terms:
\begin{align}
 \phi^*_Z V' =& (\phi_Z^* V' - V') + V', \label{eq_tdP_term_1} \\
 \phi^*_Z \nabla^g f - \nabla^g f =&  \phi_Z^* \nabla^g f - \nabla^g f - (\id + \nabla Z)^{-1} (\nabla_{\nabla^g f} Z) , \notag \\
&+(\id + \nabla Z)^{-1} (\nabla_{\nabla^g f} Z),    \label{eq_tdP_term_2} \\
\DIV_g(\phi^*_Z g ) - \tfrac12 \nabla^g \tr_g(\phi^*_Z g) =&
\DIV_g(\phi^*_Z g -g) - \tfrac12 \nabla^g \tr_g(\phi^*_Z g - g), \label{eq_tdP_term_3} \\
\DIV_g (\phi^*_Z h') - \tfrac12 \nabla^g \tr_g (\phi_Z^* h') =& \DIV_g (\phi^*_Z h'-h') + \DIV_g (h') \notag \\
&  - \tfrac12 \nabla^g \tr_g (\phi_Z^* h'-h') - \tfrac12 \nabla^g \tr_g (h') , \label{eq_tdP_term_4} \\
\DIV_g (\phi^*_Z T(\gamma')) - \tfrac12 \nabla^g \tr_g (\phi^*_Z T(\gamma')) =& \DIV_g (\phi^*_Z T(\gamma') - T(\gamma')) + \DIV_g (T(\gamma'))  \notag \\
&- \tfrac12 \nabla^g \tr_g (\phi^*_Z T(\gamma') - T(\gamma'))- \tfrac12 \nabla^g \tr_g ( T(\gamma')). \label{eq_tdP_term_5} 
\end{align}
We will check that each of the terms \eqref{eq_tdP_term_1}--\eqref{eq_tdP_term_5} depends on $(Z,h',\gamma',V')$ in the $C^{1,\alpha}$-sense, where the Banach norms are to be taken as in \eqref{eq_tdPZhpgpVp}.
The $C^{1,\alpha}$-regularity of \eqref{eq_tdP_term_1}, \eqref{eq_tdP_term_3}, \eqref{eq_tdP_term_4} follows directly from Lemma~\ref{Lem_composition} and the fact that the natural embedding $C^{k+2,\alpha}_{-1,\nabla f} (M; TM) \to C^{k+2,\alpha}_{-1}(M;TM)$ is bounded and linear.
For the $C^{1,\alpha}$-regularity of  \eqref{eq_tdP_term_2} we also need to use the fact that the map $C^{k+2,\alpha}_{-1,\nabla f} (M; TM) \to C^{k,\alpha}_{-1}(M;TM)$, $Z \mapsto \nabla_{\nabla^g f} Z$ is bounded and linear.
For \eqref{eq_tdP_term_5} we use the fact that the maps
$T : T\GenCONE^{k^*} (N) \to C^{k+2,\alpha}(M;S^2 T^*N)$ and $T\GenCONE^{k^*} (N) \to C^{k,\alpha}_{-1}(M;S^2 T^*N)$, $\gamma' \mapsto \nabla (T(\gamma'))$ are bounded and linear; see Lemma~\ref{Lem_nabf_Tgamma}.

Note that $\td P(0,0,0,0) = 0$.
The first component of the differential of $\td P$ at the origin is given by
\[ C^{k+2,\alpha}_{-1,\nabla f} (M; TM) \longrightarrow C^{k,\alpha}_{-1}(M;TM), \]
\[ Z \longmapsto \frac{d}{dt} \bigg|_{t = 0}  P(\phi_{tZ}^*g, \phi_{tZ}^* \nabla^g f)
= \LL_Z \nabla^g f + \DIV_g (\LL_Z g) - \tfrac12 \nabla^g \tr_g (\LL_Z g) 
= \triangle_f Z - \tfrac12 Z, \]
which is invertible by Proposition~\ref{Prop_Lundardi_weighted}.
So by the Implicit Function Theorem we find that for small $\eps > 0$ there is a $C^{1,\alpha}$-map
\begin{equation*}
 \td Z : \big[   C^{k+3,\alpha}_{-1} (M; S^2 T^*M) \times T\GenCONE^{k^*}(N) \times  C^{k+2,\alpha}_{-1} (M;TM) \big]_\eps  \longrightarrow C^{k+2,\alpha}_{-1,\nabla f} (M; TM) 
\end{equation*}
such that for $(h', \gamma', V')$ in the domain of $\td Z$ we have
\[ P\big(\phi_{\td Z(h',\gamma', V')}^* ( g+h'+T(\gamma')), \phi_{\td Z(h',\gamma', V')}^*  (\nabla^g f + V') \big) = 0. \]

Next, by Lemma~\ref{Lem_composition} the maps
\[ \big[   C^{k+3,\alpha}_{-1} (M; S^2 T^*M) \times T\GenCONE^{k^*}(N)  \times  C^{k+2,\alpha}_{-1} (M;TM) \big]_\eps  \longrightarrow C^{k+1,\alpha}_{-1} (M; S^2T^*M) \]
\begin{multline*}
 (h',\gamma',V') \longmapsto \phi_{\td Z(h',\gamma', V')}^* (g+h'+T(\gamma'))  - g - T(\gamma') \\
= \big(\phi_{\td Z(h',\gamma', V')}^* (g+T(\gamma')) - (g+T(\gamma')) \big) + \big(\phi_{\td Z(h',\gamma', V')}^* h' - h' \big) + h' 
\end{multline*}
and
\begin{align*}
 \big[   C^{k+3,\alpha}_{-1} (M; S^2 T^*M) \times T\GenCONE^{k^*}(N)  \times  C^{k+2,\alpha}_{-1} (M;TM) \big]_\eps  &\longrightarrow C^{k+1,\alpha}_{-1} (M; S^2T^*M) \\
 (h',\gamma',V') &\longmapsto \phi_{\td Z(h',\gamma', V')}^* g - g 
\end{align*}
are well defined and $C^{1,\alpha}$.
This concludes the proof of Assertions~\ref{Prop_gauge_a}, \ref{Prop_gauge_c}.

It remains to establish the uniqueness statement of Assertion~\ref{Prop_gauge_b}.
For this, consider the following modified map 
for some small $\eps_0 > 0$ (note the change of the first and last H\"older space):
\begin{multline*}
 \td P_0 : \big[ C^{2,\alpha}_{0,\nabla f} (M; TM) \times  C^{3,\alpha}_{-1} (M; S^2 T^*M) \times T\GenCONE^{k^*}(N)  \times  C^{2,\alpha}_{-1} (M;TM) \big]_{\eps_0} \\ \longrightarrow C^{0,\alpha}_0 (M; TM) 
\end{multline*}
\[ (Z,  h', \gamma',V') \longmapsto P(\phi_Z^* ( g+ h'+T(\gamma')), \phi_Z^*  (\nabla^g f + V') ). \]
By the same reasoning as above, $\td P_0$ is well defined and $C^{1,\alpha}$ for small enough $\eps_0$.
Moreover, we have $\td P_0 (0,0,0,0) = 0$ and the first component of its differential at the origin is invertible (here we use Corollary~\ref{Cor_Lundardi} instead of Proposition~\ref{Prop_Lundardi_weighted}).
So, by the Implicit Function Theorem and after reducing $\eps_0$, we may assume that $\td P_0 (\cdot,  h',  \gamma',  V') = 0$ has a unique solution in $[ C^{2,\alpha}_{0,\nabla f} (M; TM) ]_{\eps_0}$.

Now fix some $h', \gamma', V'$ satisfying \eqref{eq_hpgpVp_small} and suppose that $\phi : M \to M$ satisfies \eqref{eq_phi_close_id} and $P(\phi^* g',\phi^* V') = 0$.
Assuming $\eps$ to be sufficiently small, we may find a vector field $Z \in C^3(M;TM)$ with $\phi = \phi_Z$.
The condition \eqref{eq_phi_close_id} implies that $\Vert Z \Vert_{C^{2,\alpha}} \leq C \eps$ for some uniform $C < \infty$.
By our previous discussion involving \eqref{eq_tdP_term_1}--\eqref{eq_tdP_term_5}, if we plug $(Z,h',\gamma',V')$ into the formula defining $\td P_0$ (note that $(Z,h',\gamma',V')$ may a priori not be in the domain of $\td P_0$), then we obtain
\begin{equation} \label{eq_nab_f_Z_0}
 0 = \td P_0(Z,h',\gamma',V')
= (\id + \nabla Z)^{-1} (\nabla_{\nabla^g f} Z) + \td P'_0 (Z,h',\gamma',V'), 
\end{equation}
where $\td P'_0$ is a $C^{1,\alpha}$-map of the form
\begin{multline*}
 \td P'_0 : \big[ C^{2,\alpha}_{0} (M; TM) \times  C^{3,\alpha}_{-1} (M; S^2 T^*M) \times T\GenCONE^{k^*}(N)  \times  C^{2,\alpha}_{-1} (M;TM) \big]_{\eps_0} \\ \longrightarrow C^{0,\alpha} (M; TM) 
\end{multline*}
with $\td P'_0(0,0,0,0) = 0$ (note the change in the first H\"older space).
So for small enough $\eps > 0$ we obtain that $\Vert \td P'_0 (Z,h',\gamma',V') \Vert_{C^{0,\alpha}} \leq C \eps$ for some uniform $C < \infty$.
Combining this with \eqref{eq_nab_f_Z_0} yields
\[ \Vert Z \Vert_{C^{2,\alpha}_{0,\nabla f}} = \Vert Z \Vert_{C^{2,\alpha}} + \Vert \nabla_{\nabla^g f} Z \Vert_{C^{0,\alpha}} \leq C \eps. \]
So $(Z,h',\gamma',V')$ is indeed contained in the domain of $\td P_0$ for sufficiently small $\eps$, which implies the uniqueness of $Z$.
\end{proof}

\section{Smooth dependence} \label{sec_smooth_dependence}
\subsection{Introduction}
Let $(M,N,\iota)$ be an emsemble and $k^* \geq 30$ an integer and consider the spaces $\MM = \MM^{k^*} (M, N, \iota)$ and $\MM_{\grad} = \MM_{\grad}^{k^*} (M, N, \iota)$.
The goal of this section is to show that a neighborhood of $\MM_{\grad}$ within $\MM$ can be equipped with a canonical structure of a $C^{1,\alpha}$-Banach manifold such that the projection 
\[ \Pi : \MM^{k^*} (M,N, \iota) \lto \GenCONE^{k^*}(N) \] is $C^{1,\alpha}$ near points in $\MM_{\grad}$ and that is of a certain standard form in appropriate local charts.

\subsection{Gauging families of expanding solitons at infinity}
The main result of this subsection, Proposition~\ref{Prop_gauge_infinity_families}, is a rather technical result, which allows us to gauge continuous families of solitons $(M,g_x,V_x)_{x \in X}$ at infinity, i.e., it allows us to find diffeomorphisms $\psi_x : M \to M$ such that $\iota^* \psi_x^* V_x = - \frac12 r \partial_r$ on $(r_0,\infty) \times N$ for some $r_0 > 1$.
This will later be used to conclude that $(\psi^*_x g_x, \psi^*_x V_x)$ represents an element $p_x \in \MM$, which varies continuously in $x$.

As a preparation we establish the following lemma.

\begin{Lemma} \label{Lem_construct_iota_x}
Let $(M,N,\iota)$ be an ensemble, $k^* \geq 10$ and consider a $C^{k^*-2}$-regular representative $(g,V = \nabla^g f,\gamma)$ of an element of $\MM_{\grad}(M,N,\iota)$. 
Consider a family of vector fields $(V_x)_{x \in X}$ on $M$, where $X$ is a topological space.
Suppose that $V_{x_0} = V$ for some $x_0 \in X$ and that for any $x \in X$ we have $V_x - V_{x_0} \in C^1_{-1}(M;TM)$ (where $C^1_{-1}$ is defined using the soliton structure $(M,g,f)$; see Subsection~\ref{subsec_list_Holder}) such that the map
\[ X \lto C^1_{-1}(M;TM), \qquad x \longmapsto V_x - V_{x_0} \]
is continuous.
Then there is a neighborhood $x_0 \in U \subset X$ and a  family of maps $(\iota_x : \IR_+ \times N \to M)_{x \in U}$ such that the following is true:
\begin{enumerate}[label=(\alph*)]
\item \label{Lem_construct_iota_x_a} $\iota_{x_0}|_{(1,\infty) \times N} = \iota$.
\item \label{Lem_construct_iota_x_b} For any $x \in U$ the map $s \mapsto \iota_x (e^{s/2}, z)$ is a trajectory of $-V_x$. 
\item \label{Lem_construct_iota_x_c} For any $x_1, x_2 \in U$ we have
\begin{equation} \label{eq_iotax1x2_finite_dist}
 \sup_{r > 0, z \in N} d_g(\iota_{x_1} (r,z), \iota_{x_2}(r,z)) < \infty. 
\end{equation}
\item \label{Lem_construct_iota_x_d} For any $r_0 > 0$ and any convergent sequence $x_i \to x_\infty$ in $U$ we have
\[ \sup_{r \geq r_0,z \in N} d_g(\iota_{x_i} (r,z), \iota_{x_\infty}(r,z)) \xrightarrow[i\to\infty]{} 0. \]
\item \label{Lem_construct_iota_x_e} For any $x \in U$ the map $\iota_x$ is uniquely characterized by Assertion~\ref{Lem_construct_iota_x_b} and the bound \eqref{eq_iotax1x2_finite_dist} for $x_1 = x$ and $x_2 = x_0$.
\end{enumerate}
\end{Lemma}

\begin{proof}
Choose a neighborhood $x_0 \in U \subset X$ such that the radial component of $-\iota^* V_x$ is everywhere positive for all $x \in U$; so if $s \mapsto \iota(r(s), z(s))$ is a trajectory of $-V_x$, then $r'(s) > 0$.
Moreover, for any trajectory $\sigma : \IR \to M$ of any $-V_x$, $x \in U$, the parameters $s$ for which $ \sigma(s) \in \iota((1,\infty) \times N)$ is either empty or an interval of the form $(s_0, \infty)$ for some $s_0 \in \IR$.

Write $\iota^* V_x := -\frac12 r \partial_r + \td V_x$, where $(\td V_x)_{x \in U}$ is a family of vector fields on $(1,\infty) \times N$ that depends continuously on $x$ in the $C^1_{\gamma,r^{-1}}$-norm.
Note that $\td V_{x_0} = 0$ and the trajectories $\sigma(s)$ of $-V_x$ are of the form $\iota(\td\sigma(s))$ for trajectories $\td\sigma(s)$ of $\frac12 r \partial_r - \td V_x$.

\begin{Claim} \label{Cl_traj_bounds}
After possibly shrinking $U$ there is a constant $C < \infty$ such that the following is true.
Consider two $x_1, x_2 \in U$ and let $\td\sigma_i : (\underline s_i, \infty) \to (1,\infty) \times N$, $i=1,2$, be trajectories of $\frac12 r \partial_r - \td V_{x_i}$.
Let $r : (1,\infty) \times N \to (1,\infty)$ be the radial coordinate.
Then whenever
\begin{equation} \label{eq_10_conditions}
 d_{\gamma}(\td\sigma_1(s),\td\sigma_2(s)) \leq \tfrac1{10} r (\td\sigma_1(s)) \qquad \text{and} \qquad r (\td\sigma_1(s)) \geq 10,
\end{equation}
we have
\begin{equation} \label{eq_dds_dsigma12}
 \frac{d}{ds} d_{\gamma}(\td\sigma_1(s),\td\sigma_2(s))
\geq  \bigg( \frac12 - \frac{C}{r(\td\sigma_1(s))} \Vert \td V_{x_1} \Vert_{C^1_{\gamma, r^{-1}}} \bigg) d_{\gamma}(\td\sigma_1(s),\td\sigma_2(s))  - \frac{C}{r(\td\sigma_1(s))} \big\Vert \td V_{x_1} - \td V_{x_2} \big\Vert_{C^1_{\gamma,r^{-1}}}. 
\end{equation}
Moreover, there is a uniform constant $C_0 < \infty$ such that whenever the bound 
\begin{equation} \label{eq_dsigma12_lessC0}
  d_{\gamma}(\td\sigma_1(s),\td\sigma_2(s))  \leq C_0 \Vert \td V_{x_1} - \td V_{x_2} \Vert_{C^1_{\gamma,r^{-1}}}.  
\end{equation}
holds for $s = s_0$ for some $s_0 > \max \{ \underline{s}_1, \underline{s}_2 \}$ with $r (\td\sigma_1(s_0)) \geq 10$, then the same bound holds for all $s \in (\max \{ \underline{s}_1, \underline{s}_2 \}, s_0]$ with the property that $r (\td\sigma_1(s)) \geq 10$.
\end{Claim}

\begin{proof}
The bound \eqref{eq_dds_dsigma12} follows from a direct computation.
The condition \eqref{eq_10_conditions} ensures that the minimizing geodesic between $\td\sigma_1(s), \td\sigma_2(s)$ lies within $(1,\infty) \times N$.

To see the second statement, notice that we can choose $C_0$ large enough such that, after shrinking $U$, the bound \eqref{eq_dds_dsigma12} and the equality case \eqref{eq_dsigma12_lessC0} imply that
\[ \frac{d}{ds} d_{\gamma}(\td\sigma_1(s),\td\sigma_2(s)) > 0. \]
So if we have equality in \eqref{eq_dsigma12_lessC0} for some $s \in (\underline{s}_i, s_0]$ with the property that $r (\td\sigma_1(s)) > 10$, then the left-hand side of \eqref{eq_dsigma12_lessC0} is strictly increasing at $s$, which implies the preservation of \eqref{eq_dsigma12_lessC0}.
\end{proof}

\begin{Claim} \label{Cl_trajectory}
For any $x_1,x_2 \in U$ and any trajectory $\td\sigma_1 : (\underline s_1,\infty) \to (1,\infty) \times N$ of $\frac12 r \partial_r - \td V_{x_1}$ there is a unique trajectory $\td\sigma_2 : (\underline s_2,\infty) \to (1,\infty) \times N$ of $\frac12 r \partial_r - \td V_{x_2}$ such that
\begin{equation*} 
 \limsup_{s \to \infty} d_{\gamma}(\td\sigma_1(s), \td\sigma_2(s)) < \infty. 
\end{equation*}
Moreover, we have 
\[ \sup_{s > \max \{ \underline{s}_1, \underline{s}_2 \}, r (\td\sigma_1(s)) \geq 10} d_{\gamma}(\td\sigma_1(s), \td\sigma_2(s)) \leq C_0 \Vert V_{x_1} - V_{x_2} \Vert_{C^1_{\gamma,r^{-1}}}. \]
\end{Claim}

\begin{proof}
The uniqueness statement of the claim follows by setting $x_1 = x_2$ in \eqref{eq_dds_dsigma12}.
To see the existence, choose a sequence $s_j \to \infty$.
For each $j$ let $\td\sigma_{2,j} : (\underline{s}_{2,j},\infty) \to M$ be the trajectory of $\frac12 r \partial_r - \td V_{x_2}$ with $\td\sigma_{2,j}(s_j) = \td\sigma_1(s_j)$.
Then by the second statement in Claim~\ref{Cl_traj_bounds} we have
\[ d_{\gamma} (\td\sigma_1(s),\td\sigma_{2,j}(s) ) \leq C_0 \big\Vert \td V_{x_1} - \td V_{x_2} \big\Vert_{C^1_{\gamma,r^{-1}}} \]
for all $s > \max \{ \underline{s}_1, \underline{s}_{2,j} \}$ with $r(\td\sigma_1(s)) \geq 10$.
Since $r (\td\sigma_1(s)) \leq 2$ for $s > \underline{s}_{2,j}$ close to $\underline{s}_{2,j}$, it follows that $\underline{s}_{2,j}$ remains bounded and therefore, after passing to a subsequence, the trajectories $\td\sigma_{2,j}$ converge to the desired trajectory $\td\sigma_2$.
\end{proof}

We now define the maps $\iota_x$ as follows.
For any $x \in U$ and any $z \in N$ consider the trajectory $\td\sigma_1 : s \mapsto (e^{s/2},z)$ of $\frac12 r \partial_r = \tfrac12 r \partial_r - \td V_{x_0}$ and use Claim~\ref{Cl_trajectory} for $x_1 = x_0$ and $x_2 = x$ to determine the maximal trajectory $\td\sigma_2$ of $\tfrac12 r \partial_r - \td V_{x}$.
Set $\iota_x(e^{s/2},z) := \iota(\td\sigma_2(s))$, which is defined as long as $e^{s/2} > 20$, after possibly shrinking $U$.
Note that $s \mapsto \iota_x(e^{s/2},z)$ is a trajectory of $-V_x$, so it can be extended to a maximal trajectory in $M$.
Assertions~\ref{Lem_construct_iota_x_a}--\ref{Lem_construct_iota_x_c}, \ref{Lem_construct_iota_x_e} and Assertion~\ref{Lem_construct_iota_x_d} for $r_0 \geq 20$ follow immediately from Claim~\ref{Cl_trajectory}.
The case $r_0 < 20$ in Assertion~\ref{Lem_construct_iota_x_d} follows by continuous dependence on the initial conditions.
\end{proof}

The following proposition is the main result of this subsection.

\begin{Proposition} \label{Prop_gauge_infinity_families}
Fix an ensemble $(M,N,\iota)$ and integers $2 \leq k \leq k^*-10$ and set $\MM = \MM^{k^*}(M,N,\iota)$ and $\MM_{\grad} = \MM_{\grad}^{k^*}(M,N,\iota)$.
Let $(g,V = \nabla f,\gamma)$ be a $C^{k^*-2}$-regular representative of an element of $\MM_{\grad}$. 
Denote by $X$ a metrizable topological space and consider a family $(g_x, V_x, \gamma_x)_{x \in X}$ where $g_x \in C^{k}_{\loc} (M; S^2T^*M)$ is a Riemannian metric, $V_x \in C^k_{\loc}(M; TM)$ is a vector field and $\gamma_x \in \GenCONE^{k^*}(N)$.
Suppose that:
\begin{enumerate}[label=(\roman*)]
\item We have $(g_{x_0}, V_{x_0}, \gamma_{x_0}) = (g,V,\gamma)$ for some $x_0 \in X$.
\item For any $x \in X$ the soliton equation holds:
\[ \Ric_{g_x} + \tfrac12 \LL_{V_x} g_x + \tfrac12 g_x = 0 \]
and the metric $g_x$ is asymptotic to the generalized cone metric $\gamma_x$ in the sense that $\iota^* g - \gamma_x \to 0$ uniformly at infinity.
\item \label{Prop_gauge_infinity_families_iii} For any $x \in X$ we have $V_x - V_{x_0} \in C^k_{-1}(M;TM)$ and $g_x - g_{x_0} \in C^k_{-1}(M;S^2T^*M)$ and the maps (H\"older norms are taken with respect to $(M,g,f)$; see Subsection~\ref{subsec_list_Holder})
\begin{alignat*}{2}
 X &\lto C^{k}_{-1}(M;TM), & \qquad x &\longmapsto V_x - V_{x_0}, \\
 X &\lto C^k_{-2}(M;S^2T^*M), & \qquad x &\longmapsto g_x - g_{x_0} - T(\gamma_x - \gamma_{x_0}), \\
 X &\lto \GenCONE^{k^*}(N), &\qquad x &\longmapsto \gamma_x
\end{alignat*}
are continuous.
\end{enumerate}
Then there is a family of $C^{k+1}$-embeddings $(\iota_x : \IR_+ \times N \to M)_{x \in U}$, a family of $C^{k+1}$-diffeomorphisms $(\psi_x : M \to M)_{x \in U}$ and a constant $R_0 < \infty$ such that the following is true:
\begin{enumerate}[label=(\alph*)]
\item \label{Prop_gauge_infinity_families_a} $\iota_{x_0}|_{(1,\infty) \times N} = \iota$.
\item \label{Prop_gauge_infinity_families_b} $\iota_x^* V_x =  -\frac12 r \partial_r$  for all $x \in U$.
\item \label{Prop_gauge_infinity_families_c} 
For any $r_0 > 0$ and any convergent sequence $x_i \to x_\infty$ in $U$ we have
\begin{equation} \label{eq_dg_iota_x1x2}
\sup_{r \geq r_0,z \in N} d_g(\iota_{x_i} (r,z), \iota_{x_\infty}(r,z)) \xrightarrow[i\to\infty]{} 0
\end{equation}
and
\begin{equation} \label{eq_dg_psi_x1x2}
  \sup_{y \in M} d_g(\psi_{x_i} (y), \psi_{x_\infty}(y)) \xrightarrow[i\to\infty]{} 0. 
\end{equation}
\item \label{Prop_gauge_infinity_families_h} $\iota_x$ is uniquely determined by Assertion~\ref{Prop_gauge_infinity_families_b} and the fact that 
\[ \sup_{r > 1, z \in N} d_g (\iota_x(r,z),\iota(r,z)) < \infty. \]
\item \label{Prop_gauge_infinity_families_ee} $\psi_{x_0} = \id_M$.
\item \label{Prop_gauge_infinity_families_e} $\iota_x = \psi_x \circ \iota$ on $(R_0,\infty) \times N$ for all $x \in U$.
\item \label{Prop_gauge_infinity_families_f} $\iota^* \psi^*_x V_x = - \frac12 r \partial_r$ on $(R_0,\infty) \times N$ for all $x \in U$.
\item \label{Prop_gauge_infinity_families_d} There is a uniform constant $C < \infty$ such that for all $x \in X$ we have for all $0 \leq m \leq k^*-4$
\[ |\nabla^m (\iota_{x}^* g_x - \gamma_{x} ) | \leq Cr^{-2-m} \qquad \text{on} \quad (R_0,\infty) \times N. \]
Here the covariant derivatives are taken with respect to $\gamma_{x}$.
\item \label{Prop_gauge_infinity_families_g} $\psi_x$ has finite displacement for all $x \in U$, i.e., $\sup_{y \in M} d_g (\psi_x(y),y) < \infty$.
\end{enumerate}
Note that Assertions~\ref{Prop_gauge_infinity_families_f},  \ref{Prop_gauge_infinity_families_d} imply that $(\psi_x^* g_x, \psi_x^* V_x, \gamma_x)$ represents an element of $\MM$.
Moreover, due to Assertions~\ref{Prop_gauge_infinity_families_b}, \ref{Prop_gauge_infinity_families_e}, if $\td\iota_x : \IR_+ \times N \to M$ denotes the map from Lemma~\ref{Lem_MM_regular_rep}\ref{Lem_MM_regular_rep_c}, then we have $\iota_x = \psi_x \circ \td\iota_x$.
\end{Proposition}

\begin{proof}
Let $(\iota_x : \IR_+ \times N \to M)_{x \in U}$ be the family of maps from Lemma~\ref{Lem_construct_iota_x}.
Comparing these maps with those of Lemma~\ref{Lem_iota} implies that, in fact, each map $\iota_x$ is an embedding of regularity $C^{k+1}$ and that Assertions~\ref{Prop_gauge_infinity_families_a}, \ref{Prop_gauge_infinity_families_b}, \ref{Prop_gauge_infinity_families_d} and the continuity property \eqref{eq_dg_iota_x1x2} in Assertion~\ref{Prop_gauge_infinity_families_c} hold. 

Next, we need to construct the maps $\psi_x$.
Fix some $x \in U$ for a moment and consider the map 
\[ \phi_{x} := \iota^{-1} \circ \iota_{x} \big|_{(1,\infty) \times N} : (1,\infty) \times N \lto (1,\infty) \times N \]
Due to Lemma~\ref{Lem_construct_iota_x}\ref{Lem_construct_iota_x_c} and \ref{Lem_construct_iota_x_d}, we find that, given some $\eps > 0$, we can shrink $U$ such that $d_\gamma (\phi_x(r,z), (r,z)) < \eps$ for all $(r,z) \in (2,\infty) \times N$.
Moreover,
\[ \phi_{x}^* ( \iota^* g_x) = \iota^*_{x}g_x \]
and since both metrics $\iota^* g_x$, $\iota^*_{x}g_x$ are sufficiently controlled on $(R_0, \infty)\times N$, by Assumption~\ref{Prop_gauge_infinity_families_iii} and Assertions~\ref{Prop_gauge_infinity_families_d}, respectively, we obtain $C^{k+1}_{\loc}$-bounds on $\phi_{x}$ with respect to $\gamma$ on $(R_0+1, \infty) \times N$, which go to zero as $\eps \to 0$.
So if $\eps$ is chosen sufficiently small, we can perform an interpolation construction between the identity map and $\phi_x$.
As a result, we obtain a family of $C^{k+1}$-diffeomorphism $(\ov\phi_{x} : \IR_+ \times N \to \IR_+ \times N)_{x \in U}$ such that $\ov\phi_{x} = \id$ on $(0, 2) \times N$ and $\ov\phi_{x} = \phi_{x}$ on $(R_0+2, \infty) \times N$.
Thus $\iota \circ \ov\phi_{x} \circ \iota^{-1}$ can be extended smoothly by the identity wherever it is not defined.
Let $\psi_{x} : M \to M$ be the resulting map.
Then Assertions~\ref{Prop_gauge_infinity_families_e}, \ref{Prop_gauge_infinity_families_f}, \ref{Prop_gauge_infinity_families_g} follow immediately, after adjusting $R_0$.
Assertion~\ref{Prop_gauge_infinity_families_ee} and the continuity property \eqref{eq_dg_psi_x1x2}  in Assertion~\ref{Prop_gauge_infinity_families_c} can be ensured in the construction process of the maps $\psi_x$.
Lastly, Assertion~\ref{Prop_gauge_infinity_families_h} follows from Lemma~\ref{Lem_construct_iota_x}\ref{Lem_construct_iota_x_e}.
\end{proof}
\bigskip

\subsection{Linear theory}
We will first establish some crucial facts of linear nature, which concern the Einstein operator $L_g = \triangle_f + 2\Rm$ on a fixed gradient expanding soliton $(M,g,f)$.

\begin{Lemma} \label{Lem_linear_identities}
Let $(M,N,\iota)$ be an ensemble and $k^* \geq 20$.
Fix a $C^{k^*-2}$-regular representative $(g,V = \nabla^g f,\gamma)$ of an element in $\MM_{\grad}(M,N,\iota)$.
In the following all weighted H\"older spaces are taken with respect to the soliton structure $(M,g,f)$ (see Subsection~\ref{subsec_list_Holder} for further details). %
Denote by $K \subset C^{2}_{-1} (M; \lb S^2 T^*M)$ the kernel of the Einstein operator $L = L_g = \triangle_f + 2 \Rm$ on symmetric $(0,2)$-tensors.
Then the following is true for any $k \geq 0$ and $\alpha \in (0,1)$:
\begin{enumerate}[label=(\alph*)]
\item \label{Lem_linear_identities_a} $\dim K < \infty$.
\item \label{Lem_linear_identities_b} $K$ is equal to the kernel of the following map, which has Fredholm index $0$, \[ L : C^{k+2,\alpha}_{-2,\nabla f} (M;S^2 T^*M) \lto C^{k,\alpha}_{-2} (M;S^2 T^*M). \] 
\item \label{Lem_linear_identities_c} There are unique 
closed subspaces $K_k^\perp \subset C^{k,\alpha}_{-2} (M;S^2 T^*M)$ and $K_{k,\nabla f}^\perp \subset C^{k,\alpha}_{-2, \nabla f} (M;S^2 T^*M)$  such that we have the following splittings
\begin{equation} \label{eq_K_Kperp_splitting}
 C^{k,\alpha}_{-2} (M;S^2 T^*M)  = K \oplus K_k^\perp, \qquad
 C^{k,\alpha}_{-2,\nabla f} (M;S^2 T^*M)  = K \oplus K_{k,\nabla f}^\perp 
\end{equation}
which are orthogonal with respect to the natural pairing $\langle h_1, h_2 \rangle_f := \int_M ( h_1 \cdot h_2 ) e^{-f} dg$.
That is, for any $h_1 \in K$ and $h_2 \in K_k^\perp$ the integrand defining this pairing is integrable and the integral equals zero; a similar statement holds for the second splitting.
Moreover, we have $K^\perp_0 \supset K^\perp_1 \supset \ldots$ and $K^\perp_{0,\nabla f} \supset K^\perp_{1,\nabla f} \supset \ldots$ and $K^\perp_k \supset K^\perp_{k,\nabla f}$.
\item \label{Lem_linear_identities_d} We have $L (K_{k+2,\nabla f}^\perp) \subset K_k^\perp$
and the map
\[ K_{k+2,\nabla f}^\perp \lto K_k^\perp, \qquad h \mapsto Lh \]
is invertible with bounded inverse.
\item \label{Lem_linear_identities_e} The map
\[ T\GenCONE^{k^*} (N) \lto K, \qquad \gamma' \longmapsto \proj_K(L(T(\gamma')) \]
is surjective.
Here $\proj_K$ denotes the projection onto the first factor of \eqref{eq_K_Kperp_splitting} for any $k = 0, \ldots, k^* - 6$.
\item \label{Lem_linear_identities_f} Consider a tensor $h \in C^2_{\loc} (M; S^2 T^* M)$ with the property that for some $\delta > 0$ and $C < \infty$ we have the asymptotic bounds
\begin{equation} \label{eq_h_decay_bounds}
 |h|, \; |Lh|  \leq C |f|^{-\delta}, \qquad |\nabla h| \leq C. 
\end{equation}
Then $\langle \kappa, L h \rangle_f = 0$ for any $\kappa \in K$.
\end{enumerate}
\end{Lemma}

\begin{proof}
Assertions~\ref{Lem_linear_identities_a} and \ref{Lem_linear_identities_b} follow from Proposition~\ref{Prop_Lundardi_weighted} (applied with the smooth structure on $M$ supplied by Lemma~\ref{Lem_soliton_smooth}) and the fact that we have $\Rm(h) \in C^{0,\alpha}_{-2}(M; S^2 T^*M)$ for any  $h \in C^2_{-1}(M; S^2 T^*M)$.

After possibly adjusting $f$ by an additive constant, we may assume that $f < 0$.
Set $\mathbf{r} := 2\sqrt{-f} \in C^2(M)$.
We need the following claim.

\begin{Claim} \label{Cl_DS_kappa_bound}
For any kernel element $\kappa \in K$ the following is true.
Set
\[ \hat\kappa := \mathbf{r}^n e^{\mathbf{r}^2/4} \kappa. \]
Then 
\[ \lim_{\mathbf{r} \to \infty} \hat\kappa = \kappa_\infty, \]
where the limit is to be understood in the $L^2(N)$-sense along trajectories of $\nabla^g f$.
More specifically, consider the dilations for $\lambda > 0$
\begin{equation} \label{eq_omega_la}
 \omega_\lambda : \IR_+ \times N \longrightarrow \IR_+ \times N, \qquad (r,z) \longmapsto (\lambda r, z). 
\end{equation}
Then we have convergence
\begin{equation} \label{eq_L2_loc_to_k_inf}
 \hat\kappa_\lambda := \frac1{\lambda^2} \omega_\lambda^* ( \iota^*  \hat\kappa )  \xrightarrow[\lambda \to \infty]{\quad L^2_{\loc} \quad } \kappa_\infty, 
\end{equation}
where $\kappa_\infty$ is a symmetric $(0,2)$-tensor field on $\IR_+ \times N$ of regularity $L^2_{\loc}$ with the property that $\frac1{\lambda^2} \omega_\lambda^* \kappa_\infty = \kappa_\infty$ for all $\lambda > 0$, i.e., $\kappa_\infty$ is parallel along radial lines.
Moreover, if $\kappa \neq 0$, then $\kappa_\infty$ does not vanish in $L^2$.
\end{Claim}

\begin{proof}
This is essentially a consequence of \cite[Theorem~7.1]{Deruelle_Schulze_2021}, but it requires some justification.
Let us apply this theorem for $(h, \mathcal{B}) = (\kappa, 0)$.
Then equations (2.18)--(2.20) and (4.1) of \cite{Deruelle_Schulze_2021} hold as required.
Note that equation (2.11) of \cite{Deruelle_Schulze_2021} may not be true, however, this equation is not part of the assumptions of \cite[Theorem~7.1]{Deruelle_Schulze_2021} and does not enter its proof.

Next, observe that the assertion of \cite[Theorem~7.1]{Deruelle_Schulze_2021} defines convergence via parallel transport along trajectories of $\nabla^g f$ instead of convergence of pullbacks under dilations.
However, this difference is negligible as the metric $g$ is asymptotically conical and the difference between the Levi-Civita connection and the connection induced by the dilation process is of order $O(r^{-3})$.
Alternatively, one may derive the convergence \eqref{eq_L2_loc_to_k_inf} directly from equation (7.6) in \cite{Deruelle_Schulze_2021}.
In this case the connection in $\nabla_{\mathbf{N}} \hat h$ needs to be replaced by
\[ \frac1{\mathbf{r}} \left( \mathcal{L}_{2\nabla f} \hat h - 2\hat h \right) = \frac1{\mathbf{r}} \left(  \nabla_{2\nabla f} \hat h + 2\Ric * \hat h \right)
= \frac{2|\nabla f|}{\mathbf{r}}  \nabla_{\mathbf{N}} \hat h  + O(\mathbf{r}^{-3}) \hat h; \]
note that $\frac{2|\nabla f|}{\mathbf{r}} \to 1$.

Let us now prove the last statement.
Suppose by contradiction that $\kappa \neq 0$, but $\kappa_\infty = 0$.
Then \cite[Theorem~7.1]{Deruelle_Schulze_2021} implies that $\kappa$ has compact support.
Consider a point $y \in M$ in the boundary of the support. By \cite{Bando_analytic_87, Kotschwar_analytic_13} we can choose coordinates around $y$ in which the coefficients of $g$ are real-analytic.
Due to the soliton equation, we obtain that the second derivatives $\partial_{ij}f$ in these coordinates are real-analytic, so $f$ is real-analytic as well.
The equation $L \kappa = 0$ can therefore be locally expressed using real analytic coefficients, so by \cite{Morrey_book_58} the coefficients of $\kappa$ in these local coordinates are real-analytic.
Since all derivatives of these coefficients vanish outside their support, they must be zero at the origin, which corresponds to the point $y$.
Therefore $\kappa$ must vanish in a neighborhood of $y$, which contradicts the choice of $y$.

Lastly, note that the work \cite{Deruelle_Schulze_2021} assumes that $(M,g)$ is asymptotic to a \emph{smooth} cone metric, whereas here the cone metric has regularity $C^{k^*}$ for $k^* \geq 20$.
This is, however, enough to carry out the arguments in \cite{Deruelle_Schulze_2021}.
\end{proof}

Claim~\ref{Cl_DS_kappa_bound} implies that for any $\kappa \in K$ the form
\[ C^{k,\alpha}_{-2} (M;S^2 T^*M) \longrightarrow \IR, \qquad h \longmapsto \langle \kappa, h \rangle_f \]
is well defined and bounded.
So we can define $K_k^\perp$ as the kernel of the bounded map
\[ C^{k,\alpha}_{-2} (M;S^2 T^*M) \longrightarrow K^*, \qquad h \longmapsto \langle \cdot, h \rangle_f, \]
and define $K_{k,\nabla f}^\perp$ similarly, which proves Assertion~\ref{Lem_linear_identities_c}.

For the remaining assertions we need the following claim, which allows us to perform integration by parts.
This may involve picking up a term at infinity.

\begin{Claim} \label{Cl_kLh_int}
Let $h \in C^2_{\loc}(M; S^2 T^*M)$ such that one of the following is true:
\begin{enumerate}[label=(\roman*)]
\item $h$ satisfies the bounds \eqref{eq_h_decay_bounds} for some $0 < \delta \leq 1$ and $C < \infty$.
\item There is some $\gamma' \in T\GenCONE^2(N)$ such that we have $h = T(\gamma')$ on the complement of some compact subset.
\end{enumerate}
Then for any $\kappa \in K$ we have
\[ \langle \kappa, Lh \rangle_f = \begin{cases} 0 &\text{in Case (i)} \\ \frac12 \int_{\{ 1 \} \times N } (\kappa_\infty \cdot \gamma') d\gamma|_{\{ 1 \} \times N} & \text{in Case (ii)} \end{cases} \]
In particular, Assertion~\ref{Lem_linear_identities_f} holds.
\end{Claim}

\begin{proof}
Let $\eta : \IR \to [0,1]$ be some smooth cutoff function with $\eta \equiv 1$ on $(-\infty, 1]$ and $\eta \equiv 0 $ on $[2, \infty)$ and set $\eta_\la (r) := \eta(r/\la)$.
By Claim~\ref{Cl_DS_kappa_bound} and the fact that $Lh = O(\mathbf{r}^{-2\delta})$ by assumption in Case~(i) and due to Lemma~\ref{Lem_nabf_Tgamma} in Case~(ii), we obtain that
\begin{equation} \label{eq_kLh_lim}
  \langle \kappa, Lh \rangle_f  = \lim_{\la \to \infty} \big\langle (\eta_\la \circ  \mathbf{r}) Lh , \kappa \big\rangle_f. 
\end{equation}
Using integration by parts, we obtain that
\begin{align*}
 \langle (\eta_\la \circ  \mathbf{r}) L h, \kappa \rangle_f 
&= \langle (\eta_\la \circ  \mathbf{r}) Lh, \kappa \rangle_f - \langle (\eta_\la \circ  \mathbf{r})  h, L  \kappa \rangle_f 
= \langle (\eta_\la \circ  \mathbf{r}) L h,  \kappa \rangle_f -  \langle L ( (\eta_\la \circ  \mathbf{r}) h),  \kappa \rangle_f 
\\
&=- \int_M \big( \big( \triangle_f (\eta_\la \circ  \mathbf{r}) \big) (h \cdot \hat\kappa)
+ 2 \nabla_{\nabla (\eta_\la \circ  \mathbf{r})} h \cdot \hat\kappa \big) \mathbf{r}^{-n} dg. \\
&= -\int_M \big( ( \triangle (\eta_\la \circ  \mathbf{r}) ) (h \cdot \hat\kappa) + \tfrac12 \mathbf{r} ( \nabla \mathbf{r} \cdot \nabla (\eta_\la \circ  \mathbf{r}) ) (h \cdot \hat\kappa)
+ 2 \nabla_{\nabla (\eta_\la \circ  \mathbf{r})} h \cdot \hat\kappa \big) \mathbf{r}^{-n} dg.
\end{align*}
Set $\gamma' := 0$ in Case (i).
Recall the dilations $\omega_\la$ from \eqref{eq_omega_la}.
Since we have the following convergences
\begin{alignat*}{2}
 g_\la := \la^{-2} \omega_\la^* \iota^* g &\xrightarrow[\la\to\infty]{\quad C^2_{\loc} \quad} \gamma, \qquad & 
h_\la := \la^{-2} \omega_\la^* \iota^*  h &\xrightarrow[\la\to\infty]{\quad C^0_{\loc} \quad} \gamma', \qquad \\
\hat\kappa_\la := \la^{-2} \omega_\la^* \iota^*  \hat\kappa &\xrightarrow[\la\to\infty]{\quad L^2_{\loc} \quad} \kappa_\infty , \qquad &
\mathbf{r}_\la := \la^{-1} (\mathbf{r} \circ \iota \circ \omega_\la)  &\xrightarrow[\la\to\infty]{\quad C^2_{\loc} \quad} r,
\end{alignat*}
and since
\[ \big| \nabla^{g_\la} h_\la |_{g_\la} \leq C \la, \qquad | \nabla^{g_\la} (\eta \circ \mathbf{r}_\la) | \leq C, \]
we obtain
\begin{align*}
 \langle (\eta_\la \circ  \mathbf{r}) L h, \kappa \rangle_f
&= -\int_{(1,2) \times N} \Big( \la^{-2} ( \triangle_{g_\la} (\eta \circ  \mathbf{r}_\la) ) (h_\la \cdot_{g_\la} \hat\kappa_\la) + \tfrac12 \mathbf{r}_\la (\nabla^{g_\la} \mathbf{r}_\la \cdot_{g_\la} \nabla^{g_\la} (\eta \circ  \mathbf{r}_\la) ) (h_\la \cdot_{g_\la} \hat\kappa_\la) \\
&\hspace{80mm} + 2 \la^{-2} \nabla^{g_\la}_{\nabla^{g_\la} (\eta \circ  \mathbf{r}_\la)} h_\la \cdot_{g_\la} \hat\kappa_\la \Big) \mathbf{r}_\la^{-n} dg_\la \\
&\xrightarrow[\la \to \infty]{} -\frac12 \int_{(1,2) \times N}  r \, \eta'(r) \, (\gamma' \cdot \kappa_\infty) r^{-n} d\gamma =  \frac12 \int_{\{ 1 \} \times N} (\gamma' \cdot \kappa_\infty) d\gamma|_{\{ 1 \} \times N}.
\end{align*}
Combining this with \eqref{eq_kLh_lim} proves the claim.
\end{proof}

To prove Assertion~\ref{Lem_linear_identities_d}, note first that due to Claim~\ref{Cl_kLh_int} for $h \in K_{k+2,\nabla f}^\perp$ and for any $\kappa \in K$ we have $ \langle \kappa, Lh \rangle_f = 0$, so $Lh \in K_{k}^\perp$.
In addition, by definition $L$ restricted to $K_{k+2,\nabla f}^\perp$ is injective.
So since $L$ has Fredholm index $0$, we obtain that $L : K_{k+2,\nabla f}^\perp \to K_k^\perp$ is invertible with bounded inverse.

To see Assertion~\ref{Lem_linear_identities_e}, note first that $L (T(\gamma')) \in C^{k^*-5}_{-2}(M;S^2 T^*M)$ due to Lemmas~\ref{Lem_nabf_Tgamma}, \ref{Lem_Holder_norms_properties}, so $\proj_K (L (T(\gamma')))$ is defined.
To see the surjectivity, we argue by contradiction.
Suppose that there is a non-zero $\kappa \in K$ such that we have for any $\gamma' \in T\GenCONE^{k^*}(N)$, using Claim~\ref{Cl_kLh_int},
\begin{equation*} %
0 = \big\langle \kappa, \proj_K( L (T(\gamma'))) \big\rangle_f
= \langle \kappa,  L T(\gamma')) \rangle_f
=  \frac12 \int_{\{ 1 \} \times N} (\kappa_\infty \cdot \gamma') d\gamma|_{\{ 1 \} \times N}. 
\end{equation*}
This implies that $\kappa_\infty$ is of the following form
\begin{equation} \label{eq_kappa_inf_only_radial}
 \kappa_\infty = u \, dr^2, \qquad  u \in L^2 (N). 
\end{equation}

The following claim yields the desired contradiction.

\begin{Claim} \label{Cl_kappa_tangential}
If $\kappa_\infty$ is of the form \eqref{eq_kappa_inf_only_radial}, then $\kappa_\infty = 0$, hence $\kappa = 0$ by Claim~\ref{Cl_DS_kappa_bound}.
\end{Claim}

\begin{proof}
By 
Lemma~\ref{Lem_gauged_div_equation}) we have
\[ \triangle_f \DIV_f \kappa - \tfrac12 \DIV_f \kappa = 0. \]
Thus
\[ \triangle_f |\DIV_f \kappa|^2 - |\DIV_f \kappa |^2 \geq 0, \]
and an application of the maximum principle implies
\begin{equation} \label{eq_DIF_kappa_0}
 \DIV_f \kappa \equiv 0. 
\end{equation}
Our goal will be to pass this identity to the limit and to deduce a divergence-free condition on $\kappa_\infty$ in the weak sense.

Fix some compactly supported vector field of the form $Y = Y_0  \partial_r$ on $\IR_+ \times N$ for $Y_0 \in C^\infty_c (\IR_+ \times N)$.
For large $\lambda$ we can find a compactly supported vector field $Y_\lambda$ on $M$ such that $\omega_\lambda^* \iota^* Y_\lambda = Y$.
Due to \eqref{eq_DIF_kappa_0}, \eqref{eq_L2_loc_to_k_inf} and the following $C^1_{\loc}$-convergences for $\la \to \infty$
\[  g_\lambda := \lambda^{-2} \omega_\lambda^* \iota^* g \lto \gamma, \qquad \lambda^{-2} \omega_\lambda^* \iota^* (\LL_{Y_\lambda} g) \lto \LL_Y \gamma, \qquad \mathbf{r}_\lambda := \lambda^{-1} (\mathbf{r} \circ \iota \circ \omega_\lambda) \lto r, \]
it follows that
\begin{multline*} %
 \int_{\IR_+ \times N} \big( (\LL_Y \gamma) \cdot_{\gamma} \kappa_\infty \big) r^{-n}\, d\gamma 
 = \lim_{\lambda \to \infty} \int_{\IR_+ \times N} \big( \lambda^{-2} \omega_\lambda^* \iota^* (\LL_{Y_\lambda} g) \cdot_{g_\lambda}   \hat\kappa_\lambda \big)  \mathbf{r}_\lambda^{-n}\, dg_\lambda \\
= \lim_{\lambda \to \infty} \int_M (\LL_{Y_\lambda} g) \cdot   e^{\mathbf{r}^2/4} \kappa \, dg
= \lim_{\lambda \to \infty} \int_M  2Y_{\lambda} \cdot  (\DIV_f \kappa) e^{-f} dg = 0. 
\end{multline*}
Since $(\LL_{Y} \gamma)(\partial_r, \partial_r) = 2\partial_r Y_0$, it follows that
\[ 0 = \int_0^\infty \int_N 2 \partial_r Y_0 (r,z) u(z) r^{-1} d\gamma|_{\{1\} \times N}(z) dr
= \int_N \int_0^\infty 2Y_0 (r,z) u(z) r^{-2} dr d\gamma|_{\{1\} \times N}(z)   \]
for any $Y_0 \in C^\infty_c (\IR_+ \times N)$.
This implies that $u = 0$ in $L^2$.
\end{proof}

This finishes the proof.
\end{proof}
\bigskip

\subsection{Local description via the Implicit Function Theorem}
Our next goal will be to apply the Implicit Function Theorem to the expanding soliton equation in the DeTurck gauge.
As a result, we will obtain a local characterization of the space of expanding solitons in the DeTurck gauge in a neighborhood of any asymptotically conical \emph{gradient} expanding soliton.
The following proposition can be viewed as the analog of \cite[3.2]{White_1987}.

\begin{Proposition}\label{prop_smooth_dependence}
Let $(M,N,\iota)$ be an ensemble, $2 \leq k \leq k^*-20$ be integers, $\alpha \in (0,1)$ and fix some $C^{k^*-2}$-regular representative $(g,V = \nabla^g f,\gamma)$ of an element of $\MM_{\grad}^{k^*} (M,N,\iota)$.
We will use the gradient soliton $(M,g,f)$ to define the weighted H\"older spaces as in Subsection~\ref{subsec_list_Holder}.
Denote by $K \subset C^{2,\alpha}_{-2} (M; \lb S^2 T^*M)$ the kernel of the Einstein operator $L = L_g = \triangle_f + 2 \Rm$ (see also Lemma~\ref{Lem_linear_identities}).
Then there is an open neighborhood of $(\gamma,0)$
\[ U = U_1 \times U_2 \subset \GenCONE^{k^*}(N) \times K \]
and there are real-analytic maps
\[ F : U \longrightarrow C^{k,\alpha}_{-2,\nabla f}(M;S^2 T^*M), \qquad G : U \longrightarrow K \]
such that the following is true for any $(\gamma', \kappa') \in U$:
\begin{enumerate}[label=(\alph*)]
\item \label{prop_smooth_dependence_a} $F(\gamma,0) = 0$.
\item \label{prop_smooth_dependence_b} Using the Ricci DeTurck operator $Q_g$ from \eqref{eq_def_Qg}, we have 
\[ G(\gamma', \kappa') = Q_g(g + F(\gamma', \kappa') + T(\gamma'-\gamma) ) \in K. \]
\item \label{prop_smooth_dependence_c} $(D_2 F)_{(\gamma,0)} : K \to C^{k,\alpha}_{-2,\nabla f}(M;S^2 T^*M)$ is the inclusion map.
\item \label{prop_smooth_dependence_d} For any $\gamma'' \in T\GenCONE^{k^*}(N)$ we have $L (T(\gamma'')) \in C^{k-2,\alpha}_{-2}(M;S^2 T^*M)$ and
\[ (D_1 G)_{(\gamma,0)}(\gamma'') = \proj_K (L(T(\gamma''))), \]
as in Lemma~\ref{Lem_linear_identities}\ref{Lem_linear_identities_e}.
Moreover, the derivative $(D_1 G)_{(\gamma',\kappa')}$ has full rank for all $(\gamma', \kappa') \in U$ and $(D_2 G)_{(\gamma,0)} = 0$.
So $G^{-1}(0) \subset U$ is a real-analytic Banach submanifold of codimension $\dim K$.
It passes through the point $(\gamma, 0)$ and its tangent space at this point contains $\{0 \} \times K$.
\item \label{prop_smooth_dependence_e} There is a neighborhood of the origin $W \subset C^{3}_{-1} (M;S^2T^*M)$ such that the following is true.
If $(\gamma',h) \in U_1 \times W$ satisfies $Q_g (g+h+ T(\gamma'-\gamma)) = 0$, then there is a $\kappa' \in U_2$ such that $h = F(\gamma',\kappa')$.
A similar statement holds for families of solutions.
More specifically, if $U' \subset X$ is an open subset of a Banach space and $H_1 : U' \to U_1$, $H_3 : U' \to W$ are $C^{m,\beta}$-regular $(0 \leq m \leq \infty, \beta \in [0,1))$ maps with the property that $Q_g(g + H_3(x) + T(H_1(x)-\gamma)) = 0$ for all $x \in U'$, then there is a $C^{m,\beta}$-regular map $H_2 : U' \to U_2$ such that for all $x \in U'$ we have
\[  F(H_1(x),H_2(x)) = H_3(x), \qquad G(H_1(x), H_2(x)) = 0. \]
So the image of $H_3$ even lies in $W \cap C^{k,\alpha}_{-2,\nabla f}(M;S^2T^*M)$ and we even have $C^{m,\beta}$-regularity of $H_3$ with respect to the $C^{k,\alpha}_{-2,\nabla f}$-norm in the codomain.
\end{enumerate}
\end{Proposition}

\begin{proof}
We first establish the following claim.

\begin{Claim} \label{eq_Q_td_real_analytic}
There is an open neighborhood 
\[ \td U \subset  C^{k,\alpha}_{-2,\nabla f} (M; S^2 T^* M) \times \GenCONE^{k^*}(N)  . \]
of $(0,\gamma)$ such that
\[ \td Q :  \td U  \longrightarrow C^{k-2,\alpha}_{-2}(M; S^2 T^*M) \qquad (h',\gamma') \longmapsto Q_g(g+h'+T(\gamma'-\gamma)), \]
defines a real-analytic map between Banach spaces.
\end{Claim}

\begin{proof}
First recall from \eqref{eq_Q_expansion_intro} that
\begin{multline} \label{eq_Q_expansion}
 Q_g(g') = - \nabla_{\nabla f} (g'-g) + \Ric * (g'-g) \\ + (g')^{-1} * \nabla^2 g'  + g' * \nabla^2 g' + (g')^{-1} * (g')^{-1} * \nabla g' * \nabla g'  + \nabla g' * \nabla g',
\end{multline}
where all covariant derivatives are taken with respect to $g$, $(g')^{-1}$ denotes the inverse of $g'$ viewed as a $(2,0)$-tensor and ``$*$'' denotes a multilinear contraction of the listed tensors, possibly after raising or lowering indices with respect to $g$.

Next note that, due to Lemma~\ref{Lem_Holder_norms_properties}, tensorial contractions of the following form are real-analytic:
\begin{multline*}
 C^{k-2,\alpha} (M; T^{a_1}_{b_1} M) \times \ldots \times C^{k-2,\alpha} (M; T^{a_l}_{b_l} M) \times C^{k-2,\alpha}_{-2} (M; T^{a_{l+1}}_{b_{l+1}} M) \longrightarrow C^{k-2,\alpha}_{-2} (M; T^{a_0}_{b_0} M M), \qquad \\
(h_1, \ldots, h_l, h_{l+1}) \longmapsto h_1 * \ldots * h_l * h_{l+1}, 
\end{multline*}
\begin{multline*}
 C^{k-2,\alpha} (M; T^{a_1}_{b_1} M) \times \ldots \times C^{k-2,\alpha}_{-1} (M; T^{a_{l+1}}_{b_{l+1}} M) \times C^{k-2,\alpha}_{-1} (M; T^{a_{l+2}}_{b_{l+2}} M) \longrightarrow C^{k-2,\alpha}_{-2} (M; T^{a_0}_{b_0} M M), \qquad \\
(h_1, \ldots, h_{l+1}, h_{l+2}) \longmapsto h_1 * \ldots * h_{l+1} * h_{l+2}. 
\end{multline*}
So, in particular, the map
\begin{multline*}
 C^{k,\alpha}(M; S^2 T^* M) \times C^{k,\alpha}(M; S^2 TM) \lto C^{k,\alpha}(M; S^2 T^*M) \times C^{k,\alpha}(M; T^1_1 TM), \\
 (g',g'') \longmapsto g' * g'',
\end{multline*}
where $(g'*g'')_i^j = g'_{il} (g'')_{lj}$, is real-analytic.
Since its differential at $(g,g^{-1})$ is invertible, it is locally invertible and its inverse is of the form $(g',\delta) \mapsto (g', (g')^{-1})$, where $\delta \in C^{k,\alpha}(M; T^1_1 M)$ denotes the identity section with components $\delta_i^j$.
So the map $g' \mapsto (g')^{-1} \in C^{k,\alpha}(M; S^2 T^* M)$ is defined in a neighborhood of $g \in C^{k,\alpha}(M; S^2 T^* M)$ and real-analytic.

The claim can therefore be reduced to showing that the following linear operators are bounded:
\begin{alignat*}{2}
 C^{k,\alpha}_{-2,\nabla f} (M; S^2 T^* M) \times T\GenCONE^{k^*}(N)  &\to  C^{k-2,\alpha}_{-2} (M; S^2 T^*M), &\qquad (h',\gamma') &\mapsto \nabla_{\nabla f} (h'+T(\gamma')), \\
 C^{k,\alpha}_{-2,\nabla f} (M; S^2 T^* M) \times T\GenCONE^{k^*}(N)  &\to  C^{k,\alpha}(M; S^2 T^*M), &\qquad (h',\gamma') &\mapsto g+h'+T(\gamma'),  \\
 C^{k-2,\alpha}(M; S^2 T^*M) &\to C^{k-2,\alpha}_{-2}(M; S^2 T^*M), &\qquad h' &\mapsto \Ric * h', \\
 C^{k,\alpha}_{-2,\nabla f} (M; S^2 T^* M) \times T\GenCONE^{k^*}(N)  &\to  C^{k-2,\alpha}_{-1} (M; S^2 T^*M), &\qquad (h',\gamma') &\mapsto \nabla (h'+T(\gamma')), \\ 
 C^{k,\alpha}_{-2,\nabla f} (M; S^2 T^* M) \times T\GenCONE^{k^*}(N) &\to  C^{k-2,\alpha}_{-2} (M; S^2 T^*M), &\qquad (h',\gamma',\kappa') &\mapsto \nabla^2 (h'+T(\gamma')).
\end{alignat*}
This follows using Lemmas~\ref{Lem_nabf_Tgamma}, \ref{Lem_Holder_norms_properties}.
\end{proof}

Using \eqref{eq_DQg}, we find that the differential of $\td Q$ at $(\gamma,0)$ is given by
\begin{multline} \label{eq_DtdQ}
 (D \td Q)_{(\gamma,0)} : C^{2,\alpha}_{-2,\nabla f} (M; S^2 T^* M) \times T\GenCONE^{k^*}(N)  \longrightarrow C^{0,\alpha}_{-2} (M; S^2 T^* M), \qquad \\
(h',\gamma') \longmapsto L ( h' + T(\gamma')), 
\end{multline}
Where $L h := L_g h = \triangle_f h + 2 \Rm(h)$.

Consider now the splittings
\[ C^{k-2,\alpha}_{-2} (M;S^2 T^*M)  = K \oplus K_{k-2}^\perp, \qquad
C^{k,\alpha}_{-2,\nabla f} (M;S^2 T^*M)  = K \oplus K_{k,\nabla f}^\perp \]
from Lemma~\ref{Lem_linear_identities}, together with the projection operators $\proj_K, \proj_{K^\perp_{k-2}}$, $\proj_{K_{k,\nabla f}^\perp}$, and define the real-analytic map
\begin{align*}
 \Delta : \GenCONE^{k^*}(N) \times K \times K^\perp_{k,\nabla f} \longrightarrow \GenCONE^{k^*}(N) \times K \times K^\perp_{k-2},
\\
 (\gamma', \kappa', h') \longmapsto (\gamma',\kappa', \proj_{K^\perp_{k-2}} \big(\td Q (h'+\kappa',\gamma')) \big).
\end{align*}
 Its differential at $(\gamma, 0,0)$ is given by 
\begin{equation} \label{eq_derivative_Delta}
  (\gamma', \kappa', h') \longmapsto \big(\gamma', \kappa', \proj_{K_{k-2}^\perp} ( L (h'+T(\gamma') )) \big). 
\end{equation}
 Since $K^\perp_{k,\nabla f} \to K_{k-2}^\perp$, $h \mapsto \proj_{K_{k-2}^\perp} ( L  h) = L h$ is invertible by Lemma~\ref{Lem_linear_identities}\ref{Lem_linear_identities_d}, this implies that $\Delta$ is invertible near the origin.
 So we can find open neighborhoods:
 \[ (\gamma,0,0) \in U_1 \times U_2 \times U_3 \subset  \GenCONE^{k^*}(N) \times K \times K^\perp_{k-2} \]
 \[  (\gamma,0,0) \in U_1 \times U_2 \times V \subset \GenCONE^{k^*}(N) \times K \times K^\perp_{k,\nabla f} \]
 such that $\Delta |_{U_1 \times U_2 \times V} : U_1 \times U_2 \times V \to U_1 \times U_2 \times U_3$ is invertible and its inverse is of the form
 \[ (\Delta |_{U_1 \times U_2 \times V})^{-1} : U_1 \times U_2 \times U_3 \longrightarrow U_1 \times U_2 \times V, \qquad (\gamma', \kappa', h) \longmapsto (\gamma', \kappa',  F_0(\gamma', \kappa', h)) \]
 for some real-analytic $F_0 : U_1 \times U_2 \times U_3 \to V$ with $F_0 (\gamma, 0,0) = 0$.
 Set
 \begin{alignat*}{2}
  F : U = U_1 \times U_2 &\longrightarrow C^{k,\alpha}_{-2,\nabla f} (M;S^2 T^*M), &\qquad (\gamma',\kappa') &\longmapsto F_0 (\gamma', \kappa',0)+\kappa', \\
  G :  U = U_1 \times U_2 &\longrightarrow K, &\qquad (\gamma',\kappa') &\longmapsto \proj_K \big( \td Q(F(\gamma',\kappa') ,\gamma') \big) . 
\end{alignat*}
Then, by construction,
\[ \proj_{K^\perp_{k-2}} \big(\td Q (F(\gamma',\kappa'),\gamma') \big) 
=\proj_{K^\perp_{k-2}} \big(\td Q (F_0(\gamma',\kappa',0)+\kappa',\gamma') \big) 
=0 \]
and Assertions \ref{prop_smooth_dependence_a}, \ref{prop_smooth_dependence_b} follow immediately.
Assertion~\ref{prop_smooth_dependence_c} follows from the fact that $(D_2F_0)_{(\gamma,0,0)} = 0$, which follows from inverting \eqref{eq_derivative_Delta}.
For Assertion~\ref{prop_smooth_dependence_d}, we compute, using~\eqref{eq_DtdQ}, that for any $\gamma' \in T\GenCONE^{k^*}(N)$
\begin{multline*}
 (D_1 G)_{(\gamma,0)} (\gamma') = \proj_{K} \big( L \big( (D_1 F)_{(\gamma,0)} (\gamma') + T(\gamma')\big) \big) \\
= \proj_{K} \big( L \big( (D_1 F_0)_{(\gamma,0)} (\gamma') + T(\gamma')\big) \big) = \proj_{K} \big(L(T(\gamma')) \big).  
\end{multline*}
The last equality holds since the codomain of $F_0$ is $K_{k,\nabla f}^\perp$ and $L( K_{k,\nabla f}^\perp ) \subset K_{k-2}^\perp$ by Lemma~\ref{Lem_linear_identities}\ref{Lem_linear_identities_d}.
Lemma~\ref{Lem_linear_identities}\ref{Lem_linear_identities_e} implies that $(D_1G)_{\gamma,0}$ is surjective.
So after possibly shrinking $U_1, U_2$, we obtain that $(D_1G)_{(\gamma', \kappa')}$ is surjective for all $(\gamma', \kappa') \in U_1 \times U_2$.
We also verify using \eqref{eq_derivative_Delta} that for any $\kappa' \in K$ we have
\[ (D_2 G)_{(\gamma,0)} (\kappa') = \proj_K \big( L \big( (D_2 F)_{(\gamma,0)} (\kappa') \big) \big)
= \proj_K \big( L \big( (D_2 F_0)_{(\gamma,0)} (\kappa') + \kappa' \big) \big) = 0, \]
which finishes the proof of Assertion~\ref{prop_smooth_dependence_d}.

For Assertion~\ref{prop_smooth_dependence_e} let $W \subset C^3_{-1}(M; S^2 T^*M) \subset C^{2,\alpha}_{-1}(M; S^2 T^*M)$ be a small neighborhood of the origin, which we will determine in the course of the proof.
Suppose first that for some $h \in W $ we have
\begin{equation} \label{eq_Qgghtgpg}
   Q_g (g+h+T(\gamma'-\gamma))= 0. 
\end{equation}
Our goal will be to show that, in fact, $h \in C^{2,\alpha}_{-2,\nabla f}(M; S^2 T^*M)$, after possibly shrinking $U_1$.
To do this, we set $\gamma'' := \gamma' - \gamma$ and expand \eqref{eq_Q_expansion} as follows:
\begin{align}
 \triangle_f h &= - \triangle T(\gamma'') + \nabla_{\nabla f} T(\gamma'') + \Ric * h + \Ric * T(\gamma'')  \notag \\&\qquad + \big( (g')^{-1} - g^{-1} \big) * \nabla^2 (h + T(\gamma''))  + \big( h + T(\gamma'') \big) * \nabla^2 (h + T(\gamma'')) \notag \\ 
 &\qquad + (g')^{-1} * (g')^{-1} * \nabla h * \nabla h 
+ (g')^{-1} * (g')^{-1} * \nabla (T(\gamma'')) * \nabla (T(\gamma'')) \notag \\
&\qquad + (g')^{-1} * (g')^{-1} * h* \nabla (T(\gamma''))
+  \nabla h * \nabla h \notag \\
&\qquad +  \nabla (T(\gamma'')) * \nabla (T(\gamma''))
+  h* \nabla T(\gamma'') . \label{eq_lapfh_form}
\end{align}
Using the fact that
\[ (g')^{-1} - g^{-1} = - (g')^{-1} * (g' - g) * g^{-1} = - (g')^{-1} * (h + T(\gamma'')) * g^{-1}, \]
and the arguments from the proof of Claim~\ref{eq_Q_td_real_analytic}, this allows us to express \eqref{eq_lapfh_form} in the abbreviated form
\begin{equation} \label{eq_lapf_RS}
  \triangle_f h = R(\gamma'',h) + S(\gamma'',h), 
\end{equation}
where
\[ R : U_1 \times W' \longrightarrow C^{0,\alpha}_{-2} (M ; S^2 T^*M) \]
is real analytic
and
\[ S(\gamma'', h) = (g + h + T(\gamma''))^{-1} * T(\gamma'') * g^{-1} * \nabla^2 h + T(\gamma'') * \nabla^2 h \]
is linear in the second argument and real-analytic as a map of the following two forms:
\begin{alignat*}{1}
 S : U_1 \times W' &\lto C^{0,\alpha}_{-1}(M; S^2 T^*M), \\
 S : U_1 \times C^{2,\alpha}_{-2}(M; S^2 T^*M) &\lto C^{0,\alpha}_{-2}(M; S^2 T^*M). 
\end{alignat*}
By Proposition~\ref{Prop_Lundardi_weighted} (here we briefly use the smooth structure on $M$ supplied by Lemma~\ref{Lem_soliton_smooth}), the operator $\triangle_f$ is invertible as a map of both of the following forms:
\begin{alignat}{2}
 \triangle_f &: &C^{2,\alpha}_{-1,\nabla f}(M; S^2 T^*M) &\to C^{0,\alpha}_{-1}(M; S^2 T^*M), \label{eq_lap_f_inv_1} \\
\triangle_f &: & C^{2,\alpha}_{-2,\nabla f}(M; S^2 T^*M) &\to C^{0,\alpha}_{-2}(M; S^2 T^*M). \label{eq_lap_f_inv_2}
\end{alignat}
Consider now the real-analytic map
\begin{multline} \label{eq_U1WC}
U_1 \times W \times C^{2,\alpha}_{-2,\nabla f}(M; S^2 T^*M) \lto U_1 \times W \times C^{0,\alpha}_{-2}(M; S^2 T^*M) \\
 (\gamma'',h,h') \longmapsto  \big( \gamma'', h, R(\gamma'',h)  +  S(\gamma'',h')  - \triangle_f h'  \big). 
\end{multline}
Since its differential at $(0,0,0)$ is invertible, it has a real-analytic inverse near the origin (here we have used \eqref{eq_lap_f_inv_2}), which maps $(\gamma'',h,0)$ to an element of the form $(\gamma'',h,h')$ with
\begin{equation} \label{eq_lapf_RS_p}
   \triangle_f h' = R(\gamma'',h) + S(\gamma'',h'), 
\end{equation}
We claim that $h' = h$ if $\gamma''$ is sufficiently close to the origin.
To see this, we combine \eqref{eq_lapf_RS_p} with \eqref{eq_lapf_RS} to conclude
\[ \triangle_f (h'-h) = S(\gamma'', h'-h). \]
Using the invertibility of $\triangle_f$ as an operator of the form \eqref{eq_lap_f_inv_1}, this implies
\[ \Vert h' - h \Vert_{C^{2,\alpha}_{-1,\nabla f} }
\leq C \Vert S(\gamma'',h'-h) \Vert_{C^{0,\alpha}_{-1} }
\leq C \Vert \gamma'' \Vert \, \Vert h' -h \Vert_{C^{2,\alpha}_{-1,\nabla f} }. \]
So for sufficiently small $\Vert \gamma'' \Vert$, which can be enforced by shrinking $U_1$, we must have $h' = h$.
In addition, 

In summary, we have shown that after shrinking $U_1$ and choosing $W$ appropriately small, any solution $h \in W$ of \eqref{eq_Qgghtgpg} satisfies $C^{2,\alpha}_{-2,\nabla f}(M, \lb S^2 T^*M)$.
Moreover, $h$, as an element of $C^{2,\alpha}_{-2,\nabla f}(M, \lb S^2 T^*M)$, can be recovered as the image of the inverse of the real-analytic map \eqref{eq_U1WC}. 
In addition, our construction of the maps $F$ and $G$, involving the map $\Delta$, in the case $k = 2$ then implies that $h = F(\gamma',\kappa')$ for $\kappa' = \proj_K (h) \in U_2$.
Since the identity $h = F(\gamma', \kappa')$ is independent of $k$, we obtain the same statement for $k > 2$, after possibly shrinking $U_1$ and $W$.

Let us now consider the more general setting of Assertion~\ref{prop_smooth_dependence_d}, involving the maps $H_1$ and $H_3$.
Our previous discussion implies that, in fact $H_3$ is also a $C^{m,\beta}$-regular map of the form $H_3:U' \to C^{2,\alpha}_{-2,\nabla f}(M, S^2 T^*M)$.
As explained in the previous paragraph, we can moreover choose $H_2 := \proj_K \circ H_3$, which inherits the regularity $C^{m,\beta}$.
Lastly, since the identity $H_3(x) = F(H_1(x), H_2(x))$ is independent of $k$, we obtain that $H_3$ is even $C^{m,\beta}$-regular as a map of the form $H_3:U' \to C^{k,\alpha}_{-2,\nabla f}(M, S^2 T^*M)$.
This finishes the proof of Assertion~\ref{prop_smooth_dependence_d}.
\end{proof}

\bigskip

\subsection{Local structure of $\MM$}
We will now use the local characterization from the previous subsection and deduce a local characterization of the space $\MM$ near $\MMgrad$.
In doing so, we will obtain the desired $C^{1,\alpha}$-Banach manifold structure on a neighborhood of $\MMgrad$ in $\MM$, see Corollary~\ref{Cor_MM_is_Banach_manifold} below.
The topology of this Banach manifold structure agrees with the topology induced by the metric on $\MM$, as defined in Subsection~\ref{subsec_metric_on_MM}.

\begin{Proposition} \label{Prop_Banach_manifold}
Let $(M,N,\iota)$ be an ensemble, $4 \leq k \leq k^*-20$ be an integer, $\alpha \in (0,1)$ and write $\MM := \MM^{k^*} (M,N,\iota)$.
Fix a $C^{k^*-2}$-regular representative $(g,V=\nabla^g f,\gamma)$ of an element $p \in \MMgrad := \MM_{\grad}^{k^*}(M,N,\iota)$ and denote by $K_p = K$ the kernel of $L_g$ as in Proposition~\ref{prop_smooth_dependence}.
Then there is a Banach space $X_p$, an open neighborhood of the origin $U_p \subset X_p$ and a map
\[ \Phi_p : U_p \longrightarrow \MM \]
with the following properties (as in Proposition~\ref{prop_smooth_dependence}, all weighted H\"older norms are taken with respect to the gradient soliton $(M,g,f)$, unless mentioned otherwise):
\begin{enumerate}[label=(\alph*)]
\item \label{Prop_Banach_manifold_a} $\Phi_p(0) = p$.
\item \label{Prop_Banach_manifold_b} $\Phi_p(U_p) \subset \MM$ is open and $\Phi_p : U_p \to \Phi_p(U_p)$ is a homeomorphism.
\item \label{Prop_Banach_manifold_c} The map $\Pi \circ \Phi_p : U_p \to \GenCONE^{k^*}(N)$ is real-analytic and its differential at the origin has Fredholm index $0$ and nullity equal to the nullity of $L_g$.
\item \label{Prop_Banach_manifold_d} There is a real-analytic map $H_p : U_p \to C^{k,\alpha}_{-2,\nabla f}(M; S^2 T^*M)$ such that for any $x \in U_p$ the image $\Phi_p(x)$ is represented by $(\psi_x^* g_x, \psi_x^* V_x, \gamma_x)$, where $\psi_x : M \to M$ is a $C^2$-diffeomorphism of finite displacement and
\[ g_x := g + H_p(x) + T(\Pi(\Phi_p(x))) -T(\gamma), \qquad V_x = \nabla^g f - \DIV_g (g_x) + \tfrac12 \nabla^g \tr_g (g_x). \]
Moreover, the metric $g_x$ satisfies $Q_g (g_x) = 0$ for all $x \in U_p$.
\item \label{Prop_Banach_manifold_e} Suppose that $\td U \subset \td X$ is an open subset of a Banach space $\td X$ and $\td\Phi : \td U \to \Phi_p(U_p) \subset \MM$ is a continuous map such that the following is true:
\begin{itemize}
\item The map $\Pi \circ \td\Phi : \td U \to \GenCONE^{k^*}(N)$ is $C^{1,\alpha}$; set $\td\gamma_{\td x} := \Pi (\td\Phi(\td x))$ for all $\td x \in \td U$.
\item There is a $C^{k^*-2}$-regular representative $(\td g, \td V = \nabla^{\td g} \td f, \td \gamma )$ of an element of $\MMgrad$ and a $C^{1,\alpha}$-regular map 
\[ \td H : \td U \longrightarrow C^{k,\alpha}_{ \td g, -2}(M; S^2 T^*M) \] (where the $\td g$-subscript indicates that the H\"older space is taken with respect to the gradient soliton $(M, \td g, \td f)$) such that for any $\td x \in \td U$ the image $\td\Phi(\td x)$ is represented by $(\td\psi_{\td x}^* \td g_{\td x}, \td\psi_{\td x}^* \td V_{\td x}, \td\gamma_{\td x})$, where $\td\psi_{\td x} : M \to M$ is a $C^2$-diffeomorphism of finite displacement and
\[\qquad\qquad  \td g_{\td x} = \td g + \td H(\td x) + T(\Pi(\td\Phi(\td x))) - T(\td\gamma), \qquad \td V_{\td x} = \nabla^{\td g} \td f - \DIV_{\td g} (\td g_{\td x}) + \tfrac12 \nabla^{\td g} \tr_{\td g} (\td g_{\td x}).  \]
\end{itemize}
Then the map $\Phi_p^{-1} \circ \td\Phi : \td U \to X_p$ is of regularity $C^{1,\alpha}$.
Moreover, if $\td\Phi(\td x_0) = \Phi_p(0)$ for some $\td x_0 \in \td U$, then the differential of $\Phi_p^{-1} \circ \td\Phi$ at $\td x_0$ can be characterized as follows.
Suppose that we have $(M,\td\psi_{\td x_0}^* \td g_{\td x_0},\td\psi_{\td x_0}^* \td V_{\td x_0}) = (M,g,V)$, let $\td v \in \td X$ and set
\[ v := d(\Phi_p^{-1} \circ \td\Phi)_{\td x_0} (\td v), \qquad \dot\gamma := d(\Pi\circ\td\Phi)_{\td x_0} (\td v) \]
and
\begin{equation} \label{eq_infinitesimal_version}
 \dot g := (d H_p)_{0}(v) + T(\dot\gamma), \qquad 
\dot{\td g} := (d\td H_p)_{\td x_0}(\td v) + T(\dot\gamma). 
\end{equation}
Then we have
\[ \dot g = \td\psi_{\td x_0}^*  \dot{\td g} + \LL_Y g, \]
where $Y \in C^2_{g,-1,\nabla f}(M; TM)$ is the unique vector field satisfying
\[ \qquad\qquad \triangle_{g,f} Y - \tfrac12 Y = \DIV_{\td\psi_{\td x_0}^* \td g} (\td\psi_{\td x_0}^* \dot{\td g}) - \tfrac12 \nabla^{\td\psi_{\td x_0}^* \td g} \tr_{\td\psi_{\td x_0}^* \td g}  (\td\psi_{\td x_0}^* \dot{\td g}) - \DIV_g (\td\psi_{\td x_0}^* \dot{\td g}) + \tfrac12 \nabla^g \tr_g (\td\psi_{\td x_0}^* \dot{\td g}). \]
\item \label{Prop_Banach_manifold_f} Consider the setting of Proposition~\ref{prop_smooth_dependence}.
There is a real-analytic map $\Xi_p : U_p \to K$ such that the map
\[ U_p \to \GenCONE^{k^*} (N) \times K, \qquad x \mapsto (\Pi(\Phi_p(x)), \Xi_p(x)) \]
has image in $\{ G = 0 \} \subset U_1 \times U_2$ and is a diffeomorphism onto its image between Banach manifolds.
Moreover, for all $x \in U_p$ we have
\[ H_p (x) = F(\Pi(\Phi_p(x)), \Xi_p(x)). \]
\end{enumerate}
\end{Proposition}

As a consequence, we obtain

\begin{Corollary} \label{Cor_MM_is_Banach_manifold}
Let $(M,N,\iota)$ be an ensemble, $2 \leq k \leq k^*-20$ be an integer, $\alpha \in (0,1)$ and write $\MM := \MM^{k^*} (M,N,\iota)$, $\MMgrad := \MMgrad^{k^*} (M,N,\iota)$.
Then the inverses of the maps
$\{ \Phi_p : U_p \to \MM \}_{p \in \MMgrad}$
form a $C^{1,\alpha}$-Banach manifold atlas on the neighborhood $$\MMgrad \subset \MM':= \bigcup_{p \in \MMgrad}\Phi_p(U_p) \subset \MM$$ that is compatible with the subspace topology induced by the metric $\MM$ (see Subsection~\ref{subsec_metric_on_MM}).
Moreover, the map $\Pi : \MM' \to \GenCONE^{k^*}(N)$ is of regularity $C^{1,\alpha}$ with respect to this atlas and the maps $\Pi \circ \Phi_p$ are real-analytic.
\end{Corollary}

\begin{Remark}
The choice of $k$ and $\alpha$ is not essential.
An inspection of the proof of Proposition~\ref{prop_smooth_dependence} shows that different choices of $k$, $\alpha$ lead to Banach manifold atlases whose charts arise by restrictions to smaller subsets.
This fact is, however, not important in the context of this paper, so it suffices to imagine the parameters $k^*,k,\alpha$ as fixed.
\end{Remark}

\begin{proof}[Proof of Corollary~\ref{Cor_MM_is_Banach_manifold}.]
For any $p_1, p_2 \in \MM$ consider the transition map 
\[ \Phi_{p_2}^{-1} \circ \Phi_{p_1}|_{\Phi_{p_1}^{-1} (\Phi_{p_1}(U_{p_1}) \cap \Phi_{p_2} (U_{p_2}))} : \Phi_{p_1}^{-1} (\Phi_{p_1}(U_{p_1}) \cap \Phi_{p_2} (U_{p_2})) \longrightarrow \Phi_{p_2}^{-1} (\Phi_{p_1}(U_{p_1}) \cap \Phi_{p_2} (U_{p_2})) . \]
This map is of regularity $C^{1,\alpha}$ by Proposition~\ref{Prop_Banach_manifold}\ref{Prop_Banach_manifold_d} and \ref{Prop_Banach_manifold_e}.
\end{proof}
\medskip

\begin{proof}[Proof of Proposition~\ref{Prop_Banach_manifold}]
Let $F : U = U_1 \times U_2 \to C^{k,\alpha}_{-2,\nabla f}(M;S^2T^*M)$ and $G : U \to K$ be the real-analytic maps from Proposition~\ref{prop_smooth_dependence} applied to $(M,\nabla f, \gamma)$.
Let $X_p \subset T\GenCONE^{k^*}(N) \times K$ be the kernel of the differential $DG_{(\gamma,0)}$.
By Proposition~\ref{prop_smooth_dependence}\ref{prop_smooth_dependence_d} we have $X_p = X'_p \times K$ for some closed subspace $X'_p \subset T\GenCONE^{k^*}(N)$ whose codimension equals $\dim K$; let $X''_p \subset T\GenCONE^{k^*}(N)$ be a complementary subspace and write 
\begin{equation} \label{eq_GenCone_splitting}
\gamma = (\gamma'_0, \gamma''_0) \in \GenCONE^{k^*}(N) \subset X'_p \times X''_p.
\end{equation}
By the Implicit Function Theorem, we can find an open neighborhood of the origin $U_p \subset X_p = X'_p \times K$ and a real-analytic map $Y_p : U_p \to X''_p$ such that $Y_p(0,0) = \gamma''_0$ and $(D Y_p)_{(0,0)} = 0$ and such that for any $(\gamma',\kappa') \in U_p$ we have $G(\gamma'_0 + \gamma' + Y_p(\gamma',\kappa'),\kappa') = 0$.
Set
\[ H_p : U_p \to C^{k,\alpha}_{-2,\nabla f}(M; S^2 T^*M), \qquad x \mapsto F (\gamma'_0 + \gamma' + Y_p(\gamma',\kappa'),\kappa') \]
and
\[ h_{(\gamma',\kappa')} := F_p (\gamma'_0 + \gamma' + Y_p(\gamma',\kappa'),\kappa').  \]
So for any $(\gamma',\kappa') \in U_p$ we have
\begin{equation} \label{eq_Q_h_gp_kp}
 Q_g\big(g_x \big) = 0,
\end{equation}
where
\[ g_x := g + h_x + T(\gamma'+ Y_p(\gamma',\kappa')). \]
In the following we will often denote the elements of $X_p$ by $x$ and forget that these denote pairs of the form $(\gamma',\kappa')$.

The identity \eqref{eq_Q_h_gp_kp} implies that for each $x \in U_p$ the metric $g_x$ satisfies the expanding soliton equation
\[ \Ric_{g_x} + \tfrac12 \LL_{V_x} g_x + \tfrac12 g_x = 0, \]
where
\begin{equation} \label{eq_Vx_def}
 V_{x} := \nabla^g f - \DIV_g(g_{x}) + \tfrac12 \nabla^g \tr_g (g_{x}). 
\end{equation}
Note that this implies that the following map is smooth:
\[ U_p \to C^{k-1,\alpha}_{-2}(M; TM), \qquad x \mapsto V_x - V_0 = V_x- \nabla^g f. \]
We can now apply Proposition~\ref{Prop_gauge_infinity_families} to obtain the $C^k$-embeddings $(\iota_x : \IR_+ \times N \to M)_{x \in U_p}$ and a family of $C^k$-diffeomorphisms $(\psi_x : M \to M)_{x \in U_p}$, after possibly shrinking $U_p$, such that the assertions of this proposition hold for some uniform constant $R_0 < \infty$.
So
\[ (g'_x, V'_x, \gamma_x) := (\psi^*_x g_x, \psi^*_x V_x, \gamma_x) \]
represents an element $\Phi_p(x) \in \MM$ and the maps $\td\iota_x : \IR_+ \times N \to M$ from Lemma~\ref{Lem_MM_regular_rep}\ref{Lem_MM_regular_rep_c} satisfy
\begin{equation} \label{eq_iotax_tdiotax}
 \iota_x = \psi_x \circ \td\iota_x.  
\end{equation}
Since $\Pi ( \Phi_p (x)) = \gamma_{x}$, this implies Assertions~\ref{Prop_Banach_manifold_a}, \ref{Prop_Banach_manifold_c}, \ref{Prop_Banach_manifold_d}, \ref{Prop_Banach_manifold_f}.

Let us now prove Assertion~\ref{Prop_Banach_manifold_b}.

\begin{Claim} \label{Cl_Phi_injective}
After possibly shrinking $U_p$, the map $\Phi_p : U_p \to \MM$ is injective.
\end{Claim}

\begin{proof}
Suppose, by contradiction, that $\Phi(x_{1,j}) = \Phi(x_{2,j})$ for two sequences $x_{1,j}, x_{2,j} \in U_p$ with $x_{1,j} \neq x_{2,j}$ such that $x_{1,j}, x_{2,j} \to 0$.
So there is a sequence of $C^1$-diffeomorphisms $\chi'_j : M \to M$ that each equal the identity outside a compact subset such that $(\chi'_j)^* g'_{x_{2,j}} = g'_{x_{1,j}}$ and $(\chi'_j)^* V'_{x_{2,j}} = V'_{x_{1,j}}$.
Therefore, the $C^1$-diffeomorphisms
\[ \chi_j := \psi_{x_{2,j}} \circ \chi'_j \circ \psi_{x_{1,j}}^{-1} : M \longrightarrow M \]
satisfy
\begin{equation*} %
 \chi_j^* g_{x_{2,j}} = g_{x_{1,j}}, \qquad \chi_j^* V_{x_{2,j}} = V_{x_{1,j}} 
\end{equation*}
and on subsets of the form $(r_{0,j},\infty) \times N$, $r_{0,j} > R_0$, we have
\begin{equation} \label{eq_chi_iotax1x2}
 \chi_j \circ \iota_{x_{1,j}} = \chi_j \circ \psi_{x_{1,j}} \circ \iota = \psi_{x_{2,j}} \circ \chi'_j \circ \iota = \psi_{x_{2,j}} \circ \iota = \iota_{x_{2,j}}. 
\end{equation}
Since
\[ (\chi_j \circ \iota_{x_{1,j}})^* V_{x_{2,j}}
= \iota_{x_{1,j}}^* \chi_j^* V_{x_{2,j}} = \iota_{x_{1,j}}^* V_{x_{1,j}} = - \tfrac12 r \partial_r = \iota_{x_{2,j}}^* V_{x_{2,j}}, \]
we conclude that \eqref{eq_chi_iotax1x2} even holds on all of $\IR_+ \times N$.
Since the image of $\iota_{x_{i,j}}= \psi_{x_{i,j}} \circ \td\iota_{x_{i,j}}$ is dense (see Lemma~\ref{Lem_Im_iota_dense}) and since $g_{x_{1,j}}-g_{x_{2,j}} \to 0$ in $C^k_{\loc}$, we can use Assertion~\ref{Prop_gauge_infinity_families_c} of Proposition~\ref{Prop_gauge_infinity_families} to conclude that we have uniform convergence $\chi_j \to \id_M$.
Therefore, we may apply Proposition~\ref{Prop_gauge} for large $j$ and conclude that $\chi_j = \id_M$ since both metrics $g_{x_{i,j}}$ and vector fields $V_{x_{i,j}}$, $i=1,2$, are already in the DeTurck gauge due to \eqref{eq_Q_h_gp_kp} and \eqref{eq_Vx_def}.
Therefore $g_{1,j} = g_{2,j}$ for large $j$, which implies $x_{1,j} = x_{2,j}$ for large $j$ by construction.
\end{proof}

\begin{Claim} \label{Cl_continuous}
The map $\Phi_p : U_p \to \MM$ is continuous.
\end{Claim}

\begin{proof}
We need to compare the pullbacks of the length metrics of $g'_x$ via $\td\iota_x$, which equal the pullbacks of the length metrics of $g_x$ via $\iota_x$ due to \eqref{eq_iotax_tdiotax}.
These depend continuously on $x$ due to Assertion~\ref{Prop_gauge_infinity_families_c} of Proposition~\ref{Prop_gauge_infinity_families}.
\end{proof}

\begin{Claim} \label{Cl_U_star}
Let $U^* \subset U_p$ be a neighborhood of the origin.
Then $\Phi_p (U^*)$ is a neighborhood of $p$.
\end{Claim}

\begin{proof}
Suppose not, so we can find a sequence $p_i  \in \MM \setminus \Phi_p(U^*)$ with $p_i \to p$.
Then $\gamma_i := \Pi(p_i) \to \gamma$.
By Corollary~\ref{Cor_conv_weighted}, we can find $C^{k^*-5}$-regular representatives $(g_i,V_i,\gamma_i)$ of $p_i$ such that
\[ \big\Vert g_i - g - T(\gamma_i - \gamma) \big\Vert_{C^{k^*-5}_{-1}} \to 0, \qquad V_i \xrightarrow[i \to \infty]{\quad C^{k^*-5}_{\loc} \quad} V \]
and $V_i = V$ on $\iota((r_0,\infty) \times N)$ for some uniform $r_0 > 2$.
We can now apply Proposition~\ref{Prop_gauge} and obtain that for large $i$ there are $C^2$-diffeomorphisms $\chi_i : M \to M$ of finite displacement such that
\[ Q_g(\chi_i^* g_i) = 0,  \qquad 
\chi_i^* V_i = \nabla^g f - \DIV_g(\chi_i^* g_i ) + \tfrac12 \nabla^g \tr_g (\chi_i^* g) \]
and
\[ \big\Vert \chi_i^* g_i - g - T(\gamma_i - \gamma) \big\Vert_{C^{k,\alpha}_{-1}} \to 0. \]
So by Assertion~\ref{prop_smooth_dependence_e} of Proposition~\ref{prop_smooth_dependence} we have $\chi_i^* g_i = g_{x_i}$ and $\chi_i^* V_i = V_{x_i}$ for some $x_i \in X_p$ for large $i$.
Reviewing the construction of $g_{x_i}$ implies that $x_i \to 0$, so $x_i \in U^*$ for large $i$.

It remains to show that $p_i = \Phi_p(x_i)$ for large $i$.
To see this note first that
\[ (\chi_i \circ \psi_{x_i})^* g_i  = \psi_{x_i}^* g_{x_i} = g'_{x_i}, \qquad 
(\chi_i \circ \psi_{x_i})^* V_i  = \psi_{x_i}^* g_{x_i} = V'_{x_i}.  \]
Denote by $\td\iota_i : \IR_+ \times N \to M$ the map from Lemma~\ref{Lem_MM_regular_rep}\ref{Lem_MM_regular_rep_c} with $\td\iota_i^* V_i = - \tfrac12 r \partial_r$ and $\td\iota_i = \iota$ on $\iota((r_i, \infty) \times N)$ for some $r_i > R_0$.
Then
\[ ( \chi_i^{-1} \circ \td\iota_i )^* V_{x_i}
= - \tfrac12 r \partial_r. \]
So since $\chi_i^{-1}$ has finite displacement, we obtain, using Assertion~\ref{Prop_gauge_infinity_families_h} of Proposition~\ref{Prop_gauge_infinity_families}, that $\iota_{x_i} =  \chi_i^{-1} \circ \td\iota_i$.
Combining this with \eqref{eq_iotax_tdiotax} implies that on $(r'_i, \infty) \times N$, for large enough $r'_i > r_i$,
\[ (\chi_i \circ \psi_{x_i}) \circ \iota
= (\chi_i \circ \psi_{x_i}) \circ \td\iota_{x_i} = \chi_i \circ \iota_{x_i} = \td\iota_i = \iota.  \]
This implies that $\chi_i \circ \psi_{x_i} : M \to M$ equals the identity on the complement of some compact subset.
Therefore, $p_i = [(g_i,V_i, \gamma_i)] = [(g'_i, V'_i, \gamma_i)] = \Phi(x_i)$, which yields the desired contradiction.
\end{proof}

\begin{Claim}
Assertion~\ref{Prop_Banach_manifold_b} holds after possibly shrinking $U_p$.
\end{Claim}

\begin{proof}
Choose an open neighborhood of the origin $U'_p \subset U_p \subset X_p$ whose closure is contained in $U_p$.
By Claim~\ref{Cl_U_star} we know that there is an open neighborhood $p \in W' \subset \Phi_p(U'_p)$ and by Claim~\ref{Cl_continuous} we know that $\Phi_p^{-1}(W')$ is open.
So after replacing $U'_p$ by $\Phi_p^{-1}(W')$, we may assume that both $U'_p$ and $\Phi_p(U'_p)$ are open and that 
\begin{equation} \label{eq_closureUppinUp}
\ov{U}'_p \subset U_p.
\end{equation}

It remains to show that $\Phi_p |_{U'_p} : U'_p \to \Phi_p (U'_p)$ is a homeomorphism.
Consider a sequence $x_i \in U'_p$, $i \leq \infty$ such that $\Phi_p(x_i) \to \Phi_p(x_\infty)$ in $\MM$.
Recall that $x_i = (\gamma'_{x_i}, \kappa_{x_i})$, where $\gamma_{x_i} = (\gamma'_{x_i},\gamma''_{x_i})$ in the splitting \eqref{eq_GenCone_splitting} and $\kappa_{x_i} \in K$.
Since by construction
\[ \gamma_i =\Pi(\Phi_p(x_i)) \to \Pi(\Phi_p(x_\infty)) = \gamma_\infty = (\gamma'_\infty, \gamma''_\infty), \]
it remains to show that $\kappa_{x_i} \to \kappa_{x_\infty}$.
Suppose not.
Since $K$ is finite dimensional, we may pass to a subsequence such that $\kappa_{x_i} \to \kappa'_\infty \in K$, where $\kappa'_\infty \neq \kappa_{x_\infty}$.
Due to \eqref{eq_closureUppinUp} we have $x_i \to x'_\infty := (\gamma'_\infty,\kappa'_\infty) \in U_p$.
By Claim~\ref{Cl_continuous} we have $\Phi_p(x_i) \to \Phi_p(x'_\infty)$, so $\Phi_p(x'_\infty) = \Phi_p(x_\infty)$.
Claim~\ref{Cl_Phi_injective} implies that $x'_\infty = x_\infty$, giving us a contradiction.
\end{proof}

It remains to verify Assertion~\ref{Prop_Banach_manifold_e}.
Consider the maps $\td\Phi, \td H$ and the metric $\td g$ and write $\td h_{\td x} := \td H (\td x)$ for $\td x \in \td U$.
By Assertion~\ref{Prop_Banach_manifold_b}, we know that
\[ \Theta := \Phi_p^{-1} \circ \td\Phi : \td U \to U_p \]
is continuous.
Note that $\td\gamma_{\td x} = \gamma_{\Theta(\td x)}$.
Our goal will be to establish its $C^{1,\alpha}$-regularity.
For any $\td x \in \td U$ the metric and vector field representing $\td\Phi(\td x) = \Phi_p ( \Theta(\td x))$ is isometric to both $\td g_{\td x},\td V_{\td x}$ and $g_{\Theta(\td x)}, V_{\Theta(\td x)}$ via a $C^2$-diffeomorphism of finite displacement.
By composing one of these diffeomorphisms with the inverse of the other for each $\td x \in \td U$, we obtain a family of $C^2$-diffeomorphisms $(\chi_{\td x} : M \to M)_{\td x \in \td U}$ with finite displacement such that
\begin{equation} \label{eq_chi_g_V}
 \chi_{\td x}^* \td g_{\td x} = g_{\Theta(\td x)}, \qquad
\chi^*_{\td x} \td V_{\td x} = V_{\Theta(\td x)}. 
\end{equation}

Fix some $\td x_0 \in \td U$ and set $x_0 := \Theta(\td x_0) \in U_p$.
Note that we are not yet assuming that $x_0=0\in U_p$.
In the following we will establish the $C^{1,\alpha}$-regularity of $\Theta$ near $\td x_0$.
We will need the following technical statement.

\begin{Claim} \label{Cl_tdg0}
For any $\td x \in \td U$ we define $\td g^0_{\td x}, \td h^0_{\td x}, \td V^0_{\td x}$ via
\begin{equation} \label{eq_chi_H_bar_p}
  \chi_{\td x_0}^* \td g_{\td x} = \td g^0_{\td x} = g + \td h^0_{\td x} + T( \td\gamma_{\td x}) - T(\gamma), \qquad
  \chi_{\td x_0}^* \td V_{\td x} = \td V^0_{\td x}.
\end{equation}
Then
\begin{equation} \label{eq_H_bar_H}
\td{g}^0_{\td x_0} = g_{x_0}, \qquad  \td{h}^0_{\td x_0} = h_{x_0}, \qquad \td{V}^0_{\td x_0} = V_{x_0}.
\end{equation}
Moreover, the maps (note that the following norms are taken with respect to $(M,g,f)$)
\begin{alignat*}{2}
 \td U &\longrightarrow C^{k}_{g,-1}(M; S^2 T^*M), \qquad& \td x &\longmapsto \td h^0_{\td x}, \\
 \td U &\longrightarrow C^{k-1}_{g,-1}(M; TM), \qquad& \td x &\longmapsto \td V^0_{\td x} - V_{x_0} 
\end{alignat*}
take values in the respective codomains and are of regularity $C^{1,\alpha}$.
\end{Claim}

\begin{proof}
Identity \eqref{eq_H_bar_H} is a direct consequence of \eqref{eq_chi_g_V}.
Rearranging \eqref{eq_chi_H_bar_p} yields
\begin{align*}
 \td h^0_{\td x} &= \chi_{\td x_0}^* \big( \td g_{\td x} - \td g_{\td x_0} \big) + \chi_{\td x_0}^*  \td g_{\td x_0}
- g - T(\td\gamma_{\td x}) + T(\gamma) \\
 &= \chi_{\td x_0}^* \big( \td h_{\td x} - \td h_{\td x_0} \big) + \chi_{\td x_0}^* \big( T(\td\gamma_{\td x}) - T(\td\gamma_{\td x_0}) \big)  + g_{x_0}
- g - T(\td\gamma_{\td x}) + T(\gamma) \\
 &=   \chi_{\td x_0}^* \big( \td h_{\td x} - \td h_{\td x_0} \big) + \chi_{\td x_0}^* \big( T(\td\gamma_{\td x}) - T(\gamma_{x_0})  \big) 
 - \big(   
   T(\td\gamma_{\td x}) -  T(\gamma_{x_0}) \big)+ h_{x_0}.   
\end{align*}
Moreover,
\[ \td V^0_{\td x} - V_{x_0} = \chi^*_{\td x_0} \td V_{\td x} - \chi^*_{x_0} \td V_{\td x_0} = \chi^*_{\td x_0} \big( \td V_{\td x} - \td V_{\td x_0} \big). \]
So by composition of functions, we are left with showing that the following linear operators are bounded (note the different background metrics):
\begin{alignat}{2}
 C^{k}_{\td g,-1}(M; S^2 T^*M) &\longrightarrow C^{k}_{ g,-1}(M; S^2 T^*M), &\qquad  h &\longmapsto \chi_{\td x_0}^* h \label{eq_map_1} \\
 C^{k}_{\td g,-1}(M; TM) &\longrightarrow C^{k}_{ g,-1}(M;  TM), &\qquad  W &\longmapsto \chi_{\td x_0}^* W \label{eq_map_2} \\
 T\GenCONE^{k^*}(N) &\longrightarrow C^{k}_{g,-1}(M; S^2 T^*M), &\qquad  \gamma' &\longmapsto \chi_{\td x_0}^*(T(\gamma')) - T(\gamma') \label{eq_map_3}
\end{alignat}
For the map \eqref{eq_map_1} this follows using repeated application of Lemma~\ref{Lem_change_background_metric} from the fact that $g_{x_0}, g, \td g_{\td x_0}, \td g$ are bilipschitz and
\[ \Vert g_{x_0} - g \Vert_{C^k_g}, \;
 \Vert \td g_{\td x_0} - \td g \Vert_{C^k_{\td g}} \;
 < \infty, \]
via
\[ \Vert \chi_{\td x_0}^* h \Vert_{C^{k}_{-1,g}}
\leq C \Vert \chi_{\td x_0}^* h \Vert_{C^{k}_{  g_{x_0}, -1}}
= C \Vert \chi_{\td x_0}^* h \Vert_{C^{k}_{\chi_{\td x_0}^* \td g_{\td x_0},-1}}
= C \Vert  h \Vert_{C^{k}_{\td g_{\td x_0},-1}} 
\leq C \Vert  h \Vert_{C^{k}_{\td g},-1}.  \]
The boundedness of \eqref{eq_map_2} follows analogously.
An application of Lemma~\ref{Lem_change_background_metric} also implies that it suffices to show the boundedness of \eqref{eq_map_3} as a map into $C^k_{-1,g_{x_0}}$.
We can then estimate, again using Lemmas~\ref{Lem_change_background_metric}, \ref{Lem_nabf_Tgamma},
\begin{multline*}
  \big\Vert \nabla^{g_{x_0}} (\chi_{\td x_0}^*(T(\gamma'))  ) \big\Vert_{C^{k-1}_{g_{x_0},-1}} 
=  \big\Vert \nabla^{\chi_{\td x_0}^*\td g_{\td x_0}} (\chi_{\td x_0}^*(T(\gamma')) ) \big\Vert_{C^{k-1}_{ \chi_{\td x_0}^*\td g_{\td x_0},-1}}
=  \big\Vert \nabla^{\td g_{\td x_0}} (T(\gamma') ) \big\Vert_{C^{k-1}_{\td g_{\td x_0},-1}}\\
\leq C \big\Vert \nabla^{\td g_{\td x_0}} (T(\gamma') ) \big\Vert_{C^{k-1}_{\td g,-1}}
\leq C \big\Vert \nabla^{\td g} (T(\gamma') ) \big\Vert_{C^{k-1}_{-1, \td g}} + C \big\Vert \nabla^{\td g} \td g_{\td x_0} \big\Vert_{C^{k-1}_{ \td g,-1}}  \big\Vert T(\gamma')  \big\Vert_{C^{k-1}_{ \td g}} \leq C \Vert \gamma' \Vert .
\end{multline*}
and
\[ \Vert \nabla^{g_{x_0}} (T(\gamma')) \Vert_{C^{k-1}_{g,-1}}
\leq C\Vert \nabla^g (T(\gamma')) \Vert_{C^{k-1}_{g,-1}} + C\Vert \nabla^g g_{x_0} \Vert_{C^{k-1}_{g,-1}} \Vert T(\gamma') \Vert_{C^{k-1}_{g}}
 \leq C \Vert \gamma' \Vert, \]
so
\[ \Vert \nabla^{g_{x_0}} (\chi_{\td x_0}^*(T(\gamma')) - T(\gamma'))\Vert_{C^{k-1}_{g_{x_0},-1}} \leq C \Vert \gamma' \Vert. \]
Thus it remains to show that the following linear operator is bounded
\[ T\GenCONE^{k^*}(N) \longrightarrow C^{0}_{ \td g_{\td x_0},-1}(M; S^2 T^*M), \qquad  \gamma'  \longmapsto \chi_{\td x_0}^*(T(\gamma')) - T(\gamma'). \]
This follows from the fact that $|\nabla T(\gamma')| \leq C(-f)^{-1/2}|\gamma'|$ and from Lemma~\ref{Lem_dchi_small} since
\begin{multline*}
 \chi^*_{\td x_0} \td{g}_{\td x_0} - \td g_{\td x_0}
= g_{x_0} - \td g_{\td x_0}
= g + h_{x_0} + T(\gamma_{x_0}) - T(\gamma) - \big(\td g + \td h_{\td x_0} + T(\td\gamma_{\td x_0}) - T(\td\gamma)\big) \\
= (g- T(\gamma)) + h_{x_0} - (\td g - T(\td\gamma)) - \td h_{\td x_0},
\end{multline*}
which implies that $\Vert \chi^*_{\td x_0} \td{g}_{\td x_0} - \td g_{\td x_0} \Vert_{C^{1}_{\td g_{\td x_0},-2}} < \infty$ via Lemmas~\ref{Lem_MM_regular_rep}\ref{Lem_MM_regular_rep_b}, \ref{Lem_change_background_metric}.
\end{proof}

Set now for any $\td x \in \td U$
\[ \chi^0_{\td x} := \chi^{-1}_{\td x_0} \circ \chi_{\td x}. \]
Then $\chi^0_{\td x_0} = \id_M$ and by \eqref{eq_chi_g_V}, \eqref{eq_chi_H_bar_p}
\[ g_{\Theta(\td x)} = (\chi^0_{\td x})^* \td g^0_{\td x}, \qquad V_{\Theta(\td x)} = (\chi^0_{\td x})^* \td V^0_{\td x}. \]
Assuming $x_0$ to be sufficiently small (which can be achieved by shrinking $U_p$), we may use Claim~\ref{Cl_tdg0} to apply Proposition~\ref{Prop_gauge} to $\td g^0_{\td x}, \td V^0_{\td x}$ for $\td x$ close to $\td x_0$.
We obtain that there is a neighborhood $\td U_0 \subset \td U$ of $\td x_0$ and a family of $C^3$-diffeomorphisms $(\chi'_{\td x} : M \to M)_{\td x \in \td U_0}$ of finite displacement such that $P_g ((\chi'_{\td x})^* \td g^0_{\td x}, (\chi'_{\td x})^* \td V^0_{\td x}) = 0$. 
Combining Assertion~\ref{Prop_gauge_c} of Proposition~\ref{Prop_gauge} with Assertion~\ref{prop_smooth_dependence_e} of Proposition~\ref{prop_smooth_dependence} implies, again assuming $x_0$ to be small and after possibly shrinking $\td U_0$, that there is a $C^{1,\alpha}$-regular map $\Theta' : \td U_0 \to U_p$ such that
\[ g_{\Theta'(\td x)} = (\chi'_{\td x})^* \td g^0_{\td x}, \qquad V_{\Theta'(\td x)} = (\chi'_{\td x})^* \td V^0_{\td x}. \]

We will now show that $\Theta' = \Theta |_{\td U_0}$, which will establish the $C^{1,\alpha}$-regularity of $\Theta$.
Fix some $\td x \in \td U_0$ and observe that $\omega := (\chi^0_{\td x})^{-1} \circ \chi'_{\td x}$ is a $C^1$-diffeomorphism of finite displacement with
\[ \omega^* g_{\Theta(\td x)} = g_{\Theta'(\td x)}, \qquad \omega^* V_{\Theta(\td x)} = V_{\Theta'(\td x)}. \]
Since 
\[ (\omega \circ \iota_{\Theta'(\td x)})^* V_{\Theta(\td x)} 
=  \iota_{\Theta'(\td x)}^* V_{\Theta'(\td x)} 
= -\tfrac12 r \partial_r
= \iota_{\Theta(\td x)}^* V_{\Theta(\td x)} , \] 
we can apply the uniqueness property of $\iota_{\Theta(\td x)}$ from Proposition~\ref{Prop_gauge_infinity_families}\ref{Prop_gauge_infinity_families_h} and obtain that
\[ \omega \circ \iota_{\Theta'(\td x)} = \iota_{\Theta(\td x)}. \]
So if we set
\[ \omega' :=  \psi^{-1}_{\Theta(\td x)} \circ \omega \circ \psi_{\Theta'(\td x)},  \]
then by \eqref{eq_iotax_tdiotax}
\[ \omega' \circ \td\iota_{\Theta'(\td x)}
= \psi^{-1}_{\Theta(\td x)} \circ \omega \circ \iota_{\Theta'(\td x)}
= \psi^{-1}_{\Theta(\td x)} \circ \iota_{\Theta(\td x)}
= \td\iota_{\Theta(\td x)}, \]
so $\omega'$ equals the identity outside a compact subset and
\[ (\omega')^* g'_{\Theta(\td x)} = g'_{\Theta'(\td x)},  \qquad  (\omega')^* V'_{\Theta(\td x)} = V'_{\Theta'(\td x)}. \]
So $(g'_{\Theta(\td x)}, V'_{\Theta(\td x)}, \td\gamma_{\td x})$ and $(g'_{\Theta'(\td x)}, V'_{\Theta'(\td x)}, \td\gamma_{\td x})$ represent the same class in $\MM$.
Therefore $\Phi_p(\Theta(\td x)) = \Phi_p (\Theta'(\td x))$ and thus by Claim~\ref{Cl_Phi_injective} we have $\Theta(\td x) = \Theta'(\td x)$.

To see the statement involving \eqref{eq_infinitesimal_version}, note that in this setting we have $x_0 = \Theta (\td x_0) = 0$, so
\[ \td g^0_{\td x_0}  = g, \qquad \td V_{\td x_0}^0 = \nabla^g f \]
and we may take $\td\chi_{\td x_0} = \td\psi_{\td x_0}$.
Moreover, $v = d\Theta_{\td x_0} (\td v) = d\Theta'_{\td x_0} (\td v)$, which implies that
\[ \partial_s \big|_{s=0} g_{s v} = \dot g , \qquad 
\partial_s \big|_{s=0}\td g^0_{\td x_0 + s \td v} = \td\psi^*_{\td x_0} \dot{\td g}  \]
and
\begin{multline*}
 \partial_s \big|_{s=0}\td V^0_{\td x_0 + s \td v} 
= \td\psi^*_{\td x_0} \big( \partial_s \big|_{s=0}\td V_{\td x_0 + s \td v} \big)
= \td\psi^*_{\td x_0} \big( - \DIV_{\td g} \big(\partial_{s} \big|_{s=0} \td g_{\td x_0 + s \td v} \big) + \tfrac12 \nabla^{\td g} \tr_{\td g} \big(\partial_{s} \big|_{s=0} \td g_{\td x_0 + s \td v} \big) \big) \\
= \td\psi^*_{\td x_0} \big( - \DIV_{\td g} \dot{\td g} + \tfrac12 \nabla^{\td g} \tr_{\td g} \dot{\td g} \big) 
=  - \DIV_{ \td\psi^*_{\td x_0}\td g} \td\psi^*_{\td x_0}\dot{\td g} + \tfrac12 \nabla^{\td\psi^*_{\td x_0}\td g} \tr_{\td\psi^*_{\td x_0}\td g} \td\psi^*_{\td x_0}\dot{\td g} . 
\end{multline*}
The identity \eqref{eq_infinitesimal_version} now follows from \eqref{eq_gauging_differential} in Proposition~\ref{Prop_gauge}. 
\end{proof}
\bigskip

\section{Change of index formula for the Einstein operator} \label{sec_change_of_index}
In this section we establish an identity (see Proposition \ref{prop_index_formula}) relating the index of the Einstein operator $L_{g'}$ of a soliton metric $g'$ constructed via Proposition~\ref{prop_smooth_dependence} to the index of the Einstein operator $L_g$ of the base soliton metric $g$ and the index of the derivative $D_2 G_{(\gamma',\kappa')}$.
This index formula will be used in the proof of Proposition~\ref{Prop_summary} to define an orientation on the space of expanding solitons over a finite dimensional submanifold of cone metrics, which in turn will be key to define an integer degree.
We also establish various other rather technical results that will be used in this process.

The following results will not be needed for the definition of the $\IZ_2$-degree, so a reader who is only interested in this definition may skip this section.

\subsection{Comparing nearby expanding soliton metrics}
We establish some technical estimates comparing spaces of function spaces defined via the gradient expanding soliton  $(M,g,f)$ and a nearby gradient expanding soliton $(M,g',f')$.

\begin{Lemma} \label{Lem_nearby_bounds}
Consider the setting of Proposition~\ref{prop_smooth_dependence} with $k \geq 3$ and let $\delta > 0$.
Then for $(\gamma',\kappa') \in \{ G =0 \} \subset U_1 \times U_2$ sufficiently close to $(\gamma,0)$ the following is true.
Let
\begin{equation} \label{eq_gp_g}
 g' := g + F(\gamma',\kappa') + T(\gamma'-\gamma), \qquad V' := \nabla^g f - \DIV_g g' + \tfrac12 \nabla^g \tr_g g' 
\end{equation}
and suppose that $\gamma' \in \CONE^{k^*}(N)$ and $V' = \nabla^{g'} f'$ for some potential $f' \in C^1(M)$.
Suppose moreover the normalizing assumptions
\[ |\nabla^g f|_g^2 + R_g = - f, \qquad |\nabla^{g'} f'|_{g'} + R_{g'} = - f'. \]
Then the following holds:
\begin{enumerate}[label=(\alph*)]
\item  \label{Lem_nearby_bounds_a} We have the bounds
\begin{equation} \label{eq_gpg_fpf_bounds}
 \Vert g' - g \Vert_{C^0_g} \leq \delta, \qquad
\Vert \nabla^g g'  \Vert_{C^3_{-1,g}} \leq \delta, \qquad \Vert f' - f \Vert_{C^1_g} \leq \delta, 
\end{equation}
Moreover, we have the following bound on the difference of the Levi-Civita connections in the radial direction
\begin{equation} \label{eq_gpg_fpf_bounds_2}
\Vert \nabla^{g'}_{\nabla^g f} - \nabla^g_{\nabla^g f} \Vert_{C^0_{g,-1}} \leq \delta. 
\end{equation}
\item  \label{Lem_nearby_bounds_b} We have $C^{m}_{g,-1,\nabla^g f} (M;TM) = C^{m}_{g',-1,\nabla^{g'} f'} (M;TM)$ for $m=0,1,2,3$ (note that the second space uses the potential $f'$ for the definition of the weight) and the corresponding norms are equivalent.
The same is true if we replace the tensor bundle $TM$ with $S^2 T^*M$.
\item  \label{Lem_nearby_bounds_c} We have $L^2_{g,f}(M;S^2 T^*M) = L^2_{g',f'}(M;S^2 T^*M)$ and for any $\td\kappa_1, \td\kappa_2 \in L^2_{g,f}(M;S^2 T^*M)$
\[ (1-\delta) \Vert \td\kappa_1 \Vert_{L^2_{g,f}}
\leq \Vert \td\kappa_1 \Vert_{L^2_{g',f'}}
\leq (1+\delta) \Vert \td\kappa_1 \Vert_{L^2_{g,f}} \]
and
\[ \big| \langle \td\kappa_1, \td\kappa_2 \rangle_{g,f} - \langle \td\kappa_1, \td\kappa_2 \rangle_{g',f'} \big| \leq \delta \Vert \td\kappa_1 \Vert_{L^2_{g,f}} \Vert \td\kappa_2 \Vert_{L^2_{g,f}} . \]
Similar bounds hold for the norms and inner products of $H^1_{g,f}$ and $H^1_{g',f'}$.
\end{enumerate}
\end{Lemma}

\begin{proof}
The first two bounds of \eqref{eq_gpg_fpf_bounds} in Assertion~\ref{Lem_nearby_bounds_a} hold due to the continuity of $F$ and $T$; see also Lemma~\ref{Lem_nabf_Tgamma}.
For the last bound in \eqref{eq_gpg_fpf_bounds}, we use the normalizing conditions to conclude
\begin{multline*}
 |f'-f| \leq |R_{g'}-R_g| + \big| |\nabla^{g'} f'|_{g'}^2 - |\nabla^g f|_{g}^2 \big| \leq C \delta +  \big| |\nabla^{g'} f'|_{g'}^2 - |\nabla^{g} f|_{g'}^2 \big| + \big| |\nabla^{g} f|_{g'}^2 - |\nabla^g f|_g^2 \big| \\
 \leq C\delta + \big| \nabla^{g'} f' - \nabla^g f\big|_{g'} \, |\nabla^{g'} f' + \nabla^g f \big| + \big|(g' - g)(\nabla^g f, \nabla^g f)\big|
\end{multline*}
The second term on the right-hand side is $\leq C\delta$, because $\nabla^{g'} f' - \nabla^g f = -\DIV_g g' + \tfrac12\nabla^g \tr_g g'$.
The third term is bounded by $|F(\gamma',\kappa')| \, |\nabla^g f|_g^2 \leq C \delta$ outside a compact subset, because $T(\gamma'-\gamma)$ has no components in the $\nabla^g f$-direction.
For the derivative bound on $f'-f$, it suffices to show smallness of $\nabla^{g'} (f'-f) = (\nabla^{g'} f' - \nabla^g f) + (\nabla^g f - \nabla^{g'} f)$.
The smallness of the first term follows again by
\[ \nabla^{g'} f' - \nabla^g f = - \DIV_g g' + \tfrac12 \nabla^g \tr_g g'. \]
For the second term, note that
\begin{equation*} %
   (\nabla^g f - \nabla^{g'} f) 
= g^{\prime, -1} \big( (g' - g ) (\nabla^g f) \big),
\end{equation*}
which can be bounded by $C \delta$ as before.
So \eqref{eq_gpg_fpf_bounds} follows after possibly assuming $(\gamma',\kappa')$ to be closer to $(\gamma,0)$.

For the bound \eqref{eq_gpg_fpf_bounds_2} note that for any vector fields $Y,Z$ the following is true at a point at which $\nabla^g Y = \nabla^g Z = 0$ and $|Y|_g = |Z|_g = 1$
\begin{multline*}
  2 \big( \nabla^{g'}_{\nabla^g f} Y \cdot_{g'} Z - \nabla^g_{\nabla^g f} Y  \cdot_{g'} Z \big)
= (\nabla^g_{\nabla^g f} g')(Y,Z) + (\nabla^g_{Y} g')(\nabla^g f,Z) - (\nabla^g_{Z} g')(\nabla^g f,Y)
\\
= \big(\nabla^g_{\nabla^g f} T(\gamma'-\gamma)\big)(Y,Z) + \big(\nabla^g_{Y} T(\gamma'-\gamma)\big)(\nabla^g f,Z) - \big(\nabla^g_{Z} T(\gamma'-\gamma)\big)(\nabla^g f,Y) \\ + \nabla^g F(\gamma',\kappa') * \nabla^g f * Y * Z.
\end{multline*}
By Lemma~\ref{Lem_nabf_Tgamma}  we can bound the first term on the right-hand side and since $\Vert F(\gamma',\kappa') \Vert_{C^0_{-2,g}}$ can be made arbitrarily small, for $(\gamma',\kappa')$ sufficiently close to $(\gamma,0)$, we can also bound the last term.
To bound the remaining terms, notice that
\begin{multline*}
 \big(\nabla^g_{Y} T(\gamma'-\gamma)\big)(\nabla^g f,Z)
= Y \big( \big(  T(\gamma'-\gamma)\big) (\nabla^g f, Z) \big) - \big(  T(\gamma'-\gamma)\big) (\nabla^g_Y \nabla^g f, Z) \\
= 0 + \tfrac12 \big(  T(\gamma'-\gamma)\big) (Y , Z) + \big( T(\gamma'-\gamma) * \Ric_g \big)(Y,Z) , 
\end{multline*}
where in the last equality we have used the soliton identity and the fact that $\gamma' - \gamma \in T\CONE^{k^*}(N)$, which implied the vanishing of the radial component.
Therefore
\[ \big(\nabla^g_{Y} T(\gamma'-\gamma)\big)(\nabla^g f,Z) - \big(\nabla^g_{Z} T(\gamma'-\gamma)\big)(\nabla^g f,Y)
= \big( T(\gamma'-\gamma) * \Ric_g \big)(Y,Z), \]
which can be bounded as desired.

Assertion~\ref{Lem_nearby_bounds_b} follows from Assertion~\ref{Lem_nearby_bounds_a} via Lemma~\ref{Lem_change_background_metric}.

Assertion~\ref{Lem_nearby_bounds_c} follows after possibly assuming $(\gamma',\kappa')$ to be closer to $(\gamma,0)$ again, since Assertion~\ref{Lem_nearby_bounds_a} guarantees closeness of the background measures $e^{-f}dg$ and $e^{-f'} dg'$, which are used to define the two inner products.
\end{proof}

\subsection{A self-adjoint operator}

In order to relate $\Index(-L_g)$ and $\Index(-L_{g'})$, we need to make sense of the quantity $\Index (D_2G_{(\gamma',\kappa')})$. 
Unfortunately, the differential $D_2G_{(\gamma',\kappa')}:K_g\rightarrow K_g$ may not be self-adjoint, but we will see that $D_2G_{(\gamma',\kappa')}$ behaves like $L_{g'}|_{K_g}$ when $(\gamma',\kappa')$ is close to $(\gamma,0)$.

\begin{Lemma}\label{Lem_Z_map}
Suppose that we are in the setting of Lemma~\ref{Lem_nearby_bounds}.
Then there is an injective linear map
\[ Z=Z_{(\gamma',\kappa')}: K_g \longrightarrow   C^{2,\alpha}_{g',-1,\nabla^{g'}f'}(M;S^2T^*M) \cap H^1_{g',f'}(M;S^2 T^*M) \]
such that for all $\kappa\in K_g$
\begin{equation} \label{eq_LZ_d2G}
 L_{g'} (Z(\kappa))=D_2G_{(\gamma',\kappa')}(\kappa),
\end{equation}
and such that
\begin{equation} \label{eq_definition_Z}
 Z(\kappa) = (D_2 F)_{(\gamma',\kappa')}(\kappa) + \LL_Y g', 
\end{equation}
where $Y \in C^{3,\alpha}_{g',-1,\nabla^{g'}f'}(M;TM)$ is the unique solution to the equation
\begin{equation} \label{eq_Y_identity}
 \triangle_{g',f'}Y-\tfrac{1}{2} Y=\DIV_g(\td{\kappa})-\tfrac{1}{2}\nabla^g\tr_g(\td{\kappa})-\DIV_{g'}(\td{\kappa})+\tfrac{1}{2}\nabla^{g'}\tr_{g'}(\td{\kappa}), \qquad \td{\kappa}=(D_2 F)_{(\gamma',\kappa')}(\kappa). 
\end{equation}
Moreover, given $\delta > 0$, the following is true, assuming that $(\gamma',\kappa')$ is sufficiently close to $(\gamma,0)$:
\begin{enumerate}[label=(\alph*)]
\item  \label{Lem_Z_map_a} For any $\kappa \in K_g$ we have
\begin{equation} \label{eq_Zk_L2_C0_bounds}
 \big\Vert Z_{(\gamma',\kappa')}(\kappa) - \kappa \big\Vert_{L^2_{g,f}} + \big\Vert Z_{(\gamma',\kappa')}(\kappa) - \kappa \big\Vert_{C^0_{g,-1}} \leq \delta \Vert \kappa \Vert_{L^2_{g,f}}. 
\end{equation}
\item  \label{Lem_Z_map_b} Suppose that $\psi : M \to M$ is a diffeomorphism of regularity $C^3$ and of finite displacement such that $(\psi^* g', \psi^* \nabla^{g'} f', \gamma')$ represents an element of $\MMgrad (M,N,\iota)$.
Then for any $\kappa \in \ker D_2 G_{(\gamma',\kappa')}$ and $\gamma^* \in T\GenCONE^{k^*}(N)$ we have
\begin{equation} \label{eq_ZLvskL}
 \big| \big\langle \psi^* (Z(\kappa)), L_{\psi^* g'} (T(\gamma^*)) \big\rangle_{\psi^* g',\psi^* f'} - \big\langle \kappa, L_{g} (T(\gamma^*)) \big\rangle_{g,f} \big| \leq \delta \big\Vert \gamma^* \big\Vert \, \big\Vert \kappa \big\Vert_{L^2_{g,f}} .
\end{equation}
\end{enumerate}
\end{Lemma}

\begin{proof}
Fix some $\kappa \in K_g$.
Let $\td\kappa := D_2 F_{(\gamma',\kappa')}(\kappa) \in C^3_{g,-1}(M;S^2T^*M) = C^3_{g',-1}(M;S^2T^*M)$ and let $A \in C^{2}_{g',-1}(M;TM)$ be the right-hand side in the first equation of \eqref{eq_Y_identity}.
By Proposition~\ref{Prop_Lundardi_weighted} the equation \eqref{eq_Y_identity} has a unique solution $Y\in C^{3,\alpha}_{g',-1,\nabla^{g'}f'}(M;TM)$.
It follows that $\LL_Y g' \in C^{2,\alpha}_{g',-1}(M;S^2T^*M)$.
Define $Z(\kappa) \in C^{2,\alpha}_{g',-1}(M; S^2 T^*M)$ as in \eqref{eq_definition_Z}.

We will now verify \eqref{eq_LZ_d2G}.
For this recall that for an arbitrary metric $\td g$ we have
\[Q_g(\td{g})=-2\Ric(\td{g})-\td{g}-\LL_{\nabla^g f} \td{g}+\LL_{\DIV_g(\td{g})-\frac{1}{2}\nabla^g \tr_g(\td{g})}\td{g},\]
so that
\[
Q_g(\td{g})-Q_{g'}(\td{g})=-\LL_{\nabla^g f-\nabla^{g'} f'} \td{g}+\LL_{\DIV_g(\td{g})-\frac{1}{2}\nabla^g\tr_g(\td{g})}\td{g}-\LL_{\DIV_{g'}(\td{g})-\frac{1}{2}\nabla^{g'}\tr_{g'}(\td{g})}\td{g}.
\]
Taking the differential in the $\td{g}$ argument at $\td g = g'$, we obtain, using \eqref{eq_gp_g},
\begin{align*}
(D Q_g)_{g'}(h)-L_{g'}(h)&=-\LL_{\nabla^g f-\nabla^{g'} f'}h+\LL_{\DIV_g(g')-\frac{1}{2}\nabla^g\tr_g(g')}h
\\
&\qquad+\LL_{\DIV_g(h)-\frac{1}{2}\nabla^g\tr_g(h)}g'-\LL_{\DIV_{g'}(h)-\frac{1}{2}\nabla^{g'}\tr_{g'}(h)}g'
\\
&=\LL_{\DIV_g(h)-\frac{1}{2}\nabla^g\tr_g(h)}g'-\LL_{\DIV_{g'}(h)-\frac{1}{2}\nabla^{g'}\tr_{g'}(h)}g'.
\end{align*}
Plugging in $h = \td \kappa$ and noticing that by Proposition~\ref{prop_smooth_dependence}\ref{prop_smooth_dependence_b} we have $(D_2 G)_{(\gamma',\kappa')} (\kappa) = (DQ_g)_{g'}(\td\kappa)$, 
we obtain
\[L_{g'}(\td{\kappa})=D_2G_{(\gamma',\kappa')}(\kappa)-\LL_{\DIV_g(\td{\kappa})-\frac{1}{2}\nabla^g\tr_g(\td{\kappa})}g'+\LL_{\DIV_{g'}(\td{\kappa})-\frac{1}{2}\nabla^{g'}\tr_{g'}(\td{\kappa})}g'.\]
On the other hand, from the definition of $Q_g$ we have that
\begin{align*}
L_{g'}(\LL_Y g') &= (DQ_g)_{g'}(\LL_{Y} g') \\
&=-2 (D{\Ric})_{g'}(\LL_Y g')-\LL_Y g'-\LL_{\nabla^{g'}f'}\LL_Y g'+\LL_{\DIV_{g'}\LL_Y g'-\frac{1}{2}\nabla^{g'}\tr_{g'}\LL_Y g'} g'
\\
&=-\LL_Y(2\Ric_{g'}+g'+\LL_{\nabla^{g'}f'}g')+\LL_{[Y,\nabla^{g'}f']}g'+\LL_{\DIV_{g'}\LL_Y g'-\frac{1}{2}\nabla^{g'}\tr_{g'}\LL_Y g'} g'
\\
&=\LL_{[Y,\nabla^{g'}f']}g'+\LL_{\DIV_{g'}\LL_Y g'-\frac{1}{2}\nabla^{g'}\tr_{g'}\LL_Y g'} g'.
\end{align*}
So
\begin{multline*}
 L_{g'} (Z(\kappa)) = D_2 G_{(\gamma',\kappa')}(\kappa) \\
+ \LL_{[Y,\nabla^{g'}f']}g'+\LL_{\DIV_{g'}\LL_Y g'-\frac{1}{2}\nabla^{g'}\tr_{g'}\LL_Y g'} g'
-\LL_{\DIV_g(\td{\kappa})-\frac{1}{2}\nabla^g\tr_g(\td{\kappa})}g'+\LL_{\DIV_{g'}(\td{\kappa})-\frac{1}{2}\nabla^{g'}\tr_{g'}(\td{\kappa})}g'. \
\end{multline*}
The identity \eqref{eq_LZ_d2G} now follows from \eqref{eq_Y_identity} and the fact that 
\[[Y,\nabla^{g'}f']+\DIV_{g'}\LL_Y g'-\tfrac{1}{2}\nabla^{g'}\tr_{g'}\LL_Y g'= \triangle_{f',g'}Y-\tfrac{1}{2} Y.\]
Note that $Z(\kappa) \in C^{2,\alpha}_{g',-1}(M; S^2 T^*M)$.
So by \eqref{eq_LZ_d2G} we have 
\begin{equation}\label{eq_Zkappa_elliptic}
\triangle_{f',g'} Z(\kappa) = 2 \Rm_{g'}(Z(\kappa)) + D_2G_{(\gamma',\kappa')}(\kappa) \in C^{2,\alpha}_{g',-1} (M; S^2 T^*M).
\end{equation}
It follows again from Proposition~\ref{Prop_Lundardi_weighted} that $Z(\kappa) \in C^{2,\alpha}_{g',-1,\nabla^{g'}f'} (M; S^2 T^*M)$.
The fact that $Z(\kappa) \in H^1_{g',f'} (M; S^2 T^*M)=H^1_{g,f} (M; S^2 T^*M)$ follows from Proposition~\ref{Prop_L2_theory}\ref{Prop_L2_theory_b}.

Let us now prove Assertion~\ref{Lem_Z_map_a}.
Suppose without loss of generality that $\Vert \kappa \Vert_{L^2_{g,f}}=1$.
Since $K_g$ is finite dimensional, this implies bounds of the form $\Vert \kappa \Vert_{C^{3}_{g,-1}}, \Vert \kappa \Vert_{H^1_{g,f}} \leq C$, where here and in the following $C$ denotes a generic constant that is independent of the choices of $(\kappa',\gamma')$ and $\kappa$.
So by choosing $(\gamma',\kappa')$ close enough to $(\gamma,0)$, we can ensure that the $C^2_{g',-1}$-norm of the right-hand side of \eqref{eq_Y_identity} and thus $\Vert Y \Vert_{C^{3,\alpha}_{g',-1,\nabla^{g'} f'}}$ and $\Vert \LL_{Y} g' \Vert_{C^0_{g,-1}}$ are arbitrarily small.
Combining this with the continuity of $D_2 F_{(\gamma',\kappa')}$, implies the desired bound on $\Vert Z_{(\gamma',\kappa')}(\kappa) - \kappa \Vert_{C^0_{g,-1}}$ for $(\gamma',\kappa')$ close enough to $(\gamma,0)$.

To see the $L^2_{g,f}$-bound, recall first that $\Vert \kappa \Vert_{H^1_{g,f}} \leq C$.
So, using the Poincar\'e inequality from Lemma~\ref{Lem_L2f_Poincare}, we obtain that
\begin{equation} \label{eq_kappa_poinc}
 \int_M |\kappa|_g^2 \, |\nabla^g f|_g^2 e^{-f}dg \leq C. 
\end{equation}
Similarly, using \eqref{eq_Zkappa_elliptic}, the fact that $\Vert D_2 G_{(\gamma',\kappa')}(\kappa) \Vert_{L^2_{g,f}} \leq C$, the integration by parts from Proposition~\ref{Prop_L2_theory}\ref{Prop_L2_theory_a}, and Lemma~\ref{Lem_nearby_bounds}, in order to compare $g',g$ and $f',f$, we obtain a bound of the form
\begin{multline*}
 \int_M |\nabla^g (Z(\kappa))|_g^2 e^{-f} dg 
\leq C \int_M (|f| + 1)^{-1} |Z(\kappa)|_g^2 e^{-f} dg + C \int_M |D_2G_{(\gamma',\kappa')}(\kappa)|_g \, |Z(\kappa)|_g e^{-f} dg  \\
\leq C \int_M (|f| + 1)^{-1} |Z(\kappa)|_g^2 e^{-f} dg 
+ C \int_M |D_2G_{(\gamma',\kappa')}(\kappa)|_g^2 e^{-f} dg
+ \frac12 \int_M  |Z(\kappa)|_g^2 e^{-f} dg,
\end{multline*}
so
\[  \int_M |\nabla^g (Z(\kappa))|_g^2 e^{-f} dg 
\leq  C \int_M (|f| + 1)^{-1} |Z(\kappa)|_g^2 e^{-f} dg 
+ C. \]
Combining this with Lemma~\ref{Lem_L2f_Poincare} implies
\begin{equation} \label{eq_Zkappa_poinc}
  \int_M |Z(\kappa)|_g^2 \, |\nabla^g f|^2_g e^{-f} dg 
\leq  C \int_M (|f|+1)^{-1} |Z(\kappa)|_g^2 e^{-f} dg 
+ C. 
\end{equation}
Since $|\nabla^2 f|^2_g \geq 2 C(|f|+1)^{-1}$ outside a compact subset and since we have uniform $C^0$-bounds on $Z(\kappa)$, this implies
\[ \int_M |Z(\kappa)|_g^2 \, |\nabla^g f|^2_g e^{-f} dg 
\leq  C. \]
Combining \eqref{eq_kappa_poinc}, \eqref{eq_Zkappa_poinc} yields
\[ \int_M |Z(\kappa) - \kappa|_g^2 \, |\nabla^g f|^2_g e^{-f} dg 
\leq  C. \]
So for some constant $F > 0$, whose value we will determine in a moment, we have, using the first bound in \eqref{eq_Zk_L2_C0_bounds},
\[ \Vert Z(\kappa) - \kappa \Vert_{L^2_{g,f}}^2
\leq \int_{\{ |\nabla f|^2 \leq F \}} |Z(\kappa) - \kappa|_g^2  e^{-f} dg + \frac{C}{F} \leq C(F) \delta + \frac{C}{F}. \]
Choosing $F$ large enough and then adjusting $\delta$ shows that we can make the right-hand side as small as we like as long as $(\gamma',\kappa')$ is sufficiently close to $(\gamma,0)$.

Next, we prove Assertion~\ref{Lem_Z_map_b}.
Suppose without loss of generality that $\Vert \kappa \Vert_{L^2_{g,f}} =\Vert \gamma^* \Vert =  1$.
In the following, we will denote by $r : M \to \IR_+$ a smooth function such that $r \circ \iota$ agrees with the first coordinate on $(2,\infty) \times N$.
We will also denote by $C < \infty$ a generic constant, which is independent of the choice of $(\gamma',\kappa')$, $\kappa, \gamma^*$.
Note that by assumption and \eqref{eq_LZ_d2G} we have
\[ L_{g'} (Z(\kappa)) = 0. \]
So we can apply Proposition~\ref{Prop_L2_theory}\ref{Prop_L2_theory_d} and Assertion~\ref{Lem_Z_map_a} to conclude that
\begin{equation} \label{eq_kZk_bounds}
 |\kappa| \leq C r^{-4+0.1} e^{-r^2/4}, \qquad 
|Z (\kappa)| \leq C r^{-4+0.1} e^{-r^2/4} 
\leq  \big\Vert Z(\kappa) \big\Vert_{L^2_{g,f}} \leq C r^{-4+0.1} e^{-r^2/4}.
\end{equation}
Set $g'' := \psi^* g'$, $f'' := \psi^* f'$, so $(g'',\nabla^{g''} f'', \gamma')$ represents an element in $\MMgrad(M,N,\iota)$.
Using Lemma~\ref{Lem_MM_regular_rep}\ref{Lem_MM_regular_rep_b} and Lemma~\ref{Lem_nabf_Tgamma}  we obtain bounds of the form
\begin{align}
 |L_g (T(\gamma^*))|_g &\leq |\nabla^g_{\nabla^g f} (T(\gamma^*)) |_g + C r^{-2} \leq C r^{-1},  \label{eq_LT_bound} \\
 |L_{g''} (T(\gamma^*))|_{g''} &\leq |\nabla^{g''}_{\nabla^{g''} f''} (T(\gamma^*)) |_{g''} + C r^{-2} \leq C r^{-1} \qquad \text{if} \quad r > \underline{r}(\gamma',\kappa',\psi). \label{eq_LT_bound_2}
\end{align}
Note that the lower bound on $r$ in \eqref{eq_LT_bound_2} naturally depends on $\psi$; more specifically it depends on the domain on which we have $\iota^* \nabla^{g''} f'' = - \frac12 r \partial_r$ on $(1, \infty) \times N$.
Similarly, using Lemma~\ref{Lem_nearby_bounds}\ref{Lem_nearby_bounds_a}, we have
\begin{multline} \label{eq_LgpT_bound}
|L_{g'} (T(\gamma^*)) |_g 
\leq \big|\nabla^{g'}_{\nabla^{g'} f'} T(\gamma^*) \big|_g + C r^{-1} 
\leq \big|\nabla^{g'}_{\nabla^{g} f} T(\gamma^*) \big|_g + C r^{-1} \\
\leq \big|\nabla^{g}_{\nabla^{g} f} T(\gamma^*) \big|_g + C r^{-1} \leq C r^{-1}, 
\end{multline}
where we have used $|\nabla^{g'} f' - \nabla^g f| \leq C$ in the first inequality and \eqref{eq_gpg_fpf_bounds_2} in the second.
Combining \eqref{eq_LT_bound}, \eqref{eq_LgpT_bound} with the bounds \eqref{eq_kZk_bounds}, allows use to conclude that there is a constant $F < \infty$, which is independent of the choice of $(\gamma',\kappa')$, $\kappa$, $\gamma^*$, such that
\begin{align*}
 \bigg| \big\langle Z(\kappa), L_{g'} (T(\gamma^*)) \big\rangle_{g',  f'} - \int_{\{ |f| \leq F \}} \big( Z(\kappa) \cdot_{g'} L_{g'} (T(\gamma^*)) \big) e^{-f'} dg' \bigg|  &\leq \frac{\delta}{3} \\
\bigg| \big\langle \kappa, L_{g} (T(\gamma^*)) \big\rangle_{g,  f} - \int_{\{ |f| \leq F \}} \big( \kappa \cdot_g L_{g} (T(\gamma^*)) \big) e^{-f}dg  \bigg|
&\leq \frac{\delta}{3}. 
\end{align*}
If $(\gamma',\kappa')$ is sufficiently close to $(\gamma,0)$, then Lemma~\ref{Lem_nearby_bounds}\ref{Lem_nearby_bounds_a} and Assertion~\ref{Lem_Z_map_a} imply that the integrands of both integrals are sufficiently close so that the absolute value of the difference of both integrals is bounded by $\delta/3$.
This shows
\begin{equation} \label{eq_ZLvskL_intermediate}
 \big| \big\langle Z(\kappa), L_{g'} (T(\gamma^*)) \big\rangle_{g',  f'} - \big\langle \kappa, L_{g} (T(\gamma^*)) \big\rangle_{g,  f} \big| \leq \delta \big\Vert \gamma^* \big\Vert \, \big\Vert \td\kappa \big\Vert_{L^2_{g,f}}. 
\end{equation}

Next, we claim that
\begin{equation} \label{eq_psiT_T2}
  \big\langle \psi^* (Z(\kappa)), L_{\psi^* g'} (T(\gamma^*)) \big\rangle_{\psi^* g', \psi^* f'}
= \big\langle Z(\kappa), L_{g'} ( T(\gamma^*)) \big\rangle_{g',  f'}.  
\end{equation}
To prove this, notice first that by \eqref{eq_LT_bound_2}, \eqref{eq_LgpT_bound} we have the asymptotic bounds
\begin{align*} 
|L_{g''}(T(\gamma^*))|_{g''} &\leq C r^{-1}, \\
 |L_{g''} (\psi^* (T(\gamma^*))) |_{g''} &=  |L_{\psi^* g'} (\psi^*(T(\gamma^*))) |_{\psi^* g'}
= |L_{g'} (T(\gamma^*)) |_{g'} \circ \psi \leq C r^{-1}. 
\end{align*}
Since we have bounds on $\psi^* g' - g' = g'' - g'$ and its derivatives through Lemma~\ref{Lem_iota} and since $\psi$ has finite displacement, we can apply Lemma~\ref{Lem_dchi_small} and obtain that
\[ \big|\psi^* (T(\gamma^*)) - T(\gamma^*) \big|_{g''} \leq C' r^{-1} , \]
for some constant $C' < \infty$, which may depend on $(\gamma',\kappa')$, $\gamma^*$ and $\psi$.
Moreover, since $g'' = \psi^* g'$ and since we have uniform bounds on $\nabla^{g''} g'$, we also obtain
\[ |\nabla^{g''} \psi^* (T(\gamma^*)) |_{g''}, \; |\nabla^{g''} (T(\gamma^*)) |_{g''} \leq C. \]
So since
\[  L_{g''} \big(\psi^* (Z (\kappa)) \big) = \psi^* \big( L_{g'} (Z (\kappa)  \big) = 0, \]
we have the necessary bounds to conclude, using Lemma~\ref{Lem_linear_identities}\ref{Lem_linear_identities_f}, that
\begin{equation*}
 \big\langle \psi^* (Z (\kappa)), L_{g''} \big( \psi^* (T(\gamma^*)) - T(\gamma^*) \big)\big\rangle_{g'',f''} 
= 0. 
\end{equation*}
Therefore
\begin{multline*}
  \big\langle \psi^* ( Z (\kappa)), L_{\psi^* g'} (  T(\gamma^*) )\big\rangle_{\psi^*g',\psi^*f'}
=  \big\langle \psi^* ( Z(\kappa)), L_{\psi^* g'} \big( \psi^* (T(\gamma^*))  \big)\big\rangle_{\psi^*g',\psi^*f'} \\
=  \big\langle  Z (\kappa), L_{ g'} \big(T(\gamma^*)  \big)\big\rangle_{g',f'} , 
\end{multline*}
which proves \eqref{eq_psiT_T2}.
Combining \eqref{eq_psiT_T2} with \eqref{eq_ZLvskL_intermediate} implies \eqref{eq_ZLvskL}.

\end{proof}

We now describe the behavior of $Z = Z_{(\gamma',\kappa')}$ restricted to the kernel of $(D_2 G)_{(\gamma',\kappa')}$.
Roughly speaking, the following lemma states that $Z$ maps every direction $\td\kappa$ within this kernel to a direction $\td\kappa' \in K_{g'}$ in such a way that $\td\kappa$ and $\td\kappa'$ correspond to the same direction in the tangent space of $\MM^{k^*} (M,N,\iota)$, viewed with respect to the charts given by the base metrics $g$ and $g'$.
To be more specific, recall that the charts $\Phi_p$ from Proposition~\ref{Prop_Banach_manifold}, which give $\MM^{k^*} (M,N,\iota)$ the structure of a $C^{1,\alpha}$-Banach manifold can be identified with subsets of the Banach-submanifold $\{ G = 0 \} \subset U_1 \times U_2$ from Proposition~\ref{prop_smooth_dependence}.
Here $U_2 \subset K_g$ is an open subset of the kernel of $L_g$, where $p = [(g,\nabla^g f, \gamma)]$ and the submanifold $\{ G = 0 \}$ is tangent to $\{ 0 \} \times K_g$ at the point corresponding to $p$.
Call the directions at any tangent space of $\{ G = 0 \}$ with vanishing first component \emph{vertical.} 
Consider now another point $p' = [(g',\nabla^{g'} f', \gamma')]  \in \MM^{k^*} (M,N,\iota)$ near $p$ and repeat the previous construction\footnote{Note that in the statement of the following lemma, the base metric $g'$ is actually denoted by $\psi^* g'$.} with the resulting Banach submanifold $\{ G' = 0 \} \subset U_1' \times U'_2$.
Then the tangent space of $\{ G = 0 \}$ at the point corresponding to $p'$ can be identified with the tangent space of $\{ G' = 0 \}$ at its basepoint, via the differential $D(\Phi^{-1}_{p'} \circ \Phi_p)$.
Consider now a vertical direction $\td\kappa$ of the tangent space at $\{ G = 0 \}$ that corresponds to a vertical direction $\td\kappa'$ of the tangent space at $\{ G' = 0\}$.
The following lemma then states that $Z(\td\kappa) = \td\kappa'$.

\begin{Lemma} \label{Lem_Z_is_change_of_coo}
Consider the setting of Proposition~\ref{Prop_Banach_manifold}, fix $p = [(g,V=\nabla^g f, \gamma)] \in \MM_{\grad}$ as in this lemma and consider the map $\Phi_p : U_p \to \MM$.
Fix some $x' \in U_p$, set $p' := \Phi_p(x')$, $\gamma' := \Pi(p')$, $\kappa' := \Xi_p (x')$ and let as in Assertion~\ref{Prop_Banach_manifold_d} of this lemma
\[ g' := g + H_p(x') + T(\gamma') - T(\gamma), \qquad
V' := \nabla^g f - \DIV_g (g') + \tfrac12 \nabla^g \tr_g (g'). \]
Then, by the same assertion, there is a $C^2$-diffeomorphism $\psi : M \to M$ of finite displacement such that $p'$ is represented by $(\psi^* g', \psi^* V', \gamma')$.
By Lemma~\ref{Lem_MM_regular_rep} we may assume that $\psi$ is chosen such that $(\psi^* g', \psi^* V', \gamma')$ is also $C^{k^*-2}$-regular.
Suppose that $p' = [(\psi^* g', \psi^* V', \gamma')] \in \MMgrad$, so $V' = \nabla^{g'} f'$ for some potential function $f' \in C^1(M)$ and $\gamma' \in \CONE^{k^*}(N)$.

Apply Proposition~\ref{Prop_Banach_manifold} again to $(\psi^* g', \psi^* V', \gamma')$ instead of $(g,V,\gamma)$ to obtain the map $\Phi_{p'} : X_{p'} \supset U_{p'} \to \MM$.
Combining Assertions~\ref{Prop_Banach_manifold_d}, \ref{Prop_Banach_manifold_e} of this lemma implies that the map $\Phi_{p'}^{-1} \circ \Phi_p$ is defined and differentiable near $x'$.
Let $K_g$ and $K_{g'}$ be the kernels of $L_g$ and $L_{g'}$ in $C^{2,\alpha}_{g,-1,\nabla f} (M; S^2 T^*M)$ and $C^{2,\alpha}_{g',-1,\nabla f'} (M; S^2 T^*M)$, respectively.

Consider elements $\dot\kappa \in K_g$ and $\dot\kappa' \in K_{g'}$ such that for some $v \in T_{x'} U_p = X_p$ we have
\[ \dot\kappa = (D\Xi_p)_{x'}(v), \qquad 
\psi^*\dot\kappa' = (D\Xi_{p'})_{0}(D(\Phi_{p'}^{-1} \circ \Phi_p)_{x'}(v)), \qquad
(D(\Pi \circ \Phi_p))_{x'}(v) = 0. \]
Denote by $G : U_1 \times U_2 \to K_g$ the function from Proposition~\ref{prop_smooth_dependence} applied to $(g,\nabla^g f, \gamma)$ and by $Z_{(\gamma', \kappa')}$ the map from Lemma~\ref{Lem_Z_map}.
Then we have $\dot\kappa \in \ker (D_2 G)_{(\gamma',\kappa')}$ and
\[ \dot\kappa' = Z_{(\gamma',\kappa')} (\dot\kappa). \]
\end{Lemma}

\begin{proof}
For $s \in (-\eps,\eps)$, where $\eps > 0$ is a small constant, define
\[ \gamma_s := \Pi(\Phi_p(x'+sv)) \]
and
\begin{alignat*}{1}
 g_s &:= g + H_p(x'+sv) + T(\gamma_s) - T(\gamma), \\
 &\qquad\qquad\qquad   V_s := \nabla^g f - \DIV_g (g_s) + \tfrac12 \nabla^g \tr_g (g_s), \\
  g'_s &:= \psi^* g' + H_{p'}((\Phi_{p'}^{-1} \circ \Phi_p)_{x'}(x'+sv)) + T(\gamma_s) - T(\gamma),  \\
  &\qquad\qquad\qquad V'_s := \psi^* \nabla^{g'} f' - \DIV_{\psi^* g'} (g'_s) + \tfrac12 \nabla^{\psi^* g'} \tr_{\psi^* g'} (g'_s). 
\end{alignat*}
Then
\begin{multline*}
 \gamma_0 = \gamma', \qquad
g_0 = g', \qquad
g'_0 = \psi^* g', \qquad \\
\partial_s|_{s=0} g_s 
= (DH_p)_{x'}(v)
= (D_2 F)_{(\gamma',\kappa')} (\dot\kappa) = : \td{\dot\kappa}, \qquad
\partial_s|_{s=0} g'_s = \psi^*\dot\kappa'. 
\end{multline*}
Let us now apply Proposition~\ref{Prop_Banach_manifold}\ref{Prop_Banach_manifold_e} for the base triple $(\psi^* g', \psi^* V', \gamma')$ and the family $\td\Phi(s) =  g_s$. (Note that we switched the role of the base metrics here!)
We obtain that
\[ \psi^*\dot\kappa' 
= \partial_s|_{s=0} g'_s 
= \psi^*(\partial_s|_{s=0}  g_s) + \LL_{Y'} (\psi^* g' )
= \psi^* \big( (D_2 F)_{(\gamma',\kappa')} (\dot\kappa) \big) + \LL_{Y'} (\psi^*g') , \]
where $Y' \in C^{2}_{\psi^* g',-1,\psi^*\nabla^{g'} f'} (M;TM)$ is the unique vector field such that
\begin{align*}
 \triangle_{\psi^*g',\psi^*f'} Y' - \tfrac12 Y' &=  
 \DIV_{\psi^* g} (\psi^* (\partial_s |_{s=0} g_s)) - \tfrac12 \nabla^{\psi^* g} \tr_{\psi^* g}(\psi^* (\partial_s|_{s=0} g_s))  \\ &\qquad -
  \DIV_{\psi^* g'} (\psi^*(\partial_s|_{s=0} g_s))+ \tfrac12 \nabla^{\psi^*g'} \tr_{\psi^*g'} (\psi^*(\partial_s|_{s=0} g_s)) \\
&= \psi^* \big(\DIV_{ g} (\td{\dot\kappa}) - \tfrac12 \nabla^{g} \tr_{g} (\td{\dot\kappa}) -\DIV_{g'} (\td{\dot\kappa}) + \tfrac12 \nabla^{g'} \tr_{g'} (\td{\dot\kappa} )
\big).
\end{align*}
So setting $Y := \psi_* Y' \in C^{2}_{g',-1,\nabla^{g'} f'} (M;TM)$, we obtain that
\[ \dot\kappa' = (D_2F)_{(\gamma',\kappa')} (\dot\kappa) + \LL_{Y} g' \]
and
\[ \triangle_{g',f'} Y - \tfrac12 Y = \DIV_g (\td{\dot\kappa}) - \tfrac12 \nabla^g \tr_g (\td{\dot\kappa}) - \DIV_{g'} (\td{\dot\kappa}) + \tfrac12 \nabla^{g'} \tr_{g'} (\td{\dot\kappa}). \]
This characterization of $Y$ agrees with \eqref{eq_Y_identity}, which finishes the proof.
The fact that $\dot \kappa \in \ker (D_2 G)_{(\gamma',\kappa')}$ follows from \eqref{eq_LZ_d2G}.
\end{proof}

\subsection{The index formula}
We can now prove Proposition \ref{prop_index_formula}, which is the main result of this section.
Note that in the following we use the index as defined in Definition~\ref{Def_idx_null}.

\begin{Proposition}\label{prop_index_formula}
Consider the setting of Proposition~\ref{prop_smooth_dependence} with $k \geq 3$.
For $(\gamma',\kappa') \in \{ G =0 \} \subset U_1 \times U_2$ sufficiently close to $(\gamma,0)$ the following is true.
Let
\[ g' := g + F(\gamma',\kappa') + T(\gamma'-\gamma), \qquad V' := \nabla^g f - \DIV_g g' + \tfrac12 \nabla^g \tr_g g' \]
and suppose that $\gamma' \in \CONE^{k^*}(N)$ and $V' = \nabla^{g'} f'$ for some potential $f' \in C^1(M)$.

Then the map $Z := Z_{(\gamma',\kappa')} : K_g \to Z(K_g)$ from Lemma~\ref{Lem_Z_map} is invertible and the linear map:
\begin{equation} \label{eq_proj_d2G}
 \proj_{g',f',Z(K_g)} \circ (D_2 G)_{(\gamma',\kappa')} \circ Z_{(\gamma',\kappa')}^{-1} = \proj_{g',f',Z(K_g)} \circ L_{g'} : Z(K_g) \lto Z(K_g),
\end{equation}
where $\proj_{g',f',Z(K_g)}$ denotes the orthogonal projection in $L^2_{g',f'}$ onto $Z(K_g)$, 
is self-adjoint with respect to the inner product $\langle \cdot, \cdot \rangle_{g',f'}$.
Moreover, we have
\[\Index(-L_{g'})=\Index(-L_g)+\Index (-D_2G_{(\gamma',\kappa')}),\]
in the sense that
\begin{equation}
\Index(-L_{g'})=\Index(-L_g)+\Index(\proj_{Z(K_g)}\circ(-L_{g'})), \label{eq_index_formula}
\end{equation}
where the last index in \eqref{eq_index_formula} denotes the number of negative eigenvalues of \eqref{eq_proj_d2G}.
\end{Proposition}
\begin{proof}
Let $\delta > 0$ be a small constant whose value we will determine later.
In the following, we will frequently consider the spaces $L^2_{g,f}(M; S^2 T^*M)$, $L^2_{g',f'}(M; S^2 T^*M)$, $H^1_{g,f}(M; S^2 T^*M)$, $H^1_{g',f'}(M; S^2 T^*M)$.
For ease of notation, we will drop the parentheses and simply write $L^2_{g,f}, \lb L^2_{g',f'}, \lb H^1_{g,f}, \lb H^1_{g',f'}$.
Note that by Lemma~\ref{Lem_nearby_bounds}\ref{Lem_nearby_bounds_c} we have $L^2_{g,f} = L^2_{g',f'}, H^1_{g,f} = H^1_{g',f'}$ and the corresponding norms can be assumed to be equivalent.

The self-adjointness of the map \eqref{eq_proj_d2G} is a direct consequence of Lemma~\ref{Lem_Z_map}.
To establish the remaining statements, we need the following map, which relates the inner products $\langle \cdot, \cdot \rangle_{g,f}$ and $\langle \cdot, \cdot \rangle_{g',f'}$.

\begin{Claim} \label{Cl_Lambda}
There is a unique $(2,2)$-tensor $\Lambda$ on $M$ such that for any $u, v \in L^2_{g,f}$ we have
\begin{equation} \label{eq_Lambda_ggp}
   \langle \Lambda(u), v \rangle_{g',f'} = \langle u, v \rangle_{g,f}, 
\end{equation}where we define $(\Lambda(u))_{ij} = \Lambda^{kl}_{ij} u_{kl}$.
Moreover, we have
\begin{equation} \label{eq_Lambda_close_to_id}
   |\Lambda^{kl}_{ij} - \delta^k_i \delta^l_j | \leq C \delta, \qquad |\nabla^{g} \Lambda |, |\nabla^{g'} \Lambda | \leq C \delta, 
\end{equation}
for some constant $C < \infty$, which is independent of $(\gamma',\kappa')$ and $\Lambda$ induces the following bounded linear operators with bounded inverses:
\[ \Lambda:  L^2_{g,f} \lto L^2_{g,f}, \qquad \Lambda|_{H^1_{g,f}} : H^1_{g,f} \lto H^1_{g,f}. \]
\end{Claim}

\begin{proof}
Write $dg = \omega \, dg'$ for some $\omega \in C^1(M)$ and set
\[ \Lambda^{kl}_{ij} := \omega e^{f'-f} g'_{ia} g^{ak} g'_{jb} g^{bl}. \]
The identity \eqref{eq_Lambda_ggp} follows immediately.
The bound \eqref{eq_Lambda_close_to_id} is an immediate consequence of Lemma~\ref{Lem_nearby_bounds}\ref{Lem_nearby_bounds_a}.
The last statement of the claim is a direct consequence of \eqref{eq_Lambda_close_to_id}.
\end{proof}

By Corollary~\ref{Cor_index}\ref{Cor_index_a} there exists a $\langle\cdot,\cdot\rangle_{g,f}$-orthogonal decomposition
\begin{equation}\label{eq_L2_original_decomposition}
L^2_{g,f}(M;S^2 T^*M)=N_{g,f}\oplus K_{g,f}\oplus P_{g,f},
\end{equation}
where the spaces $N_{g,f}, K_{g,f} = K_g$ are contained in $H^1_{g,f} \cap C^\infty_{\loc}$  and finite dimensional with $\Index(-L_g) = \dim N_{g,f}$.
Moreover, $P_{g,f} \cap H^1_{g,f} \subset P_{g,f}$ is dense and the bilinear form $I_{g,f} : H^1_{g,f} \times H^1_{g,f} \to \IR$ with
\[ I_{g,f}(u_1, u_2) = \int_M \big( (\nabla u_1 \cdot_g \nabla u_2) - 2\Rm_g (u_1, u_2) \big) e^{-f} dg \]
is diagonal with respect to \eqref{eq_L2_original_decomposition}, negative definite on $N_{g,f}$, vanishes on $K_{g,f}$ and positive definite on $P_{g,f}$.
By Proposition~\ref{Prop_L2_theory}\ref{Prop_L2_theory_c}, we know that $N_{g,f}$ and $P_{g,f}$ are spanned by the negative and positive eigenvectors of $-L_g$
So there is a constant $c_0 > 0$ such that for all $u_1 \in N_{g,f}$ and $u_2 \in P_{g,f} \cap H^1_{g,f}$ we have
\begin{equation*} %
   I_{g,f} (u_1,u_1) \leq -c_0 \Vert u_1 \Vert_{L^2_{g,f}}^2, \qquad
I_{g,f} (u_2,u_2) \geq c_0 \Vert u_2 \Vert_{L^2_{g,f}}^2.  
\end{equation*}
Since, for some uniform constant $C_0 < \infty$,
\[ \Vert u_2 \Vert^2_{L^2_{g,f}} \leq I_{g,f}(u_2, u_2) + C_0 \Vert u_2 \Vert^2_{L^2_{g,f}}, \]
we obtain that
\[ I_{g,f} (u_2, u_2) \geq \frac{c_0}{C_0} \Vert u_2 \Vert^2_{H^1_{g,f}} - \frac{c_0}{C_0} I_{g,f}(u_2, u_2), \]
which implies that for some uniform constant $c'_0 > 0$
\[ I_{g,f} (u_2, u_2) \geq c'_0 \Vert u_2 \Vert^2_{H^1_{g,f}}. \]
Carrying out a similar argument for $I_{g,f}(u_1, u_1)$ (or using the fact that $N_{g,f}$ is finite dimensional) implies that there is a uniform constant $c_1 > 0$ such that for all $u_1 \in N_{g,f}$ and $u_2 \in P_{g,f} \cap H^1_{g,f}$
\begin{equation} \label{eq_definite_lower_bound}
   I_{g,f} (u_1,u_1) \leq -c_1 \Vert u_1 \Vert_{H^1_{g,f}}^2, \qquad
I_{g,f} (u_2,u_2) \geq c_1 \Vert u_2 \Vert_{H^1_{g,f}}^2.  
\end{equation}

\begin{Claim} \label{Cl_Lambda_splitting}
Assuming that $(\gamma', \kappa')$ is sufficiently close to $(\gamma,0)$,  we have the (not necessarily orthogonal) decomposition
\begin{equation}
L^2_{g',f'}(M;S^2 T^*M) =\Lambda(N_{g,f})\oplus Z(K_g)\oplus \Lambda(P_{g,f}).\label{eq_L2prime_decomposition}
\end{equation}
Moreover, the bilinear form $I_{g',f'} : H^1_{g',f'} \times H^1_{g',f'} \to \IR$ with
\[ I_{g',f'}(u_1, u_2) = \int_M \big( (\nabla u_1 \cdot_{g'} \nabla u_2) - 2\Rm_{g'} (u_1, u_2) \big) e^{-f'} dg' \]
 is negative-definite on $\Lambda(N_{g,f}) \subset H^1_{g',f'}$ and positive-definite on $\Lambda(P_{g,f})\cap H^1_{g',f'} = \Lambda(P_{g,f} \cap H^1_{g,f})$.
\end{Claim}
\begin{proof}

By Lemma \ref{Lem_nearby_bounds}\ref{Lem_nearby_bounds_c}, Lemma~\ref{Lem_Z_map}\ref{Lem_Z_map_a} and Claim~\ref{Cl_Lambda}, we may replace the summand $K_{g.f}$ in \eqref{eq_L2_original_decomposition} by $(\Lambda^{-1}\circ Z)(K_{g,f})$, though the resulting decomposition may no longer be orthogonal.
It now follows that
\[ L^2_{g',f'} = L^2_{g,f} = \Lambda( L^2_{g,f}) =  \Lambda(N_{g,f})\oplus Z( K_{g,f} ) \oplus \Lambda(P_{g,f}),\]
which proves, \eqref{eq_L2prime_decomposition}.

For the second statement, observe that by Lemma~\ref{Lem_nearby_bounds}\ref{Lem_nearby_bounds_a} and \eqref{eq_Lambda_close_to_id}
\[ \big| |\nabla^{g'} (\Lambda(u))|^2_{g'} - |\nabla^g u|^2_g \big|
\leq C \delta \big( |\nabla^{g'} u |^2_{g'} + |u|_{g'}^2 \big). \]
So we obtain, using again the factor $dg = \omega \, dg'$ from the proof of Claim~\ref{Cl_Lambda}, and a generic constant $C < \infty$,
\begin{multline*}
 \Big| I_{g',f'} (\Lambda(u),\Lambda(u)) - I_{g,f} (u,u) \Big| \\
\leq \int_M \Big( |\nabla^g u|^2_g \, \big|  e^{-f'+f} \omega^{-1} - 1 \big|  + C \big| {\Rm}_{g'} e^{-f'+f} \omega^{-1} - {\Rm}_g \big|_g |u|_g^2 \Big) e^{-f} dg + C\delta \Vert u \Vert^2_{H^1_{g',f'}}  \\
\leq C \delta \Vert u \Vert^2_{H^1_{g,f}}.
\end{multline*}
Combining this with \eqref{eq_definite_lower_bound} implies that for sufficiently small $\delta$ we have for any non-zero $u \in P_{g,f} \cap H^1_{g,f}$
\[ I_{g',f'} (\Lambda(u), \Lambda(u)) \geq  - C \delta \Vert u \Vert^2_{H^1_{g,f}} + c_1 \Vert u \Vert^2_{H^1_{g,f}} > 0. \]
So $I_{g',f'}$ is positive definite on
\[ \Lambda(P_{g,f} \cap H^1_{g,f}) = \Lambda(P_{g,f}) \cap H^1_{g,f} = \Lambda(P_{g,f}) \cap H^1_{g',f'}. \]
The same argument applies for $N_{g,f}$.
\end{proof}

We now consider the map $\proj_{g',f',Z(K_g)}\circ(-L_{g'}):Z(K_g)\rightarrow Z(K_g)$, which is self-adjoint with respect to $\langle \cdot, \cdot \rangle_{g',f'}$ and choose the $\langle \cdot, \cdot \rangle_{g',f'}$-orthogonal decomposition
\begin{equation*}
Z(K_g)=N^Z\oplus K^Z\oplus P^Z,%
\end{equation*}
where $N^Z$ and $P^Z$ are the spanned by the negative and positive spectral directions, respectively, and $N^Z$ is the kernel of 
$\proj_{g',f',Z(K_g)}\circ(-L_{g'})$.

\begin{Claim}\label{claim_Lg'_definiteness}
The  bilinear form $I_{g',f'}$ is negative definite on $(\Lambda(N_{g,f})\oplus N^Z) \subset H^1_{g',f'}$ and positive semidefinite on $( \Lambda(P_{g,f}) \oplus K^Z\oplus P^Z)\cap H^1_{g',f'}$.
\end{Claim}
\begin{proof}
Let $u_1\in N_{g,f} \oplus (P_{g,f} \cap H^1_{g',f'})$ and $v_2 = Z(u_2) \in Z(K_g)$ for $u_2 \in K_g$.
Recall that by Lemma~\ref{Lem_Z_map} we know that $v_2$ is of regularity $C^2$ and
\[ L_{g'} v_2 = L_{g'}(Z(u_2)) = (D_2G)_{(\gamma',\kappa')}(u_2) \in K_g \subset L^2_{g,f} = L^2_{g',f'}. \]
So by Proposition~\ref{Prop_L2_theory}\ref{Prop_L2_theory_a} and Claim~\ref{Cl_Lambda} we have
\[ I_{g',f'} ( \Lambda(u_1) , v_2 ) 
=  - \langle \Lambda(u_1), L_{g'} v_2 \rangle_{g',f'} 
= - \langle u_1, L_{g'} v_2 \rangle_{g,f} 
= 0 \]
So if $u_1 \in N_{g,f}$ and $v_2 \in N^Z$, then by Claim~\ref{Cl_Lambda_splitting}
\[ I_{g',f'} ( \Lambda(u_1) + v_2, \Lambda(u_1) + v_2 )
= I_{g',f'} ( \Lambda(u_1) , \Lambda(u_1)  )
+ I_{g',f'} (  v_2,  v_2 ) \leq 0, \]
with equality if and only if $u_1 = v_2 = 0$.
Similarly, if $u_1 \in P_{g,f} \cap H_{g,f}$, so that $\Lambda(u_1) \in \Lambda(P_{g,f} \cap H^1_{g,f}) = \Lambda(P_{g,f}) \cap H^1_{g',f'}$, and $v_2 = Z(u_2) \in K^Z \oplus P^Z \subset Z(K_g)$ for $u_2 \in K_g$, then
\[ I_{g',f'} ( \Lambda(u_1) + v_2, \Lambda(u_1) + v_2 )
= I_{g',f'} ( \Lambda(u_1) , \Lambda(u_1)  )
+ I_{g',f'} (  v_2,  v_2 ) \geq 0. \qedhere\]
\end{proof}
\medskip

The formula \eqref{eq_index_formula} now follows from Claim \ref{claim_Lg'_definiteness} together with Corollary \ref{Cor_index}\ref{Cor_index_b}.
\end{proof}
\bigskip

\section{The main argument} \label{sec_main_argument}
\subsection{Summary of previous results} \label{subsec_summary_main}
The following proposition summarizes the results from the previous sections.
We have added an asterisk next to the assertions that are only necessary for the analysis of the $\IZ$-degree.
Readers who are only interested in the $\IZ_2$-degree may disregard these assertions.

\begin{Proposition} \label{Prop_summary}
Fix an ensemble $(M,N,\iota)$, an integer $k^* \geq 30$ and an $\alpha \in (0,1)$ and consider the metric space $\MM := \MM^{k^*} (M,N, \iota)$ (see Definition~\ref{Def_MM_metric}) and the natural projection $\Pi : \MM \to \GenCONE^{k^*} (N)$.
Then there is an open neighborhood $\MMgrad := \MMgrad(M,N,\iota) \subset \MM' \subset \MM$ and a $C^{1,\alpha}$-Banach manifold structure on $\MM'$ such that the following is true:
\begin{enumerate}[label=(\alph*)]
\item  \label{Prop_summary_a} $\Pi$ is of regularity $C^{1,\alpha}$.
Moreover, around any $p \in \MM'$ there is a $C^{1,\alpha}$-coordinate chart with respect to which $\Pi$ is real-analytic.
\item  \label{Prop_summary_b} $\Pi |_{\MMgeqgrad} : \MMgeqgrad := \MMgeqgrad^{k^*}(M,N,\iota) \to \CONEgeq^{k^*}(N)$ is proper.
\item  \label{Prop_summary_c} For any $p \in \MM'$ the differential
\[ D\Pi_p : T_p \MM \longrightarrow T\GenCONE^{k^*}(N) \]
is Fredholm of index $0$.
\end{enumerate}
For any $C^{k^*-2}$-regular representative $(g,V=\nabla^g f,\gamma)$ of a point $p \in \MM_{\grad}$ the following is true:
\begin{enumerate}[label=(\alph*), start=4]
\item  \label{Prop_summary_d} We have $\dim \ker D\Pi_p = \dim K_g$, where $K_g$ is the kernel of
\[ L_g : C^{2,\alpha}_{-2,\nabla f}(M; S^2 T^*M) \lto C^{0,\alpha}_{-2}(M; S^2 T^*M). \]
\item \label{Prop_summary_e} There is a neighborhood $p \in U_g \subset \MM'$ and unique $C^{1,\alpha}$-maps
\[ \Xi_g : U_g \lto K_g, \qquad H_g : U_g \lto C^{k^*-20}_{-2,\nabla f}(M; S^2 T^* M), \]
where the values of $H_g$ are orthogonal to $K_g$ with respect to $\langle \cdot, \cdot \rangle_{g,f}$ and $\Xi_g (p) = H(p) = 0$ and such that for any $p' \in U_g$ the metric
\[ g_{p'} := g + \Xi_g (p') + H_g(p') + T (\Pi(p')) - T(\gamma) \]
solves the equation $Q_g (g_{p'}) = 0$ and there is a $C^{k^*-20}$-diffeomorphism $\psi_{p'} : M \to M$ with finite displacement such that
\[ \big( \psi_{p'}^* g_{p'}, \psi_{p'}^* (\nabla^g f - \DIV_g (g'_g) + \tfrac12 \nabla^g \tr_{g}( g_{p'})), \Pi(p') \big) \]
is a $C^{k^*-2}$-regular representative of $p' \in \MM$.
The map $\Xi_g$ induces the following isomorphism of vector spaces
\begin{equation} \label{eq_isom_ker_K}
 \ker (D\Pi_{p})  \xrightarrow{\quad\cong\quad} K_g, \qquad v \longmapsto (D\Xi_g)_p (v). 
\end{equation}
The neighborhood $U_g$ is also the domain of a $C^{1,\alpha}$-chart with respect to which $\Xi_g, H_g$ and $\Pi$ are real-analytic.
Moreover, using the setting of Proposition~\ref{prop_smooth_dependence} with $k=k^*-20$, the following map is a diffeomorphism onto its image
\begin{equation} \label{eq_PiXi}
 (\Pi,\Xi_g) : U_g \longrightarrow \{ G = 0 \} \subset U_1 \times U_2. 
\end{equation}
\item \label{Prop_summary_unique} Assertions~\ref{Prop_summary_a}, \ref{Prop_summary_d}, \ref{Prop_summary_e} determine the $C^{1,\alpha}$-Banach manifold structure on $\MM'$ uniquely in the following sense:
For any other open neighborhood $\MMgrad \subset \MM'' \subset \MM$ equipped with a $C^{1,\alpha}$-Banach manifold structure that satisfies these assertions, there is another open neighborhood $\MMgrad \subset \MM''' \subset \MM' \cap \MM''$ such that the restrictions of the $C^{1,\alpha}$-Banach manifold structures from $\MM'$ and $\MM''$ to $\MM'''$ agree.
Moreover, increasing the parameter $\alpha$ results in a refined $C^{1,\alpha}$-Banach manifold atlas on a neighborhood of $\MMgrad$.
\myitem[(g\,*)] \label{Prop_summary_f}  Let $\proj_{K_g} : C^0_{-1} (M; S^2 T^*M) \to K_g$ be the orthogonal projection onto $K_g$ with respect to $\langle \cdot, \cdot \rangle_{g,f}$.
Then the map, defined using the isomorphism \eqref{eq_isom_ker_K},
\[ S'_p : T \GenCONE^{k^*}(N) \longrightarrow \ker (D\Pi_p) \cong K_g, \qquad \gamma' \longmapsto \proj_{K_g} (L_g(T (\gamma'))) \]
is well defined, surjective and does not depend on the choice of the representative $(g, V, \gamma)$ of $p$ or the cutoff function $\eta$ used in the definition of the map $T$ (see Definition~\ref{Def_T}).
Moreover, $$\ker S'_p  = \Image D\Pi_p, $$ so $S'_g$ descends to an isomorphism (note that we take the quotient from the left for reasons that will become apparent later)
\[ S_p : \Image (D\Pi_p) \big\backslash T \GenCONE^{k^*}(N)   \longrightarrow \ker D\Pi_p \cong K_g. \]
\end{enumerate}
Next suppose that $Q \subset \GenCONE^{k^*}(N)$ is a finite dimensional $C^{1}$-submanifold that satisfies the transversality condition
\begin{equation} \label{eq_transversality_condition}
 \Image (D\Pi_p) + T_{\Pi(p)} Q = T\GenCONE^{k^*}(N) \qquad \text{for all} \quad p \in \Pi^{-1}(Q) \cap \MMgeqgrad. 
\end{equation}
Then there is an open neighborhood
\begin{equation} \label{eq_nbhd_U}
  \Pi^{-1}(Q) \cap \MMgeqgrad  \subset U \subset \MM' 
\end{equation}
such that:
\begin{enumerate}[label=(\alph*),start=8]
\item  \label{Prop_summary_g} $\Pi|_U$ is transverse to $Q$, so $P:=\Pi^{-1}(Q) \cap U$ is a $C^{1}$-submanifold of the same dimension $\dim Q$ and $\Pi^{-1} (Q) \cap \MMgrad \subset P$.
\item  \label{Prop_summary_h} If $Q$ is even real-analytic, then $U$ can be chosen such that
\begin{align}
 (\Pi|_P)^{-1}(\CONE^{k^*}(N)) &= \MMgrad \cap P, \label{eq_Piinv_grad_grad} \\
 (\Pi|_P)^{-1}(\CONEgeq^{k^*}(N)) &= \MMgeqgrad \cap P. \label{eq_Piinv_geqgrad_geqgrad}
\end{align}
\myitem[(j*)]  \label{Prop_summary_i} If $Q$ is oriented, then there is a unique orientation on $P$ such that for all $p \in \MM_{\grad} \cap P$ and any  $C^{k^*-2}$-regular  representative $(g,\nabla^g f, \gamma)$ of $p$, the isomorphisms 
\begin{align*}
 d(\Pi|_P)_p : T_p P \big/ \ker d(\Pi|_P)_p &\longrightarrow \Image (d(\Pi|_P)_p), \\
\qquad S_p :  \Image (d(\Pi|_P)_p) \big\backslash T_{\Pi(p)} Q  = \Image (D\Pi_p)\big\backslash  T \GenCONE^{k^*}(N)   &\longrightarrow \ker D\Pi_p = \ker (d(\Pi|_P)_p)   
\end{align*}
have the same orientation if $\idx (-L_g)$ is even and the opposite orientation if $\idx (-L_g)$ is odd.
Here we have chosen arbitrary orientations on $\ker (d(\Pi|_P)_p)$ and $\Image (d(\Pi|_P)_p)$ and the induced orientations\footnote{Our convention for the orientation of a quotient is the following: Suppose that $U \subset V$ is a subspace of a finite dimensional vector space $V$ and $\{ v_1, \ldots, v_k \} \subset V$ is a positively oriented basis of $V$.
If $\{ v_{k-l+1}, \ldots, v_{k} \}$ is a positively oriented basis of $U$, then $\{ v_1, \ldots, v_{k-l} \}$ descends to a positively oriented basis of $V/U$.
If $\{ v_{1}, \ldots, v_{l} \}$ is a positively oriented basis of $U$, then $\{ v_{l+1}, \ldots, v_{k} \}$ descends to a positively oriented basis of $U \backslash V$.

Note also that the orientation of a trivial vector space is simply given by the specification of a label `$+$' or `$-$'.
More specifically, $V/\{0\}$, $\{0\} \backslash V$ and $V$ have the same orientation if $\{0 \}$ is oriented via the label `$+$' and opposite orientation otherwise.
Moreover, if $U = V \subset V$ has the same orientation as $V$, then $V/U$ and $U\backslash V$ are oriented by the label `$+$'.
By contrast, if $U$ carries the opposite orientation, then these quotients are oriented by the label `$-$'.} on the quotient spaces (note that the statement is independent of these choices).
\end{enumerate}
\end{Proposition}

\begin{proof}
By Corollary~\ref{Cor_MM_is_Banach_manifold}, the maps $\{ \Phi_p : U_p \to \MM\}_{p \in \MM_{\grad}}$ from Proposition~\ref{Prop_Banach_manifold} form a $C^{1,\alpha}$-Banach manifold atlas on $\MM' := \bigcup_{p \in \MM_{\grad}} \Phi_p(U_p)$ that is compatible with the topology on $\MM$ and by Proposition~\ref{Prop_Banach_manifold}\ref{Prop_Banach_manifold_c} the maps $\Pi \circ \Phi_p$ are real-analytic and the Fredholm index of its differential is $0$.
This shows Assertions~\ref{Prop_summary_a}, \ref{Prop_summary_c}.
Assertion~\ref{Prop_summary_b} is a restatement of Proposition~\ref{Prop_properness}.

Assertion~\ref{Prop_summary_d} is a restatement of the last part of Proposition~\ref{Prop_Banach_manifold}\ref{Prop_Banach_manifold_c}.
Assertion~\ref{Prop_summary_e} follows from Propositions~\ref{Prop_Banach_manifold}\ref{Prop_Banach_manifold_d} and \ref{prop_smooth_dependence} by setting $H_g := H_p \circ \Phi_p^{-1}$ and $\Xi_g := \Xi_p \circ \Phi_p^{-1}$.
Assertion~\ref{Prop_summary_unique} follows from the fact that Assertion~\ref{Prop_summary_e} specifies unique charts near every point of $\MMgrad$.

For Assertion~\ref{Prop_summary_f} consider the diffeomorphism from \eqref{eq_PiXi}.
By Proposition~\ref{prop_smooth_dependence}\ref{prop_smooth_dependence_d} we have
\[ (D_1 G)_{(\gamma,0)}  = (D\Xi_g)_p \circ S'_p, \qquad (D_2 G)_{(\gamma,0)} = 0 \]
and the first differential is surjective onto $K_g$.
Since by \eqref{eq_PiXi} the map
\[ T_p \MM \lto \ker (DG)_{(\gamma,0)} = \ker (D_1 G)_{(\gamma,0)} \times K_g, \qquad v \longmapsto (D\Pi_p(v), D(\Xi_g)_p(v)) \]
is an isomorphism, we obtain that $\Image D\Pi_p = \ker (D_1 G)_{(\gamma,0)} = \ker S_p$.
To see the independence statement, consider first a different map $T' : T\GenCONE^{k^*}(N) \to C^{k^*}(M; S^2 T^*M)$; i.e. $T(\gamma') - T'(\gamma')$ has compact support.
Then for any $\kappa' \in K_g$ we have by Lemma~\ref{Lem_linear_identities}\ref{Lem_linear_identities_f}
\[ \big\langle \proj_{K_g} (L_g(T(\gamma'))), \kappa' \big\rangle_{g,f}  -\big\langle \proj_{K_g} (L_g(T'(\gamma'))), \kappa' \big\rangle_{g,f}
= \big\langle L_g(T(\gamma')- T'(\gamma')), \kappa' \big\rangle_{g,f} = 0. \]
So $\proj_{K_g} (L_g(T(\gamma'))) = \proj_{K_g} (L_g(T'(\gamma')))$.
Next, consider another $C^{k^*-2}$-representative $(g',V',\gamma)$ of $p$, i.e., there is a $C^1$-diffeomorphism $\phi : M \to M$ that agrees with the identity on the complement of a compact subset such that $g' = \phi^* g$ and $V' = \phi^* V$.
Since both metrics $g, g'$ are $C^{k^*-2}$, this implies that $\phi$ is $C^{k^*-1}$.
By Assertion~\ref{Prop_summary_e} the induced maps $\Xi_g$ and $\Xi_{g'}$ are related as follows
\[ \Xi_{g'}(v) = \phi^* \big( \Xi_g (v) \big) \qquad \text{for all} \quad v \in T_p \MM. \]
So it remains to show that for all $\gamma' \in T\GenCONE^{k^*}(N)$
\begin{equation} \label{eq_ggp_indep}
 \phi^* \big( \proj_{K_g} (L_g(T(\gamma')) \big) = \proj_{K_{g'}} (L_{g'}(T(\gamma')). 
\end{equation}
To see this, let $\kappa' \in K_{g}$ be arbitrary and observe that
\begin{align*}
 \big\langle \phi^* \big( \proj_{K_g} (L_g(T(\gamma'))) \big), \phi^*\kappa' \big\rangle_{g',f'}
&= \big\langle  \proj_{K_g} (L_g(T(\gamma'))), \kappa' \big\rangle_{g,f}
= \big\langle  L_g(T(\gamma')), \kappa' \big\rangle_{g,f} \\
 \big\langle   \proj_{K_{g'}} (L_{g'}(T(\gamma'))) \big), \phi^*\kappa' \big\rangle_{g',f'}
&= \big\langle  L_{g'}(T(\gamma')), \phi^*\kappa' \big\rangle_{g',f'} 
\end{align*}
Due to the independence of the definition of the map $T$, we may assume that $g = g'$, $V = V'$ and $\phi = \id_M$ on the support of $T(\gamma')$.
This shows that both right-hand sides are equal and therefore that \eqref{eq_ggp_indep} holds, concluding the proof of Assertion~\ref{Prop_summary_f}.

Consider now the submanifold $Q$.
\begin{Claim}\label{Cl_transverse}
For any $p\in \Pi^{-1}(Q)\cap\MMgeqgrad$ there exists an open neighborhood $U\subset\MM'$ containing $p$ such that \eqref{eq_transversality_condition} continues to hold for all $p' \in \Pi^{-1}(Q) \cap U$.
\end{Claim}
\begin{proof}
Fix a complementary subspace $A$ to $T_{\Pi(p)}Q$ in $T\GenCONE^{k^*}(N)$, and define $\proj_{T_{\Pi(p)}Q}:T\GenCONE^{k^*}(N)\rightarrow T_{\Pi(p)}Q$ as the projection with respect to this decomposition. Also let $B\subset T_{\Pi(p)}Q$ be a complementary subspace to $\Image(D\Pi_p)$ in $T\GenCONE^{k^*}(N)$. We then take a neighborhood $V$ about $\Pi(p)$ such that $\proj_{T_{\Pi(p)}Q}|_{T_{q}Q}:T_{q}Q\rightarrow T_{\Pi(p)}Q$ is an isomorphism for any $q\in Q\cap V$, and define the continuous map $E:(Q\cap V)\times B\rightarrow T\GenCONE^{k^*}(N)$ by $(q,b)\mapsto (\proj_{T_{\Pi(p)}Q}|_{T_{q}Q})^{-1}(b)$. Note that $E(\Pi(p),\cdot)$ is the identity map on $B$.

We can similarly find a neighborhood $\td U\subset \Pi^{-1}(V)\subset\MM'$ about $p$ and a continuous map $\td E:\td U\times T_p\MM'\rightarrow T\GenCONE^{k^*}(N)\times K_g$ such that $\td E(p',\cdot)$ is an isomorphism between $T_{p}\MM'$ and $T_{p'}\MM'$ for all $p'\in \td U$, such that $\td E(p,\cdot)$ is the identity on $T_p\MM'$. Let $H$ be a complementary subspace to $\ker(D\Pi_p)$ in $T_p\MM'$.

Now consider the continuous map $\hat{E}:(\Pi^{-1}(Q)\cap \td U)\times H\times B\rightarrow T\GenCONE^{k^*}(N)$ defined by $(p',x,b)\mapsto (D\Pi_{p'})(\td E(p',x))+E(\Pi(p'),b)$. Since $\hat{E}(p,\cdot,\cdot):H\times B\rightarrow T\GenCONE^{k^*}(N)$ is invertible, there exists a neighborhood $U\subset\td U$ about $p$ so that $\hat{E}(p',\cdot,\cdot)$ is also invertible for all $p'\in \Pi^{-1}(Q)\cap U$. In particular \eqref{eq_transversality_condition} holds for all $p'\in \Pi^{-1}(Q)\cap U$.
%and $\td E(p,\cdot)$ is the identity map on $\Image(D\Pi_p)$. Now consider the continuous map $\hat{E}:(\Pi^{-1}(Q)\cap \td U)\times\Image(D\Pi_p)\times B\rightarrow T_{\Pi(p)}\GenCONE^{k^*}(N)$ defined by $(p',x,b)\mapsto \td E(p',x)+E(\Pi(p'),b)$. Since $\hat{E}(p,\cdot,\cdot)$ is invertible, there exists a neighborhood $U\subset\td U$ about $p$ so that $\hat{E}(p',\cdot,\cdot)$ is also invertible for all $p'\in \Pi^{-1}(Q)\cap U$. In particular \eqref{eq_transversality_condition} holds for all $p'\in \Pi^{-1}(Q)\cap U$.We can similarly find a neighborhood $\td U\subset \Pi^{-1}(V)\subset\MM'$ about $p$ and a continuous map $\td E:\td U\times\Image(D\Pi_p)\rightarrow T_{\Pi(p)}\GenCONE^{k^*}(N)$ such that $\td E(p',\cdot)$ is an isomorphism between $\Image(D\Pi_p)$ and $\Image(D\Pi_{p'})$ for all $p'\in \td U$ and $\td E(p,\cdot)$ is the identity map on $\Image(D\Pi_p)$. Now consider the continuous map $\hat{E}:(\Pi^{-1}(Q)\cap \td U)\times\Image(D\Pi_p)\times B\rightarrow T_{\Pi(p)}\GenCONE^{k^*}(N)$ defined by $(p',x,b)\mapsto \td E(p',x)+E(\Pi(p'),b)$. Since $\hat{E}(p,\cdot,\cdot)$ is invertible, there exists a neighborhood $U\subset\td U$ about $p$ so that $\hat{E}(p',\cdot,\cdot)$ is also invertible for all $p'\in \Pi^{-1}(Q)\cap U$. In particular \eqref{eq_transversality_condition} holds for all $p'\in \Pi^{-1}(Q)\cap U$.
\end{proof}

Then Assertion~\ref{Prop_summary_g} follows immediately; the fact that $P$ has the same dimension as $Q$ is a consequence of the fact that $D\Pi_p$ has Fredholm index $0$ for all $p \in P$.
Note that the topology on $P$ is paracompact, because it is induced by the metric on $\MM$.

For Assertion~\ref{Prop_summary_h} we first claim that $Q \cap \CONE^{k^*}(N)$ is a real-analytic subvariety of $Q$, i.e., that it is a level set of a real-analytic function $F_0 : Q \to \IR$.
To see this, note that the non-conical directions in $\GenCONE^{k^*}(N)$ define a closed complementary subspace $V \subset \GenCONE^{k^*}(N)$ to $\CONE^{k^*}(N)$.
Let $\langle \cdot, \cdot \rangle$ be an inner product on $V$ that is continuous with respect to the induced Banach norm.
Define $F_0 : Q \to \IR$ by $F_0(\gamma) := \langle \proj_V \gamma, \proj_V \gamma \rangle$.
Then $F_0$ is real-analytic and $F_0^{-1}(0) = Q \cap \CONE^{k^*}(N)$, which shows that $Q \cap \CONE^{k^*}(N)$ is a real-analytic subvariety of $Q$.

Let now $p = [(g, V = \nabla^g f, \gamma)] \in \Pi^{-1}(Q) \cap \MMgrad$ for some $C^{k^*-2}$-regular representative $(g, V = \nabla^g f, \gamma)$, which exists by Lemma~\ref{Lem_MM_regular_rep}\ref{Lem_MM_regular_rep_a}.
Consider a local coordinate chart on a neighborhood $p \in W_p \subset U$ with respect to which $\Pi$ and the map $H_g$ from Assertion~\ref{Prop_summary_e} are real-analytic.
Then $F_p := F_0 \circ \Pi : W_p \to \IR$ is real analytic and $(\Pi|_P)^{-1} (\CONE^{k^*}(N)) \cap W_p = F_p^{-1}(0)$.
By Proposition~\ref{Prop_variety_loc_conn} there is a neighborhood $p \in W'_p \subset W_p$ such that any $p' \in (\Pi|_P)^{-1} (\CONE^{k^*}(N)) \cap W'_p$ can be connected with $p$ within $F_p^{-1}(0) \subset W_p$ via a piecewise real-analytic arc $\sigma : [0,1] \to F_p^{-1}(0)$.
Then $s \mapsto g(s) := g + H_g (\sigma(s)) + T (\Pi(\sigma(s)) - T (\gamma)$ defines a piecewise $C^1$ family of expanding soliton metrics that are all asymptotic to a cone.
Since $g(0) = g$ is the metric of a \emph{gradient} expanding soliton, we obtain, using Corollary~\ref{Cor_pres_gradientness}, that all $g(s)$ are \emph{gradient} expanding soliton metrics.
Therefore $p' \in \MM_{\grad}$.
It follows that $(\Pi|_P)^{-1} (\CONE^{k^*}(N)) \cap W'_p \subset \MM_{\grad}$.
The identity \eqref{eq_Piinv_grad_grad} in Assertion~\ref{Prop_summary_h} now follows after replacing $P$ with $P' := \bigcup_{p \in \MM_{\grad} \cap P} W'_p$ and $U$ with a neighborhood $\Pi^{-1}(Q) \cap \MM_{\grad} \subset U' \subset U$ such that $P' = P \cap U'$.
The identity \eqref{eq_Piinv_geqgrad_geqgrad} follows from \eqref{eq_Piinv_grad_grad} via Lemma~\ref{Lem_Rgeq0_CONE_implies_MM}.

It remains to prove Assertion~\ref{Prop_summary_i}.
Note first that the assertion itself defines a unique orientation at every tangent space $T_p P$ for $p \in \MM_{\grad} \cap P$.
In the following we will show that this orientation depends continuously on $p$.
Once this is achieved, the assertion follows after possibly shrinking the neighborhood $U$ in \eqref{eq_nbhd_U}.

Fix $p \in P$ and choose a $C^{k^*-2}$-regular representative $(g,V =\nabla^g f, \gamma)$,  which exists by Lem\-ma~\ref{Lem_MM_regular_rep}\ref{Lem_MM_regular_rep_a}.
Without loss of generality, we may normalize $f$ such that
\[ |\nabla^g f|_g^2 + R_g = -f. \]
Let $K_g$ be the kernel of $L_g$ and use the map $(\Pi, \Xi_g)$ of Assertion~\ref{Prop_summary_e} to identify a neighborhood of $p$ with a neighborhood $U'$ of $(\gamma,0)$ in $\{ G = 0 \} \subset U_1 \times U_2$ from Proposition~\ref{prop_smooth_dependence}.
Recall that within this identification, the map $\Pi$ is just the projection onto the first factor.
Let $p' = (\gamma',\kappa') \in U' \cap \MM_{\grad} \cap P$ be an arbitrary point, so $\gamma' \in \CONE^{k^*}(N)$, and set
\[ g' := g_{p'}, \qquad V' := \nabla^g f - \DIV_g (g') + \tfrac12 \nabla^g \tr_g (g'), \qquad \psi := \psi_{p'} \]
from Assertion~\ref{Prop_summary_e}.
Then $(\psi^* g', \psi^* V', \gamma')$ is a $C^{k^*-2}$-regular representative of $p'$ and since $p' \in \MM_{\grad}$, we have $V' := \nabla^{g'} f'$ for some potential function $f' \in C^1(M)$, which we may again normalize such that
\[ |\nabla^{g'} f'|_{g'}^2 + R_{g'} = -f'. \]
Throughout the remainder of this proof we will assume $p'$ to be sufficiently close to $p$ in order to carry out our estimates.
To be precise, let $\delta > 0$ be a small constant, whose value we will determine later, and suppose that $p'$ is chosen sufficiently close to $p$ such that the conclusions of Lemmas~\ref{Lem_nearby_bounds}, \ref{Lem_Z_map}, \ref{Lem_Z_is_change_of_coo} and Proposition~\ref{prop_index_formula} hold for $\delta$.

Let $K_{g'}$ be the kernel of the operator
\[ L_{g'} : C^{2,\alpha}_{g',-2,\nabla^{g'} f'}(M: S^2 T^*M) \to C^{0,\alpha}_{g',-2}(M: S^2 T^*M). \]
Assertion~\ref{Prop_summary_e} applied to $p' = [(\psi^* g', \psi^* V', \gamma')]$ yields an isomorphism of $\psi^* K_{g'} = K_{\psi^* g'}$ with $\ker D\Pi_{p'}$.
Since in our identification $\Pi$ is the projection of $\{ G = 0 \} \subset U_1 \times U_2$ to the first factor, we may view $\ker D\Pi_{p'}$ also as the subspace $\ker D_2 G_{(\gamma',\kappa')} \subset K_g$.
Combining this identification with the isomorphism above and applying Lemma~\ref{Lem_Z_is_change_of_coo},  yields the isomorphism 
\begin{equation*} %
 \psi^* \circ Z_0 := \psi^* \circ Z |_{ \ker D_2 G_{(\gamma',\kappa')} } : \ker D_2 G_{(\gamma',\kappa')} \longrightarrow \psi^* K_{g'} = K_{\psi^* g'},  
\end{equation*}
which is given by the map $Z:= Z_{(\gamma',p')}$ from Lemma~\ref{Lem_Z_map} restricted to $\ker D_2 G_{(\gamma',\kappa')}$ and composed with the pullback $\psi^*$.
Then the map $S'_{p'}$ from Assertion~\ref{Prop_summary_f} can be characterized by the fact that for any $\td\kappa \in \ker D_2 G_{(\gamma',\kappa')}$ and $\td\gamma \in T\GenCONE^{k^*}(N)$ we have
\begin{equation} \label{eq_Sppp}
 \big\langle S'_{p'} (\td\gamma), \psi^* (Z_0(\td\kappa)) \big\rangle_{\psi^* g', \psi^* f'}
= \big\langle  L_{\psi^* g'} (T(\td\gamma)) , \psi^* (Z(\td\kappa) \big\rangle_{\psi^* g', \psi^* f'}
\end{equation}

\begin{Claim} \label{Cl_Sp}
For any $\td\kappa \in \ker D_2 G_{(\gamma',\kappa')}$ and $\td\gamma \in T\GenCONE^{k^*}(N)$ we have
\[ \big| \langle S'_{p'} (\td\gamma), \psi^* (Z_0( \td\kappa ))\rangle_{\psi^* g',\psi^* f'} - \langle S'_{p} (\td\gamma), \td\kappa \rangle_{g,f} \big|
\leq \delta \Vert \td\gamma \Vert \, \Vert \td\kappa \Vert_{L^2_{g,f}}. \]
\end{Claim}

\begin{proof}
This follows from Lemma~\ref{Lem_Z_map}\ref{Lem_Z_map_b}, \eqref{eq_Sppp}. %
\end{proof}

Let $V_\gamma \subset T_\gamma Q$ be a complement of $\Image (d(\Pi|_P)_p) \subset T_\gamma Q$ and extend $V_\gamma$ to a continuous, oriented distribution $\{ V_{\gamma''} \subset T_{\gamma''} Q \}$ for $\gamma'' \in Q$ near $\gamma$.
Next, choose linearly independent continuous local vector fields $w_{1}, \ldots, w_{l}$ on $P$ near $p$ such that 
$\{ d(\Pi|_P)_p(w_1|_p), \ldots, d(\Pi|_P)_p(w_l|_p) \}$ is a basis of $\Image(d(\Pi|_P)_p)$.
Then for $p'$ sufficiently close to $p$ we have
\begin{align}
 T_\gamma Q &= \spann \big\{ d(\Pi|_P)_p(w_1|_p), \ldots, d(\Pi|_P)_p(w_l|_p) \big\} \oplus V_{\gamma}, \label{eq_TgQ} \\
 T_{\gamma'} Q &= \spann \big\{ d(\Pi|_P)_{p'}(w_1|_{p'}), \ldots, d(\Pi|_P)_{p'}(w_l|_{p'}) \big\} \oplus V_{\gamma'} \label{eq_TgpQ} 
\end{align}
and the set of $l$ vectors in both identities form bases of their span.

Choose a basis $\{ \td\kappa_1, \ldots, \td\kappa_k \} \subset K_{g}$ such that  we have the $\langle \cdot, \cdot \rangle_{g',f'}$-orthonormality condition
\begin{equation} \label{eq_ONB_condition_Z}
 \big\langle Z(\td\kappa_i), Z(\td\kappa_j) \big\rangle_{g',f'} = \delta_{ij}, 
\end{equation}
and such that for some $\lambda_1, \ldots, \lambda_k \in \IR$ we have
\begin{equation} \label{eq_spec_dec_Z}
 -\big\langle Z(\td\kappa_i), D_2 G_{(\gamma',\kappa')} (\td\kappa_j) \big\rangle_{g',f'} = \lambda_i \delta_{ij}. 
\end{equation}
(Recall that by Proposition~\ref{prop_index_formula} or \eqref{eq_LZ_d2G} in Lemma~\ref{Lem_Z_map} the form on the left-hand side of \eqref{eq_spec_dec_Z} is symmetric, so both symmetric forms \eqref{eq_ONB_condition_Z} and \eqref{eq_spec_dec_Z} can be simultaneously diagonalized.)
After possibly permuting indices, we may assume moreover that $\lambda_{m+1} = \ldots = \lambda_{k} = 0$ and that $\lambda_{1}, \ldots, \lambda_m \neq 0$ for some $m \in \{ 0, \ldots, k \}$.
So by Lemma~\ref{Lem_Z_map}\ref{Lem_Z_map_a} we obtain that if $\delta$ is sufficiently small, then
\[  \ker D_2 G_{(\gamma',\kappa')} = \spann \{ \td\kappa_{m+1}, \ldots, \td\kappa_{k} \}  \]
and
\[ \proj_{g',f', Z(K_g)} \big( D_2 G_{(\gamma',\kappa')} (\td\kappa_i) \big) = -\lambda_i Z(\td\kappa_i), \]
where the projection is taken with respect to the inner product $\langle \cdot, \cdot \rangle_{g',f'}$.

Next, we claim that for sufficiently small $\delta$ we can choose bases $\{ \td\gamma_{1}, \ldots, \td\gamma_k \} \subset V_{\gamma}$ and $\{ \td\gamma'_{1}, \lb \ldots,\lb \td\gamma'_k \} \subset V_{\gamma'}$ such that
\begin{equation} \label{eq_choice_td_gammai}
 S'_p (\td\gamma_i) = D_1 G_{(\gamma,0)} (\td\gamma_i) = \td\kappa_i, \qquad \proj_{g',f',Z(K_g)} \big( D_1 G_{(\gamma',\kappa')} (\td\gamma'_i) \big) 
= Z(\td \kappa_i). 
\end{equation}
To see this, note that $D_1 G_{(\gamma,0)} |_{V_\gamma} : V_\gamma \to K_g$ is an isomorphism since $D_1 G_{(\gamma,0)} |_{T_\gamma Q}$ is surjective and its kernel is given by the first factor in \eqref{eq_TgQ}; so the first equation can be solved for $\td\gamma_i \in V_\gamma$.
The second equation in \eqref{eq_choice_td_gammai} is equivalent to 
\begin{equation} \label{eq_choice_gammapi_equ}
  D_1 G_{(\gamma',\kappa')} (\td\gamma'_i)  
= \big( \proj_{g',f',Z(K_g)}\big|_{K_g} \big)^{-1} \big( Z(\td\kappa_i) \big), 
\end{equation}
By Lemma~\ref{Lem_Z_map}\ref{Lem_Z_map_a}, if $\delta$ is sufficiently small and $p'$ is sufficiently close to $p$, then $D_1 G_{(\gamma',\kappa')}$ is sufficiently close to $D_1 G_{(\gamma,0)}$ and the right-hand side of \eqref{eq_choice_gammapi_equ} is defined.
Therefore, for $p'$ sufficiently close to $p$, the equation \eqref{eq_choice_gammapi_equ} can be solved for $\td\gamma'_i \in V_{\gamma'}$.
The same analysis implies moreover that for $p'$ sufficiently close to $p$ both bases $\{ \td\gamma_{1}, \ldots, \td\gamma_k \} \subset V_{\gamma}$ and $\{ \td\gamma'_{1}, \ldots, \td\gamma'_k \} \subset V_{\gamma'}$ have the same orientation and that we have
\begin{equation} \label{eq_gamma_gammap}
 \Vert \td\gamma_i \Vert \leq C \Vert \td\kappa_i \Vert_{L^2_{g,f}}, \qquad \Vert \td\gamma_i - \td\gamma'_i \Vert \leq C\delta \Vert \td\kappa_i \Vert_{L^2_{g,f}} 
\end{equation}
for some constant $C < \infty$, which is independent of $p'$.

Summarizing our construction, we have chosen the following bases, which have the same orientation:
\begin{align}
 T_\gamma Q &= \spann \big\{ d(\Pi|_P)_p(w_1|_p), \ldots, d(\Pi|_P)_p(w_l|_p), \td\gamma_1, \ldots, \td\gamma_k \big\} \label{eq_basis_TgammaQ} \\
 T_{\gamma'} Q &= \spann \big\{ d(\Pi|_P)_{p'}(w_1|_{p'}), \ldots, d(\Pi|_P)_{p'}(w_l|_{p'}), \td\gamma'_1, \ldots, \td\gamma'_k \big\} \label{eq_basis_TgammapQ}
\end{align}
Moreover, recalling that we view a neighborhood of $p \in P$ as a subset of $\{ G = 0 \} \subset U_1 \times U_2$, we have the following bases, which have the same orientation with respect to the local manifold structure on $P$ (which is a priori unrelated to the definition of the orientation in Assertion~\ref{Prop_summary_i}):
\begin{align}
 T_p P &= \spann \big\{ w_1|_p, \ldots, w_l|_p, (0,\td\kappa_1), \ldots, (0,\td\kappa_k) \big\} \label{eq_orient_TpP} \\
 T_{p'} P &= \spann \big\{ w_1|_{p'}, \ldots, w_l|_{p'}, (\lambda_1 \td\gamma'_1, \td\kappa_1), \ldots, (\lambda_m \td\gamma'_m, \td\kappa_m), (0,\td\kappa_{m+1}), \ldots, (0,\td\kappa_k) \big\} \label{eq_orient_TppP}
\end{align}
We also have
\begin{align}
 \ker (d(\Pi|_P)_p) &= \spann \{ (0,\td\kappa_1), \ldots, (0,\td\kappa_k) \},  \label{eq_ker_dPi_p} \\
 \ker (d(\Pi|_P)_{p'}) &= \spann \{ (0,\td\kappa_{m+1}), \ldots, (0,\td\kappa_k) \} 
\end{align}
and
\begin{align}
 \Image (d(\Pi|_P)_p) &= \spann \big\{ d(\Pi|_P)_p(w_1|_p), \ldots, d(\Pi|_P)_p(w_l|_p) \big\}, \\
 \Image (d(\Pi|_{P})_{p'}) &= \spann \big\{ d(\Pi|_P)_p(w_1|_p), \ldots, d(\Pi|_P)_p(w_l|_p), \td\gamma'_1, \ldots, \td\gamma'_m \big\}. \label{eq_Image_dPi_pp}
\end{align}
After possibly flipping the orientation on $Q$, which results in a flip of the orientations on all tangent spaces of $P$, we may assume that the bases in \eqref{eq_basis_TgammaQ}, \eqref{eq_basis_TgammapQ} are positively oriented.
Choose the orientations on the spaces \eqref{eq_ker_dPi_p}--\eqref{eq_Image_dPi_pp} given by the indicated bases.
Then if we temporarily equip $T_p P$ and $T_{p'} P$ with the orientation given by the bases \eqref{eq_orient_TpP}, \eqref{eq_orient_TppP}, then the maps
\begin{align*}
 d(\Pi|_P)_p : T_p P/\ker (d(\Pi|_P)_p) &\lto \Image (d(\Pi|_P)_p), \\
d(\Pi|_P)_{p'} : T_{p'} P/\ker( d(\Pi|_P)_{p'}) &\lto \Image (d(\Pi|_P)_{p'}) 
\end{align*}
have the same orientation if and only if $\lambda_1 \cdots \lambda_m > 0$, which by Proposition~\ref{prop_index_formula} is equivalent to the condition that $\Index(-L_g)$ and $\Index(-L_{g'})$ have the same parity.
So it remains to show that the maps
\begin{align}
 S_p :  \Image (d(\Pi|_P)_p) \big\backslash T_\gamma Q &\lto \ker (d(\Pi|_P)_p),  \label{eq_Sp_Imker} \\
  Z_0^{-1} \circ \psi_* \circ S_{p'} :  \Image (d(\Pi|_P)_{p'}) \big\backslash T_{\gamma'} Q &\lto \ker (d(\Pi|_P)_{p'}) \label{eq_Spp_Imker} 
\end{align}
have the same orientation.
The map \eqref{eq_Sp_Imker} is orientation-preserving by definition (see \eqref{eq_choice_td_gammai}), so it remains to show that  $Z_0^{-1} \circ \psi_* \circ S_{p'}$ is orientation-preserving as well.
To see this, recall that
\[ \ker D_2 G_{(\gamma',\kappa')} \times \ker D_2 G_{(\gamma',\kappa')} \to \IR, \qquad
(\td\kappa', \td\kappa'') \mapsto \big\langle Z_0 (\td\kappa'), Z_0 (\td\kappa'') \big\rangle_{g',f'} \]
is an inner product on $\ker D_2 G_{(\gamma',\kappa')}$ and consider the matrix given by the inner products 
\begin{multline} \label{eq_matrix_should_pd}
 \big\langle Z_0 ((Z_0^{-1} \circ \psi_* \circ S'_{p'}) (\td\gamma'_i)), Z_0(\td\kappa_j) \big\rangle_{g',f'} = \big\langle  \psi_* S'_{p'} (\td\gamma'_i)), Z_0(\td\kappa_j) \big\rangle_{g',f'}  \\
  = \big\langle   S'_{p'} (\td\gamma'_i)), \psi^*( Z_0(\td\kappa_j) ) \big\rangle_{\psi^* g',\psi^* f'}  
\end{multline}
for $i,j = m+1, \ldots, k$.
Using Claim~\ref{Cl_Sp}, \eqref{eq_choice_td_gammai} and \eqref{eq_gamma_gammap}, we have 
\begin{align}
 \big| \big\langle S'_{p'} (\td\gamma'_i), \psi^* ( Z_0(\td\kappa_j)) \big\rangle_{g',f'} - \delta_{ij} \big|
&\leq \big| \big\langle S'_{p'} (\td\gamma'_i), \psi^*(Z_0(\td\kappa_j)) \big\rangle_{g',f'} - \langle S'_{p} (\td\gamma'_i), \td\kappa_j \rangle_{g,f} \big| \notag \\
&\qquad  
+ \big| \langle S'_{p} (\td\gamma'_i - \td\gamma_i), \td\kappa_j \rangle_{g,f} \big|  + \big| \langle S'_{p} (\td\gamma_i), \td\kappa_j \rangle_{g,f} - \delta_{ij} \big| \notag \\
&\leq \delta \Vert \td\gamma'_i \Vert \, \Vert \td\kappa_j \Vert_{L^2_{g,f}}  +C \delta \Vert \td\kappa_i \Vert_{L^2_{g,f}}  \, \Vert \td\kappa_j \Vert_{L^2_{g,f}} +  \big| \langle \td\kappa_i, \td\kappa_j \rangle_{g,f} - \delta_{ij} \big| \notag \\
&\leq C\delta \Vert \td\kappa_i \Vert_{L^2_{g,f}}  \, \Vert \td\kappa_j \Vert_{L^2_{g,f}} +  \big| \langle \td\kappa_i, \td\kappa_j \rangle_{g,f} - \delta_{ij} \big|. \label{eq_Sp_matrix_bound}
\end{align}
Note here that $C$ is a generic constant that may depend on the map $S'_p = D_1 G_{(\gamma,0)}$, which is bounded.
By Lemma~\ref{Lem_Z_map}\ref{Lem_Z_map_a} and \eqref{eq_ONB_condition_Z} we have
\[ \Vert \td\kappa_i \Vert_{L^2_{g,f}} 
\leq \Vert Z(\td\kappa_i) \Vert_{L^2_{g,f}} + \delta \Vert \td\kappa_i \Vert_{L^2_{g,f}} 
\leq (1+\delta)\Vert Z(\td\kappa_i) \Vert_{L^2_{g',f'}} + \delta \Vert \td\kappa_i \Vert_{L^2_{g,f}} 
=  1+\delta + \delta \Vert \td\kappa_i \Vert_{L^2_{g,f}}, \]
which shows that  for sufficiently small $\delta$
\[ \Vert \td\kappa_i \Vert_{L^2_{g,f}} \leq 1 + 10 \delta \]
Similarly, for $i \neq j$ we have
\begin{multline*}
 \big| \langle \td\kappa_i, \td\kappa_j \rangle_{g,f} \big|
\leq \big| \langle Z(\td\kappa_i), Z(\td\kappa_j) \rangle_{g,f} \big| + C\delta \Vert \td\kappa_i \Vert_{L^2_{g,f}} \Vert \td\kappa_j \Vert_{L^2_{g,f}} \\
\leq \big| \langle Z(\td\kappa_i), Z(\td\kappa_j) \rangle_{g',f'} \big| + C\delta \Vert \td\kappa_i \Vert_{L^2_{g,f}} \Vert \td\kappa_j \Vert_{L^2_{g,f}} = C\delta. 
\end{multline*}
Thus the right-hand side of \eqref{eq_Sp_matrix_bound} is $\leq C\delta$.
So for sufficiently small $\delta$ the matrix \eqref{eq_matrix_should_pd} is positive definite, which shows that \eqref{eq_Spp_Imker} is orientation-preserving.
This finishes the proof.
\end{proof}

The following lemma shows that the transversality condition \eqref{eq_transversality_condition} from Proposition~\ref{Prop_summary} can be easily satisfied for a good choice of $Q$.

\begin{Lemma} \label{Lem_transverse_Q_exists}
Consider the setting of Proposition~\ref{Prop_summary} and let $\gamma \in \CONEgeq^{k^*}(N)$.
Then there is a finite dimensional linear subspace $V_0 \subset T\GenCONE^{k^*}(N)$ such that the following transversality condition holds
\begin{equation} \label{eq_VS_transversal}
 \Image(D\Pi_p) + V_0 = T\GenCONE^{k^*}(N) \qquad \text{for all} \quad p \in \Pi^{-1}(\gamma) \cap \MMgeqgrad. 
\end{equation}
Moreover, if $Q' \subset \GenCONE^{k^*}(N)$ is a finite dimensional $C^1$-submanifold with $\gamma \in Q'$ such that $V_0 := T_\gamma Q'$ satisfies \eqref{eq_VS_transversal}, then we can find an open neighborhood $\gamma \in Q \subset Q'$ such that the transversality condition \eqref{eq_transversality_condition} holds.
Lastly, for any $\gamma \in \CONEgeq^{k^*}(N)$ there is a finite dimensional, real-analytic, orientable submanifold $Q \subset \GenCONE^{k^*}(N)$ such that \eqref{eq_transversality_condition} holds.
\end{Lemma}

\begin{proof}
We recall that $\Pi^{-1} (\gamma) \cap \MMgeqgrad$ is compact by Proposition~\ref{Prop_summary}\ref{Prop_summary_b}.
By Assertion~\ref{Prop_summary_c} of the same proposition, for any $p \in \Pi^{-1} (\gamma) \cap \MMgeqgrad$ we can find a finite dimensional subspace $V_p \subset T\GenCONE^{k^*}(N)$ such that
\[ \Image(D\Pi_p) + V_p = T\GenCONE^{k^*}(N). \]
Since this condition is open by an argument similar to that in the proof of Claim \ref{Cl_transverse}, we can find a neighborhood $p \in W_p \subset \Pi^{-1}(\gamma) \cap \MMgeqgrad$ such that
\[ \Image(D\Pi_{p'}) + V_{p} = T\GenCONE^{k^*}(N), \qquad \text{for all} \quad p' \in W_p. \]

By compactness, we can find a finite cover $\{ W_{p_1}, \ldots, W_{p_m} \}$ of $\Pi^{-1}(\gamma) \cap \MMgeqgrad$.
We can then set
\[ V_0 := V_{p_1} + \ldots + V_{p_m}. \]

Next we show the second statement, so assume that $Q'$ is given.
The transversality condition implies that there is an open neighborhood $\Pi^{-1}(\gamma) \cap \MMgeqgrad \subset U \subset \MM'$ such that $\Pi|_U$ is transverse to $Q'$.
It remains to show that there is a neighborhood $\gamma \in Q \subset Q'$ such that $\Pi^{-1} (Q) \cap \MMgeqgrad \subset U$.
Suppose not.
Then there is a sequence $p_i \in \MMgeqgrad$ such that $\Pi(p_i) \to \gamma$, but $p_i \not\in U$ for all $i$.
By Proposition~\ref{Prop_summary}\ref{Prop_summary_b} we can pass to a subsequence such that $p_i \to p_\infty \in \MMgeqgrad$.
Since $\Pi(p_\infty) = \gamma$, this implies that $p_\infty \in \Pi^{-1}(\gamma) \cap \MMgeqgrad \subset U$, which is impossible since $p_i \not\in U$ for all $i$.

The last statement is a direct consequence of the second statement.
\end{proof}
\bigskip

\subsection{The degree of a non-proper map}
As a preparation for the definition of the expander degree, we need to establish the following elementary lemma, which allows us to define the degree of a map between oriented manifolds near a subset in the image over which the map is proper.
This definition makes sense, even if the map as a whole is not proper.

\begin{Lemma} \label{Lem_proper_maps}
Let $M^n, N^n$ be $C^1$-manifolds of the same dimension and $f : M \to N$ a $C^1$, but not necessarily proper map.
Suppose that $X \subset N$ is a non-empty, closed, connected subset with the property that $f|_{f^{-1}(X)} : f^{-1}(X) \to X$ is proper.
Then there are open neighborhoods $X \subset N' \subset N$ and $f^{-1}(X) \subset M' \subset M$ such that $f (M') \subset N'$, $N'$ is connected and such that the map $f |_{M'}: M' \to N'$ is proper.

Moreover, if $M, N$ are oriented, then the degree $\deg(f|_{M'} : M' \to N')$ is independent of the choice of $M', N'$; it only depends on $M, N, f$ and $X$.
If $M^n, N^n$ are not oriented, then the same is true for the $\IZ_2$-degree.
\end{Lemma}

Motivated by Lemma~\ref{Lem_proper_maps}, we make the following definition.

\begin{Definition} \label{Def_degfX}
We will write $\deg (f,X) = \deg (f|_{M'} : M' \to N')$ if $M', N'$ are chosen as in Lemma~\ref{Lem_proper_maps}.
\end{Definition}

We record the following immediate consequence:

\begin{Lemma} \label{Lem_deg_f_X12}
Let $f : M \to N$ be a map as in Lemma~\ref{Lem_proper_maps} and consider non-empty, closed, connected subsets $X_1 \subset X_2 \subset N$ with the property that $f|_{f^{-1}(X_i)}$ is proper for $i = 1,2$.
Then $\deg(f,X_1) = \deg (f, X_2)$.
\end{Lemma}
\medskip

\begin{proof}[Proof of Lemma~\ref{Lem_proper_maps}.]
Choose a locally finite cover
\[ X \subset \bigcup_{i \in I} V_i, \qquad V_i \subset N \quad \text{open and relatively compact}. \]
So for each $i \in I$ the preimage $f^{-1}(V_i \cap X)$ is relatively compact in $M$ and therefore we can find an open neighborhood $f^{-1}(V_i \cap X) \subset U_i \subset f^{-1}(V_i)$ that is relatively compact in $M$.
Set $M^* := \bigcup_{i \in I} U_i$.
Then the collection $\{ U_i \}_{i\in I}$, and thus also the collection of closures $\{ \ov U_i \}_{i \in I}$, is locally finite and $$f^{-1}(X) \subset M^* \subset M.$$

We claim that for every sequence $p_j \in \ov M^* = \bigcup_{i \in I} \ov U_i$ with the property that $f(p_j) \to q_\infty \in X$ we have $p_j \to p_\infty \in f^{-1}(X) \subset M^*$ after passing to a suitable subsequence.
In fact, by local finiteness of the cover $\{ V_i \}_{i \in I}$, we can pass to a subsequence such that there is a finite subset $I' \subset I$ with the property that for any $j$ we have $f(p_j) \in V_i$ for some $i \in I'$, but for no $i \in I \setminus I'$.
So we must have $p_j \in \ov U_i$ for some $i \in I'$.
So after possibly passing to another sequence, we can find a fixed $i$ such that $p_j \in \ov U_i$ for all $j$.
Passing to another subsequence yet again, we may assume that $p_j \to p_\infty \in \ov U_i$.
Since $f(p_\infty) = q_\infty \in X$, we obtain that $p_\infty \in f^{-1}(X) \subset M^*$.

It follows that there is an open, connected neighborhood 
\[ X \subset N' \subset \Big(\bigcup_{i \in I} V_i \Big) \setminus f(\partial \ov M^*), \]
because otherwise we could find a sequence $p_j \in \partial \ov M^*$ with $f(p_j) \to q_\infty \in X$, which would imply that $p_j \to p_\infty \in M^*$ for a subsequence, in contradiction to the fact that $\partial \ov M^*$ is closed.
Let $M' := f^{-1}(N') \cap M^*$.
We need to show that $f|_{M'} : M' \to N'$ is proper.
To see this, suppose that $f(p_j) \to q_\infty \in N'$ for some sequence $p_j \in M'$.
Choose $i \in I$ such that $q_\infty \in V_i$.
As before, we may pass to a subsequence and assume that $p_j \in U_i$ for all $j$ and that $p_j \to p_\infty \in \ov U_i \subset \ov M^*$.
Since $f(p_\infty) = q_\infty \in N'$, we find that $p_\infty \in f^{-1}(N') \cap \ov M^* = M'$.

To see the last statement, consider another set of open neighborhoods $X \subset N'' \subset N$, $f^{-1}(X) \subset M'' \subset M$ with the same properties.
Let $q_j \in N'$ be a sequence of regular values of $f |_{M'}$ such that $q_j \to q_\infty \in X$.
Note that $q_j \in N''$ for large $j$.
We claim that for large $j$ we have $(f|_{M''})^{-1}(q_j) = (f|_{M'})^{-1}(q_j)$, which implies the claim.
If not then there is a sequence $p_j \in (M' \setminus M'') \cup (M'' \setminus M')$ such that $f(p_j) = q_j$.
After passing to a subsequence and possibly switching the roles of $M', M''$ and $N', N''$, we may assume that $p_j \in M'' \setminus M'$.
Due to the properness of $f|_{M''}$, we obtain that for a subsequence $p_j \to p_\infty \in M''$.
Since $q_\infty \in X$, we have $p_\infty \in f^{-1}(X) \subset M'$ and thus $p_j \in M'$ for large $j$, in contradiction to our assumption.
\end{proof}
\bigskip

\subsection{Precise definition of the expander degree}\label{subsec_deg_def}
In this subsection, we define the expander degree of an ensemble $(M,N,\iota)$.
Via the correspondence established in Lemma~\ref{Lem_ensemble_vs_X}, this will then allow us to define the expander degree as stated in Theorem~\ref{Thm_main_vague} of the introduction.

As a first step, we define the degree of the projection map $\Pi : \MM^{k^*}(M,N,\iota) \to \GenCONE^{k^*}(N)$ over some finite dimensional, real-analytic submanifold $Q \subset \GenCONE^{k^*}(N)$, which is transverse to $\Pi$.

Fix an ensemble $(M,N,\iota)$ and suppose that $k^* \geq 30$ and $\alpha \in (0,1)$, as in  Proposition~\ref{Prop_summary}.
In the following we write $\MM := \MM (M,N,\iota)$, $\MMgrad := \MMgrad(M,N,\iota)$, etc.
Let $Q \subset \GenCONE^{k^*}(N)$ be a finite dimensional, real-analytic submanifold that satisfies the transversality condition \eqref{eq_transversality_condition} and assume that $Q \cap \CONEgeq^{k^*}(N) \neq \emptyset$.
Recall that by Lemma~\ref{Lem_transverse_Q_exists}, we can always find such a $Q$ with $\gamma \in Q$ for any given $\gamma \in \CONEgeq^{k^*}(N)$.
Let $P = \Pi^{-1}(Q) \cap U$ be the submanifold from Proposition~\ref{Prop_summary}\ref{Prop_summary_g}.
For the definition of the integer degree, we also need to assume that $Q$ is oriented and consider the orientation on $P$ described in Proposition~\ref{Prop_summary}\ref{Prop_summary_i}.
For the definition of the $\IZ_2$-degree, no orientation has to be chosen on $Q$ or $P$.

\begin{Lemma} \label{Lem_deg_MNiQ}
Consider the setting from the previous paragraph.
Then for any non-empty, closed, connected subset $X \subset Q \cap \CONEgeq^{k^*}(N)$ the (integer or $\IZ_2$-valued) degree
\begin{equation} \label{eq_degMNiotaQXdef}
 \deg (M,N,\iota,Q,X) := \deg (\Pi|_P, X)
\end{equation}
is well defined and independent of $P$, the orientation of $Q$ and independent of $\alpha$.
This degree is also independent of $k^*$ in the following sense:
If for some other integer $k^{**} \geq 30$ we have $Q \subset \GenCONE^{k^{**}}(N)$ and $Q$ is real-analytic with respect to the $C^{k^{**}}$-norm, then $\deg (\Pi, Q \cap \CONE^{k^*}(N)) = \deg (\Pi, Q \cap \CONE^{k^{**}}(N))$.

Lastly, the map $\gamma \mapsto \deg (M,N,\iota,Q,\{ \gamma \})$ is constant on every path-component of $Q \cap \CONEgeq^{k^*}(N)$.
\end{Lemma}

\begin{proof}
For the well definedness of the right-hand side in \eqref{eq_degMNiotaQXdef} we apply Lemma~\ref{Lem_proper_maps} to the map $\Pi|_P : P \to Q$.
To see that 
$\Pi |_{ (\Pi|_P)^{-1}(X)} : ( \Pi|_P)^{-1}(X) \to X$ 
is proper, note first that by Proposition~\ref{Prop_summary}\ref{Prop_summary_g}, \ref{Prop_summary_h} we have
\begin{equation} \label{eq_PiPXinv}
 ( \Pi|_P)^{-1}(X) = P \cap \Pi^{-1}(X) = \MMgeqgrad \cap \Pi^{-1}(X) = (\Pi|_{\MMgeqgrad})^{-1}(X), 
\end{equation}
so the properness follows from  Proposition~\ref{Prop_summary}\ref{Prop_summary_b}.

To see the independence of $P$ note first that by Proposition~\ref{Prop_summary}\ref{Prop_summary_unique} the $C^1$-structure on $P$ is well defined and $P = \Pi^{-1}(Q) \cap U$ may only depend on the choice of the neighborhood $U$ from \eqref{eq_nbhd_U}.
So for any other such neighborhood $U'$ the corresponding manifold $P' := \Pi^{-1}(Q) \cap U'$ has the property that $P \cap P'$ is an open subset of both $P$ and $P'$.
Since by \eqref{eq_PiPXinv} the preimage $(\Pi|_P)^{-1}(X)$ is independent of the choice of $P$, we obtain, using  the last statement of Lemma~\ref{Lem_proper_maps} that
\[ \deg(\Pi|_P, X) = \deg(\Pi|_{P \cap P'}, X) = \deg(\Pi|_{P'}, X). \]
For the independence of the orientation, note that by Proposition~\ref{Prop_summary}\ref{Prop_summary_i} a flip of the orientation on $Q$ forces a flip of the orientation on $P$ as well.
The independence of $\alpha$ follows from the fact that $\MM$ does not depend on the choice of $\alpha$ and from Proposition~\ref{Prop_summary}\ref{Prop_summary_unique}.
To see the independence of $k^*$, notice that by Remark~\ref{Rmk_differen_regularity_MM} the preimages $\Pi^{-1}(Q)$ within $\MM^{k^*}(M,N,\iota)$ and $\MM^{k^{**}}(M,N,\iota)$ agree.

For the last statement note that if $X$ is the image of a path within $Q \cap \CONEgeq^{k^*}(N)$, then by Lemma~\ref{Lem_deg_f_X12} we have $\deg (M,N,\iota,Q,\{\gamma\}) = \deg (M,N,\iota,Q,X)$ for any $\gamma \in X$.
\end{proof}
\medskip

Next, we will use Lemma~\ref{Lem_deg_MNiQ} to define a local (integer or $\IZ_2$-valued) degree near a fixed  $\gamma \in \CONEgeq^{k^*}(N)$.
We will define this degree as $\deg(M,N,\iota,Q, \{ \gamma \})$ for some suitable real-analytic submanifold $Q \subset \GenCONE^{k^*}(N)$ containing $\gamma$.

\begin{Lemma} \label{Lem_deg_indep_Q}
In the setting of Lemma~\ref{Lem_deg_MNiQ} the degree
\[ \deg (M,N,\iota,\gamma) := \deg (M,N,\iota,Q,\{\gamma\}) \]
is independent of the choice of $Q$.
Moreover the degree $\deg (M,N,\iota,\gamma)$ is well defined for any $\gamma \in \CONEgeq^{k^*}(N)$. %
\end{Lemma}

\begin{proof}
Note first that by Lemma~\ref{Lem_transverse_Q_exists}, given $\gamma \in \CONEgeq^{k^*}(N)$, there is always a submanifold $\gamma \in Q \subset \GenCONE^{k^*}(N)$ satisfying the discussion preceding Lemma~\ref{Lem_deg_MNiQ}.
So we just need to establish the independence of $\deg (M,N,\iota,Q,\{\gamma\})$ from $Q$.
To do this, consider first two finite dimensional, real-analytic submanifolds $$\gamma \in Q_1 \subset Q_2 \subset \GenCONE^{k^*}(N),$$ each satisfying the transversality condition \eqref{eq_transversality_condition}.
For the definition of the \emph{integer} degree we also need to assume that $Q_1, Q_2$ are oriented.
Our goal will be to show that
\begin{equation} \label{eq_degQ1Q2}
  \deg (M,N,\iota,Q_1,\{\gamma\}) = \deg (M,N,\iota,Q_2,\{\gamma\})
\end{equation}
As in the construction preceding Lemma~\ref{Lem_deg_MNiQ}, let 
\[ P_i := \Pi^{-1}(Q_i) \cap U_i \subset \MM, \qquad i = 1,2, \]
be the corresponding submanifolds provided by Proposition~\ref{Prop_summary}\ref{Prop_summary_g}, \ref{Prop_summary_h}, with the induced orientations provided by Proposition~\ref{Prop_summary}\ref{Prop_summary_i} if $Q_1, Q_2$ are oriented.
By the same argument as in \eqref{eq_PiPXinv}, we find that
\begin{equation} \label{eq_PiPi}
   (\Pi|_{P_i})^{-1}(\gamma) = (\Pi|_{\MMgeqgrad})^{-1}(\gamma) 
\end{equation}
is independent of $i$.

By Lemma~\ref{Lem_proper_maps} there are open neighborhoods
\[ \gamma \in Q'_2 \subset Q_2, \qquad (\Pi|_{\MMgeqgrad})^{-1}(\gamma) \subset P'_2 \subset P_2,  \]
such that $\Pi|_{P'_2} : P'_2 \to Q'_2$ is proper and $Q'_2$ is connected.
Set $Q'_1 := Q_1 \cap Q'_2$ and 
$$P'_1 := P_1 \cap P'_2 = ( \Pi^{-1}(Q_1) \cap U_1) \cap P'_2 \subset \Pi^{-1}(Q_1) \cap P'_2 = (\Pi|_{P'_2})^{-1}(Q'_1).$$
Then by \eqref{eq_PiPi} we still have
$$ (\Pi|_{P'_1})^{-1}(\gamma)  = (\Pi|_{\MMgeqgrad})^{-1}(\gamma) \subset P'_1. $$
So we can apply Lemma~\ref{Lem_proper_maps} to $\Pi|_{P'_2} : P'_2 \to Q'_2$ to obtain open neighborhoods
\[ \gamma \in Q''_1 \subset Q'_1, \qquad (\Pi|_{\MMgeqgrad})^{-1}(\gamma) \subset P''_1 \subset P'_1,  \]
such that $\Pi|_{P''_1} : P''_1 \to Q''_1$ is proper and $Q''_1$ is connected.
Note that $Q''_1 \subset Q'_2$ and $P''_1 \subset P'_2$.
In order to show \eqref{eq_degQ1Q2} for the $\IZ_2$-degree, we can now argue, using Lemma~\ref{Lem_proper_maps} as well as Lemma~\ref{Lem_proper_f_MMp} below, that:
\[ \deg(M,N,\iota,Q_1, \{ \gamma \}) = \deg(\Pi|_{Q''_1}, \{ \gamma \}) = \deg(\Pi|_{Q'_2}, \{ \gamma \}) = \deg(M,N,\iota,Q_2, \{ \gamma \}). \]

For the integer degree, we need to verify the condition of Lemma~\ref{Lem_proper_f_MMp} involving the preservation of the co-orientation.
For this purpose fix a $p \in (\Pi |_{P''_1})^{-1}(\gamma) = (\Pi|_{\MMgeqgrad})^{-1}(\gamma)$ and consider the orientations on $T_p P''_1 = T_p P_1$, $T_p P'_2 = T_p P_2$, $T_\gamma Q''_1 = T_\gamma Q_1$, $T_\gamma Q'_2 = T_\gamma Q_2$.
Choose arbitrary orientations on 
\begin{equation} \label{eq_ker_same}
 \ker (d(\Pi|_{P_1})_p ) = \ker (d(\Pi|_{P_2})_p ) 
\end{equation}
and on $\Image (D (\Pi|_{P_1})_p )$.
By \eqref{eq_ker_same} and since 
\begin{equation} \label{eq_DPi_invTQTP}
(d(\Pi|_{P_2})_p)^{-1} (T_\gamma Q_1 ) = T_p P_1, 
\end{equation}
the inclusion map $T_\gamma Q_1 \hookrightarrow T_\gamma Q_2$ induces an isomorphism (we are taking quotients from the left again)
\begin{equation} \label{eq_coker_same}
  \Image (d (\Pi|_{P_1})_p ) \big\backslash T_\gamma Q_1  \xrightarrow{\;\cong \;}  \Image (d (\Pi|_{P_2})_p ) \big\backslash T_\gamma Q_2 . 
\end{equation}
Choose the orientation on $\Image (d (\Pi|_{P_2})_p )$ such that this map is orientation preserving.
Then via the identifications \eqref{eq_ker_same} and \eqref{eq_coker_same} the maps $S_p$ from Proposition~\ref{Prop_summary}\ref{Prop_summary_g}, defined for $P_1$ and $P_2$, are the same and are both either orientation-preserving or reversing. 
Next, the quotient of the isomorphisms (see again \eqref{eq_DPi_invTQTP} for the last isomorphism)
\begin{equation} \label{eq_PiTpker}
 d(\Pi|_{P_i})_p : T_p P_i \big/ \ker ( d(\Pi|_{P_i})_p) \longrightarrow \Image (d(\Pi|_{P_i})_p), \qquad i = 1,2, 
\end{equation}
yields the isomorphism
\begin{equation} \label{eq_dPi_quotient}
 T_p P_2 \big/ T_p P_1 \longrightarrow \Image (d(\Pi|_{P_2})_p) \big/ \Image (d(\Pi|_{P_1})_p) \cong T_\gamma Q_2 \big/ T_\gamma Q_1,
\end{equation}
which is induced by the differential $d(\Pi|_{P_2})_p$.
Depending on the $\idx (- L_g)$, for some representative $(g, \nabla^g f, \gamma)$ of $p$, the maps \eqref{eq_PiTpker} have the same or the opposite orientation of $S_p$, so the map \eqref{eq_dPi_quotient} must be orientation-preserving.
This shows the additional condition of Lemma~\ref{Lem_proper_f_MMp} and hence finishes the proof of \eqref{eq_degQ1Q2} for the integer degree.

Next, consider two arbitrary finite dimensional, real-analytic submanifolds $$\gamma \in Q_1, Q_2 \subset \GenCONE^{k^*}(N),$$ each satisfying the transversality condition \eqref{eq_transversality_condition}.
Then we can pass to open neighborhoods
\[ \gamma \in Q'_i \subset Q_i \]
such that there is a real-analytic submanifold $Q_3 \subset \GenCONE^{k^*}(N)$ of dimension $\dim Q_1 + \dim Q_2$ and satisfying the transversality condition \eqref{eq_transversality_condition}, with the property that $Q'_1, Q'_2 \subset Q_3$.
It follows, using \eqref{eq_degQ1Q2} repeatedly, that
\begin{multline*} 
   \deg (M,N,\iota,Q_1,\{\gamma\}) 
= \deg (M,N,\iota,Q'_1,\{\gamma\})
= \deg (M,N,\iota,Q_3,\{\gamma\}) \\
= \deg (M,N,\iota,Q'_2,\{\gamma\})
= \deg (M,N,\iota,Q_2,\{\gamma\}). 
\end{multline*}
This finishes the proof of the independence of $Q$ and shows that $\deg(M,N,\iota,\gamma)$ is well defined.
\end{proof}
\medskip

\begin{Lemma} \label{Lem_proper_f_MMp}
Suppose that $f : M \to N$ is a proper $C^1$-map between two $C^1$-manifolds with $\dim M = \dim N$.
Let $M' \subset M$ and $N' \subset N$ be submanifolds  with $\dim M' = \dim N'$ such that $f(M') \subset N'$ and $f|_{M'} : M' \to N'$ is proper.
Suppose that $N, N'$ are connected and suppose that there is a point $q \in N'$ such that $f^{-1}(q) \subset M'$.
Moreover, suppose that for every $p \in f^{-1}(q)$ we have the transversality condition $df_p(T_pM) + T_q N' = T_q N$.
Then for the $\IZ_2$-degree we have
\begin{equation} \label{eq_degf_degfMp}
   \deg(f) = \deg(f|_{M'}). 
\end{equation}
Now suppose that, in addition, $M, N, M', N'$ are oriented and that for all $p \in f^{-1} (q)$ the induced map $df_p : T_p M / T_p M' \to T_{q} N / T_q N'$ is an orientation-preserving isomorphism.
Then \eqref{eq_degf_degfMp} even holds for the integer degree.
\end{Lemma}

\begin{proof}
We first claim that for any sufficiently small neighborhood $q \in U \subset N$ we have $M' \cap f^{-1}(U) = f^{-1}(N' \cap U)$.
To see this, note that we trivially have $M' \cap f^{-1}(U) \subset f^{-1}(N' \cap U)$ for any such neighborhood.
So if the claim was wrong, then we could find a sequence of points $p_i \in M \setminus M'$ such that $f(p_i) \in N'$ and $f(p_i) \to q$.
By properness of $f$, we may pass to a subsequence and assume that $p_i \to p_\infty \in M$, where $f(p_\infty) = q$.
By assumption we have $p_\infty \in M'$.
Due to the transversality assumption, the preimage $f^{-1}(N')$ is local submanifold near $p_\infty$ of the same dimension as $M' \subset f^{-1}(N')$.
So $p_i \in M'$ for large $i$, which is a contradiction.

Next, we claim that, in addition, for any sufficiently small neighborhood $q \in U \subset N$ the restricted map $f|_{f^{-1}(U)}$ is transverse to $N'$.
To see this, we argue again by contradiction and fix a sequence of points $p_i \in M$ such that $f(p_i) \to q$ and $df_{p_i}(T_{p_i} M) + T_{f(p_i)} N' \neq T_{f(p_i)} N$.
After passing to a subsequence, we may again assume that $p_i \to p_\infty$ with $f(p_\infty) = q$, at which the transversality condition holds by assumption.
Since this condition is open, we obtain a contradiction for large $i$.

Consider now a sufficiently small neighborhood $q \in U \subset N$, in the sense of the previous two paragraphs.
By choosing $U$ to be the domain of a local slice chart for $N'$, we may also assume that $U$ and $N' \cap U$ are connected.
Then the assumptions and the assertions of the lemma are preserved if we replace $N$ with $U$, $M$ with $f^{-1}(U)$, $N'$ with $N' \cap U$ and $M'$ with $M' \cap f^{-1}(U)$. 
So we may assume without loss of generality that $M' = f^{-1}(N')$ and that $f$ is transverse to $N'$.

The lemma for the $\IZ_2$-degree now follows immediately by analyzing $f$ over a regular value of $f|_{M'}$ in $N'$.
To see the lemma for the integer degree, consider a sequence of regular points $q_i \in N'$ for $f|_{M'}$ with $q_i \to q$.
If the induced map $df_p : T_p M / T_p M' \to T_{q_i} N / T_{q_i} N'$ is still orientation-preserving for all $p \in f^{-1}(q_i)$ and some fixed $i$, then the lemma follows.
Now suppose that for each $i$ there is a $p_i \in f^{-1}(q_i)$ for which $df_{p_i} : T_{p_i} M / T_{p_i} M' \to T_{q_i} N / T_{q_i} N'$ is not orientation-preserving.
By the properness of $f$, we may pass to a subsequence and assume that $p_i \to p \in f^{-1}(q)$ and thus $df_p : T_p M / T_p M' \to T_{q} N / T_q N'$ is not orientation-preserving as well, in contradiction to our assumptions.
This finishes the proof.
\end{proof}
\medskip

We can finally state our main result.

\begin{Theorem} \label{Thm_def_degree}
Given an ensemble $(M,N,\iota)$, there is a unique degree
\[ \deg (M,N,\iota) \in \IZ \]
such that for any $k^* \geq 30$ and any $\gamma \in \CONEgeq^{k^*}(N)$ we have
\[ \deg (M,N,\iota) = \deg (M,N,\iota,\gamma), \]
where the latter denotes the integer degree.
Moreover, if $\gamma \in \CONEgeq^{k^*}(N)$ is a regular value of $\Pi$ over $\MMgeqgrad^{k^*}(M,N,\iota)$, in the sense that for any representative $(g,\nabla^g f, \gamma)$ of an element $p \in (\Pi |_{\MMgeqgrad^{k*}(M,N,\iota)})^{-1}(\gamma)$ the operator
\[ L_p := L_g : C^{2,\alpha}_{-2,g,\nabla^g f} (M; S^2 T^*M) \lto C^{0,\alpha}_{-2,g,\nabla^g f} (M; S^2 T^*M) \]
has no kernel, then
\[ \deg (M,N,\iota) = \sum_{p \in (\Pi |_{\MMgeqgrad^{k*}(M,N,\iota)})^{-1}(\gamma)} (-1)^{\idx (-L_{p})}. \]
(Note that the condition of $L_g$ having no kernel and $\idx (-L_p)$ are independent of the choice of the representative $(g,\nabla^g f, \gamma)$ of p.)

The same statement holds for the $\IZ_2$-degree, which is equal to the integer degree modulo $2$.
Moreover, in this case we have the following simpler identity for any regular value $\gamma$ over $\MMgeqgrad^{k^*}(M,\lb N, \lb \iota)$
\[ \deg (M,N,\iota) = \# \big(\Pi |_{\MMgeqgrad^{k*}(M,N,\iota)} \big)^{-1}(\gamma) \qquad \textnormal{mod} \; 2. \]
\end{Theorem}

\begin{proof}
By Lemma~\ref{Lem_transverse_Q_exists}, given $\gamma \in \CONEgeq^{k^*}(N)$, $k^* \geq 30$ and given any line segment $G \subset \CONEgeq^{k^*}(N)$ with $\gamma \in G$, we can find a finite dimensional, real-analytic, oriented submanifold $\gamma \in Q \subset \GenCONE^{k^*}(N)$ that contains a neighborhood of $\gamma$ in $G$ and that allows us to apply Lemma~\ref{Lem_deg_MNiQ}.
So by Lemmas~\ref{Lem_deg_MNiQ}, \ref{Lem_deg_indep_Q}, the quantity
\[ \deg(M,N,\iota,\gamma') = \deg(M, N,\iota, Q,\{\gamma' \}) \]
is constant on path-components of $\gamma' \in Q \cap G$.
It follows that the map $G \to \IZ$, $\gamma' \mapsto \deg(M,N,\iota,\gamma')$ is locally constant and therefore constant.

Next note that every $\gamma \in \CONEgeq^{k^*}(N)$ can be connected to an element $\gamma' \in \CONEg^{k^*}(N)$ via a line segment within $\CONEgeq^{k^*}(N)$ that corresponds to slightly shrinking the link of the cone metric $\gamma$.
So it remains to show that $\CONEg^{k^*}(N)$, or equivalently the space of $C^{k^*}$-metrics on $N$ with $\inf_M R > 6$ is connected.
This follows from \cite{Bamler_Kleiner_space_metrics} or from \cite{Marques_2012} if in addition $N \approx S^3$.
\end{proof}
\medskip

\begin{Corollary}\label{Cor_nonzero_degree}
If $(M,N,\iota)$ is an ensemble with
\[ \deg (M,N,\iota) \neq 0, \]
then for any $30 \leq k^* \leq \infty$ the map
\[ \Pi |_{\MMgeqgrad^{k*}(M,N,\iota)} : \MMgeqgrad^{k*}(M,N,\iota) \lto \CONEgeq^{k^*}(N) \]
is surjective.
In particular, for every $\gamma \in \CONEgeq^{\infty}(N)$ there is a $p = [(g,\nabla^g f, \gamma)] \in \MMgeqgrad^{\infty}(M, \lb N, \lb \iota)$, where $g$ and $f$ are smooth.
\end{Corollary}

\begin{proof}
Since every point $\CONEgeq^{k*}(M,N,\iota) \setminus \Pi (\MMgeqgrad^{k*}(M,N,\iota))$ is trivially a regular point of the map $\Pi$ over $\MMgeqgrad^{k*}(M,N,\iota)$, the statement for $k^* < \infty$ follows from Theorem~\ref{Thm_def_degree}.
The last statement is a consequence of Remark~\ref{Rmk_differen_regularity_MM} and Corollary~\ref{Cor_smooth_rep}.
\end{proof}
\bigskip

\subsection{The degree in the case of a cone over a spherical space form}
In the following we compute the degree of the standard ensemble $(\IR^4/\Gamma, S^3/\Gamma,\iota_\Gamma)$.
This is possible due to the uniqueness of the trivial Euclidean expanding soliton among all gradient expanding solitons with non-negative scalar curvature that are asymptotic to Euclidean space.

\begin{Theorem}\label{Thm_deg_std_disk}
Consider the standard ensemble $(\IR^4/\Gamma, S^3/\Gamma,\iota_\Gamma)$, where $\Gamma \subset SO(4)$ is a finite group acting freely on $S^3$ and $\iota_\Gamma : (1,\infty) \times S^3/\Gamma \to \IR^4/\Gamma$ is the standard radial embedding.
Then 
\[ \deg (\IR^4/\Gamma, S^3/\Gamma,\iota_\Gamma)
 = 1. \]
Moreover, if $(M,N,\iota)$ is an ensemble such that $N \approx S^3 / \Gamma$ and such that $M$ has an orbifold cover such that $\iota$ lifts to a map from $(1,\infty) \times S^3$, then the following is true.
If $(M,N,\iota)$ is isomorphic to $(\IR^4/\Gamma, S^3/\Gamma,\iota_\Gamma)$, meaning that there are a diffeomorphisms $\phi : M \to \IR^4 / \Gamma$ and $\psi : N \to S^3/\Gamma$ such that
\[ \iota_\Gamma \circ (\id_{(1,\infty)} ,\psi) = \phi \circ \iota, \]
then $\deg(M,N,\iota) = 1$.
Otherwise $\deg(M,N,\iota) = 0$.
\end{Theorem}

\begin{proof}
Fix some $k^* \geq 30$.
Suppose first that $(M,N,\iota) = (\IR^4/\Gamma, S^3/\Gamma,\iota_\Gamma)$.
Denote by $p_{\eucl} \in \MMgeqgrad^{k^*}(M,N,\iota)$ the element that is represented by $(g_{\eucl}, V_{\eucl} := \nabla^{g_{\eucl}} f_{\eucl}, \gamma_{\eucl})$, where $(\IR^4, \lb g_{\eucl}, \lb f_{\eucl} := - \tfrac14 \mathbf{r}^2)$ denotes the standard Euclidean gradient expanding soliton with radial coordinate $\mathbf{r}$ and $\gamma_{\eucl}   \in \CONE^\infty (S^3/\Gamma)$ denotes the standard local Euclidean cones, whose link metric is the standard round metric on $S^3/\Gamma$. 
An application of the maximum principle shows that $L_{g_{\eucl}}$ has nullity $0$ and index $0$.
By Proposition~\ref{Prop_eucl_case}, we have 
\[ \big(\Pi |_{\MMgeqgrad^{k*}(M,N,\iota)} \big)^{-1}(\gamma_{\eucl}) = [p_{\eucl} ]. \]
So, since $L_{g_{\eucl}}$ has nullity $0$, the point $\gamma_{\eucl}$ is a regular point and Theorem~\ref{Thm_def_degree} implies that
\[ \deg (M,N,\iota) = \deg(M,N,\iota,\{ \gamma_{\eucl} \}) = 1. \]

For the second statement it remains to consider the case in which $(M,N,\iota)$ is not isomorphic to $(\IR^4/\Gamma, S^3/\Gamma,\iota_\Gamma)$.
Then Proposition~\ref{Prop_eucl_case} implies that
\[ \big(\Pi |_{\MMgeqgrad^{k*}(M,N,\iota)} \big)^{-1}(\gamma_{\eucl}) = \emptyset. \]
So by Theorem~\ref{Thm_def_degree} we have $\deg (M,N,\iota)=0$.
\end{proof}
\bigskip

\subsection{Proofs of the theorems from Subsection~\ref{subsec_statement_of_results}} \label{subsec_proofs_of_main_thms}
Below we give the proofs of Theorems \ref{Thm_main_vague}, \ref{Thm_deg_1}, \ref{Thm_Banach_mf_intro}, \ref{Thm_degexp_identity}, \ref{Thm_many_deformations}, \ref{Thm_preservation_gradient_intro} and \ref{Thm_standard_disk}. In many cases this amounts to collecting the appropriate results already proved earlier.

\begin{proof}[Proof of Theorem \ref{Thm_main_vague}]
Proposition \ref{Prop_properness} shows the properness of $\Pi \big|_{\MMgeqgrad^{k^*}(X)}$, and Theorem \ref{Thm_def_degree} gives the existence of the integer expander degree as an invariant of the smooth structure of $X$ (the equivalence between the topological setups on $X$ and an ensemble $(M,N,\iota)$ is described in Lemma \ref{Lem_ensemble_vs_X}; see also Lemma~\ref{Lem_equivalence_MM}). 
Namely, we define $\deg_{\exp}(X):=\deg(M,N,\iota)$. 
The determination of $\deg_{\exp}$ from the local behavior of an extension of $\Pi$ near any $\gamma\in\CONEgeq^{k^*}(\partial X)$ is given by Lemmas \ref{Lem_deg_MNiQ} and \ref{Lem_deg_indep_Q}. Finally, the surjectivity of $\Pi \big|_{\MMgeqgrad^{k^*}(X)}$ when the expander degree is nonzero is the content of Corollary \ref{Cor_nonzero_degree}.
\end{proof}
\medskip

\begin{proof}[Proof of Theorem \ref{Thm_deg_1}]
This is the first half of Theorem \ref{Thm_deg_std_disk}.
\end{proof}
\medskip

\begin{proof}[Proof of Theorem \ref{Thm_Banach_mf_intro}]
Corollary \ref{Cor_MM_is_Banach_manifold} gives the $C^{1,\alpha}$-Banach manifold structure of $\MM'$ and the smoothness of $\Pi|_{\MM'}$. And the integer expander degree satisfying the stated properties is defined in Subsection \ref{subsec_deg_def}, as described above in the proof of Theorem \ref{Thm_main_vague}.
\end{proof}
\medskip

\begin{proof}[Proof of Theorem \ref{Thm_degexp_identity}]
This is the second half of Theorem \ref{Thm_def_degree}.
\end{proof}
\medskip

\begin{proof}[Proof of Theorem \ref{Thm_many_deformations}]
To prove Assertion \ref{Thm_many_deformations_a}, first write $\gamma=\Pi(p)$ for some $p = [(g,\nabla^g f, \gamma)] \in\MM^{k^*}_{\grad}(X)$. 
Denote by $K \subset C^{2,\alpha}_{-2}(M; S^2 T^*M)$ the kernel of the Einstein operator $L_g$.
By Proposition~\ref{prop_smooth_dependence} we can find a neighborhood
\[ (\gamma,0) \in U = U_1 \times U_2 \subset \GenCONE^{k^*}(X) \times K \]
and a real-analytic map $G : U \to K$ such that $(D_1 G)_{(\gamma,0)}$ has full rank and such that for any $(\gamma',\kappa') \in U$ with $G(\gamma',\kappa') = 0$ we have $\gamma' \in \Image \Pi$.
It follows that $\Gamma := \{ G(\cdot,0) = 0 \} \subset U_1 \subset \GenCONE^{k^*}(X)$ is a Banach-submanifold with the desired properties.

To prove Assertion~\ref{Thm_many_deformations_b}, first observe that by Lemma~\ref{Lem_transverse_Q_exists} we may enlarge $Q \subset \GenCONE^{k^*}(\partial X)$ to a finite dimensional, real-analytic submanifold (of possibly larger dimension), which in addition satisfies the transversality condition \eqref{eq_transversality_condition}.
Assertion~\ref{Thm_many_deformations_b} now follows from Lemmas~\ref{Lem_deg_MNiQ} and \ref{Lem_deg_indep_Q} together with Theorem~\ref{Thm_def_degree} and Definition~\ref{Def_degfX}.
\end{proof}
\medskip

\begin{proof}[Proof of Theorem \ref{Thm_preservation_gradient_intro}]
If $\sigma:[0,1]\rightarrow\MM^{k^*}(X)$ is a curve as described in the statement of the theorem, then by Corollary \ref{Cor_MM_is_Banach_manifold} we obtain a family $(M,g_s,v_s)_{s\in[0,1]}$ of expanding solitons on $M$ satisfying the hypotheses of Proposition \ref{Prop_gradientness}, and the proof of Proposition \ref{Prop_gradientness} shows again in this case that every member of this family is, in fact, a gradient expanding soliton, and hence in $\MM^{k^*}_{\grad}(X)$.
\end{proof}
\medskip

\begin{proof}[Proof of Theorem \ref{Thm_standard_disk}]
This is the second half of Theorem \ref{Thm_deg_std_disk}.
\end{proof}
\bigskip

\appendix

\section{Technical results} \label{appx_technical}
In this appendix, we prove several technical results that are needed in the course of the paper.

The following lemma establishes the existence of a smooth structure on an expanding soliton with respect to which both the metric and vector field are smooth.
This fact is used frequently throughout the paper, for example, in the definition of the weighted H\"older norms in Subsection~\ref{subsec_list_Holder}.

\begin{Lemma} \label{Lem_soliton_smooth}
If $(M,g,V)$ is an expanding soliton, where $M$ is a smooth orbifold, $g$ has regularity $C^2$ and $V$ has regularity $C^1$, then there is a unique smooth structure on $M$ such that $g,V$ are smooth.
\end{Lemma}

\begin{proof}
In the following, we will construct the smooth structure via certain modified harmonic coordinates.
It will be clear that this structure is unique.
Since our construction will be local, we may pass to a local cover and assume from now on that $(M,g)$ is a (possibly incomplete) Riemannian \emph{manifold.}

Let $p \in M$ be an arbitrary point and choose coordinates $(y^1, \ldots, y^n)$ in a neighborhood $p \in U \subset M$ such that the coefficient functions of $g$ and $V$ with respect to this coordinate system have regularity $C^2$ and $C^1$, respectively.
After possibly shrinking $U$, we will now choose functions $x^1, \ldots, x^n \in C^2(U)$ such that
\begin{equation} \label{eq_choice_xi_ph}
 \triangle x^i - V \cdot \nabla x^i = 0 
\end{equation}
and such that $\{ dx^i|_p \}$ are linearly independent.
Such functions can, for example, be constructed by solving a Dirichlet problem.
The desired regularity follows from Schauder estimates.
After possibly shrinking $U$ once more, we assume that the functions $x^1, \ldots, x^n$ comprise a local coordinate system near $p$.
Our goal will be to show that the coefficient functions $g_{ij}$ and $V^k$ of $g$ and $V$, respectively, with respect to $(x^1, \ldots, x^n)$ are smooth.
To do this, note that the inverse matrix $(g^{ij})$ of the metric coefficients $(g_{ij})$ is given by $g^{ij} = \nabla x^i \cdot \nabla x^j$.
Moreover, by the change-of-coordinate formula, we find that $g_{ij}$, $g^{ij}$ and $V^k$ are of regularity $C^1$ in the coordinates $(x^1, \ldots, x^n)$.

We now compute $\triangle \nabla x^i$ in the weak sense.
For any compactly supported vector field $Y \in C^2 (U; TM)$ and for any $i = 1, \ldots, n$, we can carry out the following computation, which should be understood to occur with respect to the original smooth structure on $M$,
\begin{align*}
 - \int_U \nabla^2_{kl} x^i \nabla_k Y^l \, dg
&= - \int_U \nabla^2_{lk} x^i \nabla_k Y^l \, dg
=  \int_U \nabla_{k} x^i \nabla^2_{lk} Y^l \, dg \\
&= \int_U \big(  \nabla_k x^i \nabla^2_{kl} Y^l + \Ric(\nabla x^i, Y) \big) \, dg \displaybreak[1] \\
&= \int_U \big(  - \triangle x^i \, \nabla_{l} Y^l + \Ric(\nabla x^i, Y) \big) \, dg \displaybreak[1] \\
&= \int_U \big(  -( V \cdot \nabla x^i ) \, \nabla_{l} Y^l + \Ric(\nabla x^i, Y) \big) \, dg \displaybreak[1] \\
&=  \int_U \big(  \nabla_l V^k  \nabla_k x^i \, Y^l +  V^k  \nabla_{lk}^2 x^i \, Y^l +  \Ric(\nabla x^i, Y) \big) \, dg
\\
&=  \frac12 \int_U \big(  (\nabla_l V^k - \nabla_k V^l)  \nabla_k x^i \, Y^l + 2 V^k  \nabla_{lk}^2 x^i \, Y^l - \nabla x^i \cdot Y \big) \, dg. 
\end{align*}
Thus, if we set $Y := \varphi \nabla x^j$ for some $j = 1, \ldots, n$, then we get
\[ - \int_U \nabla^2_{kl} x^i \nabla_k (\varphi  \nabla_l x^j) dg 
=  \frac12 \int_U \big(  (\nabla_l V^k - \nabla_k V^l)  \nabla_k x^i \nabla_l x^j + 2 V^k  \nabla_{lk}^2 x^i \, \nabla_l x^j - \nabla x^i \cdot \nabla x^j \big) \varphi \, dg, 
\]
hence, by symmetrizing with respect to $i,j$, we obtain the following weak characterization of $\triangle  g^{ij}$:
\begin{align}
   - \int_U \nabla_k  g^{ij} \nabla_k \varphi \, dg
&= - \int_U \big( \nabla^2_{kl} x^i \nabla_l x^j + \nabla^2_{l} x^i \nabla_{kl} x^j \big) \nabla_k \varphi \, dg \notag \\
&=  \int_U \big( V^k \nabla^2_{lk} x^i \nabla_l x^j + V^k \nabla^2_{lk} x^j \nabla_l x^i + 2 \nabla^2_{kl} x^i \nabla^2_{kl} x^j - g^{ij} \big) \varphi \,dg \notag \\
&=  \int_U \big( V^k \nabla_k g^{ij} + 2 \nabla^2_{kl} x^i \nabla^2_{kl} x^j - g^{ij} \big) \varphi \,dg. \label{eq_trgij_weak}
\end{align}
Since \eqref{eq_trgij_weak} only involves zeroth and first derivatives of the metric, this weak characterization also makes sense in the coordinate system $(x^1, \ldots, x^n)$.
Next, note that in this coordinate system, the coefficients $\nabla^2_{kl} x^i$ and $\nabla^2_{kl} x^i$ depend only on the metric coefficients $g_{ij}$ and their first derivatives.
So in the weak sense we have $\triangle g^{ij} = u^{ij}$ for some $u^{ij}$ of regularity $C^{0}$.
Therefore, using standard elliptic theory and Schauder estimates, we obtain that the $g^{ij}$, and therefore the $g_{ij}$, are even of regularity $C^{1,\alpha}$ for some arbitrary $\alpha \in (0,1)$, which we will fix henceforth.
Applying this improved regularity to \eqref{eq_trgij_weak}, implies that the functions $u^{ij}$ are even of regularity $C^{0,\alpha}$ and thus the coefficient functions $g^{ij}$ are of regularity $C^{2,\alpha}$.

Due to the improved regularity, we can now rewrite \eqref{eq_trgij_weak} in its strong form as
\begin{equation} \label{eq_trigijstrong}
 \triangle g^{ij} - \nabla_V g^{ij} = 2 \nabla_{kl}^2 x^i \nabla_{kl} x^j - g^{ij}. 
\end{equation}
Suppose by induction that the $g^{ij}$, and thus the $g_{ij}$, are of regularity $C^{m,\alpha}$ for some $m \geq 2$.
By \eqref{eq_choice_xi_ph}, the coefficients of the vector field $V$ satisfy
\[ V^k = dx^k (V) =    \nabla x^k \cdot V= \triangle x^k = - g^{kl}g^{ij} \bigg( \frac{\partial g_{jl}}{\partial x^i}  - \frac12 \, \frac{\partial g_{ij}}{\partial x^l} \bigg), \] 
so the coefficients $V^k$ have regularity $C^{m-1,\alpha}$.
Schauder estimates applied to \eqref{eq_trigijstrong} then imply that the $g^{ij}$ are even of regularity $C^{m+1,\alpha}$, which finishes the induction, showing that the coefficient functions $g_{ij}$ and $V^k$ are smooth.
\end{proof}
\medskip

The following lemma states that if a smooth Ricci flow $(g_t)_{t \in (0,T]}$ with uniformly bounded curvature converges in the $C^0$-sense, as $t \searrow 0$, to a metric $g_0$ that is isometric to a metric $\ov g$ of regularity $C^k$, then $g_0$ is of regularity $C^k$ and the convergence even occurs in the $C^{k-2}_{\loc}$-sense.

\begin{Lemma} \label{Lem_RF_Ck_convergence}
Let $(M,(g_t)_{t \in (0,T]})$ be a smooth Ricci flow with uniformly bounded curvature on a manifold $M$.
(We do not require the time-slices $(M,g_t)$ to be complete.) 
Suppose that we have local uniform convergence $g_t \to g_0$, as $t \to 0$, where $g_0$ is a $C^0$-metric with the property that the metric space $(M, d_{g_0})$ is locally isometric to the length space of a Riemannian manifold $(\ov M, \ov g)$, where $\ov g$ has regularity $C^k$, $k \geq 3$.
Then $g_0$ is even of regularity $C^{k-2}$ and the convergence $g_t \to g_0$ is in $C^{k-2}_{\loc}$.
\end{Lemma}

See also the related result \cite[Theorem~1.6]{Deruelle_Schulze_Simon_2022}, which holds in the smooth setting under weaker assumptions. 

\begin{proof}
Without loss of generality, we may assume that $M$ is an open subset of $\IR^n$ and that $g_t$ is smooth with respect to the Cartesian coordinates; denote by $(g_t)_{ij}$ the corresponding metric coefficients.
We may also assume that $M$ contains the origin $\vec 0 \in \IR^n$ and it suffices to establish the desired regularity and convergence property in a neighborhood of $\vec 0$.

The strategy of the proof is to gauge the Ricci flow $(g_t)$ into a Ricci DeTurck flow $(g'_t)$, which satisfies a strongly parabolic equation, and must therefore have better regularity.
The main complication will be to ensure the existence of a harmonic map heat flow despite the low regularity of $(g_t)$ at $t=0$.
In the following, we will use the notation and computations from \cite[Appendix~A]{Bamler_Kleiner_uniqueness}.
We will also fix some arbitrary $\alpha \in (0,\frac12)$ and denote by $C$ a generic constant.

\begin{Claim}
The metric coefficients $(g_t)_{ij}$ and the spatial derivatives $\partial (g_t)_{ij}$ are of regularity $C^{\alpha}$ in space-time $M \times [0,T]$.
In particular, the metric coefficients $(g_0)_{ij}$ are of regularity $C^{1,\alpha}$.
\end{Claim}

\begin{proof}
Shi's estimates combined with the uniform curvature assumption give us that for $l \geq 0$
\begin{equation} \label{eq_Shi_bounds}
|{\nabla^{l,g_t} \Rm_{g_t}}| \leq C t^{-l/2}.
\end{equation}
This implies that we have the following bound on the time derivative of the first spatial derivative of the metric coefficients
\[ |\partial_t \partial g_t| 
\leq C |\partial g_t| + C |\nabla {\Ric_{g_t}}|
\leq C |\partial g_t| + Ct^{-1/2}. \]
Integrating this ODI implies convergence $(g_t)_{ij} \to (g_0)_{ij}$ in $C^1_{\loc}$, as well as $C^{\alpha}$-bounds of the spatial derivatives $\partial (g_t)_{ij}$ with respect to the time variable. 

It remains to show that the spatial derivatives converge in the $C^{\alpha}_{\loc}$-sense as $t \to 0$.
To do this, consider two points $p_1,p_2 \in M$ and set $r := |p_1 - p_2|$.
A similar computation, combined with the fact that we have local uniform bounds on the metric coefficients $(g_t)_{ij}$ and their first derivatives $\partial (g_t)_{ij}$ yields
\begin{align}
 \big|\partial_t \big( (\partial g_t)_{p_2} - (\partial g_t)_{p_1} \big) \big|
&\leq  C\big| (\partial g_t)_{p_2} - (\partial g_t)_{p_1} \big| + C\big| {\Ric}_{g_t, p_2} - {\Ric}_{g_t,p_1} \big| \notag \\
&\qquad + C\big| (g_t)_{p_2} - (g_t)_{p_1} \big| 
+ C \big| (\nabla {\Ric_{g_t}})_{p_2} - (\nabla {\Ric_{g_t}})_{p_1} \big| \notag \\
&\leq  C\big| (\partial g_t)_{p_2} - (\partial g_t)_{p_1}  \big| + C t^{-1/2} r + C r + C \min\{ t^{-1/2}, t^{-1} r \}, \label{eq_diff_g_partialt}
\end{align}
where in the last term we have bounded the difference of the curvature terms using \eqref{eq_Shi_bounds} for both $l=1$ and $l=2$.
Since
\[ \min\{ t^{-1/2}, t^{-1} r \} \leq t^{-(1-\alpha)/2 -\alpha}  r^\alpha, \]
we obtain the desired result by integrating \eqref{eq_diff_g_partialt}.
\end{proof}

\begin{Claim}
There is a $C^{2,\alpha}$-regular diffeomorphism $\phi : M \to \ov M$ with $\phi^* \ov g = g$.
\end{Claim}

\begin{proof}
We will show that the local metric isometry $\phi : M \to \ov M$ between the length spaces of $(M, g_0)$ and $(\ov M, \ov g_0)$ has the desired regularity.
Obviously, $\phi$ is locally bilipschitz, so $\phi$ and $\phi^{-1}$ are differentiable almost everywhere and we have $\phi^* \ov g = g$ almost everywhere.
Consider a harmonic function $\ov y \in C^3 (\ov M)$, $\triangle_{\ov g} \ov y = 0$ and set $y := \ov y \circ \phi$.
Then $y$ is locally Lipschitz and $dy = \phi^* d\ov y$ almost everywhere.
Moreover, for any compactly supported test function $h \in C^2_c (M)$ we have
\[ \int_M (dy \cdot_g dh) dg = \int_{\ov M} \big( d\ov y \cdot_{\ov g} d (h \circ \phi^{-1}) \big) d\ov g = 0. \]
So $y$ is weakly harmonic.
So standard elliptic regularity theory implies that $y$ is of regularity $C^{2,\alpha}$ and harmonic.
It follows that if $(\ov y^1, \ldots, \ov y^n)$ are harmonic coordinates of $(\ov M, \ov g)$ of regularity $C^3$, then $(y^1 := \ov y^1 \circ y^1, \ldots, y^n := \ov y^n \circ y^n)$ are harmonic coordinates of $(M,g)$ of regularity $C^{2,\alpha}$.
Moreover, the metrics $g, \ov g$ in these coordinates have the same coefficients almost everywhere and thus everywhere.
Since $\phi$ arises by composing the inverse $(\ov y^1, \ldots, \ov y^n)$ with $(y^1, \ldots, y^n)$, we obtain that $\phi$ has regularity $C^{2,\alpha}$.
\end{proof}

Let us now choose a neighborhood $\vec 0 \subset U \subset M$ and a local diffeomorphism $\chi : U \to \IR^n$ of regularity $C^{2,\alpha}$ such that $\chi = \ov\chi  \circ \phi$ for some smooth $\ov\chi : \ov U \to \IR^n$, where $\ov U \subset \ov M$ is a neighborhood of $\phi (\vec 0)$.
So $\chi_* g = \ov\chi_* \ov g$ is of regularity $C^k$.
After possibly shrinking $U$, we can find a solution $(\chi_t : U \to \IR^n)$ to the harmonic map heat flow equation $\partial_t \chi_t^i = \triangle_{g_t} \chi_t^i$ with initial condition $\chi_0 = \chi$.
By standard parabolic estimates (see, for example, \cite[Lemma~4.3]{Appleton_2018}), and after possibly shrinking $U$ again, we obtain that the solution $\chi_t$, its first and second spatial derivative, as well as its first time derivative are of H\"older regularity on $U \times [0,T]$.
In addition, the family of maps restricted to $U \times (0,T]$ is even smooth.
So, after possibly shrinking $T$, we obtain that there is a neighborhood $\chi_0 (\vec 0) \in V \subset \IR^n$ such that $\chi_t (U) \supset V$ for all $t \in [0,T]$.

We conclude that the family of metrics $(g'_t := (\chi_t)_* g_t)_{t \in [0,T]}$ and its first spatial derivatives are H\"older continuous on $V \times [0,T]$.
Moreover, this family is smooth when restricted to $V \times (0,T]$ and evolves by the Ricci DeTurck equation with respect to the Euclidean background metric on $V$.
Recall that the Ricci DeTurck equation is non-linear, but strongly parabolic.
Since $g'_0$ is of regularity $C^k$, bootstrapping standard parabolic estimates (see, for example, again \cite[Lemma~4.3]{Appleton_2018}) allows us to conclude that $g'_t \to g'_0$ in $C^{k-1}_{\loc}$ and that we have the following spatial bound near $\chi(\vec 0)$:
\[ |\partial^{k} g'_t | \leq C t^{-1/2}. \]
Since the curvature operator is diffeomorphism invariant, these bounds imply that near $\vec 0$ we have for $t \in (0,T]$
\[ |{\Rm_{g_t}}|, \ldots, |\nabla^{k-3,g_t} {\Rm}_{g_t}| \leq C, \qquad
|\nabla^{k-2,g_t} {\Rm}_{g_t}| \leq C t^{-1/2}. \]
These bounds now allow us to establish the desired regularity properties for $(g_t)_{t \in (0,T]}$ near $0$.
To do this, assume by induction that $g_t \to g_0$ in $C^{k'-1}_{\loc}$ near $\vec 0$ for some $2 \leq k' \leq k-2$.
Then we have the following bounds on the $k'$ spatial derivatives of the metric coefficients:
\[ | \partial_t \partial^{k'} g_t | 
\leq C |\partial^{k'} g_t | + C + C |\nabla^{k',g_t} {\Rm}_{g_t}|
\leq C |\partial^{k'} g_t | + C + C t^{-1/2} . \]
Integrating this bound implies that $g_t \to g_0$ in $C^{k'}_{\loc}$ near $\vec 0$, which finishes the induction and hence the proof of the lemma.
\end{proof}

The following lemma allows us to compare $C^k$-norms taken with respect to different background metrics $g_1, g_2$, as long as both metrics are comparable to each other in the $C^k$-norm.

\begin{Lemma} \label{Lem_change_background_metric}
Let $(M,g_i,f_i)$, $i=1,2$, two gradient expanding solitons, where $g_i, f_i$ have regularity $C^2$, on the same $n$-dimensional orbifold $M$ and assume that $\sup_M f_1, \sup_M f_2 <\infty$.
Suppose that for some $A <\infty$ and $k \geq 0$ we have $|{\Rm_{g_1}}|, |{\Rm_{g_2}}| \leq A$ and
\[  A^{-1} g_1 \leq g_2 \leq A g_2, \qquad A^{-1} |\nabla^{g_1} f_1|_{g_1} - A \leq |\nabla^{g_2} f_2|_{g_2} \leq  A |\nabla^{g_1} f_1|_{g_1} + A  \]
and, using the H\"older norms from Subsection~\ref{subsec_list_Holder},
\begin{equation} \label{eq_g1g2_bounds}
 \Vert g_1 - g_2 \Vert_{C^k_{g_1}} \leq A \qquad \text{or} \qquad \Vert g_1 - g_2 \Vert_{C^k_{g_2}} \leq A. 
\end{equation}
Then for any $a \geq 0$ and any tensor field $u \in C^k_{-a,g_1} (M; T^{b_1}_{b_2} M )$ we have
\[ \Vert u \Vert_{C^{k}_{g_2,-a}} \leq C(A,n,k,a,b_1,b_2) \Vert u \Vert_{C^{k}_{g_1, -a}}. \]
\end{Lemma}

\begin{proof}
Without loss of generality, we may assume that $\sup_M f_1 = \sup_M f_2 = -1$ such that
\[ \Vert u \Vert_{C^k_{-a,g_i}} = \Vert (-f_i)^a u \Vert_{C^k_{g_i}}. \]
Due to the curvature bound and the soliton identity $R_{g_i} + |\nabla^{g_i} f_i|^2 = f_i$ we have \begin{equation} \label{eq_f1f2_comparison}
c(A,n) (-f_1) \leq -f_2 \leq C(A,n) (-f_1).
\end{equation}
Next note that it suffices to assume that the first bound in \eqref{eq_g1g2_bounds} holds, because if the lemma is true assuming the first bound, then the first bound follows from the second one after switching the roles of $g_1, g_2$ and setting $u = g_1 - g_2$ and $a = 0$.
So assume in the following that the first bound of \eqref{eq_g1g2_bounds} holds.
By Lemma~\ref{Lem_soliton_smooth}, we may moreover switch the smooth structure and assume in the following that $g_1, f_1$ are smooth.
Due to \eqref{eq_g1g2_bounds} and the soliton equation for $g_2, f_2$, this implies that $g_2, f_2$ are of regularity $C^k$. 

In the following, $C$ will denote a generic constant that only depends on $A,n,k,a,b_1,b_2$.
We will show the lemma by induction.
The case $k = 0$ follows from the bilipschitz bound and \eqref{eq_f1f2_comparison}.
Now suppose that the lemma is true for some $k \geq 0$ and that \eqref{eq_g1g2_bounds} holds for $k$ replaced with $k+1$.
Then we compute via Lemma~\ref{Lem_Holder_norms_properties} that
\begin{align*}
 \Vert u \Vert_{C^{k+1}_{g_2,-a}} 
&\leq \Vert \nabla^{g_2} u \Vert_{C^{k}_{g_2,-a}} + \Vert u \Vert_{C^{k}_{g_2,-a}} \\
&\leq \Vert \nabla^{g_1} u + g_2^{-1} * \nabla^{g_1} g_2 * u \Vert_{C^{k}_{g_1,-a}} + \Vert u \Vert_{C^{k}_{g_1,-a}} \\
&\leq C \Vert \nabla^{g_1} u \Vert_{C^{k}_{g_1,-a}} + C \Vert g_2 - g_1 \Vert_{C^{k+1}_{g_1}} *  \Vert  u \Vert_{C^{k}_{g_2,-a}}  + C \Vert  u \Vert_{C^{k}_{g_2,-a}} \\
&\leq C \Vert \nabla^{g_1} u \Vert_{C^{k}_{-a,g_1}}  + C \Vert  u \Vert_{C^{k}_{g_2,-a}} \leq  C\Vert u \Vert_{C^{k+1}_{g_2,-a}}.
\end{align*}
So the lemma also holds for $k+1$, which finishes the induction.
\end{proof}
\medskip

The next lemma states that an almost isometry $\chi$ (in the $C^1$-sense) that is close to the identity (in the $C^0$-sense) has differentials that are close to the identity (in the $C^0$-sense).

\begin{Lemma} \label{Lem_dchi_small}
Let $M$ be an $n$-dimensional Riemannian orbifold with isolated singularities, $g$ a complete Riemannian metric on $M$ of regularity $C^3$ and $p_0 \in M$, $r := d(p_0,\cdot)$.
Consider a $C^2$-diffeomorphism $\chi : M \to M$ such that for $m = 0,1$ and some $A < \infty$
\[ |{\Rm}| \leq Ar^{-2}, \qquad \inj \geq A^{-1} r - 1 \qquad |\nabla^m (\chi^* g - g)| \leq A (r+1)^{-2}, \qquad \sup_{p \in M} d_g(\chi(p),p) \leq A. \]
Then via parallel transport along minimizing geodesics we have for any $p \in M$
\begin{equation} \label{eq_dchi_id}
   |d\chi_p - \id_{T_p M} | \leq C(A,n) r^{-1}(p). 
\end{equation}
\end{Lemma}

\begin{proof}
Let $C,c$ be generic constants, which may only depend on $A, n$.

\begin{Claim}
For any $p \in M$ with $r(p) \geq 10A$ and $v \in T_p M$ with $|v| \leq \frac12 r(p)$ we have
\[d_g \big( \exp_p(v), \chi^{-1} (\exp_{\chi(p)}(d\chi_p(v))) \big)  \leq C r^{-2}(p) |v|^2 \]
\end{Claim}

\begin{proof}
Let $l := |v| \leq \frac12 r(p)$ and $v_0 := l^{-1} v$.
The curve $\gamma : [0,l] \to M$ with $$\gamma(t) := \chi^{-1} (\exp_{\chi(p)} (d\chi_p(t v_0)))$$ satisfies $\gamma(0) = p$, $\gamma'(0) = v$ and is a unit speed geodesic with respect to $\chi^* g$.
Therefore, its second derivative with respect to $g$ is bounded by $C r^{-2}(p)$.
Consider now the curve
\[ \sigma(t) :=  \exp_{\gamma( l - t)}(t \gamma'(l-t)). \]
and note that $\sigma(0) = \gamma(l)$ and $\sigma(l) = \exp_p(v)$.
By Rauch comparison
\[ |\sigma'(t)| = \bigg| \frac{d}{ds}\Big|_{s=t-l}\exp_{\gamma( l - t)}(t \gamma'(s)) \bigg|
\leq C t r^{-2} (p) \leq  C l r^{-2} (p). \]
So $d_g (\sigma(0), \sigma(l)) \leq Cl^2 r^{-2}(p)$, which proves the claim.
\end{proof}

Since $\chi^* g - g$ is bounded, the norm of the differential $d\chi_p$ and hence the left-hand side of \eqref{eq_dchi_id} are uniformly bounded.
So it suffices to assume that $r(p)$ is sufficiently large such that $\inj(p) \geq (2A)^{-1} r(p)$ and therefore the ball $B(p, (2A)^{-1} r(p))$ only consists of regular point.
Choose $0 < c \leq (2A)^{-1}$ such that $\sec \leq \frac{\pi}2 (cr(p))^{-2}$ on $B(p,cr(p))$.
Applying the Claim and the last assumption of the lemma for some $v \in T_p M$ with $|v| = c r(p)$ yields
\begin{multline}  \label{eq_dexpexpchi}
   d_g( \exp_p(v), \exp_{\chi(p)}(d\chi_p(v)))  \\
\leq d_g\big( \exp_p(v), \chi^{-1} (\exp_{\chi(p)}(d\chi_p(v))) \big) 
+ d_g \big( \chi^{-1} (\exp_{\chi(p)}(d\chi_p(v))) , \exp_{\chi(p)}(d\chi_p(v)) \big) \\
\leq C r^{-2}(p) |v|^2 + A \leq C. 
\end{multline}
Let now $v' \in T_p M$ be the result of parallel transporting $d\chi_p(v)$ along a minimizing geodesic between $\chi(p)$ and $p$.
Using comparison geometry, we obtain that
\[ d_g \big( \exp_p (v'), \exp_{\chi(p)} (d\chi_p(v)) \big) \leq C. \]
Combining this with \eqref{eq_dexpexpchi} yields
\[ d_g \big(  \exp_p (v') , \exp_p(v)\big)
\leq d_g \big( \exp_p(v), \exp_{\chi(p)}(d\chi_p(v)) \big) +
d_g \big( \exp_{\chi(p)} (d\chi_p(v)) , \exp_p (v')\big) \leq C. \]
The desired bound now follows from Alexandrov comparison since $|v| = c r(p)$.
\end{proof}
\medskip

The last lemma provides a lower bound on the $\mu$-functional for a complete Riemannian orbifold of bounded geometry.

\begin{Lemma} \label{Lem_mu_bounded_geometry}
Let $(M^n,g)$ be a complete, $n$-dimensional Riemannian orbifold, where we assume that $g$ has regularity $C^2$.
Suppose that we have the curvature bounds $|{\Rm}| \leq A$, $R \geq R_{\min}$ and the volume bound $\inf_{p \in M} |B(p,r)|\geq A^{-1}r^n$ for some $r > 0$, $A < \infty$ and $R_{\min} \in \IR$.
Then $$\mu[g,\tau] \geq -C(n,A,r) + R_{\min} \tau - \frac{n}2 (\log \tau)_+.$$
Here the $\mu$-functional is defined via functions $f : M \to (-\infty,\infty]$ such that $e^{-f}$ is smooth and has compact support; see the discussion in Subsection~\ref{subsec_soliton_noncollapsing}.
\end{Lemma}

\begin{proof}
By volume comparison and after possibly adjusting $A$, we may assume in the following that $r = 1$.
Moreover, since $\mu[g,\tau] = \mu[\tau^{-1} g, 1]$ and since the other geometric bounds remain conserved after rescaling with a constant $\geq 1$, we may assume in the following that $\tau \geq 1$.
Fix $e^{-f} \in C^\infty_c(M)$ such that $\int_M (4\pi \tau)^{-n/2} e^{-f} dg = 1$ and set $u^2 := (4\pi \tau)^{-n/2} e^{-f}$.
Then
\begin{multline} \label{eq_W_bound_logsob}
 \WW[g,f,\tau] \geq \int_M \big( \tau |\nabla f|^2 + f \big) (4\pi \tau)^{-n/2}e^{-f} dg + R_{\min} \tau - n\\
 \geq \int_M \big(4\tau |\nabla u|^2 - (\log u^2) u^2 \big)  dg  - C(n,A)+ R_{\min} \tau - \frac{n}2 \log \tau.
\end{multline}
The last integral in \eqref{eq_W_bound_logsob} can be bounded as follows using the Sobolev inequality \cite{Aubin_1976} and H\"older inequality  for some arbitrary $\alpha \in (0,1)$ with $1+\alpha < \frac{n}{n-2}$ if $n \geq 3$ and the fact that $\log u^2 \leq \frac2\alpha u^\alpha$:
\begin{align*}
 \int_M \big(4\tau |\nabla u|^2 &- (\log u^2) u^2 \big)  dg
\geq c(n,A) \bigg( \int_M u^{2+ 2\alpha} dg \bigg)^{1/(1 + \alpha)} - C(n,A) \int_M u^2 dg - \frac2\alpha \int_M u^{2+ \alpha} dg \\
&\geq c(n,A) \bigg( \int_M u^{2+ \alpha} dg \bigg)^{2/(1 + \alpha)}\bigg( \int_M u^{2} dg \bigg)^{-1/(1 + \alpha)} - C(n,A) -\frac2\alpha\int_M u^{2+ \alpha} dg \\
&= c(n,A) \bigg( \int_M u^{2+ \alpha} dg \bigg)^{2/(1 + \alpha)} - C(n,A) -\frac2\alpha \int_M u^{2+ \alpha} dg \geq - C(n,A,\alpha).
\end{align*}
Note that for the last inequality we have used the fact that the exponent $2/(1+\alpha) > 1$.
This finishes the proof of the lemma.
\end{proof}
\medskip

\section{Real-analytic varieties} \label{appx_real_analytic}
In the following proposition we show that real-analytic varieties are locally path connected.
Recall that a topological space $X$ is locally path connected if every neighborhood of any point contains an open path connected neighborhood of the same point.
This property implies that $X$ can be decomposed into the disjoint union of connected components, which are also path connected.
Proposition~\ref{Prop_variety_loc_conn} is needed in Subsection~\ref{subsec_summary_main} to show that the preimage $\Pi^{-1}(\CONE^{k^*}(N)) \subset \MM^{k^*}(M,N,\iota)$ is locally path connected near points where a Banach manifold structure exists.
This allows us to show, using Proposition~\ref{Prop_gradientness}, that the subset $\MMgrad^{k^*}(M,N,\iota) \subset \Pi^{-1}(\CONE^{k^*}(N))$ is a union of connected components.

\begin{Proposition} \label{Prop_variety_loc_conn}
Let $U \subset \IR^n$ be open and $F : U \to \IR^k$ be real-analytic.
Then $V := F^{-1} (0)$ is locally path connected.
Moreover, for any two points $p,q \in V$ in the same path component of $V$ there is a piecewise real-analytic (and thus piecewise smooth) arc $\gamma : [0,1] \to V \subset \IR^n$ connecting $p$ and $q$.
\end{Proposition}

\begin{proof}
This can be seen in several ways.
For a direct reference of this statement see \cite[Lemma~5.2]{Wallace_1960} or \cite[Proposition~2]{Whitney_Bruhat_1959}.
Another possibility is to construct a resolution $p : \td V \to V$, where $\td V$ is a real-analytic manifold and $p$ is real-analytic, via Hironaka's Theorem in the real-analytic setting \cite{Hironaka_1964, Bierstone_Milman_1989, Krantz_Steven_2002}; $\td V$ is clearly locally path connected in the real-analytic sense.

For a more elementary treatment, see \cite[Theorem~2.6]{Bierstone_Milman_1988}, which implies the first statement of the proposition.
The second statement of the proposition follows from a closer inspection of the proof.
More precisely, in each step of the induction the sets $A \subset \IR^n$ are characterized as follows.
There is a semianalytic set $B \subset \IR^{n-1}$ and two semi-analytic functions $\xi^-, \xi^+ : B \to \IR$, which either agree everywhere or satisfy $\xi^- < \xi^+$ such that
\begin{multline*}
 A = \big\{ (x_1, \ldots, x_{m-1}, y) \in B \times \IR \;\; : \;\; \xi^-(x_1, \ldots, x_{m-1}) = y = \xi^+(x_1, \ldots, x_{m-1})  \quad \text{or} \\
 \xi^-(x_1, \ldots, x_{m-1}) < y < \xi^+(x_1, \ldots, x_{m-1}) \big\},
\end{multline*}
depending on the case.
It suffices to show that, given a real-analytic arc $\gamma : [0,1] \to \ov B$ with $\gamma((0,1]) \subset B$, there is a lift $\td\gamma : [0,1] \to \ov A$ whose first $m-1$ coordinates agree with $\gamma$.
Suppose first that we are in the case $\xi^- = \xi^+ =: \xi$.
Then the last coordinate of $\td\gamma$ must be $\xi (\gamma(t))$.
A closer look at the proof of \cite[Theorem~2.6]{Bierstone_Milman_1988} implies that there is a real-analytic function $g(x_1, \ldots, x_m; y)$, which is a polynomial of bounded degree in $y$ such that $\xi(x_1, \ldots, x_{m-1})$ is a simple zero for all $(x_1, \ldots, x_{m-1} ) \in B$.
Therefore, the coefficients of the polynomial $P_t(y) := g(\gamma(t); y)$ are real-analytic in $t$ and $\xi(\gamma(t))$ is a simple zero for all $(0,1]$.
Lemma~\ref{Lem_branched_cover} below then implies that $t \mapsto \xi(\gamma(t^l))$ is real-analytic on $[0,1]$ for some integer $l$.
In the case in which $\xi^- < \xi^+$ we can argue similarly, using the fact that $t \mapsto \xi^\pm (\gamma(t^l))$ is real-analytic on $[0,1]$ for some integer $l$.
\end{proof}

\begin{Lemma} \label{Lem_branched_cover}
Consider a family of polynomials $P_t(x) = x^k + a_{k-1}(t)x^{k-1} + \ldots + a_0(t)$, $t \in [0,1]$, whose coefficients $a_i(t)$ are real-analytic in $t$.
Let $\xi(t) \in \IC$, $t \in [0,1]$ be a continuous family of zeros for $P_t(x)$.
Then there is an $l \in \{ 1, \ldots, k \}$ such that the map $[0,1) \to \IR$, $t \mapsto \xi(t^l)$ is real-analytic near $0$.
Moreover, if $\xi(t)$ is simple for all $t \in (0,1]$, then $t \mapsto \xi(t^l)$ is real-analytic on $[0,1]$.
\end{Lemma}

\begin{proof}
The goal is to establish real-analyticity at $t = 0$.
We argue by induction on $k$.
The lemma is obvious if $k = 1$, so suppose that $k \geq 2$.

Let $B_r := B(0,r) \subset \IC$ and denote by $\dot B_r := B_r \setminus \{ 0 \}$ the punctured $r$-neighborhood of the origin.
Consider the holomorphic continuations of $a_{k-1}(t), \ldots, a_0 (t)$ and $P_t(x)$ for $t \in B_\eps$ for some $\eps > 0$.

We first reduce the lemma to the case in which $P_t(x)$ has only simple zeros for all $t \in \dot B_\eps $, after possibly shrinking $\eps$.
To do this, we determine a non-trivial common divisor of $P_t(x)$ and $P'_t(x)$ using Euclid's algorithm.
Set $P^0_t(x) := P_t(x)$ and $P^1_t(x) := P'_t(x)$.
Note that the leading coefficients of $P^0_t(x), P^1_t(x)$ are non-zero for any $t \in [0,1]$.
We will construct a sequence of polynomials $P^2_t(x), P^3_t(x), \ldots$, defined for $t \in \IC \setminus \{ 0 \}$ near $0$ as follows.
Let $i \geq 0$ and suppose that $P^0_t(x), \ldots, P^{i+1}_t(x)$ have already been constructed and that the leading coefficient $a_{k_{i+1}}^{(i+1)}(t)$ of $P^{i+1}_t(x)$ is non-zero for all $t \in \dot B_{\eps_{i+1}} $ for some $0 < \eps_{i+1} < \eps$.
This allows us to compute the remainder
\[ P^i_t(x) = P^{i+1}_t(x) H^i_t(x) + P^{i+2}_t (x), \]
where the degree of $P^{i+2}_t (x)$ is strictly smaller than $k_{i+1}$ for all $t \in \dot B_{\eps_{i+1}} $.
Suppose that $P^{i+2}_t(x)$ is non-zero for for some $t \in \dot B_{\eps_{i+1}} $ and let $k_{i+2}$ be the maximum of the degree of $P^{i+2}_t (x)$  for all $t \in \dot B_{\eps_{i+1}} $; so $k_{i+2} < k_{i+1}$. 
Consider the leading coefficient $a_{k_{i+2}}^{(i+2)}(t)$ of $P^{i+2}_t(x)$, which is a non-zero meromorphic function on $ B_{\eps_i}$.
So if $a_{k_{i+2}}^{(i+2)}(t_j) = 0$ for some sequence  $t_j \to 0$, $t_j \in \dot B_{\eps_{i+1}}$, then $a_{k_{i+2}}^{(i+2)} \equiv 0$ on $\dot B_{\eps_{i+1}}$, a contradiction.
So $a_{k_{i+2}}^{(i+2)}(t) \neq 0$ for all $t \in \dot B_{\eps_{i+2}}$ for some $0 < \eps_{i+2} < \eps_{i+1}$
and we can continue the inductive process.
Our process terminates once $P^{i+2}_t(x) \equiv 0$ for all $t \in \dot B_{\eps_{i+1}}$.
Then the polynomial $P^{i+1}_t(x)$ has constant degree and is the largest common divisor of $P_t(x), P'_t(x)$ for all $t \in \dot B_{\eps_{i+1}}$.
This implies that one of the following two cases occurs:
Either $P_t(x)$ only has simple zeros for all $t \in \dot B_{\eps_{i+1}}$ or a multiple zero for all $t \in \dot B_{\eps_{i+1}}$.
Consider the latter case.
Let $P^{**}_t(x)$, $t \in \dot B_{\eps_{i+1}}$, be the normalization of $P^{i+1}_t(x)$.
Then we can write for all $t \in \dot B_{\eps_{i+1}}$
\[ P_t (x) = P^*_t (x) P^{**}_t(x), \]
where $P^*_t(\xi(t)) = 0$, both polynomials are normalized and the degree of $P^*_t(x)$ has strictly smaller degree than $P_t(x)$.
The coefficient functions of $P^*_t(x)$ are holomorphic on $\dot B_{\eps_{i+1}}$.
Next note, that since $P_t(x)$ is normalized, its zeros are uniformly bounded for $t \in B_\eps$.
So since the zeros of $P^*_t(x)$ are a subset of the zeros of $P_t(x)$, they must also be uniformly bounded and therefore the coefficient functions of $P^*_t(x)$ must be uniformly bounded as well.
So these functions can be extended to holomorphic functions on $B_{\eps_{i+1}}$ and the lemma follows by applying the induction hypothesis to $P^*_t(x)$.

We may therefore assume (after possibly reparameterizing $t$) that all zeros of $P_t(x)$ are simple for all $t \in \dot B_1$.
Then the projection
\[ \pi : \{ (t,x) \in \dot B_1 \times \IC \;\; : \;\; P_t(x) = 0 \} \to \dot B_1, \qquad (t,x) \mapsto t \]
is a covering of degree $l \leq k$ and we can find a lift $h : \dot B_1 \to \IC$ of $z \mapsto z^l$ with $P_{z^l} (h(z)) = 0$ and $h(t) = \xi(t^l)$ for $t \in (0,1)$.
Since $h$ is bounded, it can be extended to a holomorphic $\ov h : B_1 \to \IC$ with the same property, which finishes the proof.
\end{proof}
\bigskip

\bibliography{references}	
\bibliographystyle{amsalpha}

\end{document}